\pdfoutput=1

\documentclass[a4paper,abstract=true,toc=flat,fontsize=12]{scrreprt}

\usepackage{cylindrical_model_structures}

\begin{document}

\title{Cylindrical Model Structures}
\author{Richard Williamson}
\date{}

\dedication{Dedicated to Kari Chard}

\maketitle

\begin{abstract} \noindent We build a model structure from the simple point of departure of a structured interval in a monoidal category --- more generally, a structured cylinder and a structured co-cylinder in a category. \end{abstract}

\tableofcontents

\begin{chapter}{Introduction} \label{IntroductionChapter}

Model structures are cherished --- powerful, but hard to construct. In this work we build a model structure from the simple point of departure of a structured interval in a monoidal category --- more generally, a structured cylinder and a structured co-cylinder in a category.

\section*{Abstract homotopy theory} \label{PlaceWithinAbstractHomotopyTheorySubsection} \label{PlaceWithinAbstractHomotopySubsection} The first steps towards an abstract homotopy theory were, for us, taken in the early 1950s by Kan --- though there is a considerable prehistory, for example in the work of Whitehead. Kan isolated in \cite{KanAbstractHomotopyTheoryII} the notion of a cylinder in a category and of homotopy with respect to a cylinder. 

Kan's immediate interest lay in the categories of cubical and simplicial sets. Whilst both categories admit a cylinder, the corresponding homotopies cannot be composed or reversed. What are now known as Kan complexes were introduced as a remedy. 

Abstract homotopy theory has subsequently evolved in two branches, which have been explored rather independently. The first is the theory of model categories and its many variants and weakenings. The second is much less known. Its origins lie in the observation that the cylinder in topological spaces admits a much richer structure than the cylinder of the categories of simplicial or cubical sets. Two directions have been followed in capturing this richer structure of the topological cylinder in an abstract setting. 

\begin{figure}[h] \label{FigureAbstractHomotopyTheory}

\centering

\begin{tikzpicture} [>=stealth]

\node[align=center,text width=2cm,anchor=north] (kan) at (4,0) {Kan \\[0.5em]  1956 \\[0.5em] Cylinder in a category};

\node[align=center,text width=3cm,anchor=north,inner ysep=1em] (quillen) at (0,-5) {Quillen \\[0.5em]  1967 \\[0.5em] Model categories};

\node[align=center,text width=2cm,anchor=north,inner ysep=1em] (grandis) at (5,-5) {Grandis \\[0.5em]  1990s \\[0.5em] Structured cylinders};

\node[align=center,text width=2.5cm,anchor=north,inner ysep=1em] (kamps) at (9,-5) {Kamps \\[0.5em]  Late 1960s \\[0.5em] Kan conditions};

\draw[bend right=45] (kan.west) to (quillen.north);
\draw[bend left=45] (kan.east) to (7,-3);
\draw[bend right=45] (7,-3) to (grandis.north);
\draw[bend left=45] (7,-3) to (kamps.north);
\draw[<-,dashed] (quillen.east) to node[auto] {our work} ++ (2.1,0);
\end{tikzpicture}

\caption*{Figure: Approaches to abstract homotopy theory}

\end{figure}
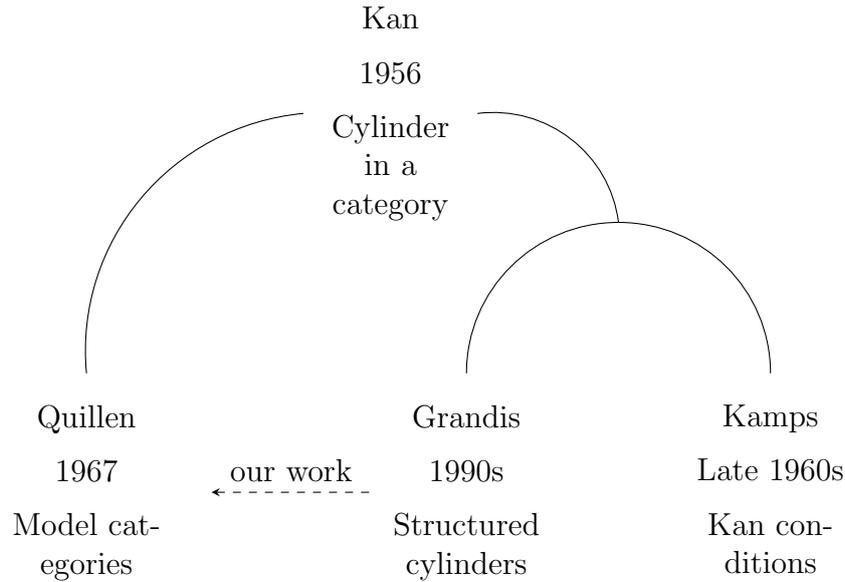

The first, begun by Kamps in the late 1960s in works such as \cite{KampsKanBedingungenUndAbstrakteHomotopietheorie}, explores the homotopy theory with respect to a cylinder whose associated cubical set satisfies properties similar to those of a Kan complex. This is a global approach, with the axioms requiring consideration of all arrows of a category. 

The second emerges out of works of Brown, Higgins, and others on cubical sets with connections, for example the paper \cite{BrownHigginsOnTheAlgebraOfCubes}. It is of a structural and categorical nature, involving a rainbow of natural transformations intertwining a cylinder and its corresponding double cylinder. This approach has been explored by Grandis in works such as \cite{GrandisCategoricallyAlgebraicFoundationsForHomotopicalAlgebra}. 

The book \cite{KampsPorterAbstractHomotopyAndSimpleHomotopyTheory} of Kamps and Porter gives a nice overview of both directions, with an emphasis on the first.

\section*{Outline}

The present work builds a bridge to model categories from the categorical approach to homotopy theory via structured intervals, cylinders, and co-cylinders. The theorems towards which all of our work leads are to be found in \ref{ModelStructureChapter}. Given a richly structured cylinder and co-cylinder in a category satisfying a certain strictness hypothesis, we prove that homotopy equivalences, cofibrations, and fibrations --- defined from an abstract point of view in the same way as for topological spaces --- equip this category with a model structure. 

More precisely, our theory typically gives rise to not one, but two model structures. We introduce in \ref{CofibrationsAndFibrationsChapter} a notion of a normally cloven cofibration, and of a normally cloven fibration. One of our model structures is defined by homotopy equivalences, cofibrations, and normally cloven fibrations, whilst the other is defined by homotopy equivalences, normally cloven cofibrations, and fibrations. 

The structures on a cylinder and a co-cylinder with which we work are defined in \ref{StructuresUponACylinderOrACoCylinderChapter}, along with our strictness hypothesis. Often, in practise, we construct a structured cylinder and a structured co-cylinder by means of a structured interval in a monoidal category. This is discussed in \ref{StructuresUponAnIntervalChapter}. 

We work in a 2-categorical setting, introduced in \ref{FormalCategoryTheoryPreliminariesChapter}, which allows us to express a duality between homotopy theory with respect to a cylinder on the one hand, and homotopy theory with respect to a co-cylinder on the other. This duality manifests itself throughout. 

In \ref{CylindricalAdjunctionsChapter}, we discuss a notion of adjunction between a cylinder and a co-cylinder, in the presence of which the corresponding homotopy theories coincide. For the remainder of this outline, we shall indicate neither the particular structures involved at different points, nor whether we are working with a cylinder, a co-cylinder, or both. These matters are carefully treated in the remainder of the work.

In \ref{HomotopyAndRelativeHomotopyChapter}, we define homotopies and relative homotopies with respect to a cylinder or co-cylinder. We demonstrate that we can compose and reverse them.

If our strictness hypothesis holds, we prove in \ref{MappingCylindersAndMappingCoCylindersChapter} that the mapping cylinder of a map gives rise to a factorisation of it into a normally cloven cofibration followed by a strong deformation retraction, and that the mapping co-cylinder of a map gives rise to a factorisation of it into a section of strong deformation retraction followed by a normally cloven fibration.

In \ref{DoldTheoremChapter}, we characterise trivial fibrations as strong deformation retractions, and characterise trivial cofibrations as sections of strong deformation retractions. This is by means of an abstraction of Dold's theorem for topological spaces, on homotopy equivalences under or over an object. 

Assuming once more that our strictness hypothesis holds, we prove in \ref{FactorisationAxiomsChapter} that a trivial cofibration is exactly a section of a strong deformation retraction, and dually that a trivial fibration is exactly a strong deformation retraction. With our mapping cylinder and mapping co-cylinder factorisations to hand, we deduce that the factorisation axioms for a model structure hold.

In \ref{LiftingAxiomsChapter}, we prove that the canonical map from the mapping cylinder of a map to the cylinder at its target admits a strong deformation retraction. We prove that normally cloven fibrations have the right lifting property with respect to sections of strong deformation retractions. We deduce that normally cloven fibrations have the covering homotopy extension property --- introduced in \ref{CoveringHomotopyExtensionPropertyChapter} --- with respect to cofibrations. 

This allows us to prove that cofibrations have the left lifting property with respect to trivial normally cloven fibrations. We conclude that the lifting axioms hold for one of our model structures. The lifting axioms for the other model structure follow by duality.  

\section*{Folk model structure} \label{ExamplesFutureDirectionsSubsection} This work was conceived as a step towards the construction of a model category of $n$-groupoids satisfying, in a strong sense, the homotopy hypothesis. The aim was to construct the folk model structure on categories and groupoids in a way which we could generalise to $n$-groupoids.

We demonstrate in \ref{ExampleChapter} that our work indeed gives a new construction of the folk model structure on categories and groupoids. The construction of a model category of $n$-groupoids by means of the present work is the point of departure of joint work with Marius Thaule which is in preparation.

\section*{Further examples} Our work gives rise to many other model structures. We discuss three. 

\begin{itemize}[itemsep=1em,topsep=1em]

\item[(1)] Let $\mathsf{Ch}(\cl{A})$ denote the category of chain complexes in an additive category $\cl{A}$ with finite limits and colimits. The cylinder and co-cylinder functors \ar{\mathsf{Ch}(\cl{A}),\mathsf{Ch}(\cl{A})} of homological algebra can be equipped with all the structures of \ref{StructuresUponACylinderOrACoCylinderChapter}. We refer the reader to \S{4.4.2} of the book \cite{GrandisDirectedAlgebraicTopology} of Grandis, for example. Our strictness hypothesis is satisfied. 

Our work thus gives a model structure on $\mathsf{Ch}(\cl{A})$ whose weak equivalences are chain homotopy equivalences. This model structure was constructed in a quite different way by Golasi{\'n}ski and Gromadzki in \cite{GolasinskiGromadzkiTheHomotopyCategoryOfChainComplexesIsAChainComplex}, appealing to a characterisation due to Kamps in \cite{KampsNoteOnNormalSequencesOfChainComplexes} of the fibrations and cofibrations.

\item[(2)] Let $\mathsf{Kan}_{\Delta}$ denote the category of algebraic Kan complexes introduced by Nikolaus in \cite{NikolausAlgebraicModelsForHigherCategories}. The objects of $\mathsf{Kan}_{\Delta}$ are Kan complexes with a chosen filling for every horn. The arrows of $\mathsf{Kan}_{\Delta}$ are morphisms of simplicial sets which respect the chosen horn fillings.

Our work gives a model structure on $\mathsf{Kan}_{\Delta}$, which we think of as akin to the model structure on topological spaces constructed by Str{\o}m in \cite{StromTheHomotopyCategoryIsAHomotopyCategory}. A different model structure on $\mathsf{Kan}_{\Delta}$ was constructed by Nikolaus in \cite{NikolausAlgebraicModelsForHigherCategories}, which we think of as akin to the Serre model structure on topological spaces.

The author conjectures that the identity functor defines a Quillen equivalence between these two model structures on $\mathsf{Kan}_{\Delta}$.

\item[(3)] Let $\mathsf{Top}$ denote the category of all topological spaces. The unit interval is exponentiable with respect to the cartesian monoidal structure on $\mathsf{Top}$ and can be equipped with all of the structures of \ref{StructuresUponAnIntervalChapter}. 

Our strictness hypothesis does not however hold.  Thus our work does not immediately give rise to a model structure on $\mathsf{Top}$. 

Nevertheless, homotopy equivalences in $\mathsf{Top}$ can be understood as homotopy equivalences with respect to the Moore co-cylinder, which to a topological space $X$ associates the set of pairs $(t,f)$ of a real number $t \in [0,\infty)$ and a map \ar{{[0,\infty)},X,f} such that $f(x) =t$ for all $x \geq t$, where this set is equipped with the subspace topology with respect to $[0,\infty) \times X^{[0,\infty)}$. Our strictness hypothesis does hold for the Moore co-cylinder.

The Moore co-cylinder does not however admit connection structures. There are two ways to get around this. Firstly, it is possible to generalise our work slightly to double cylinders and double co-cylinders which are not necessarily obtained by applying the cylinder or co-cylinder functor twice. We can then take our Moore double co-cylinder to consist of Moore rectangles as considered in the paper \cite{BrownMooreHyperrectanglesOnASpaceFormAStrictCubicalOmegaCategory} of Brown. Secondly, it should be possible to replace the connection structures in our work by the `strengths' of the paper \cite{VanDenBergGarnerTopologicalAndSimplicialModelsOfIdentityTypes} of van den Berg and Garner.

Following either route, we can obtain the model structure on $\mathsf{Top}$ constructed by Str{\o}m in \cite{StromTheHomotopyCategoryIsAHomotopyCategory} by working with respect to both the Moore co-cylinder and the usual cylinder and co-cylinder in $\mathsf{Top}$.

Around Easter 2012, Tobias Barthel and Bill Richter suggested to the author a construction of the Moore co-cylinder which could be carried out in a quite general setting. Thus our side-stepping of the failure of the strictness hypothesis to hold in $\mathsf{Top}$ may be able to be carried out more widely.   

The significance of the Moore co-cylinder for the construction of the Str{\o}m model structure on $\mathsf{Top}$ is explored in the paper \cite{BarthelRiehlOnTheConstructionOfFunctorialFactorizationsForModelCategories} of Barthel and Riehl.

\end{itemize}

\section*{Acknowledgements} I thank Kari, the wonderful funny bee to whom this work is dedicated, for everything. It is impossible to express how much your lifting of my spirit --- cakes of course playing a crucial role! --- has contributed to this work.

I thank Mum and Dad for everything over the years. I cannot express my gratitude. I would specifically like to thank you for your unwavering encouragement and belief throughout the writing of this work --- I am very lucky to have been able to phone up babbling incomprehensibly about mathematics!

I have the grandest pair of brothers in the world, Chris and Andrew. I thank you both for all the many laughter filled times --- for always being there for me.

I thank Babcia, Dziadziu, Grandma, and, beyond the grave, Grandad for everything they have done for me. I feel deeply that this work owes a great deal to your hard work throughout your lives.

To Alex, Tom, Richy, Shaun, Sian, Phil, and Katherine --- I count myself incredibly lucky to have got to know you all, and that we have such a close knit friendship. You all bring great joy to my life!

To Merav, Daniel, Nimrod, Henning, and Anne -- wonderful dinners, black and white films, and the fun brought into our lives by little Nimrod gave Kari and me great happiness in our time in Oxford, and I shall always look back with an ineffable fondness. I thank Henning for the very kind help with the final printing of the thesis, and for handing it over to the Bodleian library --- much appreciated!

I thank Rapha{\"e}l Rouquier greatly for his supervision over my four years in Oxford. That every turn in my path was greeted with enthusiasm, wisdom, and sincerity was a great encouragement to me.

I thank Tobias Barthel for many enjoyable discussions. I treasure our correspondence, and our sharing of mathematical dreams!   

I thank Nicholas Cooney for many pleasant conversations over coffee and in the pub, and for the great help with the submission of this work as my thesis! 

I thank Ulrike Tillmann for her advice and understanding at a crucial point in my time as a graduate student, and for her continuing encouragement.

I thank Michael Collins for his support and encouragement in the months before and during my time at Oxford, and for his insight in suggesting Rapha{\"e}l as a supervisor.  

To all at Oxenford Cricket Club --- especially my 2nd XI teammates --- for three years of immensely enjoyable summer Saturdays.

I thank Tim Porter and Kobi Kremnizer for their examination of my thesis, for their encouragement, and for their very helpful comments and suggestions. 

I thank Nils Baas greatly for his kindness in inviting us and helping us to settle down in Trondheim. I thank Nils and Andrew Stacey for their support in my finding a position at NTNU.

I thank Bj{\o}rn Ian Dundas for arranging the funding of my employment at NTNU as a guest researcher from January to March 2012. 

Finally to all our friends and family in Sandefjord, T{\o}nsberg, Larvik, and Trondheim for their help, well wishes, and understanding since we moved to Norway in August 2011 --- and of course for the dinners, waffles, and cakes\ldots a recurring theme!

\end{chapter}

\begin{chapter}{Formal category theory preliminaries} \label{FormalCategoryTheoryPreliminariesChapter}

In \ref{StructuresUponACylinderOrACoCylinderChapter}, we introduce various structures with which a cylinder or co-cylinder may be able to be equipped. The structures upon a co-cylinder are formally dual to those upon a cylinder. This duality, which is of a 2-categorical rather than a 1-categorical nature, as was already observed by Gray in \cite{GrayFibredAndCofibredCategories}, will manifest itself throughout this work. 

In order to express the duality we shall work throughout in a strict 2-category equipped with a strict final object. We think of the objects of this 2-category as formal categories. We now introduce these ideas as far as we shall need.

\begin{assum} Let $\cl{C}$ be a strict 2-category, and let $1$ be a final object of $\cl{C}$. \end{assum} 

\begin{defn} Let $\cl{A}$ be an object of $\cl{C}$. An {\em object} of $\cl{A}$ is a 1-arrow \ar{1,\cl{A}} of $\cl{C}$. \end{defn}

\begin{defn} Let $\cl{A}$ be an object of $\cl{C}$. Given objects $a_{0}$ and $a_{1}$ of $\cl{A}$, an {\em arrow} of $\cl{A}$ from $a_{0}$ to $a_{1}$ is  a 2-arrow \ar{a_{0},a_{1},} of $\cl{C}$. \end{defn}

\begin{rmk} Let $\underline{\mathsf{CAT}}$ denote the strict 2-category of large categories. The definitions above are motivated by the observation that, for any category $\cl{A}$, the category $\underline{\mathsf{Hom}}_{\underline{\mathsf{CAT}}}(1,\cl{A})$ is isomorphic to $\cl{A}$.  \end{rmk}

\begin{defn} Let $\cl{A}$ be an object of $\cl{C}$, and let \ar{a_{0},a_{1},f_{0}} and \ar{a_{1},a_{2},f_{1}} be arrows of $\cl{A}$. The {\em composition} $f_{1} \circ f_{0}$ of $f_{1}$ and $f_{0}$ in $\cl{A}$ is their composition in $\underline{\mathsf{Hom}}_{\cl{C}}(1,\cl{A})$. \end{defn}

\begin{defn} A {\em diagram} in $\cl{A}$ is a diagram in $\underline{\mathsf{Hom}}_{\cl{C}}(1,\cl{A})$. We define a {\em commutative diagram} in $\cl{A}$, a {\em co-cartesian square} in $\cl{A}$, and a {\em cartesian square} in $\cl{A}$ in the same way. \end{defn}

\begin{notn} Given 1-arrows 

\begin{diagram}

\begin{tikzpicture} [>=stealth]

\matrix [ampersand replacement=\&, matrix of math nodes, column sep=3 em, nodes={anchor=center}]
{ 
|(0-0)| \cl{A}_{0} \& |(1-0)| \cl{A}_{1} \& |(2-0)| \cl{A}_{2} \\ 
};
	
\draw[->] ({0-0}.north east) to node[auto] {$F_{0}$} ({1-0}.north west);
\draw[->] ({0-0}.south east) to node[auto,swap] {$F_{1}$} ({1-0}.south west);
\draw[->] (1-0) to node[auto] {$G$} (2-0); 

\end{tikzpicture} 

\end{diagram} %
of $\cl{C}$, and a 2-arrow \ar{F_{0},F_{1},\eta} of $\cl{C}$, we denote by $G \cdot \eta$ the 2-arrow \ar{GF_{0},GF_{1},} of $\cl{C}$ obtained by horizontal compositition of $id(G)$ and $\eta$. It is sometimes referred to as the {\em whiskering} of $G$ and $\eta$.

Similarly, given 1-arrows

\begin{diagram}

\begin{tikzpicture} [>=stealth]

\matrix [ampersand replacement=\&, matrix of math nodes, column sep=3 em, nodes={anchor=center}]
{ 
|(0-0)| \cl{A}_{0} \& |(1-0)| \cl{A}_{1} \& |(2-0)| \cl{A}_{2} \\ 
};
	
\draw[->] ({1-0}.north east) to node[auto] {$F_{0}$} ({2-0}.north west);
\draw[->] ({1-0}.south east) to node[auto,swap] {$F_{1}$} ({2-0}.south west);
\draw[->] (0-0) to node[auto] {$G$} (1-0); 

\end{tikzpicture} 

\end{diagram} %
of $\cl{C}$, and a 2-arrow \ar{F_{0},F_{1},\eta} of $\cl{C}$, we denote by $\eta \cdot G$ the 2-arrow \ar{F_{0}G,F_{1}G,} of $\cl{C}$ obtained by horizontal composition of $\eta$ and $id(G)$. It is also sometimes referred to as the {\em whiskering} of $\eta$ and $G$. \end{notn}

\begin{notn} Let \ar{\cl{A}_{0},\cl{A}_{1},F} be a 1-arrow of $\cl{C}$. If $a$ is an object of $\cl{A}_{0}$, we denote by $F(a)$ the object of $\cl{A}_{1}$ defined by the 1-arrow $F \circ a$ of $\cl{C}$. 

If \ar{a_{0},a_{1},f} is an arrow of $\cl{A}_{0}$, we denote by $F(f)$ the arrow of $\cl{A}_{1}$ from $F(a_{0})$ to $F(a_{1})$ defined by the whiskered 2-arrow $F \cdot f$ of $\cl{C}$. \end{notn}

\begin{rmk} \label{1ArrowsAsFunctors} In this way, we think of a 1-arrow \ar{\cl{A}_{0},\cl{A}_{1},F} of $\cl{C}$ as a functor from $\cl{A}_{0}$ to $\cl{A}_{1}$, corresponding to the functor \ar[6]{{\underline{\mathsf{Hom}}_{\cl{C}}(1,\cl{A}_{0})},{\underline{\mathsf{Hom}}_{\cl{C}}(1,\cl{A}_{1}).},{\underline{\mathsf{Hom}}_{\cl{C}}(1,F)}} \end{rmk}

\begin{notn} \label{2ArrowAsNaturalTransformationNotation} Let \pair{\cl{A}_{0},\cl{A}_{1},F_{0},F_{1}} be 1-arrows of $\cl{C}$, and let \ar{F_{0},F_{1},\eta} be a 2-arrow of $\cl{C}$. If $a$ is an object of $\cl{A}_{0}$, we denote by $\eta(a)$ the arrow of $\cl{A}_{1}$ from $F_{0}(a)$ to $F_{1}(a)$ defined by the whiskered 2-arrow $\eta \cdot a$ of $\cl{C}$. \end{notn}

\begin{rmk} Let \pair{\cl{A}_{0},\cl{A}_{1},F_{0},F_{1}} be 1-arrows of $\cl{C}$, thought of as functors from $\cl{A}_{0}$ to $\cl{A}_{1}$, as in Remark \ref{1ArrowsAsFunctors}. We think of a 2-arrow \ar{F_{0},F_{1},\eta} of $\cl{C}$ as a natural transformation from $F_{0}$ to $F_{1}$, corresponding to the natural transformation \ar[6]{{\underline{\mathsf{Hom}}_{\cl{C}}(1,F_{0})},{{\underline{\mathsf{Hom}}_{\cl{C}}(1,F_{1})}.},{\underline{\mathsf{Hom}}_{\cl{C}}(1,\eta)}} \end{rmk}

\begin{notn} We denote by $\cl{C}^{op}$ the 2-category obtained from $\cl{C}$ by reversing all 2-arrows. If $f$ is a 2-arrow of $\cl{C}$, we denote by $f^{op}$ the corresponding 2-arrow of $\cl{C}^{op}$. When viewing an object $\cl{A}$ of $\cl{C}$ as an object of $\cl{C}^{op}$, we denote it by $\cl{A}^{op}$. Thus $\underline{\mathsf{Hom}}_{\cl{C}^{op}}(1,\cl{A}^{op})$ is the opposite category of $\underline{\mathsf{Hom}}_{\cl{C}}(1,\cl{A})$. In particular, if \ar{a_{0},a_{1},f} defines an arrow of $\cl{A}$, then the 2-arrow $f^{op}$ of $\cl{C}^{op}$ defines an arrow of $\cl{A}^{op}$ from $a_{1}$ to $a_{0}$. \end{notn}

\begin{recollection} \label{AdjunctionRecollection} An {\em adjunction} between a pair $(F,G)$ of 1-arrows \adjunction{\cl{A}_{0},\cl{A}_{1},F,G} of $\cl{C}$ is a 2-arrow \ar{id(\cl{A}_{0}),GF,\eta} of $\cl{C}$, and a 2-arrow \ar{FG,id(\cl{A}_{1}),\zeta} of $\cl{C}$, such that the diagram \tri{F,FGF,F,F \eta, \zeta F, id} in $\underline{\mathsf{Hom}}_{\cl{C}}(\cl{A}_{0},\cl{A}_{1})$ commutes, and such that the diagram \tri{G,GFG,G,\eta G, G \zeta,id} in $\underline{\mathsf{Hom}}_{\cl{C}}(\cl{A}_{1},\cl{A}_{0})$ commutes. We refer to $F$ as a {\em left adjoint} of $G$, and to $G$ as a {\em right adjoint} of $F$. 

If we have an adjunction between $F$ and $G$, then for any object $\cl{A}$ of $\cl{C}$, the natural transformations $\underline{\mathsf{Hom}}_{\cl{C}}(\cl{A},\eta)$ and $\underline{\mathsf{Hom}}_{\cl{C}}(\cl{A},\zeta)$ define an adjunction between the following pair of functors. \adjunction[7]{{\underline{\mathsf{Hom}}_{\cl{C}}(\cl{A},\cl{A}_{0})},{\underline{\mathsf{Hom}}_{\cl{C}}(\cl{A},\cl{A}_{1})},{\underline{\mathsf{Hom}}_{\cl{C}}(\cl{A},F)},{\underline{\mathsf{Hom}}_{\cl{C}}(\cl{A},G)}} %
In particular, we have a natural isomorphism \ar{{\mathsf{Hom}_{\underline{\mathsf{Hom}}_{\cl{C}}(1,\cl{A}_{1})}\Big(\underline{\mathsf{Hom}}_{\cl{C}}\big(1,F(-)\big),-\Big)},{\mathsf{Hom}_{\underline{\mathsf{Hom}}_{\cl{C}}(1,\cl{A}_{0})}\Big(-,\underline{\mathsf{Hom}}_{\cl{C}}\big(1,G(-)\big)\Big)},\mathsf{adj}} of functors \ar{{\big( \mathsf{Hom}_{\cl{C}}(1,\cl{A}_{0}) \big)^{op} \times \mathsf{Hom}_{\cl{C}}(1,\cl{A}_{1})},{\mathsf{Set}.}} %
Adopting the shorthand 

\begin{diagram}

\begin{tikzpicture} [>=stealth, column 1/.style={anchor=east}, column 2/.style={anchor=west}]

\matrix [ampersand replacement=\&, matrix of math nodes, column sep=3 em, row sep=1em, nodes={anchor=center}]
{
|(0-0)| \cl{A} \& |(1-0)| \underline{\mathsf{Hom}}_{\cl{C}}(1,\cl{A}), \\  
|(0-1)| F   \& |(1-1)| \underline{\mathsf{Hom}}_{\cl{C}}\big(1,F \big), \\ 
|(0-2)| G   \& |(1-2)| \underline{\mathsf{Hom}}_{\cl{C}}\big(1,G \big), \\
};

\draw[<->] (0-0) to (1-0);
\draw[<->] (0-1) to (1-1);
\draw[<->] (0-2) to (1-2);

\end{tikzpicture}

\end{diagram} %
we shall write the above natural isomorphism as \ar{{\mathsf{Hom}_{\cl{A}_{1}}\big(F(-),-\big)},{\mathsf{Hom}_{\cl{A}_{0}}\big(-,G(-)\big).},\mathsf{adj}} \end{recollection}

\begin{defn} \label{PushoutsPullbacks2ArrowsArePushoutsPullbacksInFormalCategoriesTerminology} If for any objects $\cl{A}_{0}$ and $\cl{A}_{1}$ of $\cl{C}$, any object $a$ of $\cl{A}_{0}$, and any co-cartesian (respectively cartesian) square \sq{F_{0},F_{1},F_{2},F_{3},\eta_{0},\eta_{1},\eta_{2},\eta_{3}} in $\underline{\mathsf{Hom}}_{\cl{C}}(\cl{A}_{0},\cl{A}_{1})$, the square \sq{F_{0}(a),F_{1}(a),F_{2}(a),F_{3}(a),\eta_{0}(a),\eta_{1}(a),\eta_{2}(a),\eta_{3}(a)} in $\cl{A}_{1}$ is co-cartesian (respectively cartesian), we write that {\em pushouts (respectively pullbacks) of 2-arrows of $\cl{C}$ give rise to pushouts (respectively pullbacks) in formal categories}. \end{defn}

\begin{rmk} In both $\underline{\mathsf{CAT}}$ and $\underline{\mathsf{CAT}}^{op}$, pushouts and pullbacks of 2-arrows give rise to pushouts and pullbacks in formal categories, more or less by definition. The author expects that colimits and limits of 2-arrows equally give rise to colimits and limits in formal categories for quite general 2-categories, for instance 2-topoi. \end{rmk}
 
\begin{rmk} It would certainly be possible for us to work with weak rather than strict 2-categories. However, we are motivated by ordinary categories. In addition to $\underline{\mathsf{CAT}}$, the only 2-category of importance to us is $\underline{\mathsf{CAT}}^{op}$, exactly in order to capture duality. Thus it is sufficient for us to work with strict 2-categories, and we do so in order to avoid the distraction of coherency. \end{rmk}

\end{chapter}
 
\begin{chapter}{Structures upon a cylinder or a co-cylinder} \label{StructuresUponACylinderOrACoCylinderChapter}

We introduce the notion of a cylinder or a co-cylinder in a formal category. We define the structures --- contraction, involution, subdivision, and three flavours of connection --- upon a cylinder and a co-cylinder which play a role in our work, and introduce axioms expressing their compatibility. We refer the reader to \ref{ExampleChapter} for an example of these structures in the category of categories. 

These structures and axioms have previously appeared in the literature. However, the precise definitions and terminology vary from author to author and paper to paper. Thus we collect in one place and from a single point of view all that we shall need. Most of our structures and axioms can be found in \S{4} of Chapter I and at the end of \S{3} of Chapter II of the book \cite{KampsPorterAbstractHomotopyAndSimpleHomotopyTheory} of Kamps and Porter, or in \S{2.1} and \S{2.3} of the paper \cite{GrandisCategoricallyAlgebraicFoundationsForHomotopicalAlgebra} of Grandis. Compatibility of right connections with subdivision appears implicitly, and in a slightly different context, in \S{6.4} of \cite{BrownHigginsSiveraNonabelianAlgebraicTopology}.
 
We also introduce strictness hypotheses which our structures upon a cylinder or a co-cylinder may satisfy. We shall come in \ref{HomotopyAndRelativeHomotopyChapter} to define the notion of a homotopy with respect to a cylinder or a co-cylinder. The first of our strictness hypotheses, which we shall refer to as strictness of left or right identities, ensures that the left or right composition of an identity homotopy with a homotopy $h$ is exactly $h$. 

Strictness of identities will also allow us to prove in \ref{FactorisationAxiomsChapter} that the mapping cylinder (respectively the mapping co-cylinder) of any arrow $f$ yields a factorisation of $f$ into a normally cloven cofibration followed by a trivial fibration (respectively into a trivial cofibration followed by a normally cloven fibration). It will moreover be crucial in establishing that the lifting axioms for a model category hold, in \ref{LiftingAxiomsChapter}.

The significance of strictness of identities with regard to lifting has to the author's knowledge not previously been observed. Its importance with regard to factorisation was independently identified by van den Berg and Garner in \cite{VanDenBergGarnerTopologicalAndSimplicialModelsOfIdentityTypes}. We particularly draw the reader's attention to Remark 4.3.3 in \cite{VanDenBergGarnerTopologicalAndSimplicialModelsOfIdentityTypes}. Our strictness of left (respectively right) identities condition corresponds to the left (respectively right) unitality condition of van den Berg and Garner. 

The second of our hypotheses, which we shall refer to as strictness of left inverses, ensures that the composition of a homotopy with its inverse is exactly an identity homotopy. It also allows us, given a lower right connection structure $\Gamma_{lr}$, to construct an upper right connection structure $\Gamma_{ur}$ such that $\Gamma_{lr}$ and $\Gamma_{ur}$ are compatible with subdivision. The compatibility of right connections with subdivision will be vital for us when in \ref{CoveringHomotopyExtensionPropertyChapter} we investigate the covering homotopy extension property.

\begin{assum} Let $\cl{C}$ be a 2-category, and let $\cl{A}$ be an object of $\cl{C}$. \end{assum}

\begin{defn} A {\em cylinder} in $\cl{A}$ is a 1-arrow \ar{\cl{A},\cl{A},\cyl} of $\cl{C}$, together with a pair of 2-arrows \pair{\id_{\cl{A}},\cyl,i_0,i_1} of $\cl{C}$. \end{defn}

\begin{defn} A {\em co-cylinder} in $\cl{A}$ is a 1-arrow \ar[4]{\cl{A},\cl{A},\cocyl} of $\cl{C}$, together with a pair of 2-arrows \pair{\cocyl,\id_{\cl{A}},e_{0},e_{1}} of $\cl{C}$. \end{defn}

\begin{rmk} Let $\cocylinder = \big(\cocyl,e_{0},e_{1} \big)$ be a co-cylinder in $\cl{A}$. Then $\big( \cocyl, e_{0}^{op}, e_{1}^{op} \big)$ defines a cylinder in $\cl{A}^{op}$, which we denote by $\cocylinder^{op}$. \end{rmk}

\begin{defn} Let $\cylinder = \big( \cyl, i_0, i_1 \big)$ be a cylinder in $\cl{A}$. A {\em contraction structure} with respect to $\cylinder$ is a 2-arrow \ar{\cyl,\id_{\cl{A}},p} of $\cl{C}$, such that the following diagrams in $\underline{\mathsf{Hom}}_{\cl{C}}(\cl{A},\cl{A})$ commute. \twotriangles{\id_{\cl{A}},\cyl,\id_{\cl{A}},i_{0},p,id,\id_{\cl{A}},\cyl,\id_{\cl{A}},i_{1},p,id} \end{defn}

\begin{defn} Let $\cocylinder = \big( \cocyl, e_0, e_1 \big)$ be a co-cylinder in $\cl{A}$. A {\em contraction structure} with respect to $\cocylinder$ is a 2-arrow \ar{\id_{\cl{A}},\cocyl,c} of $\cl{C}$, such that $c^{op}$ equips the cylinder $\cocylinder^{op}$ in $\cl{A}^{op}$ with a contraction structure. \end{defn}

\begin{defn} Let $\cylinder = \big( \cyl,i_{0},i_{1} \big)$ be a cylinder in $\cl{A}$. An {\em involution structure} with respect to $\cylinder$ is a 2-arrow \ar{\cyl,\cyl,v} of $\cl{C}$, such that the following diagrams in $\underline{\mathsf{Hom}}_{\cl{C}}(\cl{A},\cl{A})$ commute. \twotriangles{\id_{\cl{A}},\cyl,\cyl,i_{0},v,i_{1},\id_{\cl{A}},\cyl,\cyl,i_{1},v,i_{0}} \end{defn}
 
\begin{defn} Let $\cocylinder = \big( \cocyl,e_{0},e_{1} \big)$ be a co-cylinder in $\cl{A}$. An {\em involution structure} with respect to $\cocylinder$ is a 2-arrow \ar{\cocyl,\cocyl,v} of $\cl{C}$, such that $v^{op}$ defines an involution structure with respect to the cylinder $\cocylinder^{op}$ in $\cl{A}^{op}$. \end{defn}

\begin{defn} Let $\cylinder = \big( \cyl, i_{0}, i_{1}, p \big)$ be a cylinder in $\cl{A}$ equipped with a contraction structure $p$. An involution structure $v$ with respect to $\cylinder$ is {\em compatible with $p$} if the following diagram in $\underline{\mathsf{Hom}}_{\cl{C}}(\cl{A},\cl{A})$ commutes. \tri{\cyl,\cyl,\id_{\cl{A}},v,p,p} \end{defn}

\begin{defn} Let $\cocylinder = \big( \cocyl, e_{0}, e_{1}, c \big)$ be a co-cylinder in $\cl{A}$ equipped with a contraction structure $c$. An involution structure $v$ with respect to $\cocylinder$ is {\em compatible with $c$} if the involution structure $v^{op}$ with respect to the cylinder $\cocylinder^{op}$ in $\cl{A}^{op}$ is compatible with the contraction structure defined by $c^{op}$. \end{defn}

\begin{defn} Let $\cylinder = \big( \cyl,i_{0},i_{1} \big)$ be a cylinder in $\cl{A}$. A {\em subdivision structure} with respect to $\cylinder$ is a 1-arrow \ar{\cl{A},\cl{A},\subdiv} of $\cl{C}$, together with a pair of 2-arrows \pair{\cyl,\subdiv,r_{0},r_{1}} of $\cl{C}$, such that the diagram \sq{\id_{\cl{A}},\cyl,\cyl,\subdiv,i_{0},r_{0},i_{1},r_{1}}in $\underline{\mathsf{Hom}}_{\cl{C}}(\cl{A},\cl{A})$ is co-cartesian, and a 2-arrow \ar{\cyl,\subdiv,s} of $\cl{C}$, such that the following diagrams in $\underline{\mathsf{Hom}}_{\cl{C}}(\cl{A},\cl{A})$ commute. \twosq{\id_{\cl{A}},\cyl,\cyl,\subdiv,i_{0},s,i_{0},r_{1},\id_{\cl{A}},\cyl,\cyl,\subdiv,i_{1},s,i_{1},r_{0}} \end{defn}

\begin{defn} Let $\cocylinder = \big( \cocyl,e_{0},e_{1} \big)$ be a co-cylinder in $\cl{A}$. A {\em subdivision structure} with respect to $\cocylinder$ is a 1-arrow \ar{\cl{A},\cl{A},\subdiv} of $\cl{C}$, together with a pair of 2-arrows \pair{\subdiv,\cocyl,r_{0},r_{1},s} of $\cl{C}$ and a 2-arrow \ar{\subdiv,\cocyl,s} of $\cl{C}$, such that $\big( \subdiv,r_{0}^{op},r_{1}^{op},s^{op} \big)$ defines a subdivision structure with respect to the cylinder $\cocylinder^{op}$ in $\cl{A}^{op}$. \end{defn}

\begin{defn} \label{SubdivisionCompatibleWithContractionDefinition} Let $\cylinder = \big( \cyl, i_{0}, i_{1}, p, \subdiv, r_{0}, r_{1}, s \big)$ be a cylinder in $\cl{A}$ equipped with a contraction structure $p$ and a subdivision structure $\big(\subdiv,r_{0},r_{1},s \big)$. Let \ar{\subdiv,\id_{\cl{A}},\overline{p}} denote the canonical 2-arrow of $\cl{C}$ such that the diagram \pushout{\id_{\cl{A}},\cyl,\cyl,\subdiv,\id_{\cl{A}},i_{0},r_{0},i_{1},r_{1},p,p,\overline{p}} in $\underline{\mathsf{Hom}}_{\cl{C}}(\cl{A},\cl{A})$ commutes. The subdivision structure $\big( \subdiv, r_{0}, r_{1}, s \big)$ is {\em compatible with $p$} if the following diagram in $\underline{\mathsf{Hom}}_{\cl{C}}(\cl{A},\cl{A})$ commutes. \tri{\cyl,\subdiv,\id_{\cl{A}},s,\overline{p},p}  \end{defn}

\begin{defn} Let $\cocylinder = \big( \cocyl, e_{0}, e_{1}, c \big)$ be a co-cylinder in $\cl{A}$ equipped with a contraction structure $c$. A subdivision structure $\big( \subdiv, r_{0}, r_{1}, s \big)$ with respect to $\cocylinder$ is {\em compatible with $c$} if the subdivision structure $\big( \subdiv, r_{0}^{op}, r_{1}^{op}, s^{op} \big)$ with respect to the cylinder $\cocylinder^{op}$ in $\cl{A}^{op}$ is compatible with the contraction structure defined by $c^{op}$. \end{defn}

\begin{defn} Let $\cylinder = \big( \cyl, i_0, i_1, \subdiv, r_0, r_1, s \big)$ be a cylinder in $\cl{A}$ equipped with a subdivision structure $\big( \subdiv, r_{0}, r_{1}, s \big)$. Then {\em $\cyl$ preserves subdivision} with respect to $\cylinder$ if the following diagram in $\underline{\mathsf{Hom}}_{\cl{C}}(\cl{A},\cl{A})$ is co-cartesian. \sq[{4,3}]{\cyl,\cyl^{2},\cyl^{2},\cyl \circ \subdiv,\cyl \cdot i_0, \cyl \cdot r_0, \cyl \cdot i_1, \cyl \cdot r_1} \end{defn} 

\begin{defn} Let $\cocylinder = \big( \cocyl, e_0, e_1, \subdiv, r_0, r_1, s \big)$ be a co-cylinder in $\cl{A}$ equipped with a subdivision structure $\big( \subdiv, r_{0}, r_{1}, s \big)$. Then {\em $\cocyl$ preserves subdivision} with respect to $\cocylinder$ if $\cocyl$ preserves subdivision with respect to the cylinder $\cocylinder^{op}$ in $\cl{A}^{op}$ equipped with the subdivision structure $\big( \subdiv, r_{0}^{op}, r_{1}^{op}, s^{op} \big)$. \end{defn}

\begin{defn} Let $\cylinder = \big( \cyl, i_{0}, i_{1}, p \big)$ be a cylinder in $\cl{A}$ equipped with a contraction structure $p$. An {\em upper left connection structure} with respect to $\cylinder$ is a 2-arrow \ar{\cyl^{2},\cyl,\Gamma_{ul}} of $\cl{C}$, such that the following diagrams in $\underline{\mathsf{Hom}}_{\cl{C}}(\cl{A},\cl{A})$ commute. 

\begin{diagram}

\begin{tikzpicture} [>=stealth]

\matrix [ampersand replacement=\&, matrix of math nodes, column sep=4 em, row sep=3 em, nodes={anchor=center}]
{ 
|(0-0)| \cyl         \& |(1-0)| \cyl^{2} \&[1em] |(2-0)| \cyl         \& |(3-0)| \cyl^{2} \\ 
                     \& |(1-1)| \cyl     \&                           \& |(3-1)| \cyl     \\
|(0-2)| \cyl         \& |(1-2)| \cyl^{2} \&      |(2-2)| \cyl         \& |(3-2)| \cyl^{2} \\
|(0-3)| \id_{\cl{A}} \& |(1-3)| \cyl     \&      |(2-3)| \id_{\cl{A}} \& |(3-3)| \cyl     \\      
};
	
\draw[->] (0-0) to node[auto] {$i_{0} \cdot \cyl$} (1-0);
\draw[->] (1-0) to node[auto] {$\Gamma_{ul}$} (1-1);
\draw[->] (0-0) to node[auto,swap] {$id$} (1-1);
\draw[->] (2-0) to node[auto] {$\cyl \cdot i_{0}$} (3-0);
\draw[->] (3-0) to node[auto] {$\Gamma_{ul}$} (3-1);
\draw[->] (2-0) to node[auto,swap] {$id$} (3-1);
\draw[->] (0-2) to node[auto] {$i_{1} \cdot \cyl$} (1-2);
\draw[->] (1-2) to node[auto] {$\Gamma_{ul}$} (1-3);
\draw[->] (0-2) to node[auto,swap] {$p$} (0-3);
\draw[->] (0-3) to node[auto,swap] {$i_{1}$} (1-3);
\draw[->] (2-2) to node[auto] {$\cyl \cdot i_{1}$} (3-2);
\draw[->] (3-2) to node[auto] {$\Gamma_{ul}$} (3-3);
\draw[->] (2-2) to node[auto,swap] {$p$} (2-3);
\draw[->] (2-3) to node[auto,swap] {$i_{1}$} (3-3);

\end{tikzpicture} 

\end{diagram} \end{defn}

\begin{defn} Let $\cocylinder = \big( \cocyl, e_{0}, e_{1}, c \big)$ be a co-cylinder in $\cl{A}$ equipped with a contraction structure $c$. An {\em upper left connection structure} with respect to $\cocylinder$ is a 2-arrow \ar{\cyl,\cyl^{2},\Gamma_{ul}} of $\cl{C}$, such that $(\Gamma_{ul})^{op}$ defines an upper left connection structure with respect to the cylinder $\cocylinder^{op}$ in $\cl{A}^{op}$ equipped with the contraction structure $c^{op}$. \end{defn}

\begin{defn} Let $\cylinder = \big( \cyl, i_{0}, i_{1}, p \big)$ be a cylinder in $\cl{A}$ equipped with a contraction structure $p$. A {\em lower right connection structure} with respect to $\cylinder$ is a 2-arrow \ar{\cyl^{2},\cyl,\Gamma_{lr}} of $\cl{C}$, such that the following diagrams in $\underline{\mathsf{Hom}}_{\cl{C}}(\cl{A},\cl{A})$ commute. 

\begin{diagram}

\begin{tikzpicture} [>=stealth]

\matrix [ampersand replacement=\&, matrix of math nodes, column sep=4 em, row sep=3 em, nodes={anchor=center}]
{ 
|(0-0)| \cyl         \& |(1-0)| \cyl^{2} \&[1em] |(2-0)| \cyl         \& |(3-0)| \cyl^{2} \\ 
                     \& |(1-1)| \cyl     \&                           \& |(3-1)| \cyl     \\
|(0-2)| \cyl         \& |(1-2)| \cyl^{2} \&      |(2-2)| \cyl         \& |(3-2)| \cyl^{2} \\
|(0-3)| \id_{\cl{A}} \& |(1-3)| \cyl     \&      |(2-3)| \id_{\cl{A}} \& |(3-3)| \cyl     \\      
};
	
\draw[->] (0-0) to node[auto] {$i_{1} \cdot \cyl$} (1-0);
\draw[->] (1-0) to node[auto] {$\Gamma_{lr}$} (1-1);
\draw[->] (0-0) to node[auto,swap] {$id$} (1-1);
\draw[->] (2-0) to node[auto] {$\cyl \cdot i_{1}$} (3-0);
\draw[->] (3-0) to node[auto] {$\Gamma_{lr}$} (3-1);
\draw[->] (2-0) to node[auto,swap] {$id$} (3-1);
\draw[->] (0-2) to node[auto] {$i_{0} \cdot \cyl$} (1-2);
\draw[->] (1-2) to node[auto] {$\Gamma_{lr}$} (1-3);
\draw[->] (0-2) to node[auto,swap] {$p$} (0-3);
\draw[->] (0-3) to node[auto,swap] {$i_{0}$} (1-3);
\draw[->] (2-2) to node[auto] {$\cyl \cdot i_{0}$} (3-2);
\draw[->] (3-2) to node[auto] {$\Gamma_{lr}$} (3-3);
\draw[->] (2-2) to node[auto,swap] {$p$} (2-3);
\draw[->] (2-3) to node[auto,swap] {$i_{0}$} (3-3);

\end{tikzpicture} 

\end{diagram} \end{defn}

\begin{defn} Let $\cocylinder = \big( \cocyl, e_{0}, e_{1}, c \big)$ be a co-cylinder in $\cl{A}$ equipped with a contraction structure $c$. A {\em lower right connection structure} with respect to $\cocylinder$ is a 2-arrow \ar{\cyl,\cyl^{2},\Gamma_{lr}} of $\cl{C}$, such that $(\Gamma_{lr})^{op}$ defines a lower right connection structure with respect to the cylinder $\cocylinder^{op}$ in $\cl{A}^{op}$ equipped with the contraction structure $c^{op}$. \end{defn}

\begin{defn} Let $\cylinder = \big( \cyl, i_{0}, i_{1}, p \big)$ be a cylinder in $\cl{A}$ equipped with a contraction structure $p$. A lower right connection structure $\Gamma_{lr}$ with respect to $\cylinder$ is {\em compatible with $p$} if the following diagram in $\underline{\mathsf{Hom}}_{\cl{C}}(\cl{A},\cl{A})$ commutes. \sq{\cyl^{2},\cyl,\cyl,\id_{\cl{A}},\Gamma_{lr},p,p \cdot \cyl,p} \end{defn} 

\begin{defn} Let $\cocylinder = \big( \cocyl, e_{0}, e_{1}, c \big)$ be a co-cylinder in $\cl{A}$ equipped with a contraction structure $c$. A lower right connection structure $\Gamma_{lr}$ with respect to $\cocylinder$ is {\em compatible with $c$} if the lower right connection structure $(\Gamma_{lr})^{op}$ with respect to the cylinder $\cocylinder^{op}$ in $\cl{A}^{op}$ is compatible with the contraction structure $c^{op}$. \end{defn} 

\begin{rmk} We shall not need to consider compatibility of an upper left connection structure with a contraction structure, or compatibility of an upper right connection structure, which we shall define next, with a contraction structure. \end{rmk}

\begin{defn} Let $\cylinder = \big( \cyl, i_{0}, i_{1}, p, v \big)$ be a cylinder in $\cl{A}$ equipped with a contraction structure $p$ and an involution structure $v$. An {\em upper right connection structure} with respect to $\cylinder$ is a 2-arrow \ar{\cyl^{2},\cyl,\Gamma_{ur}} of $\cl{C}$, such that the following diagrams in $\underline{\mathsf{Hom}}_{\cl{C}}(\cl{A},\cl{A})$ commute. 

\begin{diagram}

\begin{tikzpicture} [>=stealth]

\matrix [ampersand replacement=\&, matrix of math nodes, column sep=4 em, row sep=3 em, nodes={anchor=center}]
{ 
|(0-0)| \cyl         \& |(1-0)| \cyl^{2} \&[1em] |(2-0)| \cyl         \& |(3-0)| \cyl^{2} \\ 
                     \& |(1-1)| \cyl     \&                           \& |(3-1)| \cyl     \\
|(0-2)| \cyl         \& |(1-2)| \cyl^{2} \&      |(2-2)| \cyl         \& |(3-2)| \cyl^{2} \\
|(0-3)| \id_{\cl{A}} \& |(1-3)| \cyl     \&      |(2-3)| \id_{\cl{A}} \& |(3-3)| \cyl     \\      
};
	
\draw[->] (0-0) to node[auto] {$i_{0} \cdot \cyl$} (1-0);
\draw[->] (1-0) to node[auto] {$\Gamma_{ur}$} (1-1);
\draw[->] (0-0) to node[auto,swap] {$id$} (1-1);
\draw[->] (2-0) to node[auto] {$\cyl \cdot i_{1}$} (3-0);
\draw[->] (3-0) to node[auto] {$\Gamma_{ur}$} (3-1);
\draw[->] (2-0) to node[auto,swap] {$v$} (3-1);
\draw[->] (0-2) to node[auto] {$\cyl \cdot i_{0}$} (1-2);
\draw[->] (1-2) to node[auto] {$\Gamma_{ur}$} (1-3);
\draw[->] (0-2) to node[auto,swap] {$p$} (0-3);
\draw[->] (0-3) to node[auto,swap] {$i_{0}$} (1-3);
\draw[->] (2-2) to node[auto] {$i_{1} \cdot \cyl$} (3-2);
\draw[->] (3-2) to node[auto] {$\Gamma_{ur}$} (3-3);
\draw[->] (2-2) to node[auto,swap] {$p$} (2-3);
\draw[->] (2-3) to node[auto,swap] {$i_{0}$} (3-3);

\end{tikzpicture} 

\end{diagram} \end{defn}

\begin{defn} Let $\cocylinder = \big( \cocyl, e_{0}, e_{1}, c, v \big)$ be a co-cylinder in $\cl{A}$ equipped with a contraction structure $c$ and an involution structure $v$. An {\em upper right connection structure} with respect to $\cocylinder$ is a 2-arrow \ar{\cyl,\cyl^{2},\Gamma_{ur}}of $\cl{C}$, such that $(\Gamma_{ur})^{op}$ defines an upper right connection structure with respect to the cylinder $\cocylinder^{op}$ in $\cl{A}^{op}$ equipped with the contraction structure $c^{op}$ and the involution structure $v^{op}$. \end{defn}

\begin{rmk} \label{LowerLeftConnectionNotNecessaryRemark} Analogously, one can define a {\em lower left connection structure} with respect to a cylinder or a co-cylinder. Everything concerning upper and lower right connections below can equally be carried out for upper and lower left connections. \end{rmk}

\begin{defn} \label{RightConnectionsCylinderCompatibilityDefinition} Let $\cylinder = \big(\cyl,i_0,i_1,p,v,\subdiv,r_0,r_1,s,\Gamma_{lr},\Gamma_{ur} \big)$ be a cylinder in $\cl{A}$ equipped with a contraction structure $p$, an involution structure $v$, a subdivision structure $\big( \subdiv, r_{0}, r_{1}, s\big)$, an upper right connection structure $\Gamma_{ur}$, and a lower right connection structure $\Gamma_{lr}$. %
Let \ar{\subdiv \circ \cyl,\cyl,x} denote the canonical 2-arrow of $\cl{C}$ such that the following diagram in $\underline{\mathsf{Hom}}_{\cl{C}}(\cl{A},\cl{A})$ commutes. \pushout[{4,3,-1}]{\cyl,\cyl^{2},\cyl^{2},\subdiv \circ \cyl,\cyl,i_0 \cdot \cyl,r_0 \cdot \cyl,i_1 \cdot \cyl,r_1 \cdot \cyl,\Gamma_{lr},\Gamma_{ur},x} %
Then $\Gamma_{lr}$ and $\Gamma_{ur}$ are {\em compatible with $\big( \subdiv, r_{0}, r_{1}, s \big)$} if the following diagram in $\underline{\mathsf{Hom}}_{\cl{C}}(\cl{A},\cl{A})$ commutes. \tri[{4,3}]{\cyl^{2},\subdiv \circ \cyl,\cyl,s \cdot \cyl,x,p \cdot \cyl} \end{defn}

\begin{defn} \label{RightConnectionsCoCylinderCompatibilityDefinition} Let $\cocylinder = \big(\cocyl,e_0, e_1, c, v, \subdiv, r_0, r_1, s, \Gamma_{lr}, \Gamma_{ur} \big)$ be a co-cylinder in $\cl{A}$ equipped with a contraction structure $c$, an involution structure $v$, a subdivision structure $\big( \subdiv, r_{0}, r_{1}, s \big)$, an upper right connection structure $\Gamma_{ur}$, and a lower right connection structure $\Gamma_{lr}$. 

Then $\Gamma_{lr}$ and $\Gamma_{ur}$ are {\em compatible with $\big( \subdiv, r_{0}, r_{1}, s \big)$} if the right connections $\Gamma_{lr}^{op}$ and $\Gamma_{ur}^{op}$ with respect to the cylinder $\cocylinder^{op}$ in $\cl{A}^{op}$ equipped with the contraction structure $c^{op}$ and the involution structure $v^{op}$ are compatible with the subdivision structure $\big( \subdiv, r_{0}^{op}, r_{1}^{op}, s^{op} \big)$. \end{defn}

\begin{prpn} \label{LowerRightConnectionPlusInvolutionGivesUpperRightConnectionCylinderProposition} Let $\cylinder = \big( \cyl, i_0, i_1, p, v, \Gamma_{lr} \big)$ be a cylinder in $\cl{A}$ equipped with a contraction structure $p$, an involution structure $v$ compatible with $p$, and a lower right connection structure $\Gamma_{lr}$. Then the 2-arrow \ar[7]{\cyl^{2},\cyl,\Gamma_{lr} \circ (v \cdot \cyl)} of $\cl{C}$ defines an upper right connection structure with respect to $\cylinder$. \end{prpn}

\begin{proof} Firstly, the following diagram in $\underline{\mathsf{Hom}}_{\cl{C}}(\cl{A},\cl{A})$ commutes. 

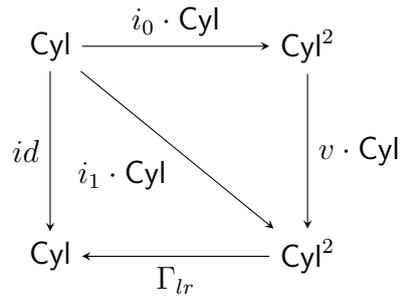
\begin{diagram}

\begin{tikzpicture} [>=stealth]

\matrix [ampersand replacement=\&, matrix of math nodes, column sep=6 em, row sep=3 em, nodes={anchor=center}]
{ 
|(0-0)| \cyl \& |(1-0)| \cyl^{2} \\[2em] 
|(0-1)| \cyl \& |(1-1)| \cyl^{2} \\
};
	
\draw[->] (0-0) to node[auto] {$i_{0} \cdot \cyl$} (1-0);
\draw[->] (1-0) to node[auto] {$v \cdot \cyl$} (1-1);
\draw[->] (0-0) to node[auto,swap] {$i_{1} \cdot \cyl$} (1-1);
\draw[->] (1-1) to node[auto] {$\Gamma_{lr}$} (0-1);
\draw[->] (0-0) to node[auto,swap] {$id$} (0-1);
\end{tikzpicture} 

\end{diagram} %
Secondly, the following diagram in $\underline{\mathsf{Hom}}_{\cl{C}}(\cl{A},\cl{A})$ commutes, appealing to the compatibility of $v$ with $p$. \trapeziums[{3,3,1,0}]{\cyl,\cyl^{2},\cyl,\cyl^{2},\id_{\cl{A}},\cyl,\cyl \cdot i_{0}, v \cdot \cyl, \Gamma_{lr}, p, i_{0}, v, \cyl \cdot i_{0}, p}  %
Thirdly, the following diagram in $\underline{\mathsf{Hom}}_{\cl{C}}(\cl{A},\cl{A})$ commutes.  

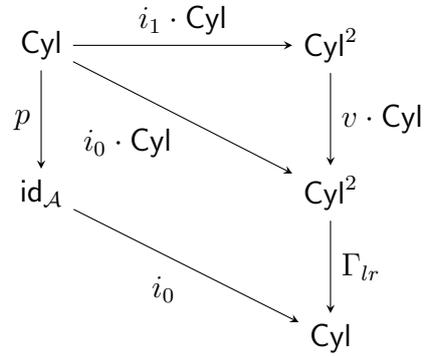
\begin{diagram}

\begin{tikzpicture} [>=stealth]

\matrix [ampersand replacement=\&, matrix of math nodes, column sep=7 em, row sep=3 em, nodes={anchor=center}]
{ 
|(0-0)| \cyl         \& |(1-0)| \cyl^{2} \\
|(0-1)| \id_{\cl{A}} \& |(1-1)| \cyl^{2} \\
                     \& |(1-2)| \cyl     \\
};
	
\draw[->] (0-0) to node[auto] {$i_{1} \cdot \cyl$} (1-0);
\draw[->] (1-0) to node[auto] {$v \cdot \cyl$} (1-1);
\draw[->] (0-0) to node[auto,swap] {$i_{0} \cdot \cyl$} (1-1);
\draw[->] (0-0) to node[auto,swap] {$p$} (0-1);
\draw[->] (0-1) to node[auto,swap] {$i_{0}$} (1-2);
\draw[->] (1-1) to node[auto] {$\Gamma_{lr}$} (1-2);

\end{tikzpicture} 

\end{diagram} %
Fourthly, the following diagram in $\underline{\mathsf{Hom}}_{\cl{C}}(\cl{A},\cl{A})$ commutes. \squareabovetriangle[{4,3,0}]{\cyl,\cyl^{2},\cyl,\cyl^{2},\cyl,\cyl \cdot i_{1},v \cdot \cyl, v, \cyl \cdot i_{1},\Gamma_{lr},id} \end{proof}

\begin{cor} \label{LowerRightConnectionPlusInvolutionGivesUpperRightConnectionCoCylinderCorollary} Let $\cocylinder = \big( \cocyl, e_0, e_1, c, v, \Gamma_{lr} \big)$ be a co-cylinder in $\cl{A}$ equipped with a contraction structure $c$, an involution structure $v$ compatible with $c$, and a lower right connection structure $\Gamma_{lr}$. Then the 2-arrow \ar[8]{\cocyl,\cocyl^{2},{(v \cdot \cocyl) \circ \Gamma_{lr}}} of $\cl{C}$ defines an upper right connection structure with respect to $\cocylinder$. \end{cor}

\begin{proof} Follows immediately from Proposition \ref{LowerRightConnectionPlusInvolutionGivesUpperRightConnectionCylinderProposition} by duality. \end{proof}

\begin{defn} \label{StrictnessLeftIdentitiesCylinderDefinition} Let $\cylinder = \big(\cyl, i_0, i_1, p, \subdiv, r_0, r_1, s \big)$ be a cylinder in $\cl{A}$, equipped with a contraction structure $p$, and a subdivision structure $\big( \subdiv, r_{0}, r_{1}, s \big)$. Let \ar{\subdiv,\cyl,q_{l}} denote the canonical 2-arrow of $\cl{C}$ such that the following diagram in $\underline{\mathsf{Hom}}_{\cl{C}}(\cl{A},\cl{A})$ commutes. \pushout{\id_{\cl{A}},\cyl,\cyl,\subdiv,\cyl,i_{0},r_{0},i_{1},r_{1},id,i_{0} \circ p, q_{l}} %
Then $\cylinder$ has {\em strictness of left identities} if the following diagram in $\underline{\mathsf{Hom}}_{\cl{C}}(\cl{A},\cl{A})$ commutes. \tri{\cyl,\subdiv,\cyl,s,q_{l},id} \end{defn}

\begin{defn} \label{StrictnessLeftIdentitiesCoCylinderDefinition} Let $\cocylinder = \big( \cocyl, e_0, e_1, c, \subdiv, r_0, r_1, s \big)$ be a co-cylinder in $\cl{A}$ equipped with a contraction structure $c$, and a subdivision structure $\big( \subdiv, r_{0}, r_{1}, s \big)$. Then $\cocylinder$ has {\em strictness of left identities} if the cylinder $\cocylinder^{op}$ in $\cl{A}^{op}$ equipped with the contraction structure $c^{op}$ and the subdivision structure $\big( \subdiv^{op}, r_{0}^{op}, r_{1}^{op}, s^{op} \big)$ has strictness of left identities. \end{defn} 

\begin{defn} \label{StrictnessRightIdentitiesCylinderDefinition} Let $\cylinder = \big(\cyl, i_0, i_1, p, \subdiv, r_0, r_1, s \big)$ be a cylinder in $\cl{A}$ equipped with a contraction structure $p$ and a subdivision structure $\big( \subdiv, r_{0}, r_{1}, s \big)$. Let \ar{\subdiv,\cyl,q_{r}} denote the canonical 2-arrow of $\cl{C}$ such that the following diagram in $\underline{\mathsf{Hom}}_{\cl{C}}(\cl{A},\cl{A})$ commutes. \pushout{\id_{\cl{A}},\cyl,\cyl,\subdiv,\cyl,i_{0},r_{0},i_{1},r_{1},i_{1} \circ p,id,q_{r}} %
Then $\cylinder$ has {\em strictness of right identities} if the following diagram in $\underline{\mathsf{Hom}}_{\cl{C}}(\cl{A},\cl{A})$ commutes. \tri{\cyl,\subdiv,\cyl,s,q_{r},id} \end{defn}

\begin{defn} \label{StrictnessRightIdentitiesCoCylinderDefinition} Let $\cocylinder = \big( \cocyl, e_0, e_1, c, \subdiv, r_0, r_1, s \big)$ be a co-cylinder in $\cl{A}$ equipped with a contraction structure $c$, and a subdivision structure $\big( \subdiv, r_{0}, r_{1}, s \big)$. Then $\cocylinder$ has {\em strictness of right identities} if the cylinder $\cocylinder^{op}$ in $\cl{A}^{op}$ equipped with the contraction structure $c^{op}$ and the involution structure $v^{op}$ has strictness of right identities. \end{defn} 

\begin{defn} \label{StrictnessIdentitiesCylinderDefinition} Let $\cylinder = \big(\cyl, i_0, i_1, p, \subdiv, r_0, r_1, s \big)$ be a cylinder in $\cl{A}$ equipped with a contraction structure $p$ and a subdivision structure $\big( \subdiv, r_{0}, r_{1}, s \big)$. Then $\cylinder$ has {\em strictness of identities} if it has both strictness of left identities and strictness of right identities. \end{defn}

\begin{defn} \label{StrictnessIdentitiesCoCylinderDefinition} Let $\cocylinder = \big( \cocyl, e_0, e_1, c, \subdiv, r_0, r_1, s \big)$ be a co-cylinder in $\cl{A}$ equipped with a contraction structure $c$, and a subdivision structure $\big( \subdiv, r_{0}, r_{1}, s \big)$. Then $\cocylinder$ has {\em strictness of identities} if it has both strictness of left identities and strictness of right identities. \end{defn}

\begin{defn} \label{StrictnessLeftInversesCylinderDefinition} Let $\cylinder = \big(\cyl, i_0, i_1, v, \subdiv, r_0, r_1, s \big)$ be a cylinder in $\cl{A}$ equipped with an involution structure $v$ and a subdivision structure $\big( \subdiv, r_{0}, r_{1}, s \big)$. Let \ar{\subdiv,\cyl,w} denote the canonical 2-arrow of $\cl{C}$ such that the following diagram in $\underline{\mathsf{Hom}}_{\cl{C}}(\cl{A},\cl{A})$ commutes. \pushout{\id_{\cl{A}},\cyl,\cyl,\subdiv,\cyl,i_{0},r_{0},i_{1},r_{1},id,v,w} Then $\cylinder$ has {\em strictness of left inverses} if the following diagram in $\underline{\mathsf{Hom}}_{\cl{C}}(\cl{A},\cl{A})$ commutes. \sq{\cyl,\subdiv,\id_{\cl{A}},\cyl,s,w,p,i_1} \end{defn}

\begin{defn} \label{StrictnessLeftInversesCoCylinderDefinition} Let $\cocylinder = \big(\cocyl, e_0, e_1, v, \subdiv, r_0, r_1, s \big)$ be a co-cylinder in $\cl{A}$ equipped with an involution structure $v$, and a subdivision structure $\big( \subdiv, r_{0}, r_{1}, s \big)$. Then $\cocylinder$ has {\em strictness of left inverses} if the cylinder $\cocylinder^{op}$ in $\cl{A}^{op}$ has strictness of left inverses. \end{defn}

\begin{rmk} We shall not need to consider strictness of right inverses, for which there is an analogue of Proposition \ref{CriterionCompatibilityRightConnectionsWithSubdivisionCylinderProposition}. \end{rmk}

\begin{prpn} \label{CriterionCompatibilityRightConnectionsWithSubdivisionCylinderProposition} Let $\cylinder = \big(\cyl, i_0, i_1, p, v, \subdiv, r_0, r_1, s, \Gamma_{lr} \big)$ be a cylinder in $\cl{A}$ equipped with a contraction structure $p$, an involution structure $v$ compatible with $p$, a subdivision structure $\big( \subdiv, r_{0}, r_{1}, s\big)$, and a lower right connection structure $\Gamma_{lr}$. Suppose that $\cylinder$ has strictness of left inverses. Let $\Gamma_{ur}$ denote the upper right connection structure with respect to $\cylinder$ constructed in Proposition \ref{LowerRightConnectionPlusInvolutionGivesUpperRightConnectionCylinderProposition}. Then $\Gamma_{lr}$ and $\Gamma_{ur}$ are compatible with $\big( \subdiv,r_{0},r_{1},s \big)$. \end{prpn}

\begin{proof} Let \ar{\subdiv,\cyl,w} denote the canonical 2-arrow of $\cl{C}$ of Definition \ref{StrictnessLeftInversesCylinderDefinition}. The following diagram in $\underline{\mathsf{Hom}}_{\cl{C}}(\cl{A},\cl{A})$ commutes. \tri[{5,3}]{\cyl^{2},\subdiv \circ \cyl,\cyl^{2},r_{0} \cdot \cyl, w \cdot \cyl, id} Hence the following diagram in $\underline{\mathsf{Hom}}_{\cl{C}}(\cl{A},\cl{A})$ commutes. \tri[{5,3}]{\cyl^{2},\subdiv \circ \cyl, \cyl, r_{0} \cdot \cyl, \Gamma_{lr} \circ (w \cdot \cyl), \Gamma_{lr}} The following diagram in $\underline{\mathsf{Hom}}_{\cl{C}}(\cl{A},\cl{A})$ also commutes.  

\begin{diagram}

\begin{tikzpicture} [>=stealth]

\matrix [ampersand replacement=\&, matrix of math nodes, column sep=5 em, row sep=4 em, nodes={anchor=center}]
{ 
|(0-0)| \cyl^{2} \& |(1-0)| \subdiv \cyl \\ 
|(0-1)| \cyl \& |(1-1)| \cyl^{2} \\
};
	
\draw[->] (0-0) to node[auto] {$r_{1} \cyl$} (1-0);
\draw[->] (1-0) to node[auto] {$w \cdot \cyl$} (1-1);
\draw[->] (0-0) to node[auto,swap] {$v \cdot \cyl$} (1-1);
\draw[->] (1-1) to node[auto] {$\Gamma_{lr}$} (0-1);
\draw[->] (0-0) to node[auto,swap] {$\Gamma_{ur}$} (0-1);
\end{tikzpicture} 

\end{diagram} %
Putting the last two observations together, we have that the following diagram in $\underline{\mathsf{Hom}}_{\cl{C}}(\cl{A},\cl{A})$ commutes. \pushout[{4,3,-1}]{\cyl,\cyl^{2},\cyl^{2},\subdiv \circ \cyl, \cyl, i_{0} \cdot  \cyl, r_{0} \cdot \cyl, i_{1} \cdot \cyl, r_{1} \cdot \cyl, \Gamma_{lr}, \Gamma_{ur}, \Gamma_{lr} \circ (w \cdot \cyl)} %
By the universal property of $\subdiv \circ \cyl$, we deduce that $\Gamma_{lr} \circ (w \cdot \cyl) =x$, where \ar{S \circ \cyl,\cyl,x} is the canonical 2-arrow of $\cl{C}$ of Definition \ref{RightConnectionsCylinderCompatibilityDefinition}. 

The following diagram in $\underline{\mathsf{Hom}}_{\cl{C}}(\cl{A},\cl{A})$ commutes. \squareabovetriangle[{5,3,0}]{\cyl^{2},\subdiv \circ \cyl,\cyl,\cyl^{2},\cyl,s \cdot \cyl,w \cdot \cyl,p \cdot \cyl, i_{1} \cdot \cyl, \Gamma_{lr},id} %
We conclude that the following diagram in $\underline{\mathsf{Hom}}_{\cl{C}}(\cl{A},\cl{A})$ commutes, as required. \tri[{4,3}]{\cyl^{2},\subdiv \circ \cyl,\cyl,s \cdot \cyl,x,p \cdot \cyl}  
\end{proof}

\begin{cor} \label{CriterionCompatibilityRightConnectionsWithSubdivisionCoCylinderCorollary} Let $\cocylinder = \big(\cocyl, e_0, e_1, c, v, \subdiv, r_0, r_1, s, \Gamma_{lr} \big)$ be a co-cylinder in $\cl{A}$ equipped with a contraction structure $c$, an involution structure $v$ compatible with $c$, a subdivision structure $\big( \subdiv, r_{0}, r_{1}, s\big)$, and a lower right connection structure $\Gamma_{lr}$. Suppose that $\cocylinder$ has strictness of left inverses. Let $\Gamma_{ur}$ denote the upper right connection with respect to $\cocylinder$ of Corollary \ref{LowerRightConnectionPlusInvolutionGivesUpperRightConnectionCoCylinderCorollary}. Then $\Gamma_{lr}$ and $\Gamma_{ur}$ are compatible with $\big( \subdiv, r_{0}, r_{1}, s \big)$. \end{cor}

\begin{proof} Follows immediately from Proposition \ref{CriterionCompatibilityRightConnectionsWithSubdivisionCylinderProposition} by duality. \end{proof}

\end{chapter}

\begin{chapter}{Cylindrical adjunctions} \label{CylindricalAdjunctionsChapter} 

We introduce a notion of adjunction between a cylinder and a co-cylinder, which is discussed for example in \S{3} of \cite{KampsPorterAbstractHomotopyAndSimpleHomotopyTheory}. In \ref{StructuresUponAnIntervalChapter} we shall explain that an interval in a category gives rise to a cylinder and co-cylinder which are adjoint. 

In \ref{HomotopyAndRelativeHomotopyChapter} we shall define a notion of homotopy with respect to a cylinder or a co-cylinder. Given both a cylinder and a co-cylinder, it will be vital for us to know that the corresponding notions of homotopy equivalence coincide. If the cylinder and co-cylinder are adjoint, we shall see that this is the case.

Furthermore, we shall in \ref{CofibrationsAndFibrationsChapter} define cofibrations with respect to a cylinder, and, dually, define fibrations with respect to a co-cylinder. If we have both a cylinder $\cylinder$ and a co-cylinder $\cocylinder$, with $\cylinder$ left adjoint to $\cocylinder$, we shall be able to characterise fibrations with respect to $\cocylinder$ via a homotopy lifting property with respect to $\cylinder$, and shall be able to characterise cofibrations with respect to $\cylinder$ via a homotopy lifting property with respect to $\cocylinder$. 

We refer the reader to Recollection \ref{AdjunctionRecollection} for the notion of an adjunction between 1-arrows of $\cl{C}$.

\begin{assum} Let $\cl{C}$ be a 2-category with a final object, and let $\cl{A}$ be an object of $\cl{C}$. \end{assum}

\begin{defn} \label{CylindricalAdjunctionDefinition}  Let $\cylinder = \big( \cyl, i_{0}, i_{1} \big)$ be a cylinder in $\cl{A}$, and let $\cocylinder = \big( \cocyl, e_{0}, e_{1} \big)$ be a co-cylinder in $\cl{A}$. Then $\cylinder$ is {\em left adjoint} to $\cocylinder$ if the following conditions are satisfied. 

\begin{itemize}[itemsep=1em,topsep=1em] 

\item[(i)]  $\cyl$ is left adjoint to $\cocyl$. 

\item[(ii)] Suppose that (i) holds. Let \ar{{\mathsf{Hom}_{\cl{A}}\big( \cyl(-),- \big)},{\mathsf{Hom}_{\cl{A}}\big(-,\cocyl(-)\big)},\mathsf{adj}} denote the corresponding natural isomorphism, adopting the shorthand of Recollection \ref{AdjunctionRecollection}. We require that for every arrow \ar{\cyl(a_{0}),a_{1},h} of $\cl{A}$, the following diagrams in $\cl{A}$ commute. \twosq{a_{0},\cyl(a_{0}),\cocyl(a_{1}),a_{1},i_{0}(a_{0}),h,\mathsf{adj}(h),e_{0}(a_{1}),a_{0},\cyl(a_{0}),\cocyl(a_{1}),a_{1},i_{1}(a_{0}),h,\mathsf{adj}(h),e_{1}(a_{1})} 

\end{itemize} 
\end{defn}

\begin{defn} \label{CylindricalAdjunctionCompatibleWithContractionAndExpansionDefinition} Let $\cylinder = \big( \cyl, i_{0}, i_{1}, p \big)$ be a cylinder in $\cl{A}$ equipped with a contraction structure $p$, and let $\cocylinder = \big( \cocyl, e_{0}, e_{1}, c \big)$ be a co-cylinder in $\cl{A}$ equipped with a contraction structure $c$. Suppose that $\cylinder$ is left adjoint to $\cocylinder$. 

Let \ar{{\mathsf{Hom}_{\cl{A}}\big( \cyl(-),- \big)},{\mathsf{Hom}_{\cl{A}}\big(-,\cocyl(-)\big)},\mathsf{adj}} denote the corresponding natural isomorphism, adopting the shorthand of Recollection \ref{AdjunctionRecollection}. The adjunction between $\cyl$ and $\cocyl$ is {\em compatible with $p$ and $c$} if, for every arrow \ar{a_{0},a_{1},f} of $\cl{A}$, the following diagram in $\cl{A}$ commutes. \tri{a_{0},a_{1},\cocyl(a_{1}),f,c(a_{1}),\mathsf{adj}\big(f \circ p(a_{0}) \big)} \end{defn}

\begin{rmk} Given a cylinder $\cylinder$ in $\cl{A}$, and a 1-arrow \ar[4]{\cl{A},\cl{A},\cocyl} of $\cl{C}$ which is left adjoint to $\cyl$, one can always equip $\cocyl$ with the structure of a co-cylinder $\cocylinder$ in $\cl{A}$ via the adjunction. 

Moreover, one can transfer structures upon $\cylinder$ across the adjunction to structures upon $\cocylinder$. This is explained for example in \S{1.8} of the paper \cite{GrandisMacDonaldHomotopyStructuresForAlgebrasOverAMonad} of Grandis and MacDonald. It goes back at least to \S{7} of the paper \cite{KampsKanBedingungenUndAbstrakteHomotopietheorie} of Kamps. \end{rmk}

\end{chapter}

\begin{chapter}{Monoidal category theory preliminaries} \label{MonoidalCategoryTheoryPreliminariesChapter}

In \ref{StructuresUponAnIntervalChapter} we shall introduce the notion of an interval in a monoidal category $\cl{A}$. Under natural hypotheses, a structured interval in $\cl{A}$ will give rise to a structured cylinder and a structured co-cylinder in $\cl{A}$. 

We shall need the notion of an exponentiable object of $\cl{A}$. There are two possible definitions, as we shall not assume that our monoidal structures are symmetric. We make the following choice.

\begin{assum} Let $(\cl{A},\otimes)$ be a monoidal category. Let $1$ denote its unit object, and let $\lambda$ denote its natural isomorphism \ar{- \otimes 1, {-.}} \end{assum}

\begin{defn} An object $a$ of $\cl{A}$ is {\em exponentiable} with respect to $\otimes$ if the functor \ar[4]{\cl{A},\cl{A},- \otimes a} admits a right adjoint, which we shall denote by \ar{\cl{A},\cl{A}.,{(-)^{a}}} \end{defn}

\begin{rmk} \label{UnitExponentiableRemark} We have that $1$ is exponentiable with respect to $\otimes$, since the natural isomorphism \ar[7]{{\mathsf{Hom}_{\cl{A}}(- \otimes 1,-)},{\mathsf{Hom}_{\cl{A}}(-,-)},{\mathsf{Hom}_{\cl{A}}(\lambda^{-1},-)}} exhibits the identity functor as a right adjoint of $- \otimes 1$. \end{rmk} 

\begin{notn} Let \ar{a_{0},a_{1},f} be an arrow of $\cl{A}$ such that both $a_{0}$ and $a_{1}$ are exponentiable with respect to $\otimes$. Then, for any object $a$ of $\cl{A}$, we have a natural isomorphism \ar[4]{{\mathsf{Hom}_{\cl{A}}(- \times a_{0}, a)},{\mathsf{Hom}_{\cl{A}}(-,a^{a_{0}})},\mathsf{adj}(a_{0})} and a natural isomorphism \ar[4]{{\mathsf{Hom}_{\cl{A}}(- \times a_{1}, a)},{\mathsf{Hom}_{\cl{A}}(-,a^{a_{1}}).},\mathsf{adj}(a_{1})} %
Let \ar{{a^{a_{1}}},{a^{a_{0}}},{a^{f}}} denote the arrow of $\cl{A}$ corresponding via the Yoneda lemma to the following natural transformation.  

\begin{diagram}

\begin{tikzpicture} [>=stealth]

\matrix [ampersand replacement=\&, matrix of math nodes, column sep=5 em, row sep=4 em, nodes={anchor=center}]
{ 
|(0-0)| \mathsf{Hom}_{\cl{A}}(-,a^{a_{1}}) \& |(1-0)| \mathsf{Hom}_{\cl{A}}(- \otimes a_{1},a) \\ 
|(0-1)| \mathsf{Hom}_{\cl{A}}(-,a^{a_{0}}) \& |(1-1)| \mathsf{Hom}_{\cl{A}}(- \otimes a_{0}, a) \\
};
	
\draw[->] (0-0) to node[auto] {$\mathsf{adj}(a_{1})^{-1}$} (1-0);
\draw[->] (1-0) to node[auto] {$\mathsf{Hom}_{\cl{A}}(- \otimes f,a)$} (1-1);
\draw[->,dashed] (0-0) to (0-1);
\draw[<-] (0-1) to node[auto,swap] {$\mathsf{adj}(a_{0})$} (1-1);

\end{tikzpicture} 

\end{diagram} 

\end{notn} 

\begin{rmk} Let \ar{a_{0},a_{1},f} be an arrow of $\cl{A}$ such that both $a_{0}$ and $a_{1}$ are exponentiable with respect to $\otimes$. Associating to an object $a$ of $\cl{A}$ the arrow \ar{{a^{a_{1}}},{a^{a_{0}}},{a^{f}}} of $\cl{A}$ defines a natural transformation \ar{{(-)^{a_{1}}},{(-)^{a_{0}}.},{{(-)}^{f}}} \end{rmk}

\begin{defn} The monoidal structure upon $\cl{A}$ defined by $\otimes$ is {\em closed} if, for every object $a$ of $\cl{A}$, the functor \ar[4]{\cl{A},\cl{A},a \otimes -} admits a right adjoint. \end{defn}

\begin{rmk} Suppose that the monoidal structure upon $\cl{A}$ defined by $\otimes$ is symmetric. This monoidal structure is closed if and only if every object of $\cl{A}$ is exponentiable with respect to $\otimes$. \end{rmk}

\begin{rmk} Let \sq{a_{0},a_{1},a_{2},a_{3},f_{0},f_{1},f_{2},f_{3}} be a co-cartesian diagram in $\cl{A}$.  Let $a$ be an object of $\cl{A}$ such that the following diagram in $\cl{A}$ is co-cartesian. \sq[{4,3}]{a \otimes a_{0}, a \otimes a_{1}, a \otimes a_{2}, a \otimes a_{3}, a \otimes f_{0}, a \otimes f_{1}, a \otimes f_{2}, a \otimes f_{3}} %
If $a_{0}$, $a_{1}$, $a_{2}$, and $a_{3}$ are exponentiable with respect to $\otimes$, the following diagram in $\cl{A}$ is cartesian. \sq{{a^{a_{3}}},{a^{a_{2}}},{a^{a_{1}}},{a^{a_{0}}},{a^{f_{3}}},{a^{f_{2}}},{a^{f_{1}}},{a^{f_{0}}}} \end{rmk} 

\end{chapter}

\begin{chapter}{Structures upon an interval} \label{StructuresUponAnIntervalChapter} 

We define the notion of an interval in a monoidal category $\cl{A}$. We introduce --- exactly in parallel with \ref{StructuresUponACylinderOrACoCylinderChapter} --- structures with which this interval may be able to be equipped. 

An interval $\interval$ in $\cl{A}$ gives rise, under natural hypotheses, to a cylinder $\cylinderinterval$ and a co-cylinder $\cocylinderinterval$ in $\cl{A}$. Structures upon $\interval$ pass to structures upon $\cylinderinterval$ and $\cocylinderinterval$. 

Moreover, $\cylinderinterval$ is left adjoint to $\cocylinderinterval$. Whereas in an abstract setting we work with cylinders and co-cylinders, in practise we often construct a cylinder and co-cylinder via an interval in a monoidal category. 

In parallel with \ref{StructuresUponACylinderOrACoCylinderChapter} once more, we introduce a strictness of left identities hypothesis, a strictness of right identities hypothesis, and a strictness of left inverses hypothesis. If these hypotheses hold for $\interval$, then they hold for $\cylinderinterval$ and $\cocylinderinterval$. 

We refer the reader to \ref{MonoidalCategoryTheoryPreliminariesChapter} for our conventions regarding exponential objects in a monoidal category, and for other preliminary observations to which we shall appeal.

\begin{assum} Let $(\cl{A},\otimes)$ be a monoidal category. Let $1$ denote its unit object, and let $\lambda$ denote its natural isomorphism \ar{- \otimes 1, {-.}} %
Let $\alpha$ denote its natural isomorphism \ar{{(- \otimes -) \otimes -}, {- \otimes (- \otimes -).}} \end{assum}

\begin{defn} An {\em interval} in $\cl{A}$ is an object $I$ of $\cl{A}$, together with a pair of arrows \pair{1,I,i_0,i_1} of $\cl{A}$. \end{defn} 

\begin{rmk} We let \ar{\cl{A},\cl{A},{(-)^{1}}} denote the identity functor, by virtue of Remark \ref{UnitExponentiableRemark}. 

We shall frequently implicitly identify the functor \ar[4]{\cl{A},\cl{A},- \otimes 1} with the identity functor, via the natural isomorphism $\lambda$. \end{rmk}

\begin{defn} \label{CylinderAndCoCylinderFromInterval} Let $\interval = \big( I, i_0, i_1 \big)$ be an interval in $\cl{A}$. Regarding $\cl{A}$ as an object of the 2-category of categories, we denote by $\cylinderinterval$ the cylinder in $\cl{A}$ defined by the functor \ar[4]{\cl{A},\cl{A},- \otimes I} and the natural transformations \pair[{5,0.75}]{\id_{\cl{A}},{- \otimes I.},- \otimes i_{0}, - \otimes i_{1}} 

If $I$ is exponentiable with respect to $\otimes$, we denote by $\cocylinderinterval$ the co-cylinder in $\cl{A}$ defined by the functor \ar{\cl{A},\cl{A},{(-)^{I}}} and the natural transformations \pair[{4,0.75}]{{(-)^{I}},{\id_{\cl{A}}.},{(-)^{i_{0}}},{(-)^{i_{1}}}} \end{defn}

\begin{prpn} \label{CylinderAndCoCylinderFromIntervalAreAdjointProposition} Let $\interval = \big( I, i_0, i_1 \big)$ be an interval in $\cl{A}$. Suppose that $I$ is exponentiable with respect to $\otimes$. Then the cylinder $\cylinderinterval$ in $\cl{A}$ is left adjoint to the co-cylinder $\cocylinderinterval$ in $\cl{A}$. \end{prpn} 

\begin{proof} Let \ar{a_{0} \otimes I,a_{1},h} be an arrow of $\cl{A}$. Since $I$ is exponentiable with respect to $\otimes$, we have a natural isomorphism \ar{{\mathsf{Hom}_{\cl{A}}(- \otimes I, a_{1})},{\mathsf{Hom}_{\cl{A}}(-,(a_{1})^{I}).},\sim} In particular, we have an isomorphism \ar{{\mathsf{Hom}_{\cl{A}}(a_{0} \otimes I,a_{1})},{\mathsf{Hom}_{\cl{A}}(a_{0}, (a_{1})^{I}).},\sim} %
Let us denote this isomorphism by $\mathsf{adj}$. 

By definition of \ar[4]{{(a_{1})^{I}},{a_{1},},{(a_{1})^{i_{0}}}} the following diagram in the category of sets commutes. 

\begin{diagram}

\begin{tikzpicture} [>=stealth]

\matrix [ampersand replacement=\&, matrix of math nodes, column sep=10 em, row sep=4 em, nodes={anchor=center}]
{ 
|(0-0)| \mathsf{Hom}_{\cl{A}}(a_{0},(a_{1})^{I}) \& |(1-0)| \mathsf{Hom}_{\cl{A}}(a_{0} \otimes I,a_{1}) \\ 
|(0-1)| \mathsf{Hom}_{\cl{A}}(a_{0},a_{1}) \& |(1-1)| \mathsf{Hom}_{\cl{A}}(a_{0} \otimes 1, a_{1}) \\
};
	
\draw[->] (0-0) to node[auto] {$\mathsf{adj}^{-1}$} (1-0);
\draw[->] (1-0) to node[auto] {$\mathsf{Hom}_{\cl{A}}(a_{0} \otimes i_{0},a_{1})$} (1-1);
\draw[->] (0-0) to node[auto,swap] {$\mathsf{Hom}_{\cl{A}}(a_{0},(a_{1})^{i_{0}})$} (0-1);
\draw[<-] (0-1) to node[auto,swap] {$\mathsf{Hom}_{\cl{A}}(\lambda^{-1}(a_{0}),a_{1})$} (1-1);

\end{tikzpicture} 

\end{diagram} %
Thus the following diagram in $\cl{A}$ commutes. \sq[{4,3}]{a_{0},a_{0} \otimes I,{(a_{1})^{I}},a_{1},a_{0} \otimes i_{0},h,\mathsf{adj}(h),{(a_{1})^{i_{0}}}} 

Similarly, by definition of \ar[4]{{(a_{1})^{I}},{a_{1},},{(a_{1})^{i_{1}}}} the following diagram in the category of sets commutes. 

\begin{diagram}

\begin{tikzpicture} [>=stealth]

\matrix [ampersand replacement=\&, matrix of math nodes, column sep=10 em, row sep=4 em, nodes={anchor=center}]
{ 
|(0-0)| \mathsf{Hom}_{\cl{A}}(a_{0},(a_{1})^{I}) \& |(1-0)| \mathsf{Hom}_{\cl{A}}(a_{0} \otimes I,a_{1}) \\ 
|(0-1)| \mathsf{Hom}_{\cl{A}}(a_{0},a_{1}) \& |(1-1)| \mathsf{Hom}_{\cl{A}}(a_{0} \otimes 1, a_{1}) \\
};
	
\draw[->] (0-0) to node[auto] {$\mathsf{adj}^{-1}$} (1-0);
\draw[->] (1-0) to node[auto] {$\mathsf{Hom}_{\cl{A}}(a_{0} \otimes i_{1},a_{1})$} (1-1);
\draw[->] (0-0) to node[auto,swap] {$\mathsf{Hom}_{\cl{A}}(a_{0},(a_{1})^{i_{1}})$} (0-1);
\draw[<-] (0-1) to node[auto,swap] {$\mathsf{Hom}_{\cl{A}}(\lambda^{-1}(a_{0}),a_{1})$} (1-1);

\end{tikzpicture} 

\end{diagram} Hence the following diagram in $\cl{A}$ commutes. \sq[{4,3}]{a_{0},a_{0} \otimes I,{a_{1}^{I}},a_{1},a_{0} \otimes i_{1},h,\mathsf{adj}(h),{a_{1}^{i_{1}}}} 

\end{proof}

\begin{defn} Let $\interval =\big( I, i_0, i_1 \big)$ be an interval in $\cl{A}$. A {\em contraction structure} with respect to $\interval$ is an arrow \ar{I,1,p} of $\cl{A}$, such that the following diagrams in $\cl{A}$ commute. \twotriangles{1,I,1,i_{0},p,id,1,I,1,i_{1},p,id} \end{defn}

\begin{rmk} Let $\interval = \big( I, i_0, i_1 \big)$ be an interval in $\cl{A}$. Suppose that $1$ is a final object of $\cl{A}$. This is the case, for example, if the monoidal structure upon $\cl{A}$ is cartesian. Then the canonical arrow \ar{I,1,} of $\cl{A}$ defines a contraction structure with respect to $\interval$. \end{rmk}

\begin{rmk} \label{IntervalWithContractionGivesCylinderWithContractionAndCocylinderWithExpansionRemark} Let $\interval = \big( I, i_0, i_1, p \big)$ be an interval in $\cl{A}$ equipped with a contraction structure $p$. Then the natural transformation \ar[4]{- \otimes I, \id_{\cl{A}},- \otimes p} equips the cylinder $\cylinderinterval$ in $\cl{A}$ with a contraction structure. 

If $I$ is exponentiable with respect to $\otimes$, the natural transformation \ar{\id_{\cl{A}},{(-)^{I}},{(-)^{p}}} equips the co-cylinder $\cocylinderinterval$ in $\cl{A}$ with a contraction structure. \end{rmk}

\begin{prpn} \label{AdjunctionBetweenCylinderAndCoCylinderFromIntervalCompatibleWithContractionAndExpansionProposition} Let $\interval = \big( I, i_0, i_1,p \big)$ be an interval in $\cl{A}$ equipped with a contraction structure $p$. Suppose that $I$ is exponentiable with respect to $\otimes$. We regard the cylinder $\cylinderinterval$ in $\cl{A}$ as equipped with the contraction structure defined by $- \otimes p$, and regard the co-cylinder $\cocylinderinterval$ in $\cl{A}$ as equipped with the contraction structure defined by $(-)^{p}$. 

Recall by Proposition \ref{CylinderAndCoCylinderFromIntervalAreAdjointProposition} that $\cylinderinterval$ is left adjoint to $\cocylinderinterval$. We moreover have that the adjunction between $- \otimes I$ and $(-)^{I}$ is compatible with $- \otimes p$ and $(-)^{p}$. \end{prpn}

\begin{proof} Let \ar{a_{0},a_{1},f} be an arrow of $\cl{A}$. Since $I$ is exponentiable with respect to $\otimes$, we have a natural isomorphism \ar{{\mathsf{Hom}_{\cl{A}}(- \otimes I, a_{1})},{\mathsf{Hom}_{\cl{A}}(-,(a_{1})^{I}).},\sim} %
In particular, we have an isomorphism \ar{{\mathsf{Hom}_{\cl{A}}(a_{0} \otimes I,a_{1})},{\mathsf{Hom}_{\cl{A}}(a_{0}, (a_{1})^{I}).},\sim} 

Let us denote this isomorphism by $\mathsf{adj}$. By definition of \ar[4]{a_{1},{(a_{1})^{I},},{(a_{1})^{p}}} the following diagram in the category of sets commutes. 

\begin{diagram}

\begin{tikzpicture} [>=stealth]

\matrix [ampersand replacement=\&, matrix of math nodes, column sep=10 em, row sep=4 em, nodes={anchor=center}]
{ 
|(0-0)| \mathsf{Hom}_{\cl{A}}(a_{0},a_{1}) \& |(1-0)| \mathsf{Hom}_{\cl{A}}(a_{0} \otimes 1,a_{1}) \\ 
|(0-1)| \mathsf{Hom}_{\cl{A}}(a_{0},(a_{1})^{I}) \& |(1-1)| \mathsf{Hom}_{\cl{A}}(a_{0} \otimes I, a_{1}) \\
};

\draw[->] (0-0) to node[auto] {$\mathsf{Hom}_{\cl{A}}(\lambda(a_{0}),a_{1})$} (1-0);
\draw[->] (1-0) to node[auto] {$\mathsf{Hom}_{\cl{A}}(a_{0} \otimes p,a_{1})$} (1-1);
\draw[->] (0-0) to node[auto,swap] {$\mathsf{Hom}_{\cl{A}}(a_{0},(a_{1})^{p})$} (0-1);

\draw[<-] (0-1) to node[auto,swap] {$\mathsf{adj}$} (1-1);

\end{tikzpicture} 

\end{diagram} Thus the following diagram in $\cl{A}$ commutes. \tri{a_{0},a_{1},{(a_{1})^{I}},f,{(a_{1})^{p}},\mathsf{adj}\big( f \circ ( a_{0} \otimes p) \big)} \end{proof}  

\begin{defn}  Let $\interval =\big( I, i_0, i_1 \big)$ be an interval in $\cl{A}$. An {\em involution structure} with respect to $\interval$ is an arrow \ar{I,I,v}of $\cl{A}$, such that the following diagrams in $\cl{A}$ commute. \twotriangles{1,I,I,i_{0},v,i_{1},1,I,I,i_{1},v,i_{0}} \end{defn}

\begin{rmk} Let $\interval = \big( I, i_0, i_1, v \big)$ be an interval in $\cl{A}$ equipped with an involution structure $v$. Then the natural transformation \ar[5]{- \otimes I, - \otimes I, - \otimes v} equips the cylinder $\cylinderinterval$ in $\cl{A}$ with an involution structure. 

If $I$ is exponentiable with respect to $\otimes$, the natural transformation \ar{{(-)^{I}},{(-)^{I}},{(-)^{v}}} equips the co-cylinder $\cocylinderinterval$ in $\cl{A}$ with an involution structure. \end{rmk}

\begin{defn} Let $\interval = \big( I, i_{0}, i_{1}, p \big)$ be an interval in $\cl{A}$ equipped with a contraction structure $p$. An involution structure $v$ with respect to $\interval$ is {\em compatible with $p$} if the following diagram in $\cl{A}$ commutes. \tri{I,I,1,v,p,p} \end{defn}

\begin{rmk} Let $\interval = \big( I, i_{0}, i_{1}, p, v \big)$ be an interval in $\cl{A}$ equipped with a contraction structure $p$ and an involution structure $v$ compatible with $p$. Then the involution structure $- \otimes v$ with respect to the cylinder $\cylinderinterval$ in $\cl{A}$ is compatible with the contraction structure $- \otimes p$.

If $I$ is exponentiable with respect to $\otimes$, the involution structure $(-)^{v}$ with respect to the co-cylinder $\cocylinderinterval$ in $\cl{A}$ is compatible with the contraction structure $(-)^{p}$. \end{rmk}

\begin{defn} Let $\interval = \big( I, i_0, i_1 \big)$ be an interval in $\cl{A}$. A {\em subdivision structure} with respect to $\interval$ is an object $S$ of $\cl{A}$, together with a pair of arrows \pair{I,S,r_{0},r_{1}} of $\cl{A}$, such that the square \sq{1,I,I,S,i_{0},r_{0},i_{1},r_{1}} in $\cl{A}$ is co-cartesian, and an arrow \ar{I,S,s} of $\cl{A}$, such that the following diagrams in $\cl{A}$ commute. \twosq{1,I,I,S,i_{0},s,i_{0},r_{1},1,I,I,S,i_{1},s,i_{1},r_{0}} \end{defn}

\begin{requirement} \label{SubdivisionRequirement} Let $\interval = \big( I, i_{0}, i_{1} \big)$ be an interval in $\cl{A}$, and let $\big( S, r_{0}, r_{1}, s \big)$ be a subdivision structure with respect to $\interval$. Then for any object $a$ of $\cl{A}$, the square \sq{a,a \otimes I, a \otimes I, a \otimes S, a \otimes i_{0}, a \otimes r_{0}, a \otimes i_{1}, a \otimes r_{1}} in $\cl{A}$ is co-cartesian. \end{requirement}

\begin{rmk} \label{SubdivisionRequirementRemark} Requirement \ref{SubdivisionRequirement} is satisfied whenever the monoidal structure on $\cl{A}$ is closed. It is also satisfied, for example, by the usual interval in the category of topological spaces, equipped with its cartesian monoidal structure. This monoidal structure is not closed. 

Since by Proposition \ref{CylinderAndCoCylinderFromIntervalAreAdjointProposition} we have that $\cylinderinterval$ is left adjoint to $\cocylinderinterval$, Requirement \ref{SubdivisionRequirement} is equivalent to the dual requirement that, for any object $a$ of $\cl{A}$, the square \sq{a^{S},a^{I},a^{I},a,a^{r_{0}},a^{i_{0}},a^{r_{1}},a^{i_{1}}} in $\cl{A}$ is cartesian. \end{rmk}

\begin{rmk} Let $\interval = \big( I, i_{0}, i_{1} \big)$ be an interval in $\cl{A}$, and let $\big( S, r_{0}, r_{1}, s \big)$ be a subdivision structure with respect to $\interval$, such that Requirement \ref{SubdivisionRequirement} holds. Then the functor \ar[4]{\cl{A},\cl{A},- \otimes S} and the natural transformations \pair[{5,0.75}]{- \otimes I, - \otimes S, - \otimes r_{0}, - \otimes r_{1}} and \ar[5]{- \otimes I, - \otimes S, - \otimes s} define a subdivision structure with respect to the cylinder $\cylinderinterval$ in $\cl{A}$. Moreover, $- \otimes I$ preserves subdivision with respect to $\cylinderinterval$ and $\big(- \otimes S, - \otimes r_{0}, - \otimes r_{1}, - \otimes s \big)$. 

If both $I$ and $S$ are exponentiable with respect to $\otimes$ then the functor \ar[4]{\cl{A},\cl{A},{(-)^{S}}} and the natural transformations \pair[{5,0.75}]{{(-)^{S}}, {(-)^{I}}, {(-)^{r_{0}}}, {(-)^{r_{1}}}} and \ar[5]{{(-)^{S}}, {(-)^{I}}, {(-)^{s}}} defines a subdivision structure with respect to the co-cylinder $\cocylinderinterval$ in $\cl{A}$. Moreover $(-)^{I}$ preserves subdivision with respect to $\cocylinderinterval$ and $\big( {(-)^{S}}, {(-)^{r_{0}}}, {(-)^{r_{1}}}, {(-)^{s}} \big)$.  \end{rmk} 

\begin{defn} Let $\interval = \big( I, i_{0}, i_{1}, p, S, r_{0}, r_{1}, s \big)$ be an interval in $\cl{A}$ equipped with a contraction structure $p$ and a subdivision structure $\big( \subdiv, r_{0}, r_{1}, s \big)$. Let \ar{S,1,\overline{p}} denote the canonical arrow of $\cl{A}$ such that the following diagram in $\cl{A}$ commutes. \pushout{1,I,I,S,1,i_0,r_0,i_1,r_1,p,p,\overline{p}} %
The subdivision structure $\big( S, r_{0}, r_{1}, s \big)$ is {\em compatible with $p$} if the following diagram in $\cl{A}$ commutes. \tri{I,S,1,s,\overline{p},p} \end{defn}

\begin{rmk} Let $\interval = \big( I, i_{0}, i_{1}, p, S, r_{0}, r_{1}, s\big)$ be an interval in $\cl{A}$ equipped with a contraction structure $p$, and a subdivision structure $\big( \subdiv, r_{0}, r_{1}, s \big)$ compatible with $p$. Suppose that Requirement \ref{SubdivisionRequirement} holds. Then the subdivision structure $\big( - \otimes S, - \otimes r_{0}, i- \otimes r_{1}, - \otimes s \big)$ with respect to the cylinder $\cylinderinterval$ in $\cl{A}$ is compatible with the contraction structure $- \otimes p$. 

If both $I$ and $S$ are exponentiable with respect to $\otimes$, then the subdivision structure $\big( {(-)^{S}}, {(-)^{r_{0}}}, {(-)^{r_{1}}}, {(-)^{s}} \big)$ with respect to the co-cylinder $\cocylinderinterval$ in $\cl{A}$ is compatible with the contraction structure $(-)^{p}$. \end{rmk}

\begin{notn} Let $\big( I, i_0, i_1 \big)$ be an interval in $\cl{A}$. We denote by $I^{2}$ the object $I \otimes I$ of $\cl{A}$. \end{notn}

\begin{rmk} Let $\big( I, i_0, i_1 \big)$ be an interval in $\cl{A}$. We shall frequently implicitly identify the functor $(- \otimes I) \otimes I$ with the functor $- \otimes I^{2}$, via the natural isomorphism $\alpha$. Similarly, we shall frequently implicitly identify the functor  $(- \otimes I) \otimes 1$ with the functor $- \otimes (I \otimes 1)$, via $\alpha$. \end{rmk} 

\begin{defn} Let $\interval = \big( I, i_0, i_1, p \big)$ be an interval in $\cl{A}$ equipped with a contraction structure $p$. An {\em upper left connection structure} with respect to $\interval$ is an arrow \ar{I^{2},I,\Gamma_{ul}}of $\cl{A}$, such that the following diagrams in $\cl{A}$ commute.  \twotriangles[{4,3,0}]{I,I^{2},I,i_{0} \otimes I, \Gamma_{ul},id,I,I^{2},I,I \otimes i_{0},\Gamma_{ul},id} %
\twosq[{4,3,0}]{I,I^{2},1,I,i_{1} \otimes I,\Gamma_{ul},p,i_{1},I,I^{2},1,I,I \otimes i_{1},\Gamma_{ul},p,i_{1}} \end{defn}

\begin{rmk} Let $\interval = \big( I, i_0, i_1, p, \Gamma_{ul} \big)$ be an interval in $\cl{A}$ equipped with a contraction structure $p$, and an upper left connection structure $\Gamma_{ul}$. Let us regard the cylinder $\cylinderinterval$ in $\cl{A}$ as equipped with the contraction structure $- \otimes p$, Then the natural transformation \ar[5]{- \otimes I^{2}, - \otimes I, - \otimes \Gamma_{ul}} equips $\cylinderinterval$ with an upper left connection structure. 

Suppose that $I$ is exponentiable with respect to $\otimes$. Let us regard the co-cylinder $\cocylinderinterval$ in $\cl{A}$ as equipped with the contraction structure $(-)^{p}$. Then the natural transformation \ar[4]{{(-)^{I}},{(-)^{I^{2}}},{(-)^{\Gamma_{ul}}}} equips $\cocylinderinterval$ with an upper left connection structure. \end{rmk}

\begin{defn} Let $\interval = \big( I, i_0, i_1, p \big)$ be an interval in $\cl{A}$ equipped with a contraction structure $p$. A {\em lower right connection structure} with respect to $\interval$ is an arrow \ar{I^{2},I,\Gamma_{lr}} of $\cl{A}$, such that the following diagrams in $\cl{A}$ commute.  \twotriangles[{4,3,0}]{I,I^{2},I,i_{1} \otimes I, \Gamma_{lr},id,I,I^{2},I,I \otimes i_{1},\Gamma_{lr},id} %
\twosq[{4,3,0}]{I,I^{2},1,I,i_{0} \otimes I,\Gamma_{lr},p,i_{0},I,I^{2},1,I,I \otimes i_{0},\Gamma_{lr},p,i_{0}} \end{defn}

\begin{rmk} Let $\interval = \big( I, i_0, i_1, p, \Gamma_{lr} \big)$ be an interval in $\cl{A}$ equipped with a contraction structure $p$, and a lower right connection structure $\Gamma_{lr}$. Let us regard the cylinder $\cylinderinterval$ in $\cl{A}$ as equipped with the contraction structure $- \otimes p$. Then the natural transformation \ar[5]{- \otimes I^{2}, - \otimes I, - \otimes \Gamma_{lr}} equips $\cylinderinterval$ with a lower right connection structure. 

Suppose that $I$ is exponentiable with respect to $\otimes$. Let us regard the co-cylinder $\cocylinderinterval$ in $\cl{A}$ as equipped with the contraction structure $(-)^{p}$. Then the natural transformation \ar[4]{{(-)^{I}},{(-)^{I^{2}}},{(-)^{\Gamma_{lr}}}} equips $\cocylinderinterval$ with a lower right connection structure. \end{rmk}

\begin{defn} Let $\interval = \big( I, i_{0}, i_{1}, p \big)$ be an interval in $\cl{A}$ equipped with a contraction structure $p$. A lower right connection structure $\Gamma_{lr}$ with respect to $\interval$ is {\em compatible with $p$} if the following diagram in $\cl{A}$ commutes. \sq{I^{2},I,I,1,\Gamma_{lr},p,I \otimes p,p} \end{defn} 

\begin{rmk} Let $\interval = \big( I, i_{0}, i_{1}, p, \Gamma_{lr} \big)$ be an interval in $\cl{A}$ equipped with a contraction structure $p$ and a lower right connection structure $\Gamma_{lr}$. If $\Gamma_{lr}$ is compatible with $p$, then the lower right connection structure $- \otimes \Gamma_{lr}$ with respect to the cylinder $\cylinderinterval$ in $\cl{A}$ is compatible with the contraction structure $- \otimes p$. 

Suppose that $I$ is exponentiable with respect to $\otimes$. Then the lower right connection structure $(-)^{\Gamma_{lr}}$ with respect to the co-cylinder $\cocylinderinterval$ in $\cl{A}$ is compatible with the contraction structure $(-)^{p}$. \end{rmk} 

\begin{rmk} We shall not need to consider compatibility of an upper left connection structure upon an interval with a contraction structure, or compatibility of an upper right connection structure upon an interval, which we shall define next, with a contraction structure. \end{rmk}

\begin{defn} Let $\interval = \big( I, i_0, i_1, p, v \big)$ be an interval in $\cl{A}$ equipped with a contraction structure $p$, and an involution structure $v$. An {\em upper right connection structure} with respect to $\interval$ is an arrow \ar{I^{2},I,\Gamma_{ur}} of $\cl{A}$, such that the following diagrams in $\cl{A}$ commute. 

\begin{diagram}

\begin{tikzpicture} [>=stealth]

\matrix [ampersand replacement=\&, matrix of math nodes, column sep=5 em, row sep=3 em, nodes={anchor=center}]
{ 
|(0-0)| I \& |(1-0)| I^{2} \&[2em] |(2-0)| I \& |(3-0)| I^{2} \\ 
          \& |(1-1)| I     \&                \& |(3-1)| I     \\
|(0-2)| I \& |(1-2)| I^{2} \&      |(2-2)| I \& |(3-2)| I^{2} \\
|(0-3)| 1 \& |(1-3)| I     \&      |(2-3)| 1 \& |(3-3)| I     \\      
};
	
\draw[->] (0-0) to node[auto] {$I \otimes i_{0}$} (1-0);
\draw[->] (1-0) to node[auto] {$\Gamma_{ur}$} (1-1);
\draw[->] (0-0) to node[auto,swap] {$id$} (1-1);
\draw[->] (2-0) to node[auto] {$i_{1} \otimes I$} (3-0);
\draw[->] (3-0) to node[auto] {$\Gamma_{ur}$} (3-1);
\draw[->] (2-0) to node[auto,swap] {$v$} (3-1);
\draw[->] (0-2) to node[auto] {$i_{0} \otimes I$} (1-2);
\draw[->] (1-2) to node[auto] {$\Gamma_{ur}$} (1-3);
\draw[->] (0-2) to node[auto,swap] {$p$} (0-3);
\draw[->] (0-3) to node[auto,swap] {$i_{0}$} (1-3);
\draw[->] (2-2) to node[auto] {$I \otimes i_{1}$} (3-2);
\draw[->] (3-2) to node[auto] {$\Gamma_{ur}$} (3-3);
\draw[->] (2-2) to node[auto,swap] {$p$} (2-3);
\draw[->] (2-3) to node[auto,swap] {$i_{0}$} (3-3);

\end{tikzpicture} 

\end{diagram} \end{defn}

\begin{rmk} Let $\interval = \big( I, i_0, i_1, p, v, \Gamma_{ur} \big)$ be an interval in $\cl{A}$ equipped with a contraction structure $p$, an involution structure $v$, and an upper right connection structure $\Gamma_{ur}$. Let us regard the cylinder $\cylinderinterval$ in $\cl{A}$ as equipped with the contraction structure $- \otimes p$ and the involution structure $- \otimes v$. Then the natural transformation \ar[5]{- \otimes I^{2}, - \otimes I, - \otimes \Gamma_{ur}} equips $\cylinderinterval$ in $\cl{A}$ with an upper right connection structure. 

Suppose that $I$ is exponentiable with respect to $\otimes$. Let us regard the co-cylinder $\cocylinderinterval$ in $\cl{A}$ as equipped with the contraction structure $(-)^{p}$ and the involution structure $(-)^{v}$. Then the natural transformation \ar[4]{{(-)^{I}},{(-)^{I^{2}}},{(-)^{\Gamma_{ur}}}} equips $\cocylinderinterval$ in $\cl{A}$ with an upper right connection structure. \end{rmk}

\begin{rmk} Analogously, one can define a {\em lower left connection structure} with respect to an interval. Everything concerning upper and lower right connections below can equally be carried out for upper and lower left connections. \end{rmk}

\begin{defn} \label{RightConnectionsIntervalCompatibilityDefinition} Let $\interval = \big( I, i_0, i_1, p, v, S, r_0, r_1, s, \Gamma_{lr}, \Gamma_{ur} \big)$ be an interval in $\cl{A}$ equipped with a contraction structure $p$, an involution structure $v$, a subdivision structure $\big( S, r_{0}, r_{1}, s \big)$, a lower right connection structure $\Gamma_{lr}$, and an upper right connection structure $\Gamma_{ur}$. Let \ar{I \otimes S,I,x} denote the canonical arrow of $\cl{A}$ such that the following diagram in $\cl{A}$ commutes. \pushout[{4,3,-1}]{I,I^{2},I^{2},I \otimes S,I,I \otimes i_{0},I \otimes r_0,I \otimes i_1,I \otimes r_1, \Gamma_{lr},\Gamma_{ur},x} Then $\Gamma_{lr}$ and $\Gamma_{ur}$ are {\em compatible with $\big( S, r_{0}, r_{1}, s \big)$} if the following diagram in $\cl{A}$ commutes. \tri[{4,3}]{I^{2},I \otimes S,I,I \otimes s,x,I \otimes p} \end{defn}

\begin{rmk} Let $\interval = \big( I, i_0, i_1, p, v, S, r_0, r_1, s, \Gamma_{lr}, \Gamma_{ur} \big)$ be an interval in $\cl{A}$ equipped with a contraction structure $p$, an involution structure $v$, a subdivision structure $\big( S, r_{0}, r_{1}, s \big)$, a lower right connection structure $\Gamma_{lr}$, and an upper right connection structure $\Gamma_{ur}$. Suppose that $\Gamma_{lr}$ and $\Gamma_{ur}$ are compatible with $\big( S, r_{0}, r_{1}, s \big)$, and that Requirement \ref{SubdivisionRequirement} holds. 

Then the right connections $- \otimes \Gamma_{lr}$ and $- \otimes \Gamma_{ur}$ are compatible with the subdivision structure $\big( - \otimes S, - \otimes r_{0}, - \otimes r_{1}, - \otimes s \big)$ upon the cylinder $\cylinderinterval$ in $\cl{A}$ equipped with the contraction structure $- \otimes p$. 

Suppose that both $I$ and $S$ are exponentiable with respect to $\otimes$. Then the right connections $(-)^{\Gamma_{lr}}$ and $(-)^{\Gamma_{ur}}$ are compatible with the subdivision structure $\big( (-)^{S}, (-)^{r_{0}}, (-)^{r_{1}}, (-)^{s} \big)$ upon the co-cylinder $\cocylinderinterval$ in $\cl{A}$ equipped with the contraction structure $(-)^{p}$. \end{rmk}

\begin{defn} \label{StrictnessLeftIdentitiesIntervalDefinition} Let $\interval = \big( I, i_0, i_1, p, S, r_0, r_1, s \big)$ be an interval in $\cl{A}$ equipped with a contraction structure $p$, and a subdivision structure $\big( S, r_{0}, r_{1}, s \big)$. Let \ar{S,I,q_{l}} denote the canonical arrow of $\cl{A}$ such that the following diagram in $\cl{A}$ commutes. \pushout{1,I,I,S,I,i_{0},r_{0},i_{1},r_{1},id,i_{0} \circ p,q_{l}} %
Then $\interval$ has {\em strictness of left identities} if the following diagram in $\cl{A}$ commutes. \tri{I,S,I,s,q_{l},id} \end{defn}

\begin{rmk} Let $\interval = \big( I, i_0, i_1, p, S, r_0, r_1, s \big)$ be an interval in $\cl{A}$ equipped with a contraction structure $p$ and a subdivision structure $\big( S, r_{0}, r_{1}, s \big)$. Suppose that Requirement \ref{SubdivisionRequirement} holds and that $\interval$ has strictness of left identities. 

Then the cylinder $\cylinderinterval$ in $\cl{A}$ equipped with the contraction structure $- \otimes p$ and the subdivision structure $\big( - \otimes S, - \otimes r_{0}, - \otimes r_{1}, - \otimes s \big)$ has strictness of left identities. 

If both $I$ and $S$ are exponentiable with respect to $\otimes$, then the co-cylinder $\cocylinderinterval$ in $\cl{A}$ equipped with the contraction structure $(-)^{p}$ and the subdivision structure $\big( (-)^{S}, (-)^{r_{0}}, (-)^{r_{1}}, (-)^{s} \big)$ has strictness of left identities. \end{rmk} 

\begin{defn} \label{StrictnessRightIdentitiesIntervalDefinition} Let $\interval = \big( I, i_0, i_1, p, S, r_0, r_1, s \big)$ be an interval in $\cl{A}$ equipped with a contraction structure $p$, and a subdivision structure $\big( S, r_{0}, r_{1}, s \big)$. Let \ar{S,I,q_{r}} denote the canonical arrow of $\cl{A}$ such that the following diagram in $\cl{A}$ commutes. \pushout{1,I,I,S,I,i_{0},r_{0},i_{1},r_{1},i_{1} \circ p,id,q_{r}} %
Then $\interval$ has {\em strictness of right identities} if the following diagram in $\cl{A}$ commutes. \tri{I,S,I,s,q_{r},id} \end{defn}

\begin{rmk} Let $\interval = \big( I, i_0, i_1, p, S, r_0, r_1, s \big)$ be an interval in $\cl{A}$ equipped with a contraction structure $p$ and a subdivision structure $\big( S, r_{0}, r_{1}, s \big)$. Suppose that Requirement \ref{SubdivisionRequirement} holds, and that $\interval$ has strictness of right identities. 

Then the cylinder $\cylinderinterval$ in $\cl{A}$ equipped with the contraction structure $- \otimes p$ and the subdivision structure $\big( - \otimes S, - \otimes r_{0}, - \otimes r_{1}, - \otimes s \big)$ has strictness of right identities. 

If both $I$ and $S$ are exponentiable with respect to $\otimes$, then the co-cylinder $\cocylinderinterval$ in $\cl{A}$ equipped with the contraction structure $(-)^{p}$ and the subdivision structure $\big( (-)^{S}, (-)^{r_{0}}, (-)^{r_{1}}, (-)^{s} \big)$ has strictness of right identities. \end{rmk} 

\begin{defn}  Let $\interval = \big( I, i_0, i_1, p, S, r_0, r_1, s \big)$ be an interval in $\cl{A}$ equipped with a contraction structure $p$ and a subdivision structure $\big( S, r_{0}, r_{1}, S \big)$. Then $\interval$ has {\em strictness of identities} if it has both strictness of left identities and strictness of right identities. \end{defn}

\begin{defn} \label{StrictnessLeftInversesDefinition} Let $\interval = \big( I, i_0, i_1, v, S, r_0, r_1, s \big)$ be an interval in $\cl{A}$ equipped with an involution structure $v$, and a subdivision structure $\big( S, r_{0}, r_{1}, s \big)$. Let \ar{S,I,w} denote the canonical arrow of $\cl{A}$ such that the following diagram in $\cl{A}$ commutes. \pushout{1,I,I,S,I,i_{0},r_{0},i_{1},r_{1},id,v,w} Then $\interval$ has {\em strictness of left inverses} if the following diagram in $\cl{A}$ commutes. \sq{I,S,1,I,s,w,p,i_1} \end{defn}

\begin{rmk} Let $\interval = \big( I, i_0, i_1, v, S, r_0, r_1, s \big)$ be an interval in $\cl{A}$ equipped with an involution structure $v$ and a subdivision structure $\big( S, r_{0}, r_{1}, s \big)$. Suppose that Requirement \ref{SubdivisionRequirement} holds. Then the cylinder $\cylinderinterval$ in $\cl{A}$ equipped with the involution structure $- \otimes v$ and the subdivision structure $\big( - \otimes S, - \otimes r_{0}, - \otimes r_{1}, - \otimes s \big)$ has strictness of left inverses. 

If both $I$ and $S$ are exponentiable with respect to $\otimes$, then the co-cylinder $\cocylinderinterval$ in $\cl{A}$ equipped with the involution structure $(-)^{v}$ and the subdivision structure  $\big( (-)^{S}, (-)^{r_{0}}, (-)^{r_{1}}, (-)^{s} \big)$ has strictness of left inverses. \end{rmk}

\begin{rmk} We shall not need to consider strictness of right inverses. \end{rmk}

\end{chapter}

\begin{chapter}{Homotopy and relative homotopy} \label{HomotopyAndRelativeHomotopyChapter}

In \ref{StructuresUponACylinderOrACoCylinderChapter} we introduced a notion of a cylinder or a co-cylinder in a formal category $\cl{A}$. We now explain that a cylinder gives rise to a notion of homotopy between arrows of $\cl{A}$. In a dual manner, a co-cylinder gives rise to a notion of homotopy between arrows of $\cl{A}$. Thus a cylinder or a co-cylinder allows us to define a notion of homotopy equivalence in $\cl{A}$. 

If $\cl{A}$ admits both a cylinder and a co-cylinder, as we shall assume later in this work, it will be crucial for us to know that the corresponding notions of homotopy coincide. In \ref{CylindricalAdjunctionsChapter} we mentioned that an adjunction between the cylinder and the co-cylinder ensures that this holds. We shall now be able to observe this.

Homotopy theory with respect to a cylinder or a co-cylinder alone is rather spartan. The structures upon a cylinder or co-cylinder defined in \ref{StructuresUponACylinderOrACoCylinderChapter} allow for a much richer theory, which we shall explore in the remainder of this work. 

At first, all of our constructions will be abstractions from the homotopy theory of topological spaces. Later on, for instance in \ref{MappingCylindersAndMappingCoCylindersChapter} when we shall require that the strictness of identities hypotheses introduced in \ref{StructuresUponACylinderOrACoCylinderChapter} hold, topological spaces will no longer be our guide.

A contraction structure allows us to construct identity homotopies. An involution structure allows us to reverse homotopies. A subdivision structure allows to compose homotopies. In the presence of both an involution structure and a subdivision structure, homotopy equivalences in $\cl{A}$ have the two out of three property. 

Given a pair of commutative diagrams \twosq{a_{2},a_{0},a_{3},a_{1},g_{0},f,f',g_{1},a_{0},a_{2},a_{1},a_{3},r_{0},f,f',r_{1}} in $\cl{A}$, where $r_{0}$ is a retraction of $g_{0}$ and $r_{1}$ is a retraction of $g_{1}$, we demonstrate that if $f$ is a homotopy equivalence then so is $f'$. We make a technical observation concerning homotopy inverses, which we shall appeal to in our proof in \ref{DoldTheoremChapter} of Dold's theorem. 

We introduce the notion of a double homotopy with respect to a cylinder, and explain a pictorial notation. Double homotopies will play an indispensable role throughout this work. Our three flavours of connection structures allow us to construct double homotopies with specific boundary homotopies.

With respect to a cylinder or a co-cylinder equipped with a contraction structure, we define a notion of homotopy under or over an object of $\cl{A}$. If $\cl{A}$ admits both a cylinder $\cylinder$ equipped with a contraction structure $p$, and a co-cylinder $\cocylinder$ equipped with a contraction structure $c$, then an adjunction between $\cylinder$ and $\cocylinder$ which is compatible with $p$ and $c$ allows us to observe that the notion of a homotopy equivalence under (respectively over) an object with respect to $\cylinder$ coincides with that of a homotopy equivalence under (respectively over) an object with respect to $\cocylinder$.

Identity homotopies are also identity homotopies under or over an object. An involution structure which is compatible with contraction allows us to construct reverse homotopies under or over an object. A subdivision structure which is compatible with contraction allows us to compose homotopies under or over an object.  

We conclude by introducing the notion of a strong deformation retraction with respect to a cylinder or a co-cylinder. All consideration of homotopies under or over an object in this work relates to strong deformation retractions.

As discussed in \ref{IntroductionChapter}, homotopy with respect to a cylinder in an abstract setting was first considered by Kan in \cite{KanAbstractHomotopyTheoryII}. The insight that further structure upon a cylinder can give rise to a richer theory is due to Kamps, presented in works such as \cite{KampsKanBedingungenUndAbstrakteHomotopietheorie} from around 1970.  

In a setting closer to that in which we are working, the observation that a subdivision structure upon a cylinder allows one to compose homotopies was first explored, to the author's knowledge, by Grandis in  papers such as \cite{GrandisCategoricallyAlgebraicFoundationsForHomotopicalAlgebra}, written in the 1990s. The abstract notion of homotopy under and over an object is for example discussed for in the book \cite{KampsPorterAbstractHomotopyAndSimpleHomotopyTheory} of Kamps and Porter, which also treats much of the rest of this section.

\begin{assum} Let $\cl{C}$ be a 2-category with a final object. Suppose that pushouts and pullbacks of 2-arrows of $\cl{C}$ give rise to pushouts and pullbacks in formal categories, in the sense of Definition \ref{PushoutsPullbacks2ArrowsArePushoutsPullbacksInFormalCategoriesTerminology}. Let $\cl{A}$ be an object of $\cl{C}$. Recall that we view $\cl{A}$ as a formal category, writing of objects and arrows of $\cl{A}$. This terminology and all other formal category theory preliminaries can be found in \ref{FormalCategoryTheoryPreliminariesChapter}. \end{assum}

\begin{defn} \label{HomotopyCylinderDefinition} Let $\cylinder = \big( \cyl, i_0, i_1, \big)$ be a cylinder in $\cl{A}$, and let \pair{a_{0},a_{1},f_{0},f_{1}} be arrows of $\cl{A}$. A {\em homotopy} from $f_{0}$ to $f_{1}$ with respect to $\cylinder$ is an arrow \ar{\cyl(a_{0}),a_{1},h} of $\cl{A}$, such that the following diagrams in $\cl{A}$ commute. \twotriangles[{4,3,2}]{a_{0},\cyl(a_{0}),a_{1},i_{0}(a_{0}),h,f_{0},a_{0},\cyl(a_{0}),a_{1},i_{1}(a_{0}),h,f_{1}} \end{defn}

\begin{defn} Let $\cocylinder = \big( \cocyl, e_0, e_1 \big)$ be a co-cylinder in $\cl{A}$, and let \pair{a_{0},a_{1},f_{0},f_{1}} be arrows of $\cl{A}$. 

A {\em homotopy} from $f_{0}$ to $f_{1}$ with respect to $\cocylinder$ is an arrow \ar{a_{0},\cocyl(a_{1}),h} of $\cl{A}$, such that the following diagrams in $\cl{A}$ commute. \twotriangles[{4,3,2}]{a_{0},\cocyl(a_{1}),a_{1},h,e_{0}(a_{1}),f_{0},a_{0},\cocyl(a_{1}),a_{1},h,e_{1}(a_{1}),f_{1}} \end{defn} 

\begin{rmk} Let $\cocylinder = \big( \cocyl, e_0, e_1 \big)$ be a co-cylinder in $\cl{A}$, and let \pair{a_{0},a_{1},f_{0},f_{1}} be arrows of $\cl{A}$. An arrow \ar{a_{0},\cocyl(a_{1},h} of $\cl{A}$ is a homotopy from $f_{0}$ to $f_{1}$ if and only if $h^{op}$ is a homotopy from $f_{0}^{op}$ to $f_{1}^{op}$ with respect to the cylinder $\cocylinder^{op}$ in $\cl{A}^{op}$. \end{rmk}

\begin{prpn} \label{HomotopyCylinderGivesHomotopyCoCylinderProposition} Let $\cylinder = \big( \cyl, i_0, i_1 \big)$ be a cylinder in $\cl{A}$, and let $\cocylinder = \big( \cocyl, e_0, e_1 \big)$ be a co-cylinder in $\cl{A}$. Suppose that $\cylinder$ is left adjoint to $\cocylinder$. Let \ar{{\mathsf{Hom}_{\cl{A}}\big(\cyl(-),- \big)},{\mathsf{Hom}_{\cl{A}}\big(-,\cocyl(-) \big)},\mathsf{adj}} denote the corresponding natural isomorphism, adopting the shorthand of Recollection \ref{AdjunctionRecollection}. 

Let \pair{a_{0},a_{1},f_{0},f_{1}} be arrows of $\cl{A}$, and suppose that \ar{\cyl(a_{0}),a_{1},h} defines a homotopy from $f_{0}$ to $f_{1}$ with respect to $\cylinder$. Then the arrow \ar[4]{a_{0},\cocyl(a_{1}),\mathsf{adj}(h)} of $\cl{A}$ defines a homotopy from $f_{0}$ to $f_{1}$ with respect to $\cocylinder$. \end{prpn} 

\begin{proof} Follows immediately from the fact that $\cylinder$ is left adjoint to $\cocylinder$. \end{proof}

\begin{cor} \label{HomotopyCoCylinderGivesHomotopyCylinderCorollary} Let $\cylinder = \big( \cyl, i_0, i_1 \big)$ be a cylinder in $\cl{A}$, and let $\cocylinder = \big( \cocyl, e_0, e_1 \big)$ be a co-cylinder in $\cl{A}$. Suppose that $\cylinder$ is left adjoint to $\cocylinder$. Let \ar{{\mathsf{Hom}_{\cl{A}}\big(\cyl(-),- \big)},{\mathsf{Hom}_{\cl{A}}\big(-,\cocyl(-) \big)},\mathsf{adj}} denote the corresponding natural isomorphism, adopting the shorthand of Recollection \ref{AdjunctionRecollection}. 

Let \pair{a_{0},a_{1},f_{0},f_{1}} be arrows of $\cl{A}$, and suppose that \ar{a_{0},\cocyl(a_{1}),h} defines a homotopy from $f_{0}$ to $f_{1}$ with respect to $\cocylinder$. Then the arrow \ar[4]{\cyl(a_{0}),a_{1},\mathsf{adj}^{-1}(h)} of $\cl{A}$ defines a homotopy from $f_{0}$ to $f_{1}$ with respect to $\cylinder$. \end{cor} 

\begin{proof} Follows immediately from Proposition \ref{HomotopyCylinderGivesHomotopyCoCylinderProposition} by duality. \end{proof}

\begin{prpn} \label{IdentityHomotopyProposition} Let $\cylinder = \big( \cyl, i_0, i_1, p \big)$ be a cylinder in $\cl{A}$ equipped with a contraction structure $p$. Let \ar{a_{0},a_{1},f} be an arrow of $\cl{A}$. Then the arrow \ar[5]{\cyl(a_{0}), a_{1}, f \circ p(a_{0})} of $\cl{A}$ defines a homotopy from $f$ to itself with respect to $\cylinder$. \end{prpn}

\begin{proof} Follows immediately from the fact that $p$ defines a contraction structure with respect to $\cylinder$. \end{proof}

\begin{rmk} Let $\cylinder = \big( \cyl, i_0, i_1, p \big)$ be a cylinder in $\cl{A}$ equipped with a contraction structure $p$. Given an arrow $f$ of $\cl{A}$, we refer to the corresponding homotopy of Proposition \ref{IdentityHomotopyProposition} from $f$ to itself as the {\em identity homotopy} from $f$ to itself with respect to $\cylinder$. \end{rmk} 

\begin{prpn} \label{ReverseHomotopyProposition} Let $\cylinder = \big( \cyl, i_0, i_1, v \big)$ be a cylinder in $\cl{A}$ equipped with an involution structure $v$. Let \pair{a_{0}, a_{1}, f_{0}, f_{1}} be arrows of $\cl{A}$, and let \ar{\cyl(a_{0}), a_{1}, h} be a homotopy from $f_{0}$ to $f_{1}$ with respect to $\cylinder$. The arrow \ar[5]{\cyl(a_{0}),a_{1}, h \circ v(a_{0})} of $\cl{A}$ defines a homotopy from $f_{1}$ to $f_{0}$ with respect to $\cylinder$. \end{prpn}
 
\begin{proof} Follows immediately from the fact that $v$ defines an involution structure with respect to $\cylinder$. \end{proof}

\begin{rmk} Let $\cylinder = \big( \cyl, i_0, i_1, v \big)$ be a cylinder in $\cl{A}$ equipped with an involution structure $v$. Given a homotopy $h$ with respect to $\cylinder$ between a pair of arrows of $\cl{A}$, we refer to the corresponding homotopy of Proposition \ref{ReverseHomotopyProposition} as the {\em reverse} of $h$, and denote it by $h^{-1}$. \end{rmk} 

\begin{prpn} \label{CompositionHomotopiesProposition} Let $\cylinder = \big( \cyl, i_0, i_1, \subdiv, r_{0}, r_{1}, s \big)$ be an interval in $\cl{A}$ equipped with a subdivision structure $(\subdiv, r_{0}, r_{1}, s \big)$. Let \ar{\cyl(a_{0}), a_{1}, h} be a homotopy with respect to $\cylinder$ from an arrow \ar{a_{0},a_{1},f_{0}} of $\cl{A}$ to an arrow \ar{a_{0},a_{1},f_{1}}of $\cl{A}$. Let \ar{\cyl(a_{0}), a_{1}, k} be a homotopy with respect to $\cylinder$ from $f_{1}$ to a third arrow \ar{a_{0},a_{1},f_{2}} of $\cl{A}$. 

Let \ar{\subdiv(a_{0}), a_{1}, r} denote the canonical arrow of $\cl{A}$ such that the following diagram in $\cl{A}$ commutes. \pushout{a_{0}, \cyl(a_{0}), \cyl(a_{0}), \subdiv(a_{0}), a_{1}, i_{0}(a_{0}), r_{0}(a_{0}), i_{1}(a_{0}), r_{1}(a_{0}), k, h, r} Then the arrow \ar[5]{\cyl(a_{0}), a_{1}, r \circ s(a_{0})} of $\cl{A}$ defines a homotopy from $f_{0}$ to $f_{2}$ with respect to $\cyl$. \end{prpn}

\begin{proof} The following diagram in $\cl{A}$ commutes. 

\begin{diagram}

\begin{tikzpicture} [>=stealth]

\matrix [ampersand replacement=\&, matrix of math nodes, column sep=6 em, row sep=4 em, nodes={anchor=center}]
{ 
|(0-0)| a_{0} \& |(1-0)| \cyl(a_{0}) \\ 
|(0-1)| \cyl(a_{0}) \& |(1-1)| \subdiv(a_{0}) \\[1em]
|(0-2)| a_{1} \& \\
};
	
\draw[->] (0-0) to node[auto] {$i_{0}(a_{0})$} (1-0);
\draw[->] (1-0) to node[auto] {$s(a_{0})$} (1-1);
\draw[->] (0-0) to node[auto] {$i_{0}(a_{0})$} (0-1);
\draw[->] (0-1) to node[auto,swap] {$r_{1}(a_{0})$} (1-1);
\draw[->] (1-1) to node[auto] {$r$} (0-2);
\draw[->] (0-1) to node[auto] {$h$} (0-2);
\draw[->, bend right=60] (0-0) to node[auto,swap] {$f_{0}$} (0-2);

\end{tikzpicture} 

\end{diagram} The following diagram in $\cl{A}$ also commutes.

\begin{diagram}

\begin{tikzpicture} [>=stealth]

\matrix [ampersand replacement=\&, matrix of math nodes, column sep=6 em, row sep=4 em, nodes={anchor=center}]
{ 
|(0-0)| a_{0} \& |(1-0)| \cyl(a_{0}) \\ 
|(0-1)| \cyl(a_{0}) \& |(1-1)| \subdiv(a_{0}) \\[1em]
|(0-2)| a_{1} \& \\
};
	
\draw[->] (0-0) to node[auto] {$i_{1}(a_{0})$} (1-0);
\draw[->] (1-0) to node[auto] {$s(a_{0})$} (1-1);
\draw[->] (0-0) to node[auto] {$i_{1}(a_{0})$} (0-1);
\draw[->] (0-1) to node[auto,swap] {$r_{0}(a_{0})$} (1-1);
\draw[->] (1-1) to node[auto] {$r$} (0-2);
\draw[->] (0-1) to node[auto] {$k$} (0-2);
\draw[->, bend right=60] (0-0) to node[auto,swap] {$f_{2}$} (0-2);

\end{tikzpicture} 

\end{diagram} \end{proof}

\begin{rmk} Let $\cylinder = \big( \cyl, i_0, i_1, \subdiv, r_{0}, r_{1}, s \big)$ be a cylinder in $\cl{A}$ equipped with a subdivision structure $\big( \subdiv, r_{0}, r_{1}, s \big)$. Let $f_{0}$, $f_{1}$, and $f_{2}$ be arrows of $\cl{A}$, let $h$ be a homotopy from $f_{0}$ to $f_{1}$ with respect to $\cylinder$, and let $k$ be a homotopy from $f_{1}$ to $f_{2}$ with respect to $\cylinder$. 

We denote by $h + k$ the corresponding homotopy of Proposition \ref{CompositionHomotopiesProposition} from $f_{0}$ to $f_{2}$ with respect to $\cylinder$, and refer to it as a {\em composite homotopy}. \end{rmk}

\begin{rmk} Thus if $\cylinder = \big( \cyl, i_0, i_1, p, v, \subdiv, r_{0}, r_{1}, s \big)$ is a cylinder in $\cl{A}$ equipped with a contraction structure $p$, an involution structure $v$, and a subdivision structure $\big( \subdiv, r_{0}, r_{1}, s \big)$, then homotopy with respect to $\cylinder$ defines an equivalence relation upon the arrows of $\cl{A}$. \end{rmk}

\begin{defn} Let $\cylinder = \big( \cyl, i_0, i_1 \big)$ be a cylinder in $\cl{A}$, and let \ar{a_{0},a_{1},f} be an arrow of $\cl{A}$. A {\em homotopy inverse} of $f$ with respect to $\cylinder$ is an arrow \ar{a_{1},a_{0},f^{-1}} of $\cl{A}$, together with a homotopy from $f^{-1}f$ to $id(a_{0})$ with respect to $\cylinder$, and a homotopy from $ff^{-1}$ to $id(a_{1})$ with respect to $\cylinder$. \end{defn}

\begin{defn} Let $\cylinder = \big( \cyl, i_0, i_1 \big)$ be a cylinder in $\cl{A}$. An arrow \ar{a_{0},a_{1},f} of $\cl{A}$ is a {\em homotopy equivalence} with respect to $\cylinder$ if it admits a homotopy inverse with respect to $\cylinder$. \end{defn} 

\begin{defn} Let $\cocylinder = \big( \cocyl, e_0, e_1 \big)$ be a co-cylinder in $\cl{A}$. An arrow \ar{a_{0},a_{1},f} of $\cl{A}$ is a {\em homotopy equivalence} with respect to $\cocylinder$ if $f^{op}$ is a homotopy equivalence with respect to the cylinder $\cocylinder^{op}$ in $\cl{A}^{op}$. \end{defn} 

\begin{prpn} \label{HomotopyEquivalenceCylinderIffHomotopyEquivalenceCoCylinderProposition} Let $\cylinder = \big( \cyl, i_0, i_1 \big)$ be a cylinder in $\cl{A}$, and let $\cocylinder = \big( \cocyl, e_0, e_1 \big)$ be a co-cylinder in $\cl{A}$. Suppose that $\cylinder$ is left adjoint to $\cocylinder$. Then an arrow \ar{a_{0},a_{1},f} of $\cl{A}$ is a homotopy equivalence with respect to $\cylinder$ if and only if it is a homotopy equivalence with respect to $\cocylinder$. \end{prpn}

\begin{proof} Follows immediately from Proposition \ref{HomotopyCylinderGivesHomotopyCoCylinderProposition} and Corollary \ref{HomotopyCoCylinderGivesHomotopyCylinderCorollary}. \end{proof} 

\begin{lem} \label{HomotopyPreAndPostCompositionLemma} Let $\cylinder = \big( \cyl, i_0, i_1 \big)$ be a cylinder in $\cl{A}$. Suppose that we have four arrows

\begin{diagram}

\begin{tikzpicture} [>=stealth]

\matrix [ampersand replacement=\&, matrix of math nodes, column sep=3 em, nodes={anchor=center}] 
{ 
|(0-0)| a_{0} \& |(1-0)| a_{1} \& |(2-0)| a_{2} \& |(3-0)| a_{3} \\ 
};
	
\draw[->] (0-0) to node[auto] {$g_{0}$} (1-0);
\draw[->] ({1-0}.north east) to node[auto] {$f_{0}$} ({2-0}.north west);
\draw[->] ({1-0}.south east) to node[auto,swap] {$f_{1}$} ({2-0}.south west);
\draw[->] (2-0) to node[auto] {$g_{1}$} (3-0);
\end{tikzpicture} 

\end{diagram} of $\cl{A}$, and a homotopy \ar{\cyl(a_{1}),a_{2},h}from $f_{0}$ to $f_{1}$ with respect to $\cylinder$. Then the arrow \ar[8]{\cyl(a_{0}),a_{3},g_{1} \circ h \circ \cyl(g_{0})} of $\cl{A}$ defines a homotopy from $g_{1}f_{0}g_{0}$ to $g_{1}f_{1}g_{0}$ with respect to $\cylinder$. 

\end{lem}

\begin{proof} The following diagram in $\cl{A}$ commutes.

\begin{diagram}

\begin{tikzpicture} [>=stealth]

\matrix [ampersand replacement=\&, matrix of math nodes, column sep=4 em, row sep=3 em, nodes={anchor=center}]
{ 
|(0-0)| a_{0} \& |(1-0)| \cyl(a_{0}) \\ 
|(0-1)| a_{1} \& |(1-1)| \cyl(a_{1}) \\
              \& |(1-2)| a_{2}       \\
};
	
\draw[->] (0-0) to node[auto] {$i_{0}(a_{0})$} (1-0);
\draw[->] (0-0) to node[auto,swap] {$g_{0}$} (0-1);
\draw[->] (1-0) to node[auto] {$\cyl(g_{0})$} (1-1);
\draw[->] (0-1) to node[auto] {$i_{0}(a_{1})$} (1-1);
\draw[->] (1-1) to node[auto] {$h$} (1-2);
\draw[->] (0-1) to node[auto,swap] {$f_{0}$} (1-2);
\end{tikzpicture} 

\end{diagram} Hence the following diagram in $\cl{A}$ commutes. \tri[{4,3}]{a_{0},\cyl(a_{0}),a_{3},i_{0}(a_{0}),g_{1} \circ h \circ \cyl(g_{0}),g_{1} \circ f_{0} \circ g_{0}} The following diagram in $\cl{A}$ also commutes.  
 
\begin{diagram}

\begin{tikzpicture} [>=stealth]

\matrix [ampersand replacement=\&, matrix of math nodes, column sep=4 em, row sep=3 em, nodes={anchor=center}]
{ 
|(0-0)| a_{0} \& |(1-0)| \cyl(a_{0}) \\ 
|(0-1)| a_{1} \& |(1-1)| \cyl(a_{1}) \\
              \& |(1-2)| a_{2}       \\
};
	
\draw[->] (0-0) to node[auto] {$i_{1}(a_{0})$} (1-0);
\draw[->] (0-0) to node[auto,swap] {$g_{0}$} (0-1);
\draw[->] (1-0) to node[auto] {$\cyl(g_{0})$} (1-1);
\draw[->] (0-1) to node[auto] {$i_{1}(a_{1})$} (1-1);
\draw[->] (1-1) to node[auto] {$h$} (1-2);
\draw[->] (0-1) to node[auto,swap] {$f_{1}$} (1-2);
\end{tikzpicture} 

\end{diagram} Hence the following diagram in $\cl{A}$ commutes. \tri[{4,3}]{a_{0},\cyl(a_{0}),a_{3},i_{1}(a_{0}),g_{1} \circ h \circ \cyl(g_{0}),g_{1} \circ f_{1} \circ g_{0}} 
\end{proof}

\begin{lem} \label{TwoOutOfThreeHomotopyEquivalencesFirstLemma} Let $\cylinder = \big( \cyl, i_0, i_1, v, \subdiv, r_0, r_1, s \big)$ be a cylinder in $\cl{A}$ equipped with an involution structure $v$ and a subdivision structure $\big( \subdiv, r_{0}, r_{1}, s  \big)$. Suppose that we have a commutative diagram in $\cl{A}$ as follows. \tri{a_{0},a_{1},a_{2},f_{0},f_{1},f_{2}} If $f_{1}$ and $f_{2}$ are homotopy equivalences with respect to $\cylinder$, then $f_{0}$ is a homotopy equivalence with respect to $\cylinder$. \end{lem}

\begin{proof} Let $f_{1}^{-1}$ be a homotopy inverse of $f_{1}$ with respect to $\cylinder$, and let \ar{\cyl(a_{1}),a_{1},h_{1}} denote the corresponding homotopy from $f_{1}^{-1}f_{1}$ to $id(a_{1})$ with respect to $\cylinder$. Let $f_{2}^{-1}$ be a homotopy inverse of $f_{2}$ with respect to $\cylinder$, and let \ar{\cyl(a_{2}),a_{2},h_{2}} denote the corresponding homotopy from $f_{2}f_{2}^{-1}$ to $id(a_{2})$ with respect to $\cylinder$. We claim that the arrow \ar[5]{a_{1},a_{0},{f_{2}^{-1}} \circ f_{1}} of $\cl{A}$ defines a homotopy inverse to $f_{0}$ with respect to $\cylinder$.

Let \ar{\cyl(a_{1}),a_{1},k_{0}} denote the arrow $h_{1}^{-1} \circ \cyl(f_{0}f_{2}^{-1}f_{1})$ of $\cl{A}$. By Lemma \ref{HomotopyPreAndPostCompositionLemma}, we have that $k_{0}$ defines a homotopy from $f_{0}f_{2}^{-1}f_{1}$ to $f_{1}^{-1}f_{1}f_{0}f_{2}^{-1}f_{1}$ with respect to $\cylinder$. We have that \[ f_{1}^{-1}f_{1}f_{0}f_{2}^{-1}f_{1} = f_{1}^{-1}f_{2}f_{2}^{-1}f_{1}. \] 

Let \ar{\cyl(a_{1}),a_{1},k_{1}} denote the arrow $f_{1}^{-1} \circ h_{2} \circ \cyl(f_{1})$ of $\cl{A}$. Appealing again to Lemma \ref{HomotopyPreAndPostCompositionLemma}, we have that $k_{1}$ defines a homotopy from $f_{1}^{-1}f_{2}f_{2}^{-1}f_{1}$ to $f_{1}^{-1}f_{1}$ with respect to $\cylinder$. 

We deduce that the arrow \ar[7]{\cyl(a_{1}),a_{1},{(k_{0} + k_{1}) + h_{1}}} of $\cl{A}$ defines a homotopy from $f_{0}f_{2}^{-1}f_{1}$ to $id(a_{1})$ with respect to $\cylinder$. 

We also have that $h_{2}^{-1}$ defines a homotopy from \[ f_{2}^{-1}f_{1}f_{0} = f_{2}^{-1}f_{2} \] to $id(a_{0})$ with respect to $\cylinder$. This completes the proof of the claim. \end{proof}

\begin{lem} \label{TwoOutOfThreeHomotopyEquivalencesSecondLemma} Let $\cylinder = \big( \cyl, i_0, i_1, v, \subdiv, r_0, r_1, s \big)$ be a cylinder in $\cl{A}$ equipped with an involution structure $v$ and a subdivision structure $\big( \subdiv, r_{0}, r_{1}, s \big)$. Suppose that we have a commutative diagram in $\cl{A}$ as follows. \tri{a_{0},a_{1},a_{2},f_{0},f_{1},f_{2}} If $f_{0}$ and $f_{2}$ are homotopy equivalences with respect to $\cylinder$, then $f_{1}$ is a homotopy equivalence with respect to $\cylinder$. \end{lem}

\begin{proof} Let $f_{0}^{-1}$ be a homotopy inverse of $f_{0}$ with respect to $\cylinder$, and let \ar{\cyl(a_{1}),a_{1},h_{0}} denote the corresponding homotopy from $f_{0}f_{0}^{-1}$ to $id(a_{1})$ with respect to $\cylinder$. Let $f_{2}^{-1}$ be a homotopy inverse of $f_{2}$ with respect to $\cylinder$, and let \ar{\cyl(a_{2}),a_{2},h_{2}} denote the corresponding homotopy from $f_{2}f_{2}^{-1}$ to $id(a_{2})$ with respect to $\cylinder$. We claim that the arrow \ar[5]{a_{2},a_{1},f_{0} \circ f_{1}^{-1}} of $\cl{A}$ defines a homotopy inverse to $f_{1}$ with respect to $\cylinder$.

Let \ar{\cyl(a_{1}),a_{1},k_{0}} denote the arrow $f_{0}f_{2}^{-1}f_{1} \circ h_{0}^{-1}$ of $\cl{A}$. By Lemma \ref{HomotopyPreAndPostCompositionLemma}, we have that $k_{0}$ defines a homotopy from $f_{0}f_{2}^{-1}f_{1}$ to $f_{0}f_{2}^{-1}f_{1}f_{0}f_{0}^{-1}$ with respect to $\cylinder$. We also have that \[ f_{0}f_{2}^{-1}f_{1}f_{0}f_{0}^{-1} = f_{0}f_{2}^{-1}f_{2}f_{0}^{-1}. \] %
Let \ar{\cyl(a_{1}),a_{1},k_{1}} denote the arrow $f_{0} \circ h_{2} \circ \cyl(f_{0}^{-1})$ of $\cl{A}$. Appealing again to Lemma \ref{HomotopyPreAndPostCompositionLemma}, we have that $k_{1}$ defines a homotopy from $f_{0}f_{2}^{-1}f_{2}f_{0}^{-1}$ to $f_{0}f_{0}^{-1}$ with respect to $\cylinder$. We deduce that the arrow \ar[7]{\cyl(a_{1}),a_{1},{(k_{0} + k_{1}) + h_{0}}} of $\cl{A}$ defines a homotopy from $f_{0}f_{2}^{-1}f_{1}$ to $id(a_{1})$ with respect to $\cylinder$. 

We also have that $h_{2}$ defines a homotopy from \[ f_{1}f_{0}f_{2}^{-1} = f_{2}f_{2}^{-1} \] to $id(a_{2})$ with respect to $\cylinder$. This completes the proof of the claim. \end{proof}

\begin{prpn} \label{TwoOutOfThreeHomotopyEquivalencesProposition} Let $\cylinder = \big( \cyl, i_0, i_1, v, \subdiv, r_0, r_1, s \big)$ be a cylinder in $\cl{A}$ equipped with an involution structure $v$ and a subdivision structure $\big( \subdiv, r_{0}, r_{1}, s \big)$. Then homotopy equivalences with respect to $\cylinder$ have the two-out-of-three property. \end{prpn}

\begin{proof} Follows immediately from Lemma \ref{TwoOutOfThreeHomotopyEquivalencesFirstLemma}, Lemma \ref{TwoOutOfThreeHomotopyEquivalencesSecondLemma}, and Proposition \ref{CompositionHomotopiesProposition}. \end{proof}

\begin{prpn} \label{RetractionHomotopyEquivalenceIsHomotopyEquivalenceProposition} Let $\cylinder = \big( \cyl, i_{0}, i_{1} \big)$ be a cylinder in $\cl{A}$. Let \ar{a_{0},a_{1},f} be an arrow of $\cl{A}$ which is a homotopy equivalence with respect to $\cylinder$. Suppose that we have commutative diagrams \twosq{a_{2},a_{0},a_{3},a_{1},g_{0},f,f',g_{1},a_{0},a_{2},a_{1},a_{3},r_{0},f',f,r_{1}} in $\cl{A}$, such that $r_{0}$ is a retraction of $g_{0}$, and such that $r_{1}$ is a retraction of $g_{1}$. Then $f'$ is a homotopy equivalence with respect to $\cylinder$. \end{prpn} 

\begin{proof} Let \ar{a_{1},a_{0},f^{-1}} be a homotopy inverse of $f$. Let $h_{0}$ be a homotopy from $f^{-1}f$ to $id(a_{0})$ with respect to $\cylinder$, and let $h_{1}$ be a homotopy from $ff^{-1}$ to $id(a_{1})$ with respect to $\cylinder$. 

Let \ar{a_{3},a_{2},{(f')^{-1}}} denote the arrow $r_{0} \circ f^{-1} \circ g_{1}$ of $\cl{A}$. Let \ar{\cyl(a_{2}),a_{2},h_{0}'} denote the arrow $r_{0} \circ h_{0} \circ \cyl(g_{0})$ of $\cl{A}$. We claim that $h_{0}'$ defines a homotopy from $(f')^{-1}f'$ to $id(a_{2})$ with respect to $\cylinder$. 

Firstly, the following diagram in $\cl{A}$ commutes. \trapeziumstwo{a_{2},\cyl(a_{2}),a_{3},a_{0},\cyl(a_{0}),a_{1},a_{0},i_{0}(a_{2}),\cyl(g_{0}),h_{0},f',g_{1},f^{-1},g_{0},i_{0}(a_{0}),f} %
Hence the following diagram in $\cl{A}$ commutes. \sq[{7,3}]{a_{2},\cyl(a_{2}),a_{3},a_{2},i_{0}(a_{2}),r_{0} \circ h_{0} \circ \cyl(g_{0}),f',r_{0} \circ f^{-1} \circ \cyl(g_{0})} %
Thus the following diagram in $\cl{A}$ commutes. \sq{a_{2},\cyl(a_{2}),a_{3},a_{2},i_{0}(a_{2}),h_{0}',f',{(f')^{-1}}} 

Secondly, the following diagram in $\cl{A}$ commutes. \squareabovetriangle{a_{2},\cyl(a_{2}),a_{0},\cyl(a_{0}),a_{0},i_{1}(a_{2}),\cyl(g_{0}),g_{0},i_{1}(a_{0}),h_{0},id} %
Hence the following diagram in $\cl{A}$ commutes. \squarewithdiagonal{a_{2},\cyl(a_{2}),a_{2},a_{0},i_{1}(a_{2}),h_{0} \circ \cyl(g_{0}),id,r_{0},g_{0}} %
Thus the following diagram in $\cl{A}$ commutes. \tri{a_{2},\cyl(a_{2}),a_{2},i_{1}(a_{2}),h_{0}',id} %
This completes the proof of the claim.

Let \ar{\cyl(a_{3}),a_{3},h_{1}'} denote the arrow $r_{1} \circ h_{1} \circ \cyl(g_{1})$ of $\cl{A}$. We claim that $h_{1}'$ defines a homotopy from $(f')^{-1}f'$ to $id(a_{3})$ with respect to $\cylinder$. 

Firstly, the following diagram in $\cl{A}$ commutes. \stackedsq{a_{1},\cyl(a_{1}),a_{0},a_{1},a_{2},a_{3},i_{0}(a_{1}),h_{1},f^{-1},f,r_{0},r_{1},f'} %
Hence, appealing to the commutativity of the diagram \sq{a_{3},\cyl(a_{3}),a_{1},\cyl(a_{1}),i_{0}(a_{3}),\cyl(g_{1}),g_{1},i_{0}(a_{1})} in $\cl{A}$, we have that the following diagram in $\cl{A}$ commutes. \sq{a_{3},\cyl(a_{3}),a_{2},a_{3},i_{0}(a_{3}),r_{1} \circ h_{1} \circ \cyl(g_{1}),r_{0} \circ f^{-1} \circ g_{1},f'} %
Thus the following diagram in $\cl{A}$ commutes. \sq{a_{3},cyl(a_{3}),a_{2},a_{3},i_{0}(a_{3}),h_{1}',{(f')^{-1}},f'}

Secondly, the following diagram in $\cl{A}$ commutes. \squareabovetriangle{a_{3},\cyl(a_{3}),a_{1},\cyl(a_{1}),a_{1},i_{1}(a_{3}),\cyl(g_{1}),g_{1},i_{1}(a_{1}),h_{1},id} %
Hence the following diagram in $\cl{A}$ commutes. \squarewithdiagonal{a_{3},\cyl(a_{3}),a_{3},a_{1},i_{1}(a_{3}),h_{1} \circ \cyl(g_{1}),id,r_{1},g_{1}} %
Thus the following diagram in $\cl{A}$ commutes. \tri{a_{3},\cyl(a_{3}),a_{3},i_{1}(a_{3}),h_{1}',id} %
This completes the proof of the claim. 

\end{proof}

\begin{lem} \label{RightHomotopyInverseIsHomotopyInverseLemma} Let $\cylinder = \big( \cyl, i_0, i_1, v, \subdiv, r_0, r_1, s \big)$ be a cylinder in $\cl{A}$ equipped with an involution structure $v$, and a subdivision structure $\big( \subdiv, r_{0}, r_{1}, s \big)$. Let \ar{a_{0},a_{1},f} be an arrow of $\cl{A}$ which is a homotopy equivalence with respect to $\cylinder$, and let \ar{a_{1},a_{0},g} be an arrow of $\cl{A}$ such that there is a homotopy \ar{\cyl(a_{1}),a_{1},h} from $fg$ to $id(a_{1})$ with respect to $\cylinder$. Then $g$ is a homotopy inverse of $f$ with respect to $\cylinder$. \end{lem} 

\begin{proof} Let \ar{a_{1},a_{0},f^{-1}} be a homotopy inverse of $f$ with respect to $\cylinder$, and let \ar{\cyl(a_{0}),a_{0},k} denote the corresponding homotopy from $f^{-1}f$ to $id(a_{0})$ with respect to $\cylinder$. By Lemma \ref{HomotopyPreAndPostCompositionLemma}, we have that the arrow \ar[5]{\cyl(a_{1}),a_{0},k \circ \cyl(g)} of $\cl{A}$ defines a homotopy from $f^{-1}fg$ to $g$ with respect to $\cylinder$. 

By Lemma \ref{HomotopyPreAndPostCompositionLemma} once more, we also have that the arrow \ar[5]{\cyl(a_{1}),a_{0}, {f^{-1} \circ h^{-1}}} of $\cl{A}$ defines a homotopy from $f^{-1}$ to $f^{-1}fg$ with respect to $\cylinder$. %
Hence the arrow \ar[12]{\cyl(a_{1}),a_{0},\big( k \circ \cyl(g) \big) + \big( f^{-1} \circ h^{-1} \big)} of $\cl{A}$ defines a homotopy from $f^{-1}$ to $g$ with respect to $\cylinder$. Let us denote it by $l$ for brevity. 

Appealing again to Lemma \ref{HomotopyPreAndPostCompositionLemma}, we have that the arrow \ar[5]{\cyl(a_{0}),a_{0}, l \circ \cyl(f)} of $\cl{A}$ defines a homotopy from $f^{-1}f$ to $gf$ with respect to $\cylinder$. Then \ar[8]{\cyl(a_{0}),a_{0},{\big(l \circ \cyl(f)\big) + k^{-1}}} defines a homotopy from $id(a_{0})$ to $gf$ with respect to $\cylinder$. 

Thus \ar[10]{\cyl(a_{0}),a_{0},{\Big( \big(l \circ \cyl(f)\big) + k^{-1}\Big)^{-1}}} defines a homotopy from $gf$ to $id(a_{0})$ with respect to $\cylinder$. \end{proof}

\begin{rmk} An analogous argument shows that if \ar{a_{0},a_{1},f} satisfies the hypotheses of Lemma \ref{RightHomotopyInverseIsHomotopyInverseLemma}, and if $g$ is an arrow of $\cl{A}$ such that there is a homotopy from $gf$ to $id(a_{0})$ with respect to $\cylinder$, then $g$ is a homotopy inverse of $f$ with respect to $\cylinder$. We shall not need this. \end{rmk}

\begin{defn} Let $\cylinder = \big( \cyl, i_0, i_1 \big)$ be a cylinder in $\cl{A}$. Let $a_{0}$ and $a_{1}$ be objects of $\cl{A}$. We refer to an arrow \ar{\cyl^{2}(a_{0}),a_{1},\sigma} of $\cl{A}$ as a {\em double homotopy} with respect to $\cylinder$. \end{defn}

\begin{defn} Let $\cylinder = \big( \cyl, i_0, i_1 \big)$ be a cylinder in $\cl{A}$. Let \ar{\cyl^{2}(a_{0}),a_{1},\sigma} be a double homotopy with respect to $\cylinder$. We refer to the arrows $h_{0}, h_{1}, h_{2}$, and $h_{3}$ of $\cl{A}$ defined by the commutative diagrams below as the {\em  boundary homotopies} of $\sigma$ with respect to $\cylinder$. \twotriangles[{5,3,0}]{\cyl(a_{0}),\cyl^{2}(a_{0}),a_{1},i_{0}\big(\cyl(a_{0})\big),\sigma,h_{0},\cyl(a_{0}),\cyl^{2}(a_{0}),a_{1},\cyl\big(i_{1}(a_{0})\big),\sigma,h_{1}} \twotriangles[{5,3,0}]{\cyl(a_{0}),\cyl^{2}(a_{0}),a_{1},\cyl\big(i_{0}(a_{0})\big),\sigma,h_{2},\cyl(a_{0}),\cyl^{2}(a_{0}),a_{1},i_{1}\big(\cyl(a_{0})\big),\sigma,h_{3}} \end{defn}

\begin{rmk} \label{DoubleHomotopyPictorialNotationRemark} The following diagrams in $\cl{A}$ commute. 

\begin{diagram}

\begin{tikzpicture} [>=stealth]

\matrix [ampersand replacement=\&, matrix of math nodes, column sep=3 em, row sep=3 em, nodes={anchor=center}]
{ 
|(0-0)| a_{0}       \& |(1-0)| \cyl(a_{0}) \&[2em] |(2-0)| a_{0}       \& |(3-0)| \cyl(a_{0}) \\ 
|(0-1)| \cyl(a_{0}) \& |(1-1)| a_{1}       \&      |(2-1)| \cyl(a_{0}) \& |(3-1)| a_{1} \\[1em]
|(0-2)| a_{0}       \& |(1-2)| \cyl(a_{0}) \&[2em] |(2-2)| a_{0}       \& |(3-2)| \cyl(a_{0}) \\ 
|(0-3)| \cyl(a_{0}) \& |(1-3)| a_{1}       \&      |(2-3)| \cyl(a_{0}) \& |(3-3)| a_{1} \\
};
	
\draw[->] (0-0) to node[auto] {$i_{0}(a_{0})$} (1-0);
\draw[->] (1-0) to node[auto] {$h_{2}$} (1-1);
\draw[->] (0-0) to node[auto,swap] {$i_{0}(a_{0})$} (0-1);
\draw[->] (0-1) to node[auto,swap] {$h_{0}$} (1-1);
\draw[->] (2-0) to node[auto] {$i_{0}(a_{0})$} (3-0);
\draw[->] (3-0) to node[auto] {$h_{1}$} (3-1);
\draw[->] (2-0) to node[auto,swap] {$i_{1}(a_{0})$} (2-1);
\draw[->] (2-1) to node[auto,swap] {$h_{0}$} (3-1);
\draw[->] (0-2) to node[auto] {$i_{1}(a_{0})$} (1-2);
\draw[->] (1-2) to node[auto] {$h_{1}$} (1-3);
\draw[->] (0-2) to node[auto,swap] {$i_{1}(a_{0})$} (0-3);
\draw[->] (0-3) to node[auto,swap] {$h_{3}$} (1-3);
\draw[->] (2-2) to node[auto] {$i_{1}(a_{0})$} (3-2);
\draw[->] (3-2) to node[auto] {$h_{2}$} (3-3);
\draw[->] (2-2) to node[auto,swap] {$i_{0}(a_{0})$} (2-3);
\draw[->] (2-3) to node[auto,swap] {$h_{3}$} (3-3);

\end{tikzpicture} 

\end{diagram} We denote by $f_{0}$, $f_{1}$, $f_{2}$, and $f_{3}$ the arrows of $\cl{A}$ obtained by taking either route through each of these four commutative diagrams, proceeding clockwise respectively from the top left diagram. Thus the following diagrams in $\cl{A}$ commute.

\begin{diagram}

\begin{tikzpicture} [>=stealth]

\matrix [ampersand replacement=\&, matrix of math nodes, column sep=3 em, row sep=4 em, nodes={anchor=center}]
{ 
|(0-0)| a_{0})      \& |(1-0)| \cyl(a_{0}) \&[2em] |(2-0)| a_{0} \& |(3-0)| \cyl(a_{0}) \\ 
|(0-1)| \cyl(a_{0}) \& |(1-1)| a_{1}       \&      |(2-1)| \cyl(a_{0}) \& |(3-1)| a_{1} \\
|(0-2)| a_{0}      \& |(1-2)| \cyl(a_{0}) \& |(2-2)| a_{0} \& |(3-2)| \cyl(a_{0}) \\ 
|(0-3)| \cyl(a_{0}) \& |(1-3)| a_{1}      \& |(2-3)| \cyl(a_{0}) \& |(3-3)| a_{1} \\
};

\draw[->] (0-0) to node[auto] {$i_{0}(a_{0})$} (1-0);
\draw[->] (1-0) to node[auto] {$h_{2}$} (1-1);
\draw[->] (0-0) to node[auto,swap] {$i_{0}(a_{0})$} (0-1);
\draw[->] (0-1) to node[auto,swap] {$h_{0}$} (1-1);
\draw[->] (0-0) to node[auto,swap] {$f_{0}$} (1-1); 
\draw[->] (2-0) to node[auto] {$i_{0}(a_{0})$} (3-0);
\draw[->] (3-0) to node[auto] {$h_{1}$} (3-1);
\draw[->] (2-0) to node[auto,swap] {$i_{1}(a_{0})$} (2-1);
\draw[->] (2-1) to node[auto,swap] {$h_{0}$} (3-1);
\draw[->] (2-0) to node[auto,swap] {$f_{1}$} (3-1); 
\draw[->] (0-2) to node[auto] {$i_{1}(a_{0})$} (1-2);
\draw[->] (1-2) to node[auto] {$h_{1}$} (1-3);
\draw[->] (0-2) to node[auto,swap] {$i_{1}(a_{0})$} (0-3);
\draw[->] (0-3) to node[auto,swap] {$h_{3}$} (1-3);
\draw[->] (0-2) to node[auto,swap] {$f_{3}$} (1-3); 
\draw[->] (2-2) to node[auto] {$i_{1}(a_{0})$} (3-2);
\draw[->] (3-2) to node[auto] {$h_{2}$} (3-3);
\draw[->] (2-2) to node[auto,swap] {$i_{0}(a_{0})$} (2-3);
\draw[->] (2-3) to node[auto,swap] {$h_{3}$} (3-3);
\draw[->] (2-2) to node[auto,swap] {$f_{2}$} (3-3); 

\end{tikzpicture} 

\end{diagram} %
In summary:

\begin{itemize}[itemsep=1em, topsep=1em]

\item[(i)] $h_{0}$ defines a homotopy from $f_{0}$ to $f_{1}$ with respect to $\cylinder$,

\item[(ii)] $h_{1}$ defines a homotopy from $f_{1}$ to $f_{3}$ with respect to $\cylinder$,

\item[(iii)] $h_{2}$ defines a homotopy from $f_{0}$ to $f_{2}$ with respect to $\cylinder$, 

\item[(iv)] $h_{3}$ defines a homotopy from $f_{2}$ to $f_{3}$ with respect to $\cylinder$. 

\end{itemize} %
To express (i)--(iv) we shall frequently depict $\sigma$ as follows. \doublehomotopy{h_{0},h_{1},h_{2},h_{3},\sigma,f_{0},f_{1},f_{2},f_{3}} \end{rmk}

\begin{prpn} Let $\cylinder = \big( \cyl, i_0, i_1, p, \Gamma_{ul} \big)$ be a cylinder in $\cl{A}$ equipped with a contraction structure $p$, and an upper left connection structure $\Gamma_{ul}$. Let \pair{a_{0},a_{1},f_{0},f_{1}} be arrows of $\cl{A}$, and let \ar{\cyl(a_{0}),a_{1},h} be a homotopy from $f_{0}$ to $f_{1}$ with respect to $\cylinder$. Let \ar{\cyl^{2}(a_{0}),a_{1},\sigma} denote the arrow $h \circ \Gamma_{ul}(a_{0})$ of $\cl{A}$. Then $\sigma$ defines a double homotopy with respect to $\cylinder$, with the boundary homotopies depicted below. \doublehomotopy{h,id,h,id,\sigma,f_{0},f_{1},f_{1},f_{1}} \end{prpn}

\begin{proof} Follows immediately from the fact that $\Gamma_{ul}$ defines an upper left connection structure with respect to $\cylinder$. \end{proof}

\begin{prpn} Let $\cylinder = \big( \cyl, i_0, i_1, p, \Gamma_{lr} \big)$ be a cylinder in $\cl{A}$ equipped with a contraction structure $p$, and a lower right connection structure $\Gamma_{lr}$. Let \pair{a_{0},a_{1},f_{0},f_{1}} be arrows of $\cl{A}$, and let \ar{\cyl(a_{0}),a_{1},h} be a homotopy from $f_{0}$ to $f_{1}$ with respect to $\cylinder$. Let \ar{\cyl^{2}(a_{0}),a_{1},\sigma} denote the arrow $h \circ \Gamma_{lr}(a_{0})$ of $\cl{A}$. Then $\sigma$ defines a double homotopy with respect to $\cylinder$, with the boundary homotopies depicted below. \doublehomotopy{id,h,id,h,\sigma,f_{0},f_{0},f_{0},f_{1}} \end{prpn}

\begin{proof} Follows immediately from the fact that $\Gamma_{lr}$ defines a lower right connection structure with respect to $\cylinder$. \end{proof} 

\begin{prpn} Let $\cylinder = \big( \cyl, i_0, i_1, p, \Gamma_{ur} \big)$ be a cylinder in $\cl{A}$ equipped with a contraction structure $p$, and an upper right connection structure $\Gamma_{ur}$. Let \pair{a_{0},a_{1},f_{0},f_{1}} be arrows of $\cl{A}$, and let \ar{\cyl(a_{0}),a_{1},h} be a homotopy from $f_{0}$ to $f_{1}$ with respect to $\cylinder$. Let \ar{\cyl^{2}(a_{0}),a_{1},\sigma} denote the arrow $h \circ \Gamma_{ur}(a_{0})$ of $\cl{A}$. Then $\sigma$ defines a double homotopy with respect to $\cylinder$, with the boundary homotopies depicted below. \doublehomotopy{h,h^{-1},id,id,\sigma,f_{0},f_{1},f_{0},f_{0}} \end{prpn} 

\begin{proof} Follows immediately from the fact that $\Gamma_{ur}$ defines an upper right connection structure with respect to $\cylinder$. \end{proof}

\begin{defn} Let $\cylinder = \big( \cyl, i_0, i_1 ,p \big)$ be a cylinder in $\cl{A}$ equipped with a contraction structure $p$. Suppose that we have a pair of commutative diagrams in $\cl{A}$ as follows. \twotriangles{a,a_{0},a_{1},j_{0},f_{0},j_{1},a,a_{0},a_{1},j_{0},f_{1},j_{1}} A {\em homotopy under $a$} from $f_{0}$ to $f_{1}$ with respect to $\cylinder$ and $(j_{0},j_{1})$ is a homotopy \ar{\cyl(a_{0}),a_{1},h} from $f_{0}$ to $f_{1}$ with respect to $\cylinder$, such that the following diagram in $\cl{A}$ commutes. \sq{\cyl(a),a,\cyl(a_{0}),a_{1},p(a),j_{1},\cyl(j_{0}),h} \end{defn}

\begin{defn} Let $\cylinder = \big( \cyl, i_0, i_1 ,p \big)$ be a cylinder in $\cl{A}$ equipped with a contraction structure $p$. Suppose that we have a pair of commutative diagrams in $\cl{A}$ as follows. \twoothertriangles{a_{0},a_{1},a,f_{0},j_{1},j_{0},a_{0},a_{1},a,f_{1},j_{0},j_{1}} A {\em homotopy over $a$} from $f_{0}$ to $f_{1}$ with respect to $\cylinder$ and $(j_{0},j_{1})$ is a homotopy \ar{\cyl(a_{0}),a_{1},h} from $f_{0}$ to $f_{1}$ with respect to $\cylinder$, such that the following diagram in $\cl{A}$ commutes. \sq{\cyl(a_{0}),a_{1},\cyl(a),a,h,j_{1},\cyl(j_{0}),p(a)} \end{defn}

\begin{defn} \label{UnderHomotopyCoCylinderDefinition} Let $\cocylinder = \big( \cocyl, e_0, e_1 , c \big)$ be a co-cylinder in $\cl{A}$ equipped with a contraction structure $c$. Suppose that we have a pair of commutative diagrams in $\cl{A}$ as follows. \twotriangles{a,a_{0},a_{1},j_{0},f_{0},j_{1},a,a_{0},a_{1},j_{0},f_{1},j_{1}} A {\em homotopy under $a$} from $f_{0}$ to $f_{1}$ with respect to $\cocylinder$ and $(j_{0},j_{1})$ is a homotopy \ar{a_{0},\cocyl(a_{1}),h} from $f_{0}$ to $f_{1}$ with respect to $\cocylinder$, such that the following diagram in $\cl{A}$ commutes. \sq{a,\cocyl(a),a_{0},\cocyl(a_{1}),c(a),\cocyl(j_{1}),j_{0},h} \end{defn}

\begin{rmk} Let $\cocylinder = \big( \cocyl, e_0, e_1 , c \big)$ be a co-cylinder in $\cl{A}$ equipped with a contraction structure $c$. Suppose that we have a pair of commutative diagrams in $\cl{A}$ as follows. \twotriangles{a,a_{0},a_{1},j_{0},f_{0},j_{1},a,a_{0},a_{1},j_{0},f_{1},j_{1}} Then an arrow \ar{a_{0},\cocyl(a_{1}),h} of $\cl{A}$ is a homotopy under $a$ from $f_{0}$ to $f_{1}$ with respect to $\cocylinder$ if and only if $h^{op}$ is a homotopy over $a$ from $f_{0}^{op}$ to $f_{1}^{op}$ with respect to the cylinder $\cocylinder^{op}$ in $\cl{A}^{op}$ equipped with the contraction structure $c^{op}$, and with respect to the arrows $(j_{1}^{op},j_{0}^{op})$ of $\cl{A}^{op}$. \end{rmk}

\begin{defn} \label{OverHomotopyCoCylinderDefinition} Let $\cocylinder = \big( \cocyl, e_0, e_1 ,c \big)$ be a co-cylinder in $\cl{A}$ equipped with a contraction structure $c$. Suppose that we have a pair of commutative diagrams in $\cl{A}$ as follows. \twoothertriangles{a_{0},a_{1},a,f_{0},j_{1},j_{0},a_{0},a_{1},a,f_{1},j_{1},j_{0}} A {\em homotopy over $a$} from $f_{0}$ to $f_{1}$ with respect to $\cocylinder$ and $(j_{0},j_{1})$ is a homotopy \ar{a_{0},\cocyl(a_{1}),h} from $f_{0}$ to $f_{1}$ with respect to $\cocylinder$, such that the following diagram in $\cl{A}$ commutes. \sq{a_{0},\cocyl(a_{1}),a,\cocyl(a),h,\cocyl(j_{1}),j_{0},c(a)} \end{defn}

\begin{rmk} Let $\cocylinder = \big( \cocyl, e_0, e_1 ,c \big)$ be a co-cylinder in $\cl{A}$ equipped with a contraction structure $c$. Suppose that we have a pair of commutative diagrams in $\cl{A}$ as follows. \twoothertriangles{a_{0},a_{1},a,f_{0},j_{1},j_{0},a_{0},a_{1},a,f_{1},j_{1},j_{0}} An arrow \ar{a_{0},\cocyl(a_{1}),h} of $\cl{A}$ is a homotopy over $a$ from $f_{0}$ to $f_{1}$ with respect to $\cocylinder$ and $(j_{0},j_{1})$ if and only if $h^{op}$ is a homotopy under $a$ from $f_{0}^{op}$ to $f_{1}^{op}$ with respect to the cylinder $\cocylinder^{op}$ in $\cl{A}^{op}$ equipped with the contraction structure $c^{op}$, and with respect to the arrows $(j_{1}^{op},j_{0}^{op})$ of $\cl{A}^{op}$. \end{rmk}

\begin{prpn} \label{OverHomotopyCylinderGivesOverHomotopyCoCylinderProposition} Let $\cylinder = \big( \cyl, i_0, i_1, p \big)$ be a cylinder in $\cl{A}$ equipped with a contraction structure $p$, and let $\cocylinder = \big( \cocyl, e_0, e_1, c \big)$ be a co-cylinder in $\cl{A}$ equipped with a contraction structure $c$. Suppose that $\cylinder$ is left adjoint to $\cocylinder$. Let \ar{{\mathsf{Hom}_{\cl{A}}\big(\cyl(-),- \big)},{\mathsf{Hom}_{\cl{A}}\big(-,\cocyl(-) \big)},\mathsf{adj}} denote the corresponding natural isomorphism, adopting the shorthand of Recollection \ref{AdjunctionRecollection}. Suppose that the adjunction between $\cyl$ and $\cocyl$ is compatible with $p$ and $c$. 

Suppose that we have a pair of commutative diagrams in $\cl{A}$ as follows. \twoothertriangles{a_{0},a_{1},a,f_{0},j_{1},j_{0},a_{0},a_{1},a,f_{1},j_{1},j_{0}} If an arrow \ar{\cyl(a_{0}),a_{1},h} of $\cl{A}$ defines a homotopy over $a$ from $f_{0}$ to $f_{1}$ with respect to $\cylinder$ and $(j_{0},j_{1})$, then the arrow \ar[4]{a_{0},\cocyl(a_{1}),\mathsf{adj}(h)} of $\cl{A}$ defines a homotopy over $a$ from $f_{0}$ to $f_{1}$ with respect to $\cocylinder$ and $(j_{0},j_{1})$. \end{prpn}

\begin{proof} Firstly, by Proposition \ref{HomotopyCylinderGivesHomotopyCoCylinderProposition} we have that if $h$ is a homotopy from $f_{0}$ to $f_{1}$ with respect to $\cylinder$, then $\mathsf{adj}(h)$ is a homotopy from $f_{0}$ to $f_{1}$ with respect to $\cocylinder$. 

Secondly, since $h$ is a homotopy over $a$ with respect to $\cylinder$ and $(j_{0},j_{1})$, the following diagram in $\cl{A}$ commutes. \sq{\cyl(a_{0}),a_{1},\cyl(a),a,h,j_{1},\cyl(j_{0}),p(a)} %
Moreover, the following diagram in $\cl{A}$ commutes. \sq[{4,3}]{\cyl(a_{0}),\cyl(a),a_{0},a,\cyl(j_{0}),p(a),j_{0},p(a_{0})} %
Putting the last two observations together, we have that the following diagram in $\cl{A}$ commutes. \sq{\cyl(a_{0}),a_{1},a_{0},a,h,j_{1},p(a_{0}),j_{0}} Thus we have that \[ \mathsf{adj}(j_{1} \circ h) = \mathsf{adj}\big( j_{0} \circ p(a_{0}) \big). \] %
Moreover, by the naturality of the isomorphism $\mathsf{adj}$, the following diagram in $\cl{A}$ commutes. \tri[{4,3}]{a_{0},\cocyl(a_{1}),\cocyl(a),{\mathsf{adj}(h)},\cocyl(j_{1}),{\mathsf{adj}(j_{1} \circ h)}} Hence the following diagram in $\cl{A}$ commutes. \tri[{4,3}]{a_{0},\cocyl(a_{1}),\cocyl(a),{\mathsf{adj}(h)},\cocyl(j_{1}),{\mathsf{adj} \big( j_{0} \circ p(a_{0}) \big) }} %
Since the adjunction between $\cyl$ and $\cocyl$ is compatible with $p$ and $c$, the following diagram in $\cl{A}$ also commutes. \triother{a_{0},a,\cocyl(a),j_{0},c(a),\mathsf{adj}(j_{0} \circ p(a_{0})\big)} Putting the last two observations together, we have that the following diagram in $\cl{A}$ commutes. \sq[{4,3}]{a_{0},\cocyl(a_{1}),a,\cocyl(a),\mathsf{adj}(h),\cocyl(j_{1}),j_{0},c(a)}
\end{proof}

\begin{prpn} \label{UnderHomotopyCylinderGivesUnderHomotopyCoCylinderProposition} Let $\cylinder = \big( \cyl, i_0, i_1, p \big)$ be a cylinder in $\cl{A}$ equipped with a contraction structure $p$, and let $\cocylinder = \big( \cocyl, e_0, e_1, c \big)$ be a co-cylinder in $\cl{A}$ equipped with a contraction structure $c$. Suppose that $\cylinder$ is left adjoint to $\cocylinder$. Let \ar{{\mathsf{Hom}_{\cl{A}}\big(\cyl(-),- \big)},{\mathsf{Hom}_{\cl{A}}\big(-,\cocyl(-) \big)},\mathsf{adj}} denote the corresponding natural isomorphism, adopting the shorthand of Recollection \ref{AdjunctionRecollection}. Suppose that the adjunction between $\cyl$ and $\cocyl$ is compatible with $p$ and $c$. 

Suppose that we have a pair of commutative diagrams in $\cl{A}$ as follows. \twotriangles{a,a_{0},a_{1},j_{0},f_{0},j_{1},a,a_{0},a_{1},j_{0},f_{1},j_{1}}If an arrow \ar{\cyl(a_{0}),a_{1},h} of $\cl{A}$ defines a homotopy under $a$ from $f_{0}$ to $f_{1}$ with respect to $\cylinder$ and $(j_{0},j_{1})$, then the arrow \ar[4]{a_{0},\cocyl(a_{1}),\mathsf{adj}(h)} of $\cl{A}$ defines a homotopy under $a$ from $f_{0}$ to $f_{1}$ with respect to $\cocylinder$ and $(j_{0},j_{1})$. \end{prpn}

\begin{proof} Firstly, by Proposition \ref{HomotopyCylinderGivesHomotopyCoCylinderProposition} we have that if $h$ is a homotopy from $f_{0}$ to $f_{1}$ with respect to $\cylinder$, then $\mathsf{adj}(h)$ is a homotopy from $f_{0}$ to $f_{1}$ with respect to $\cocylinder$. 

Secondly, since $h$ is a homotopy under $a$ with respect to $\cylinder$ and $(j_{0},j_{1})$, the following diagram in $\cl{A}$ commutes. \sq{\cyl(a),a,\cyl(a_{0}),a_{1},p(a),j_{1},\cyl(j_{0}),h} Thus we have that \[ \mathsf{adj}\big(h \circ \cyl(j_{0})\big) = \mathsf{adj}(j_{1} \circ p(a) \big). \] %
Moreover, by the naturality of the isomorphism $\mathsf{adj}$, the following diagram in $\cl{A}$ commutes. \triother[{4,3}]{a,a_{0},\cocyl(a_{1}),j_{0},\mathsf{adj}(h),\mathsf{adj}\big(h \circ \cyl(j_{0})\big)} %
Hence the following diagram in $\cl{A}$ commutes. \triother[{4,3}]{a,a_{0},\cocyl(a_{1}),j_{0},\mathsf{adj}(h),\mathsf{adj}\big(j_{1} \circ p(a)\big)} %
Since the adjunction between $\cyl$ and $\cocyl$ is compatible with $p$ and $c$, the following diagram in $\cl{A}$ also commutes. \tri{a,\cocyl(a),\cocyl(a_{1}),c(a),\cocyl(j_{1}),\mathsf{adj}\big(j_{1} \circ p(a)\big)} %
Putting the last two observations together, we have that the following diagram in $\cl{A}$ commutes. \sq[{4,3}]{a,\cocyl(a),a_{0},\cocyl(a_{1}),c(a),\cocyl(j_{1}),j_{0},\mathsf{adj}(h)} \end{proof}

\begin{cor} \label{OverHomotopyCoCylinderGivesOverHomotopyCylinderCorollary} Let $\cylinder = \big( \cyl, i_0, i_1, p \big)$ be a cylinder in $\cl{A}$ equipped with a contraction structure $p$, and let $\cocylinder = \big( \cocyl, e_0, e_1, c \big)$ be a co-cylinder in $\cl{A}$ equipped with a contraction structure $c$. Suppose that $\cylinder$ is left adjoint to $\cocylinder$. Let \ar{{\mathsf{Hom}_{\cl{A}}\big(\cyl(-),- \big)},{\mathsf{Hom}_{\cl{A}}\big(-,\cocyl(-) \big)},\mathsf{adj}} denote the corresponding natural isomorphism, adopting the shorthand of Recollection \ref{AdjunctionRecollection}. Suppose that the adjunction between $\cyl$ and $\cocyl$ is compatible with $p$ and $c$. 

Suppose that we have a pair of commutative diagrams in $\cl{A}$ as follows. \twoothertriangles{a_{0},a_{1},a,f_{0},j_{1},j_{0},a_{0},a_{1},a,f_{1},j_{1},j_{0}}If an arrow \ar{a_{0},\cocyl(a_{1}),h} of $\cl{A}$ defines a homotopy over $a$ from $f_{0}$ to $f_{1}$ with respect to $\cocylinder$ and $(j_{0},j_{1})$, then the arrow \ar{\cyl(a_{0}),a_{1},\mathsf{adj}^{-1}(h)} of $\cl{A}$ defines a homotopy over $a$ from $f_{0}$ to $f_{1}$ with respect to $\cylinder$ and $(j_{0},j_{1})$. \end{cor}

\begin{proof} Follows immediately from Proposition \ref{UnderHomotopyCylinderGivesUnderHomotopyCoCylinderProposition} by duality. \end{proof}

\begin{cor} \label{UnderHomotopyCoCylinderGivesUnderHomotopyCylinderCorollary} Let $\cylinder = \big( \cyl, i_0, i_1, p \big)$ be a cylinder in $\cl{A}$ equipped with a contraction structure $p$, and let $\cocylinder = \big( \cocyl, e_0, e_1, c \big)$ be a co-cylinder in $\cl{A}$ equipped with a contraction structure $c$. Suppose that $\cylinder$ is left adjoint to $\cocylinder$. Let \ar{{\mathsf{Hom}_{\cl{A}}\big(\cyl(-),- \big)},{\mathsf{Hom}_{\cl{A}}\big(-,\cocyl(-) \big)},\mathsf{adj}} denote the corresponding natural isomorphism, adopting the shorthand of Recollection \ref{AdjunctionRecollection}. Suppose that the adjunction between $\cyl$ and $\cocyl$ is compatible with $p$ and $c$. 

Suppose that we have a pair of commutative diagrams in $\cl{A}$ as follows. \twotriangles{a,a_{0},a_{1},j_{0},f_{0},j_{1},a,a_{0},a_{1},j_{0},f_{1},j_{1}} If an arrow \ar{a_{0},\cocyl(a_{1}),h} of $\cl{A}$ defines a homotopy under $a$ from $f_{0}$ to $f_{1}$ with respect to $\cocylinder$ and $(j_{0},j_{1})$, then the arrow \ar[4]{\cyl(a_{0}),a_{1},\mathsf{adj}^{-1}(h)} of $\cl{A}$ defines a homotopy under $a$ from $f_{0}$ to $f_{1}$ with respect to $\cylinder$ and $(j_{0},j_{1})$. \end{cor}

\begin{proof} Follows immediately from Proposition \ref{OverHomotopyCylinderGivesOverHomotopyCoCylinderProposition} by duality. \end{proof}

\begin{prpn} \label{IdentityHomotopyIsOverHomotopyProposition} Let $\cylinder = \big( \cyl, i_0, i_1, p \big)$ be a cylinder in $\cl{A}$ equipped with a contraction structure. Suppose that we have a commutative diagram in $\cl{A}$ as follows. \triother{a_{0},a_{1},a,f,j_{1},j_{0}} %
Then the identity homotopy from $f$ to itself with respect to $\cylinder$ is moreover a homotopy over $a$ with respect to $\cylinder$ and $(j_{0},j_{1})$. \end{prpn}

\begin{proof} Let $h$ denote the identity homotopy from $f$ to itself with respect to $\cylinder$. By definition of $h$, the following diagram in $\cl{A}$ commutes. \tri[{4,3}]{\cyl(a_{0}),a_{0},a_{1},p(a_{0}),f,h} %
Thus the following diagram in $\cl{A}$ commutes. \squareofthreetrianglesthree{\cyl(a_{0}),a_{1},a_{0},\cyl(a),a,h,j_{1},\cyl(j_{0}),p(a),p(a_{0}),f,j_{0}}
 
\end{proof}

\begin{prpn} \label{ReverseHomotopyIsOverHomotopyProposition} Let $\cylinder = \big( \cyl, i_0, i_1, p, v \big)$ be a cylinder in $\cl{A}$ equipped with a contraction structure $p$, and an involution structure $v$ compatible with $p$. 

Suppose that we have arrows \ar{a_{0},a,j_{0}} and \ar{a_{1},a,j_{1}} of $\cl{A}$, and a homotopy \ar{\cyl(a_{0}), a_{1}, h} with respect to $\cylinder$, such that the following diagram in $\cl{A}$ commutes. \sq{\cyl(a_{0}),a_{1},\cyl(a),a,h,j_{1},\cyl(j_{0}),p(a)} Then the following diagram in $\cl{A}$ commutes. \sq{\cyl(a_{0}),a_{1},\cyl(a),a,{h^{-1}},j_{1},\cyl(j_{0}),p(a)} \end{prpn}

\begin{proof} By definition of $h^{-1}$, the following diagram in $\cl{A}$ commutes. \tri{\cyl(a_{0}),\cyl(a_{0}),a_{1},v(a_{0}),h,h^{-1}} %
Thus we have that the following diagram in $\cl{A}$ commutes. \squareofthreetrianglesfour{\cyl(a_{0}),a_{1},\cyl(a),\cyl(a_{0}),\cyl(a),a,h^{-1},j_{1},\cyl(j_{0}),v(a),p(a),v(a_{0}),h,\cyl(j_{0})} %
Since $v$ is compatible with $p$, we also have that the following diagram in $\cl{A}$ commutes. \tri{\cyl(a),\cyl(a),a,v(a),p(a),p(a)} %
Putting the last two observations together, we have that the following diagram in $\cl{A}$ commutes, as required. \sq{\cyl(a_{0}),a_{1},\cyl(a),a,h^{-1},j_{1},\cyl(j_{0}),p(a)}
\end{proof}

\begin{cor} \label{ReverseHomotopyIsOverHomotopyCorollary} Let $\cylinder = \big( \cyl, i_0, i_1, p, v \big)$ be a cylinder in $\cl{A}$ equipped with a contraction structure $p$, and an involution structure $v$ compatible with $p$. 

Suppose that we have a pair of commutative diagrams in $\cl{A}$ as follows. \twoothertriangles{a_{0},a_{1},a,f_{0},j_{1},j_{0},a_{0},a_{1},a,f_{1},j_{1},j_{0}} Let \ar{\cyl(a_{0}), a_{1}, h} be a homotopy over $a$ from $f_{0}$ to $f_{1}$ with respect to $\cylinder$ and $(j_{0},j_{1})$. Then the reverse homotopy \ar{\cyl(a_{0}),a_{1},h^{-1}} from $f_{1}$ to $f_{0}$ with respect to $\cylinder$ is moreover a homotopy under $a$ with respect to $\cylinder$ and $(j_{0},j_{1})$ . \end{cor}

\begin{proof} Follows immediately from Proposition \ref{ReverseHomotopyIsOverHomotopyProposition}.
\end{proof}

\begin{prpn} \label{CompositeHomotopyIsOverHomotopyProposition} Let $\cylinder = \big( \cyl, i_0, i_1, p, \subdiv, r_0, r_1, s \big)$ be a cylinder in $\cl{A}$ equipped with a contraction structure $p$, and a subdivision structure $\big( \subdiv, r_{0}, r_{1}, s \big)$ compatible with $p$. 

Suppose that we have arrows \ar{a_{0},a,j_{0}} and \ar{a_{1},a,j_{1}} of $\cl{A}$, and homotopies \ar{\cyl(a_{0}), a_{1}, h} and \ar{\cyl(a_{0}), a_{1}, k} with respect to $\cylinder$, such that the diagrams \twosq{\cyl(a_{0}),a_{1},\cyl(a),a,h,j_{1},\cyl(j_{0}),p(a),\cyl(a_{0}),a_{1},\cyl(a),a,k,j_{1},\cyl(j_{0}),p(a)} and \sq{a_{0},\cyl(a_{0}),\cyl(a_{0}),a_{1},i_{0}(a_{0}),k,i_{1}(a_{0}),h} in $\cl{A}$ commute. Then the following diagram in $\cl{A}$ commutes. \sq[{4,3}]{\cyl(a_{0}),a_{1},\cyl(a),a,h+k,j_{1},\cyl(j_{0}),p(a)} \end{prpn}

\begin{proof} Appealing to the universal property of $\subdiv(a_{1})$, there is an arrow \ar{\subdiv(a_{0}),a_{1},r} of $\cl{A}$ such that the following diagram in $\cl{A}$ commutes. \pushout{a_{0},\cyl(a_{0}),\cyl(a_{0}),\subdiv(a_{0}),a_{1},{i_{0}(a_{0})},{r_{0}(a_{0})},{i_{1}(a_{0})},{r_{1}(a_{0})},h,k,r} %
The following diagram in $\cl{A}$ commutes. \squarewithdiagonal{\cyl(a_{0}),\subdiv(a_{0}),a,a_{1},r_{0}(a_{0}),r,p(a) \circ \cyl(j_{0}),j_{1},h} %
The following diagram in $\cl{A}$ also commutes. \squarewithdiagonal{\cyl(a_{0}),\subdiv(a_{0}),a,a_{1},r_{1}(a_{0}),r,p(a) \circ \cyl(j_{0}),j_{1},k} %
Putting the last two observations together, we have that the following diagram in $\cl{A}$ commutes. \pushout{a_{0},\cyl(a_{0}),\cyl(a_{0}),\subdiv(a_{0}),a,i_{0}(a_{0}),r_{0}(a_{0}),i_{1}(a_{0}),r_{1}(a_{0}),p(a_{0}) \circ \cyl(j_{0}),p(a_{0}) \circ \cyl(j_{0}), j_{1} \circ r} %
Let \ar{\subdiv,\id_{\cl{A}},\overline{p}} denote the canonical 2-arrow of $\cl{C}$ of Definition \ref{SubdivisionCompatibleWithContractionDefinition}. The following diagram in $\cl{A}$ commutes. \pushout{a_{0},\cyl(a_{0}),\cyl(a_{0}),\subdiv(a_{0}),a_{0},i_{0}(a_{0}),r_{0}(a_{0}),i_{1}(a_{0}),r_{1}(a_{0}),p(a_{0}),p(a_{0}),\overline{p}(a_{0})} %
Thus the following diagram in $\cl{A}$ commutes. \pushout{a_{0},\cyl(a_{0}),\cyl(a_{0}),\subdiv(a_{0}),a,i_{0}(a_{0}),r_{0}(a_{0}),i_{1}(a_{0}),r_{1}(a_{0}),j_{0} \circ p(a_{0}),j_{0} \circ p(a_{0}), j_{0} \circ \overline{p}(a_{0})} %
Moreover, the following diagram in $\cl{A}$ commutes. \sq{\cyl(a_{0}),a_{0},\cyl(a),a,p(a_{0}),j_{0},\cyl(j_{0}),p(a)} %
Putting the last two observations together, we have that the following diagram in $\cl{A}$ commutes. \pushout{a_{0},\cyl(a_{0}),\cyl(a_{0}),\subdiv(a_{0}),a,i_{0}(a_{0}),r_{0}(a_{0}),i_{1}(a_{0}),r_{1}(a_{0}),p(a) \circ \cyl(j_{0}), p(a) \circ \cyl(j_{0}), j_{0} \circ \overline{p}(a_{0})} %
Appealing to the universal property of $\subdiv(a_{0})$, we deduce that the following diagram in $\cl{A}$ commutes. \sq{\subdiv(a_{0}),a_{1},a_{0},a,r,j_{1},\overline{p}(a_{0}),j_{0}} %
Since the subdivision structure $\big( \subdiv, r_{0}, r_{1}, s \big)$ is compatible with $p$, we also have that that the following diagram in $\cl{A}$ commutes. \tri{\cyl(a_{0}),\subdiv(a_{0}),a_{0},s(a_{0}),\overline{p}(a_{0}),p(a_{0})} %
By definition of $h+k$, the following diagram in $\cl{A}$ commutes. \tri{\cyl(a_{0}),\subdiv(a_{0}),a_{1},s(a_{0}),r,{h+k}} %
Putting the last two observations together we have that the following diagram in $\cl{A}$ commutes. \squareofthreetrianglestwo{\cyl(a_{0}),a_{1},\subdiv(a_{0}),a_{0},a,h+k,j_{1},p(a_{0}),j_{0},s(a_{0}),r,\overline{p}(a_{0})} %
Appealing once more to the commutativity of the diagram \sq{\cyl(a_{0}),a_{0},\cyl(a),a,p(a_{0}),j_{0},\cyl(j_{0}),p(a)} in $\cl{A}$, we conclude that the following diagram in $\cl{A}$ commutes. \sq[{3,3.5,0,0}]{\cyl(a_{0}),a_{1},\cyl(a),a,h+k,j_{1},\cyl(j_{0}),p(a)}
 
\end{proof} 

\begin{cor} \label{CompositeHomotopyIsOverHomotopyCorollary} Let $\cylinder = \big( \cyl, i_0, i_1, p, \subdiv, r_0, r_1, s \big)$ be a cylinder in $\cl{A}$ equipped with a contraction structure $p$, and a subdivision structure $\big( \subdiv, r_{0}, r_{1}, s \big)$ compatible with $p$. 

Suppose that we have three commutative diagrams in $\cl{A}$ as follows.

\begin{diagram}

\begin{tikzpicture} [>=stealth]

\matrix [ampersand replacement=\&, matrix of math nodes, column sep=3 em, row sep=3 em, nodes={anchor=center}]
{ 
|(0-0)| a_{0} \& |(1-0)|   \&[2em] |(2-0)| a_{0} \& |(3-0)|   \&[2em] |(4-0)| a_{0}  \& |(5-0)|   \\ 
|(0-1)| a_{1} \& |(1-1)| a \&      |(2-1)| a_{1} \& |(3-1)| a \&      |(4-1)| a_{1}  \& |(5-1)| a \\
};
	
\draw[->] (0-0) to node[auto] {$j_{0}$} (1-1);
\draw[->] (0-0) to node[auto,swap] {$f_{0}$} (0-1);
\draw[->] (0-1) to node[auto,swap] {$j_{1}$} (1-1);
\draw[->] (2-0) to node[auto] {$j_{0}$} (3-1);
\draw[->] (2-0) to node[auto,swap] {$f_{1}$} (2-1);
\draw[->] (2-1) to node[auto,swap] {$j_{1}$} (3-1);
\draw[->] (4-0) to node[auto] {$j_{0}$} (5-1);
\draw[->] (4-0) to node[auto,swap] {$f_{2}$} (4-1);
\draw[->] (4-1) to node[auto,swap] {$j_{1}$} (5-1);

\end{tikzpicture} 

\end{diagram} %
Let \ar{\cyl(a_{0}), a_{1}, h} be a homotopy over $a$ from $f_{0}$ to $f_{1}$ with respect to $\cylinder$ and $(j_{0},j_{1})$. Let \ar{\cyl(a_{0}), a_{1}, k} be a homotopy under $a$ from $f_{1}$ to $f_{2}$ with respect to $\cylinder$ and $(j_{0},j_{1})$. Then the homotopy \ar[4]{\cyl(a_{0}),a_{1},h + k} from $f_{0}$ to $f_{2}$ with respect to $\cylinder$ is moreover a homotopy under $a$ with respect to $\cylinder$ and $(j_{0},j_{1})$. \end{cor}

\begin{proof} Follows immediately from Proposition \ref{CompositeHomotopyIsOverHomotopyProposition}. \end{proof} 

\begin{rmk} Analogues of Proposition \ref{IdentityHomotopyIsOverHomotopyProposition}, Corollary \ref{ReverseHomotopyIsOverHomotopyCorollary}, and Corollary \ref{CompositeHomotopyIsOverHomotopyCorollary} for homotopies under an object can all be proven. Moreover, all of these results dualise to the setting of over and under homotopies with respect to a co-cylinder. We shall not need any of this. \end{rmk}
 
\begin{defn} Let $\cylinder = \big( \cyl, i_0, i_1, p \big)$ be a cylinder in $\cl{A}$ equipped with a contraction structure $p$, and suppose that we have a commutative diagram in $\cl{A}$ as follows. \tri{a,a_{0},a_{1},j_{0},f,j_{1}} A {\em homotopy inverse under $a$} of $f$ with respect to $\cylinder$ and $(j_{0},j_{1})$ is an arrow \ar{a_{1},a_{0},f^{-1}} of $\cl{A}$ such that the diagram \tri{a,a_{1},a_{0},j_{1},f^{-1},j_{0}} in $\cl{A}$ commutes, together with a homotopy under $a$ from $f^{-1}f$ to $id(a_{0})$ with respect to $\cylinder$ and $(j_{0},j_{0})$, and a homotopy under $a$ from $ff^{-1}$ to $id(a_{1})$ with respect to $\cylinder$ and $(j_{1},j_{1})$. \end{defn}

\begin{defn} Let $\cylinder = \big( \cyl, i_0, i_1, p \big)$ be a cylinder in $\cl{A}$ equipped with a contraction structure $p$. Suppose that we have a commutative diagram in $\cl{A}$ as follows. \tri{a,a_{0},a_{1},j_{0},f,j_{1}} Then $f$ is a {\em homotopy equivalence under $a$} with respect to $\cylinder$ and $(j_{0},j_{1})$ if it admits a homotopy inverse under $a$ with respect to $\cylinder$ and $(j_{0},j_{1})$. \end{defn} 

\begin{defn} Let $\cylinder = \big( \cyl, i_0, i_1, p \big)$ be a cylinder in $\cl{A}$ equipped with a contraction structure $p$. Suppose that we have a commutative diagram in $\cl{A}$ as follows. \triother{a_{0},a_{1},a,f,j_{1},j_{0}} A {\em homotopy inverse over $a$} of $f$ with respect to $\cylinder$ and $(j_{0},j_{1})$ is an arrow \ar{a_{1},a_{0},f^{-1}} of $\cl{A}$ such that the diagram \triother{a_{1},a_{0},a,f^{-1},j_{0},j_{1}} in $\cl{A}$ commutes together with a homotopy over $a$ from $f^{-1}f$ to $id(a_{0})$ with respect to $\cylinder$ and $(j_{0},j_{0})$ and a homotopy over $a$ from $ff^{-1}$ to $id(a_{1})$ with respect to $\cylinder$ and $(j_{1},j_{1})$. \end{defn}

\begin{defn} Let $\cylinder = \big( \cyl, i_0, i_1, p \big)$ be a cylinder in $\cl{A}$ equipped with a contraction structure $p$, and suppose that we have a commutative diagram in $\cl{A}$ as follows. \triother{a_{0},a_{1},a,f,j_{1},j_{0}} Then $f$ is a {\em homotopy equivalence over $a$} with respect to $\cylinder$ and $(j_{0},j_{1})$ if it admits a homotopy inverse over $a$ with respect to $\cylinder$ and $(j_{0},j_{1})$. \end{defn} 

\begin{defn} Let $\cocylinder = \big( \cocyl, e_0, e_1, c \big)$ be a co-cylinder in $\cl{A}$ equipped with a contraction structure $c$. Suppose that we have a commutative diagram in $\cl{A}$ as follows. \tri{a,a_{0},a_{1},j_{0},f,j_{1}} Then $f$ is a {\em homotopy equivalence under $a$} with respect to $\cocylinder$ if $f^{op}$ is a homotopy equivalence over $a$ with respect to the cylinder $\cocylinder^{op}$ in $\cl{A}^{op}$ equipped with the contraction structure $c^{op}$, and $(j_{1}^{op},j_{0}^{op})$. \end{defn} 

\begin{defn} Let $\cocylinder = \big( \cocyl, e_0, e_1, c \big)$ be a co-cylinder in $\cl{A}$ equipped with a contraction structure $c$. Suppose that we have a commutative diagram in $\cl{A}$ as follows. \triother{a_{0},a_{1},a,f,j_{1},j_{0}} Then $f$ is a {\em homotopy equivalence over $a$} with respect to $\cocylinder$ if $f^{op}$ is a homotopy equivalence under $a$ with respect to the cylinder $\cocylinder^{op}$ in $\cl{A}^{op}$ equipped with the contraction structure $c^{op}$, and $(j_{1}^{op},j_{0}^{op})$. \end{defn} 

\begin{prpn} \label{UnderHomotopyEquivalenceCylinderIffUnderHomotopyEquivalenceCoCylinderProposition} Let $\cylinder = \big( \cyl, i_0, i_1, p \big)$ be a cylinder in $\cl{A}$ equipped with a contraction structure $p$, and let $\cocylinder = \big( \cocyl, e_0, e_1, c \big)$ be a co-cylinder in $\cl{A}$ equipped with a contraction structure $c$. Suppose that $\cylinder$ is left adjoint to $\cocylinder$ and that the adjunction between $\cyl$ and $\cocyl$ is compatible with $p$ and $c$. 

Suppose that we have a commutative diagram in $\cl{A}$ as follows. \tri{a, a_{0}, a_{1}, j_{0}, f, j_{1}} Then $f$ is a homotopy equivalence under $a$ with respect to $\cylinder$ if and only if it is a homotopy equivalence under $a$ with respect to $\cocylinder$. \end{prpn}

\begin{proof} Follows immediately from Proposition \ref{UnderHomotopyCylinderGivesUnderHomotopyCoCylinderProposition} and Corollary \ref{UnderHomotopyCoCylinderGivesUnderHomotopyCylinderCorollary}. \end{proof} 

\begin{prpn}  \label{OverHomotopyEquivalenceCylinderIffOverHomotopyEquivalenceCoCylinderProposition} Let $\cylinder = \big( \cyl, i_0, i_1, p \big)$ be a cylinder in $\cl{A}$ equipped with a contraction structure $p$, and let $\cocylinder = \big( \cocyl, e_0, e_1, c \big)$ be a co-cylinder in $\cl{A}$ equipped with a contraction structure $c$. Suppose that $\cylinder$ is left adjoint to $\cocylinder$, and that the adjunction between $\cyl$ and $\cocyl$ is compatible with $p$ and $c$. 

Suppose that we have a commutative diagram in $\cl{A}$ as follows. \triother{a,a_{0},a_{1},f,j_{0},j_{1}} Then $f$ is a homotopy equivalence over $a$ with respect to $\cylinder$ if and only if it is a homotopy equivalence over $a$ with respect to $\cocylinder$. \end{prpn}

\begin{proof} Follows immediately from Proposition \ref{OverHomotopyCylinderGivesOverHomotopyCoCylinderProposition} and Corollary \ref{UnderHomotopyCoCylinderGivesUnderHomotopyCylinderCorollary}. \end{proof} 

\begin{defn} Let \ar{a_{0},a_{1},j} be an arrow of $\cl{A}$. An arrow \ar{a_{0},a_{1},f} of $\cl{A}$ is a {\em retraction} of $j$ if the diagram \tri{a_{0},a_{1},a_{0},j,f,id} in $\cl{A}$ commutes. \end{defn} 
 
\begin{defn} \label{StrongDeformationRetractionCylinderDefinition} Let $\cylinder = \big( \cyl, i_0, i_1, p \big)$ be a cylinder in $\cl{A}$ equipped with a contraction structure $p$. Let \ar{a_{0},a_{1},j} be an arrow of $\cl{A}$. 

An arrow \ar{a_{1},a_{0},f} of $\cl{A}$ is a {\em strong deformation retraction} of $j$ with respect to $\cylinder$ if $f$ is a retraction of $j$, and if there is a homotopy under $a_{0}$ from $jf$ to $id(a_{1})$ with respect to $\cylinder$ and $(j,j)$. \end{defn}

\begin{defn}  Let $\cylinder = \big( \cyl, i_0, i_1, p \big)$ be a cylinder in $\cl{A}$ equipped with a contraction structure $p$. An arrow \ar{a_{0},a_{1},j} of $\cl{A}$ {\em admits a strong deformation retraction} with respect to $\cylinder$ if there is an arrow \ar{a_{1},a_{0},f} of $\cl{A}$ which defines a strong deformation retraction of $j$ with respect to $\cylinder$. \end{defn}

\begin{defn} \label{StrongDeformationRetractionCoCylinderDefinition} Let $\cocylinder = \big( \cocyl, e_0, e_1, c \big)$ be a co-cylinder in $\cl{A}$ equipped with a contraction structure $c$. Let \ar{a_{0},a_{1},j} be an arrow of $\cl{A}$. 

An arrow \ar{a_{1},a_{0},f} of $\cl{A}$ is a {\em strong deformation retraction} of $j$ with respect to $\cocylinder$ if $f$ is a retraction of $j$, and if there is a homotopy over $a_{0}$ from $jf$ to $id(a_{1})$ with respect to $\cocylinder$ and $(f,f)$. \end{defn} 

\begin{rmk} Let $\cocylinder = \big( \cocyl, e_0, e_1, c \big)$ be a co-cylinder in $\cl{A}$ equipped with a contraction structure $c$. Let \ar{a_{0},a_{1},j} be an arrow of $\cl{A}$. 

An arrow \ar{a_{1},a_{0},f} of $\cl{A}$ is a {\em strong deformation retraction} of $j$ with respect to $\cocylinder$ if and only if $j^{op}$ defines a strong deformation retraction of $f^{op}$ with respect to the cylinder $\cocylinder^{op}$ in $\cl{A}$, equipped with the contraction structure defined by $c^{op}$.  \end{rmk}

\begin{defn} Let $\cocylinder = \big( \cocyl, e_0, e_1, c \big)$ be a co-cylinder in $\cl{A}$ equipped with a contraction structure $c$. An arrow \ar{a_{0},a_{1},j} of $\cl{A}$ {\em admits a strong deformation retraction} with respect to $\cocylinder$ if there is an arrow \ar{a_{1},a_{0},f} of $\cl{A}$ which defines a strong deformation retraction of $j$ with respect to $\cocylinder$. \end{defn} 

\begin{prpn} \label{StrongDeformationRetractionCoCylinderCharacterisationProposition} Let $\cylinder = \big( \cyl, i_0, i_1, p \big)$ be a cylinder in $\cl{A}$ equipped with a contraction structure $p$, and let $\cocylinder = \big( \cocyl, e_0, e_1, c \big)$ be a co-cylinder in $\cl{A}$ equipped with a contraction structure $c$. Suppose that $\cylinder$ is left adjoint to $\cocylinder$, and that the adjunction between $\cyl$ and $\cocyl$ is compatible with $p$ and $c$. 

Let \ar{a_{0},a_{1},j} be an arrow of $\cl{A}$. Then an arrow \ar{a_{1},a_{0},f} of $\cl{A}$ is a strong deformation retraction of $j$ with respect to $\cocylinder$ if and only if the following conditions are satisfied:

\begin{itemize}

\item[(i)] $f$ is a retraction of $j$,

\item[(ii)] there exists a homotopy over $a_{0}$ from $jf$ to $id(a_{1})$ with respect to $\cylinder$ and $(f,f)$. 

\end{itemize}

\end{prpn}

\begin{proof} Follows immediately from Corollary \ref{OverHomotopyCoCylinderGivesOverHomotopyCylinderCorollary}. \end{proof} 

\begin{lem} \label{HomotopyPreAndPostCompositionIsOverHomotopyLemma} Let $\cylinder = \big( \cyl, i_0, i_1, p \big)$ be a cylinder in $\cl{A}$ equipped with a contraction structure. Suppose that we have four commutative diagrams in $\cl{A}$ as follows. \twoothertriangles{a_{0},a_{1},a,g_{0},j_{1},j_{0},a_{1},a_{2},a,f_{0},j_{2},j_{1}} \twoothertriangles{a_{1},a_{2},a,f_{1},j_{2},j_{1},a_{2},a_{3},a,g_{1},j_{3},j_{2}} %
Suppose also that we have a homotopy \ar{\cyl(a_{1}),a_{2},h} over $a$ from $f_{0}$ to $f_{1}$ with respect to $\cylinder$ and $(j_{1},j_{2})$. Then the arrow \ar[7]{\cyl(a_{0}),a_{3},g_{1} \circ h \circ \cyl(g_{0})} of $\cl{A}$ defines a homotopy over $a$ from $g_{1}f_{0}g_{0}$ to $g_{1}f_{1}g_{0}$ with respect to $\cylinder$ and $(j_{0},j_{3})$. \end{lem}

\begin{proof} That the arrow \ar[7]{\cyl(a_{0}),a_{3},g_{1} \circ h \circ \cyl(g_{0})} of $\cl{A}$  defines a homotopy from $g_{1}f_{0}g_{0}$ to $g_{1}f_{1}g_{0}$ with respect to $\cylinder$ is Lemma \ref{HomotopyPreAndPostCompositionLemma}. 

In addition, since $h$ defines a homotopy over $a$ from $f_{0}$ to $f_{1}$ with respect to $\cylinder$ and $(j_{1},j_{2})$, the following diagram in $\cl{A}$ commutes. \sq{\cyl(a_{1}),a_{2},\cyl(a),a,h,j_{2},\cyl(j_{1}),p(a)} %
Since the diagram \triother{a_{0},a_{1},a,g_{0},j_{1},j_{0}} in $\cl{A}$ commutes, we also have that the following diagram in $\cl{A}$ commutes. \tri[{4,3}]{\cyl(a_{0}),\cyl(a_{1}),\cyl(a),\cyl(g_{0}),\cyl(j_{1}),\cyl(j_{0})} %
Moreover, we have that the following diagram in $\cl{A}$ commutes. \tri{a_{2},a_{3},a,g_{1},j_{3},j_{2}} %
Putting the last three observations together, we have that the following diagram in $\cl{A}$ commutes. \sq[{7,3}]{\cyl(a_{0}),a_{3},\cyl(a),a,g_{1} \circ h \circ \cyl(g_{0}),j_{3},\cyl(j_{0}),p(a)}%

\end{proof} 

\begin{rmk} An analogous result holds for homotopies under an object with respect to $\cylinder$. We shall not need this. \end{rmk}

\end{chapter}

\begin{chapter}{Cofibrations and fibrations} \label{CofibrationsAndFibrationsChapter}

We introduce the notion of a cofibration with respect to a cylinder or a co-cylinder in a formal category $\cl{A}$, and the dual notion of a fibration with respect to a cylinder or a co-cylinder in $\cl{A}$. Given both a cylinder $\cylinder$ and a co-cylinder $\cocylinder$ in $\cl{A}$ which are adjoint, we characterise fibrations with respect to $\cocylinder$ as fibrations with respect to $\cylinder$, and characterise cofibrations with respect to $\cylinder$ as cofibrations with respect to $\cocylinder$. 

Thus far, this is as for the homotopy theory of topological spaces. In an abstract setting, cofibrations and fibrations are treated in \cite{KampsKanBedingungenUndAbstrakteHomotopietheorie} and \cite{KampsPorterAbstractHomotopyAndSimpleHomotopyTheory}, for example. We must go a little further. 

Let us assume once more that we have a cylinder $\cylinder$ and a co-cylinder $\cocylinder$ in a formal category $\cl{A}$. We introduce the notion of a normally cloven fibration with respect to a cylinder $\cylinder$ or $\cocylinder$, which is a strengthening of the notion of a fibration. Roughly speaking, we impose two requirements upon the lifts of homotopies that define a fibration: that these lifts are compatible, and that identity homotopies lift to identity homotopies. If $\cylinder$ is left adjoint to $\cocylinder$, we characterise normally cloven fibrations with respect to $\cocylinder$ exactly as normally cloven fibrations with respect to $\cylinder$. 

We introduce the dual notion of a normally cloven cofibration with respect to $\cylinder$ or $\cocylinder$. If $\cylinder$ is left adjoint to $\cocylinder$, we characterise normally cloven cofibrations with respect to $\cylinder$ exactly as normally cloven cofibrations with respect to $\cocylinder$.

A composition of cofibrations is a cofibration, and a composition of fibrations is a fibration. Moreover, suppose that we have commutative diagrams \twosq{a_{2},a_{0},a_{3},a_{1},g_{0},j,j',g_{1},a_{0},a_{2},a_{1},a_{3},r_{0},j',j,r_{1}} in $\cl{A}$ such that $r_{0}$ is a retraction of $g_{0}$, and such that $r_{1}$ is a retraction of $g_{1}$. If $j$ is a cofibration, then $j'$ is also a cofibration. All this holds equally for normally cloven cofibrations, and dually for fibrations and normally cloven fibrations. 

In \ref{ExampleChapter}, we shall construct a homotopy theory of categories by means of our abstract theory. A normally cloven fibration in our sense with respect to this homotopy theory is exactly what is known as a normally cloven iso-fibration. This motivates our choice of terminology. 

In an abstract setting, van den Berg and Garner recently discussed the notion of a normally cloven fibration in \cite{VanDenBergGarnerTopologicalAndSimplicialModelsOfIdentityTypes}, independently of the author. We refer the reader to around Proposition 6.1.5. We shall present further ideas from this paper in \ref{MappingCylindersAndMappingCoCylindersChapter}.

\begin{assum} Let $\cl{C}$ be a 2-category with a final object. Suppose that pushouts and pullbacks of 2-arrows of $\cl{C}$ give rise to pushouts and pullbacks in formal categories, in the sense of Definition \ref{PushoutsPullbacks2ArrowsArePushoutsPullbacksInFormalCategoriesTerminology}. Let $\cl{A}$ be an object of $\cl{C}$. As before, we view $\cl{A}$ as a formal category, writing of objects and arrows of $\cl{A}$. \end{assum}

\begin{defn} Let $\cylinder = \big( \cyl, i_0, i_1 \big)$ be a cylinder in $\cl{A}$. An arrow \ar{a_{0},a_{1},j} of $\cl{A}$ is a {\em cofibration} with respect to $\cylinder$ if, for any commutative diagram \sq{a_{0},\cyl(a_{0}),a_{1},a_{2},{i_{0}(a_{0})},h,j,f} in $\cl{A}$, there is an arrow \ar{\cyl(a_{1}),a_{2},k} with respect to $\cylinder$ such that the following diagram in $\cl{A}$ commutes. \pushout[{4,3,-1}]{a_{0},\cyl(a_{0}),a_{1},\cyl(a_{1}),a_{2},{i_{0}(a_{0})},\cyl(j),j,{i_{0}(a_{1})},h,f,k} \end{defn}

\begin{defn} Let $\cocylinder = \big( \cocyl, e_0, e_1 \big)$ be a co-cylinder in $\cl{A}$. An arrow \ar{a_{0},a_{1},f} of $\cl{A}$ is a {\em fibration} with respect to $\cocylinder$ if $f^{op}$ is a cofibration with respect to the cylinder $\cocylinder^{op}$ in $\cl{A}^{op}$. \end{defn}

\begin{defn} Let $\cylinder = \big( \cyl, i_0, i_1 \big)$ be a cylinder in $\cl{A}$. An arrow \ar{a_{0},a_{1},j} of $\cl{A}$ is a {\em trivial cofibration} with respect to $\cylinder$ if it is both a cofibration and a homotopy equivalence with respect to $\cylinder$. \end{defn}

\begin{defn} Let $\cocylinder = \big( \cocyl, e_0, e_1 \big)$ be a co-cylinder in $\cl{A}$. An arrow \ar{a_{0},a_{1},j} of $\cl{A}$ is a {\em trivial fibration} with respect to $\cocylinder$ if it is both a fibration and a homotopy equivalence with respect to $\cocylinder$. \end{defn}

\begin{defn} Let $\cylinder = \big( \cyl, i_0, i_1 \big)$ be a cylinder in $\cl{A}$. An arrow \ar{a_{1},a_{2},f} of $\cl{A}$ is a {\em fibration} with respect to $\cylinder$ if, for any commutative diagram \sq{a_{0},a_{1},\cyl(a_{0}),a_{2},g,f,i_{0}(a_{0}),h} in $\cl{A}$, there is a homotopy \ar{\cyl(a_{0}),a_{1},l} with respect to $\cylinder$ such that the following diagram in $\cl{A}$ commutes. \liftingsquare{a_{0},a_{1},\cyl(a_{0}),a_{2},g,f,i_{0}(a_{0}),h,l} \end{defn}

\begin{prpn} \label{FibrationCoCylinderIffFibrationCylinderProposition} Let $\cylinder = \big( \cyl, i_0, i_1 \big)$ be a cylinder in $\cl{A}$, and let $\cocylinder = \big( \cocyl, e_0, e_1 \big)$ be a co-cylinder in $\cl{A}$. Suppose that $\cylinder$ is left adjoint to $\cocylinder$. An arrow \ar{a_{1},a_{2},f} is a fibration with respect to $\cocylinder$ if and only if it is a fibration with respect to $\cylinder$. \end{prpn}

\begin{proof} We first prove that if $f$ is a fibration with respect to $\cocylinder$, then it is a fibration with respect to $\cylinder$. To this end, suppose that we have a commutative diagram in $\cl{A}$ as follows. \sq{a_{0},a_{1},\cyl(a_{0}),a_{2},g,f,i_{0}(a_{0}),h} %
Adopting the shorthand of Recollection \ref{AdjunctionRecollection}, let \ar{{\mathsf{Hom}_{\cl{A}}\big(\cyl(-),- \big)},{\mathsf{Hom}_{\cl{A}}\big(-,\cocyl(-) \big)},\mathsf{adj}} denote the natural isomorphism which the adjunction between $\cyl$ and $\cocyl$ gives rise to. 

Since $\cylinder$ is left adjoint to $\cocylinder$, the following diagram in $\cl{A}$ commutes. \sq{a_{0},\cyl(a_{0}),\cocyl(a_{2}),a_{2},i_{0}(a_{0}),h,\mathsf{adj}(h),e_{0}(a_{2})} %
By the commutativity of the last two diagrams, we have that the following diagram in $\cl{A}$ commutes. \sq[{4,3}]{a_{0},a_{1},a_{2},\cocyl(a_{2}),a_{2},g,f,\mathsf{adj}(h),e_{0}(a_{2})} Thus there is a homotopy \ar{a_{0},\cocyl(a_{1}),k} with respect to $\cocylinder$ such that the following diagram in $\cl{A}$ commutes, since $f$ is a fibration with respect to $\cocylinder$. \pullback[{4,4,-1}]{\cocyl(a_{1}),a_{1},\cocyl(a_{2}),a_{2},a_{0},e_{0}(a_{1}),f,\cocyl(f),e_{0}(a_{2}),g,{\mathsf{adj}(h)},k} %
By the naturality of $\mathsf{adj}$, we also have that the following diagram in $\cl{A}$ commutes. \tri[{4,3}]{a_{0},\cocyl(a_{1}),\cocyl(a_{2}),k,\cocyl(f),{\mathsf{adj}\Big(f \circ \mathsf{adj}^{-1}(k) \Big)}} We deduce that $\mathsf{adj}\big( f \circ \mathsf{adj}^{-1}(k) \big) = \mathsf{adj}(h)$, and hence that the following diagram in $\cl{A}$ commutes. \tri[{5,3}]{\cyl(a_{0}),a_{1},a_{2},{\mathsf{adj}^{-1}(k)},f,h} %
Moreover, since $\cylinder$ is left adjoint to $\cocylinder$ the following diagram in $\cl{A}$ commutes. \sq{a_{0},\cocyl(a_{1}),\cyl(a_{0}),a_{1},k,e_{0}(a_{1}),i_{0}(a_{0}),\mathsf{adj}^{-1}(k)} Hence, by the commutativity of the diagram which defines $k$, the following diagram in $\cl{A}$ commutes. \triother[{5,3}]{a_{0},\cyl(a_{0}),a_{1},i_{0}(a_{0}),\mathsf{adj}^{-1}(k),g} %
Putting this all together, we have shown that the following diagram in $\cl{A}$ commutes, concluding this direction of the proof. \liftingsquare[{9,4}]{a_{0},a_{1},\cyl(a_{0}),a_{2},g,f,i_{0}(a_{0}),h,\mathsf{adj}^{-1}(k)} 

We now prove that if $f$ is a fibration with respect to $\cylinder$, then $f$ is a fibration with respect to $\cocylinder$. To this end, suppose now that we have a commutative diagram in $\cl{A}$ as follows. \sq[{4,3}]{a_{0},a_{1},\cocyl(a_{2}),a_{2},g,f,h,e_{0}(a_{2})} %
Adopting again the shorthand of Recollection \ref{AdjunctionRecollection}, let \ar{{\mathsf{Hom}_{\cl{A}}\big(\cyl(-),- \big)},{\mathsf{Hom}_{\cl{A}}\big(-,\cocyl(-) \big)},\mathsf{adj}} denote the natural isomorphism which the adjunction between $\cyl$ and $\cocyl$ gives rise to. Since $\cylinder$ is left adjoint to $\cocylinder$, the following diagram in $\cl{A}$ commutes. \sq{a_{0},\cocyl(a_{2}),\cyl(a_{0}),a_{2},h,e_{0}(a_{2}),i_{0}(a_{0}),\mathsf{adj}^{-1}(h)} %
By the commutativity of the last two diagrams, we have that the following diagram in $\cl{A}$ commutes. \sq[{5,3}]{a_{0},a_{1},\cyl(a_{0}),a_{2},g,f,i_{0}(a_{0}),\mathsf{adj}^{-1}(h)} Thus there is a homotopy \ar{\cyl(a_{0}),a_{1},l} with respect to $\cylinder$ such that the following diagram in $\cl{A}$ commutes, since $f$ is a fibration with respect to $\cylinder$. \liftingsquare[{6,4}]{a_{0},a_{1},\cyl(a_{0}),a_{2},g,f,i_{0}(a_{0}),{\mathsf{adj}^{-1}(h)},l} %
By the naturality of $\mathsf{adj}$, we also have that the following diagram in $\cl{A}$ commutes. \tri[{4,3}]{a_{0},\cocyl(a_{1}),\cocyl(a_{2}),{\mathsf{adj}(l)},\cocyl(f),{\mathsf{adj}\big(f \circ l \big)}} We deduce that the following diagram in $\cl{A}$ commutes. \tri[{5,3}]{a_{0},\cocyl(a_{1}),\cocyl(a_{2}),{\mathsf{adj}(l)},\cocyl(f),h} %
Moreover, since $\cylinder$ is left adjoint to $\cocylinder$, the following diagram in $\cl{A}$ commutes. \sq{a_{0},\cyl(a_{0}),\cocyl(a_{1}),a_{1},i_{0}(a_{0}),k,\mathsf{adj}(l),e_{0}(a_{1})} Hence, by the commutativity of the diagram which defines $l$, the following diagram in $\cl{A}$ commutes. \triother[{4,3}]{a_{0},\cocyl(a_{1}),a_{1},\mathsf{adj}(l),e_{0}(a_{1}),g} %
Putting this all together, we have shown that the following diagram in $\cl{A}$ commutes, concluding this direction of the proof. \pullback[{4,4,-1}]{\cocyl(a_{1}),a_{1},\cocyl(a_{2}),a_{2},a_{0},e_{0}(a_{1}),f,\cocyl(f),e_{0}(a_{2}),g,h,{\mathsf{adj}(l)}} \end{proof} 

\begin{defn} Let $\cocylinder = \big( \cocyl, i_0, i_1 \big)$ be a co-cylinder in $\cl{A}$. An arrow \ar{a_{0},a_{1},j} of $\cl{A}$ is a {\em cofibration} with respect to $\cocylinder$ if, for any commutative diagram \sq{a_{0},\cocyl(a_{2}),a_{1},a_{3},h,e_{0}(a_{2}),j,g} in $\cl{A}$, there is a homotopy \ar{a_{1},\cocyl(a_{2}),l} with respect to $\cocylinder$ such that the following diagram in $\cl{A}$ commutes. \liftingsquare{a_{0},\cocyl(a_{2}),a_{1},a_{2},h,e_{0}(a_{2}),j,g,l} \end{defn}

\begin{rmk}  Let $\cocylinder = \big( \cocyl, i_0, i_1 \big)$ be a co-cylinder in $\cl{A}$. An arrow \ar{a_{0},a_{1},j} of $\cl{A}$ is a cofibration with respect to $\cocylinder$ if and only if $j^{op}$ is a fibration with respect to the cylinder $\cocylinder^{op}$ in $\cl{A}^{op}$. \end{rmk} 

\begin{cor} \label{CofibrationCoCylinderIffCofibrationCylinderCorollary} Let $\cylinder = \big( \cyl, i_0, i_1 \big)$ be a cylinder in $\cl{A}$, and let $\cocylinder = \big( \cocyl, e_0, e_1 \big)$ be a co-cylinder in $\cl{A}$. Suppose that $\cylinder$ is left adjoint to $\cocylinder$. 

An arrow \ar{a_{0},a_{1},j} is a cofibration with respect to $\cylinder$ if and only if it is a cofibration with respect to $\cocylinder$. \end{cor}

\begin{proof} Follows immediately from Proposition \ref{FibrationCoCylinderIffFibrationCylinderProposition} by duality. \end{proof}

\begin{prpn} \label{IdentityIsCofibrationProposition} Let $\cylinder = \big( \cyl, i_0, i_1 \big)$ be a cylinder in $\cl{A}$. Then for any object $a$ of $\cl{A}$, the arrow \ar{a,a,id} of $\cl{A}$ is a cofibration with respect to $\cylinder$. \end{prpn} 

\begin{proof} Suppose that we have a commutative diagram in $\cl{A}$ as follows. \sq{a,\cyl(a),a,a',i_{0}(a),h,id,g} Then the following diagram in $\cl{A}$ commutes. \pushout[{4,3,-1}]{a,\cyl(a),a,\cyl(a),a',i_{0}(a),id,id,i_{0}(a),h,g,h} \end{proof} 

\begin{prpn} \label{CompositionOfCofibrationsIsACofibrationProposition} Let $\cylinder = \big( \cyl, i_0, i_1 \big)$ be a cylinder in $\cl{A}$. Let \ar{a_{0},a_{1},j_{0}} and \ar{a_{1},a_{2},j_{1}} be arrows of $\cl{A}$ which are cofibrations with respect to $\cylinder$. Then $j_{1} \circ j_{0}$ is a cofibration with respect to $\cylinder$. 
\end{prpn} 

\begin{proof} Suppose that we have a commutative diagram in $\cl{A}$ as follows. \sq{a_{0},\cyl(a_{0}),a_{2},a_{3},i_{0}(a_{0}),h,j_{1} \circ j_{0},g} %
Since $j_{0}$ is a cofibration with respect to $\cylinder$, there is an arrow \ar{\cyl(a_{1}),a_{3},k_{0}} of $\cl{A}$ such that the following diagram in $\cl{A}$ commutes. \pushout{a_{0},\cyl(a_{0}),a_{2},\cyl(a_{2}),a_{3},i_{0}(a_{0}),\cyl(j_{0}),j_{0},i_{0}(a_{1}),h,g \circ j_{1},k_{0}} %
Since $j_{1}$ is a cofibration with respect to $\cylinder$, there is an arrow \ar{\cyl(a_{2}),a_{3},k_{1}} of $\cl{A}$ such that the following diagram in $\cl{A}$ commutes. \pushout{a_{1},\cyl(a_{1}),a_{2},\cyl(a_{2}),a_{3},i_{0}(a_{1}),\cyl(j_{1}),j_{1},i_{0}(a_{2}),k_{0},g,k_{1}} %
Putting the last two observations together, we have that the following diagram in $\cl{A}$ commutes. \pushout[{3.5,3,2}]{a_{0},\cyl(a_{0}),a_{2},\cyl(a_{2}),a_{3},i_{0}(a_{0}),\cyl(j_{1} \circ j_{0}),j_{1} \circ j_{0},i_{0}(a_{2}),h,g,k_{1}}  
\end{proof}

\begin{cor} \label{CompositionOfFibrationsIsAFibrationCorollary} Let $\cocylinder = \big( \cocyl, e_0, e_1 \big)$ be a co-cylinder in $\cl{A}$. Let \ar{a_{0},a_{1},f_{0}} and \ar{a_{1},a_{2},f_{1}} be arrows of $\cl{A}$ which are fibrations with respect to $\cocylinder$. Then $f_{1} \circ f_{0}$ is a fibration with respect to $\cocylinder$. \end{cor} 

\begin{proof} Follows immediately from Proposition \ref{CompositionOfCofibrationsIsACofibrationProposition} by duality. \end{proof}

\begin{prpn} \label{RetractionCofibrationIsCofibrationProposition} Let $\cylinder = \big( \cyl, i_{0}, i_{1} \big)$ be a cylinder in $\cl{A}$. Let \ar{a_{0},a_{1},j} be an arrow of $\cl{A}$ which is a cofibration with respect to $\cylinder$. 

Suppose that we have commutative diagrams \twosq{a_{2},a_{0},a_{3},a_{1},g_{0},j,j',g_{1},a_{0},a_{2},a_{1},a_{3},r_{0},j',j,r_{1}} in $\cl{A}$, such that $r_{0}$ is a retraction of $g_{0}$, and such that $r_{1}$ is a retraction of $g_{1}$. Then $j'$ is a cofibration with respect to $\cylinder$. \end{prpn} 

\begin{proof} Suppose that we have a commutative diagram in $\cl{A}$ as follows. \sq{a_{2},\cyl(a_{2}),a_{3},a_{4},i_{0}(a_{2}),h,j',f} %
Then the following diagram in $\cl{A}$ commutes. \trapeziumstwo[{3,3.5,0,0}]{a_{0},\cyl(a_{0}),a_{1},a_{2},\cyl(a_{2}),a_{3},a_{4},i_{0}(a_{0}),\cyl(r_{0}),h,j,r_{1},f,r_{0},i_{0}(a_{2}),j'} %
Thus, since $j$ is a cofibration with respect to $\cylinder$, there is an arrow \ar{\cyl(a_{1}),a_{4},k} of $\cl{A}$ such that the following diagram in $\cl{A}$ commutes. \pushout{a_{0},\cyl(a_{0}),a_{1},\cyl(a_{1}),a_{4},i_{0}(a_{0}),\cyl(j),j,i_{0}(a_{1}),h \circ \cyl(r_{0}),f \circ r_{1}, k} %
Let \ar{\cyl(a_{3}),a_{4},l} denote the arrow $k \circ \cyl(g_{1})$ of $\cl{A}$. We claim that the following diagram in $\cl{A}$ commutes. \pushout{a_{2},\cyl(a_{2}),a_{3},\cyl(a_{3}),a_{4},i_{0}(a_{2}),\cyl(j'),j',i_{0}(a_{3}),h,f,l} 

Firstly, we have that the following diagram in $\cl{A}$ commutes. \trapeziumsfive{a_{3},\cyl(a_{3}),a_{1},\cyl(a_{1}),a_{3},a_{4},i_{0}(a_{1}),l,id,f,g_{1},\cyl(g_{1}),i_{0}(a_{2}),r_{1},k} 

Secondly, we have that the following diagram in $\cl{A}$ commutes. \trapeziumsfive{\cyl(a_{2}),\cyl(a_{3}),\cyl(a_{0}),\cyl(a_{1}),\cyl(a_{3}),a_{4},\cyl(j'),l,id,h,\cyl(g_{0}),\cyl(g_{1}),\cyl(j),\cyl(r_{0}),k}  
    
\end{proof}

\begin{cor} \label{RetractionTrivialCofibrationIsTrivialCofibrationCorollary} Let $\cylinder = \big( \cyl, i_{0}, i_{1} \big)$ be a cylinder in $\cl{A}$. Let \ar{a_{0},a_{1},j} be an arrow of $\cl{A}$ which is a trivial cofibration with respect to $\cylinder$. 

Suppose that we have commutative diagrams \twosq{a_{2},a_{0},a_{3},a_{1},g_{0},j,j',g_{1},a_{0},a_{2},a_{1},a_{3},r_{0},j',j,r_{1}} in $\cl{A}$, such that $r_{0}$ is a retraction of $g_{0}$, and such that $r_{1}$ is a retraction of $g_{1}$. Then $j'$ is a trivial cofibration with respect to $\cylinder$. \end{cor}

\begin{proof} Follows immediately from Proposition \ref{RetractionCofibrationIsCofibrationProposition} and Proposition \ref{RetractionHomotopyEquivalenceIsHomotopyEquivalenceProposition}. \end{proof}

\begin{cor} \label{RetractionFibrationIsFibrationCorollary} Let $\cocylinder = \big( \cocyl, e_{0}, e_{1} \big)$ be a co-cylinder in $\cl{A}$. Let \ar{a_{0},a_{1},f} be an arrow of $\cl{A}$ which is a fibration with respect to $\cocylinder$. Suppose that we have commutative diagrams \twosq{a_{2},a_{0},a_{3},a_{1},g_{0},f,f',g_{1},a_{0},a_{2},a_{1},a_{3},r_{0},f',f,r_{1}} in $\cl{A}$, such that $r_{0}$ is a retraction of $g_{0}$ and such that $r_{1}$ is a retraction of $g_{1}$. Then $f'$ is a fibration with respect to $\cocylinder$. \end{cor} 

\begin{proof} Follows immediately from Proposition \ref{RetractionCofibrationIsCofibrationProposition} by duality. \end{proof} 

\begin{cor} \label{RetractionTrivialFibrationIsTrivialFibrationCorollary} Let $\cocylinder = \big( \cocyl, e_{0}, e_{1} \big)$ be a co-cylinder in $\cl{A}$. Let \ar{a_{0},a_{1},f} be an arrow of $\cl{A}$ which is a trivial fibration with respect to $\cocylinder$. 

Suppose that we have commutative diagrams \twosq{a_{2},a_{0},a_{3},a_{1},g_{0},f,f',g_{1},a_{0},a_{2},a_{1},a_{3},r_{0},f',f,r_{1}} in $\cl{A}$, such that $r_{0}$ is a retraction of $g_{0}$, and such that $r_{1}$ is a retraction of $g_{1}$. Then $f'$ is a trivial fibration with respect to $\cocylinder$. \end{cor} 

\begin{proof} Follows immediately from Corollary \ref{RetractionFibrationIsFibrationCorollary} and Proposition \ref{RetractionHomotopyEquivalenceIsHomotopyEquivalenceProposition}. \end{proof} 

\begin{notn} Let $\cylinder = \big( \cyl,i_0,i_1 \big)$ be a cylinder in $\cl{A}$. Let \ar{a_{0},a_{1},j} be an arrow of $\cl{A}$, and let $a_{2}$ be an object of $\cl{A}$. 

Let $\Sigma_{j,a_{2}}^{\cylinder}$ denote the set of pairs $(g,h)$ consisting of an arrow \ar{a_{1},a_{2},g} of $\cl{A}$ and a homotopy \ar{\cyl(a_{0}),a_{2},h} with respect to $\cylinder$, such that the following diagram in $\cl{A}$ commutes. \sq{a_{0},\cyl(a_{0}),a_{1},a_{2},i_{0}(a_{0}),h,j,g} 

Let $\Upsilon_{j,a_{2}}^{\cylinder}$ denote the set of homotopies \ar{\cyl(a_{1}),a_{2},k} with respect to $\cylinder$ such that the following diagram in $\cl{A}$ commutes. \pushout[{4,3,-1}]{a_{0},\cyl(a_{0}),a_{1},\cyl(a_{1}),a_{2},i_{0}(a_{0}),\cyl(j),j,i_{0}(a_{1}),h,g,k} \end{notn}

\begin{defn} Let $\cylinder = \big( \cyl,i_0,i_1 \big)$ be a cylinder in $\cl{A}$. A {\em cofibration equipped with a cleavage} with respect to $\cylinder$ is an arrow \ar{a_{0},a_{1},j} of $\cl{A}$, together with a map \ar{\Sigma_{j,a_{2}}^{\cylinder},\Upsilon_{j,a_{2}}^{\cylinder},k_{a_{2}}} for every object $a_{2}$ of $\cl{A}$. \end{defn}   

\begin{defn} \label{NormallyClovenCofibrationCylinderDefinition} Let $\cylinder = \big( \cyl,i_0,i_1,p \big)$ be a cylinder in $\cl{A}$ equipped with a contraction structure $p$. A {\em normally cloven cofibration} with respect to $\cylinder$ is an arrow \ar{a_{0},a_{1},j} of $\cl{A}$ which is a cofibration equipped with a cleavage with respect to $\cylinder$, satisfying the following conditions, for which we denote by \ar{\Sigma_{j,a_{2}}^{\cylinder},\Upsilon_{j,a_{2}}^{\cylinder},k_{a_{2}}} the map of the cleavage corresponding to an object $a_{2}$ of $\cl{A}$. 

\begin{itemize}[itemsep=1em,topsep=1em] 

\item[(i)] Suppose that we have a commutative diagram in $\cl{A}$ as follows. \triother{a_{0},a_{1},a_{2},j,g_{1},g_{0}} %
Then the following diagram in $\cl{A}$ commutes. \tri{\cyl(a_{1}),a_{1},a_{2},p(a_{1}),g_{1},{k_{a_{2}}\big(g_{1},g_{0} \circ p(a_{0})\big)}} 

\item[(ii)] Suppose that we have a commutative diagram in $\cl{A}$ as follows. \sq{a_{0},\cyl(a_{0}),a_{1},a_{2},i_{0}(a_{0}),h,j,g_{1}} Then for any arrow \ar{a_{2},a_{3},g_{2}} of $\cl{A}$, the following diagram in $\cl{A}$ commutes. \tri[{5,3}]{\cyl(a_{1}),a_{2},a_{3},{k_{a_{2}}(g_{1},h)},g_{2},{k_{a_{3}}(g_{2} \circ g_{1}, g_{2} \circ h)}} \end{itemize} \end{defn}

\begin{rmk} Let $\cylinder = \big( \cyl,i_0,i_1,p \big)$ be a cylinder in $\cl{A}$ equipped with a contraction structure $p$. Let $j$ be an arrow of $\cl{A}$ which is a cofibration equipped with a cleavage with respect to $\cylinder$.

We shall refer to condition (i) of Definition \ref{NormallyClovenCofibrationCylinderDefinition} as {\em lifting of identities}, and to condition (ii) of Definition \ref{NormallyClovenCofibrationCylinderDefinition} as {\em compatibility of liftings}. \end{rmk}

\begin{rmk} Let $\cylinder = \big( \cyl,i_0,i_1,p \big)$ be a cylinder in $\cl{A}$ equipped with a contraction structure $p$. If an arrow \ar{a_{0},a_{1},j} of $\cl{A}$ is a cofibration with respect to $\cylinder$, we can think of $\cyl(a_{1})$ as a weak pushout of $j$ along the arrow \ar{a_{0},\cyl(a_{0}),i_{0}(a_{0})} of $\cl{A}$. 

The lifting of identities and compatibility of liftings conditions bring $\cyl(a_{1})$ closer to an actual pushout of $j$ along $i_{0}(a_{0})$. \end{rmk}

\begin{terminology} \label{NormallyClovenCofibrationTerminology} Let $\cylinder = \big( \cyl,i_0,i_1,p \big)$ be a cylinder in $\cl{A}$ equipped with a contraction structure $p$. Let \ar{a_{0},a_{1},j} be a cofibration equipped with a cleavage with respect to $\cylinder$, which is moreover a normally cloven cofibration with respect to $\cylinder$. 

We shall typically refer to $j$ as a normally cloven cofibration, without explicitly mentioning its cleavage. \end{terminology}

\begin{notn} Let $\cocylinder = \big( \cocyl,e_0,e_1 \big)$ be a co-cylinder in $\cl{A}$. Let \ar{a_{1},a_{2},f} be an arrow of $\cl{A}$, and let $a_{0}$ be an object of $\cl{A}$. 

Let $\Sigma_{f,a_{0}}^{\cocylinder}$ denote the set of pairs $(g,h)$ consisting of an arrow \ar{a_{0},a_{1},g} of $\cl{A}$ and a homotopy \ar{a_{0},\cocyl(a_{2}),h} with respect to $\cocylinder$, such that the following diagram in $\cl{A}$ commutes. \sq[{4,3}]{a_{0},a_{1},\cocyl(a_{2}),a_{2},g,f,h,e_{0}(a_{2})} 

Let $\Upsilon_{f,a_{0}}^{\cocylinder}$ denote the set of homotopies \ar{a_{0},\cocyl(a_{1}),k} with respect to $\cocylinder$ such that the following diagram in $\cl{A}$ commutes. \pullback[{4,3,-1}]{\cocyl(a_{1}),a_{1},\cocyl(a_{2}),a_{2},a_{0},e_{0}(a_{1}),f,\cocyl(f),e_{0}(a_{2}),g,h,k} \end{notn}

\begin{defn} Let $\cocylinder = \big( \cocyl,e_0,e_1 \big)$ be a co-cylinder in $\cl{A}$. A {\em fibration equipped with a cleavage} with respect to $\cocylinder$ is an arrow \ar{a_{1},a_{2},f} of $\cl{A}$, together with a map \ar{\Sigma_{f,a_{0}}^{\cocylinder},\Upsilon_{f,a_{0}}^{\cocylinder},k_{a_{0}}} for every object $a_{0}$ of $\cl{A}$. \end{defn}   

\begin{rmk} \label{CleavageDualityCofibCylinderFibCoCylinderRemark} Let $\cocylinder = \big( \cocyl,e_0,e_1 \big)$ be a co-cylinder in $\cl{A}$, and let \ar{a_{1},a_{2},f} be a fibration equipped with a cleavage with respect to $\cocylinder$.  Let $a_{0}$ be an object of $\cl{A}$, and let \ar{\Sigma_{f,a_{0}}^{\cocylinder},\Upsilon_{f,a_{0}}^{\cocylinder},k_{a_{0}}} denote the corresponding map of the cleavage. 

Associating to a pair $(g^{op},h^{op})$ in $\Sigma_{f^{op},a_{0}}^{\cocylinder^{op}}$ the arrow $\big(k_{a_{0}}(g,h)\big)^{op}$ of $\cl{A}^{op}$, defines a map \ar{\Sigma_{f^{op},a_{0}}^{\cocylinder^{op}},{\Upsilon_{f^{op},a_{0}}^{\cocylinder^{op}}}}, which we shall denote by $(k^{op})_{a_{0}}$. Thus a cleavage with respect to $f$ and $\cocylinder$ gives rise to a cleavage with respect to $f^{op}$ and the cylinder $\cocylinder^{op}$ in $\cl{A}^{op}$. \end{rmk} 

\begin{defn} \label{NormallyClovenFibrationCoCylinderDefinition} Let $\cocylinder = \big( \cocyl,e_0,e_1,c \big)$ be a co-cylinder in $\cl{A}$ equipped with a contraction structure $c$. A {\em normally cloven fibration} with respect to $\cocylinder$ is an arrow \ar{a_{1},a_{2},f} of $\cl{A}$ which is a fibration equipped with a cleavage with respect to $\cocylinder$, such that $f^{op}$ equipped with the cleavage of Remark \ref{CleavageDualityCofibCylinderFibCoCylinderRemark} is a normally cloven cofibration with respect to the cylinder $\cocylinder^{op}$ in $\cl{A}^{op}$ equipped with the contraction structure $c^{op}$. \end{defn}

\begin{terminology} \label{NormallyClovenFibrationCoCylinderTerminology} Let $\cocylinder = \big( \cocyl,e_0,e_1,c \big)$ be a co-cylinder in $\cl{A}$ equipped with a contraction structure $c$. Let \ar{a_{1},a_{2},f} be a fibration equipped with a cleavage with respect to $\cocylinder$, which is moreover a normally cloven fibration with respect to $\cocylinder$. 

We shall typically refer to $f$ as a normally cloven fibration, without explicitly mentioning its cleavage. \end{terminology}

\begin{defn} Let $\cylinder = \big( \cyl, i_0, i_1, p \big)$ be a cylinder in $\cl{A}$ equipped with a contraction structure $p$. An arrow \ar{a_{0},a_{1},j} of $\cl{A}$ is a {\em trivial normally cloven cofibration} with respect to $\cylinder$ if it is both a normally cloven cofibration and a homotopy equivalence with respect to $\cylinder$. \end{defn}

\begin{defn} Let $\cocylinder = \big( \cocyl, e_0, e_1, c \big)$ be a co-cylinder in $\cl{A}$ equipped with a contraction structure $c$. An arrow \ar{a_{0},a_{1},f} of $\cl{A}$ is a {\em trivial normally cloven fibration} with respect to $\cocylinder$ if it is both a normally cloven fibration and a homotopy equivalence with respect to $\cocylinder$. \end{defn}

\begin{notn} \label{CleavageFibrationCylinderNotation} Let $\cylinder = \big( \cyl,i_0,i_1 \big)$ be a cylinder in $\cl{A}$. Let \ar{a_{1},a_{2},f} be an arrow of $\cl{A}$, and let $a_{0}$ be an object of $\cl{A}$. 

Let $\Delta_{f,a_{0}}^{\cylinder}$ denote the set of pairs $(g,h)$ consisting of an arrow \ar{a_{0},a_{1},g} of $\cl{A}$ and a homotopy \ar{\cyl(a_{0}),a_{1},h} with respect to $\cylinder$, such that the following diagram in $\cl{A}$ commutes. \sq{a_{0},a_{1},\cyl(a_{0}),a_{2},g,f,i_{0}(a_{0}),h} 

Let $\Omega_{f,a_{0}}^{\cylinder}$ denote the set of homotopies \ar{\cyl(a_{0}),a_{1},l} with respect to $\cylinder$ such that the following diagram in $\cl{A}$ commutes. \liftingsquare{a_{0},a_{1},\cyl(a_{0}),a_{2},g,f,i_{0}(a_{0}),h,l} \end{notn}

\begin{defn} Let $\cylinder = \big( \cyl,i_0,i_1 \big)$ be a cylinder in $\cl{A}$. A {\em fibration equipped with a cleavage} with respect to $\cylinder$ is an arrow \ar{a_{1},a_{2},f} of $\cl{A}$, together with a map \ar{\Delta_{f,a_{0}}^{\cylinder},\Omega_{f,a_{0}}^{\cylinder},l_{a_{0}}} for every object $a_{0}$ of $\cl{A}$. \end{defn}   

\begin{defn} \label{NormallyClovenFibrationCylinderDefinition} Let $\cylinder = \big( \cyl, i_0, i_1, p \big)$ be a cylinder in $\cl{A}$ equipped with a contraction structure $p$. A {\em normally cloven fibration} with respect to $\cylinder$ is an arrow \ar{a_{1},a_{2},f} of $\cl{A}$ which is a fibration equipped with a cleavage with respect to $\cylinder$, satisfying the following conditions, for which we denote by \ar{\Delta_{f,a_{0}}^{\cylinder},\Omega_{f,a_{0}}^{\cylinder},l_{a_{0}}} the map of the cleavage corresponding to an object $a_{0}$ of $\cl{A}$.  

\begin{itemize}[itemsep=1em,topsep=1em]

\item[(i)] Suppose that we have a commutative diagram in $\cl{A}$ as follows. \triother{a_{0},a_{1},a_{2},g_{1},f,g_{2}} Then the following diagram in $\cl{A}$ commutes. \tri{\cyl(a_{0}),a_{0},a_{1},p(a_{0}),g_{1},{l_{a_{0}}(g_{1},g_{2} \circ p(a_{0}))}} 

\item[(ii)] Suppose that we have a commutative diagram in $\cl{A}$ as follows. \sq{a_{0},a_{1},\cyl(a_{0}),a_{2},g_{1},f,i_{0}(a_{0}),h} Then for any arrow \ar{a_{-1},a_{0},g_{0}} of $\cl{A}$ the following diagram in $\cl{A}$ commutes. \tri[{4,3}]{\cyl(a_{-1}),\cyl(a_{0}),a_{1},\cyl(g_{0}),{l_{a_{0}}(g_{1},h)},{l_{a_{-1}}\big(g_{1} \circ g_{0}, h \circ \cyl(g_{0})\big)}} 

\end{itemize} 

\end{defn}

\begin{terminology} Let $\cylinder = \big( \cyl,i_0,i_1,p \big)$ be a cylinder in $\cl{A}$ equipped with a contraction structure $c$. Let \ar{a_{1},a_{2},f} be a fibration equipped with a cleavage with respect to $\cylinder$, which is moreover a normally cloven fibration with respect to $\cylinder$. 

We shall typically refer to $f$ as a normally cloven fibration, without explicitly mentioning its cleavage. \end{terminology}

\begin{notn} Let $\cocylinder = \big( \cocyl,e_0,e_1 \big)$ be a co-cylinder in $\cl{A}$. Let \ar{a_{0},a_{1},j} be an arrow of $\cl{A}$, and let $a_{2}$ be an object of $\cl{A}$. 

Let $\Delta_{j,a_{2}}^{\cocylinder}$ denote the set of pairs $(g,h)$ consisting of an arrow \ar{a_{1},a_{2},g} of $\cl{A}$ and a homotopy \ar{a_{0},\cocyl(a_{2}),h} with respect to $\cocylinder$, such that the following diagram in $\cl{A}$ commutes. \sq{a_{0},\cocyl(a_{2}),a_{1},a_{2},h,e_{0}(a_{2}),j,g} 

Let $\Omega_{j,a_{2}}^{\cocylinder}$ denote the set of homotopies \ar{a_{1},\cocyl(a_{2}),l} with respect to $\cocylinder$ such that the following diagram in $\cl{A}$ commutes. \liftingsquare{a_{0},\cocyl(a_{2}),a_{1},a_{2},h,e_{0}(a_{2}),j,l} \end{notn}

\begin{defn} Let $\cocylinder = \big( \cocyl,e_0,e_1 \big)$ be a co-cylinder in $\cl{A}$. A {\em cofibration equipped with a cleavage} with respect to $\cocylinder$ is an arrow \ar{a_{0},a_{1},j} of $\cl{A}$, together with a map \ar{\Delta_{j,a_{2}}^{\cocylinder},\Omega_{j,a_{2}}^{\cocylinder},l_{a_{2}}} for every object $a_{2}$ of $\cl{A}$. \end{defn}  

\begin{rmk} \label{CleavageDualityFibCylinderCofibCoCylinderRemark} Let $\cocylinder = \big( \cocyl,e_0,e_1 \big)$ be a co-cylinder in $\cl{A}$, and let \ar{a_{0},a_{1},j} be a cofibration equipped with a cleavage with respect to $\cocylinder$.  Let $a_{2}$ be an object of $\cl{A}$, and let \ar{\Delta_{j,a_{2}}^{\cocylinder},\Omega_{j,a_{2}}^{\cocylinder},l_{a_{2}}} denote the corresponding map of the cleavage. 

Associating to a pair $(g^{op},h^{op})$ in $\Delta_{j^{op},a_{2}}^{\cocylinder^{op}}$ the arrow $\big(l_{a_{2}}(g,h)\big)^{op}$ of $\cl{A}^{op}$, defines a map \ar{\Delta_{j^{op},a_{2}}^{\cocylinder^{op}},{\Omega_{j^{op},a_{2}}^{\cocylinder^{op}},}} which we shall denote by $(l^{op})_{a_{2}}$. Thus a cleavage with respect to $j$ and $\cocylinder$ gives rise to a cleavage with respect to $j^{op}$ and the cylinder $\cocylinder^{op}$ in $\cl{A}^{op}$. \end{rmk}

\begin{defn} Let $\cocylinder = \big( \cocyl,e_0,e_1,c \big)$ be a co-cylinder in $\cl{A}$ equipped with a contraction structure $c$. A {\em normally cloven cofibration} with respect to $\cocylinder$ is an arrow \ar{a_{0},a_{1},j} of $\cl{A}$ which is a cofibration equipped with a cleavage with respect to $\cocylinder$, such that $j^{op}$ equipped with the cleavage of Remark \ref{CleavageDualityFibCylinderCofibCoCylinderRemark} is a normally cloven fibration with respect to the cylinder $\cocylinder^{op}$ in $\cl{A}^{op}$ equipped with the contraction structure $c^{op}$. \end{defn}

\begin{terminology} Let $\cocylinder = \big( \cocyl,i_0,i_1,c \big)$ be a cylinder in $\cl{A}$ equipped with a contraction structure $c$. Let \ar{a_{0},a_{1},j} be a cofibration equipped with a cleavage with respect to $\cocylinder$, which is moreover a normally cloven cofibration with respect to $\cocylinder$. 

We shall typically refer to $j$ as a normally cloven fibration, without explicitly mentioning its cleavage. \end{terminology}

\begin{prpn} \label{NormallyClovenFibrationCoCylinderIffNormallyClovenFibrationCylinderProposition} Let $\cylinder = \big( \cyl, i_0, i_1, p \big)$ be a cylinder in $\cl{A}$ equipped with a contraction structure $p$, and let $\cocylinder = \big( \cocyl, e_0, e_1, c \big)$ be a co-cylinder in $\cl{A}$ equipped with a contraction structure $c$. 

Suppose that $\cylinder$ is left adjoint to $\cocylinder$, and that the adjunction between $\cyl$ and $\cocyl$ is compatible with $p$ and $c$. Then an arrow \ar{a_{1},a_{2},f} of $\cl{A}$ is a normally cloven fibration with respect to $\cocylinder$ if and only if it is a normally cloven fibration with respect to $\cylinder$. \end{prpn}

\begin{proof} We first prove that if $f$ is a normally cloven fibration with respect to $\cocylinder$, then it is a normally cloven fibration with respect to $\cylinder$. To this end, adopting the shorthand of Recollection \ref{AdjunctionRecollection}, let \ar{{\mathsf{Hom}_{\cl{A}}\big(\cyl(-),- \big)},{\mathsf{Hom}_{\cl{A}}\big(-,\cocyl(-) \big)},\mathsf{adj}} denote the natural isomorphism which the adjunction between $\cyl$ and $\cocyl$ gives rise to. Suppose that we have a commutative diagram in $\cl{A}$ as follows. \sq{a,a_{1},\cyl(a),a_{2},g_{1},f,i_{0}(a),h} %
As in the proof of Proposition \ref{FibrationCoCylinderIffFibrationCylinderProposition}, we have that the following diagram in $\cl{A}$ commutes. \sq[{4,3}]{a,a_{1},a_{2},\cocyl(a_{2}),g_{1},f,\mathsf{adj}(h),e_{0}(a_{2})} %
Let \ar{\Sigma_{f,a}^{\cocylinder},\Upsilon_{f,a}^{\cylinder},k_{a}} denote the map, of the cleavage with which $f$ is equipped, corresponding to the object $a$ of $\cl{A}$. Then the homotopy \ar[4]{a,\cocyl(a_{1}),k_{a}(g_{1}),\mathsf{adj}(h))} with respect to $\cocylinder$ fits into a commutative diagram in $\cl{A}$ as follows. \pullback[{4,4,-1}]{\cocyl(a_{1}),a_{1},\cocyl(a_{2}),a_{2},a,e_{0}(a_{1}),f,\cocyl(f),e_{0}(a_{2}),g_{1},{\mathsf{adj}(h)},{k_{a}\big(g_{1},\mathsf{adj}(h)\big)}} %
Let $l_{a}(g_{1},h)$ denote the arrow \ar[9]{\cyl(a),a_{1},{\mathsf{adj}^{-1}(k_{a}(g_{1},h))}} of $\cl{A}$. Following again the proof of Proposition \ref{FibrationCoCylinderIffFibrationCylinderProposition}, we have that the following diagram in $\cl{A}$ commutes. \liftingsquare[{9,3}]{a,a_{1},\cyl(a),a_{2},g_{1},f,i_{0}(a),h,{l_{a}(g_{1},h)}} 

We claim that the cleavage given by the maps \ar{\Delta_{f,a}^{\cylinder},\Omega_{f,a}^{\cylinder},l_{a}} defined by $(g_{1},h) \mapsto l_{a}(g_{1},h)$, for an object $a$ of $\cl{A}$, equips $f$ with the structure of a normally cloven fibration with respect to $\cylinder$. 

Indeed, suppose that we have a commutative triangle in $\cl{A}$ as follows. \tri{\cyl(a_{0}),a_{0},a_{2},p(a_{0}),g_{2},h} Then \[ \mathsf{adj}(h) = \mathsf{adj}\big( g_{2} \circ p(a_{0}) \big). \] %
Moreover, the following diagram in $\cl{A}$ commutes since the adjunction between $\cyl$ and $\cocyl$ is compatible with $p$ and $c$. \tri{a_{0},a_{2},\cocyl(a_{2}),g_{2},c(a_{2}),\mathsf{adj}\big(g_{2} \circ p(a_{0})\big)} %
We deduce that the following diagram in $\cl{A}$ commutes. \tri{a_{0},a_{2},\cocyl(a_{2}),g_{2},c(a_{2}),\mathsf{adj}(h)} Hence the following diagram in $\cl{A}$ commutes, since $f$ satisfies the lifting of identities condition of Definition \ref{NormallyClovenFibrationCoCylinderDefinition}. \tri{a_{0},a_{1},\cocyl(a_{1}),g_{1},c(a_{1}),{k_{a_{0}}\big(g_{1},\mathsf{adj}(h)\big)}} %
Thus we have that \[ l_{a_{0}}(g_{1},h)  = \mathsf{adj}^{-1}\big( c(a_{1}) \circ g_{1} \big). \] %
Moreover, the following diagram in $\cl{A}$ commutes since the adjunction between $\cyl$ and $\cocyl$ is compatible with $p$ and $c$. \tri{\cyl(a_{0}),a_{0},a_{1},p(a_{0}), g_{1}, \mathsf{adj}^{-1}(c(a_{1}) \circ g_{1})} %
Putting the last two observations together, we have that the following diagram in $\cl{A}$ commutes. \tri{\cyl(a_{0}),a_{0},a_{1},p(a_{0}),g_{1},{l_{a_{0}}(g_{1},h)}} %
This proves that $f$ satisfies the lifting of identities condition of Definition \ref{NormallyClovenFibrationCylinderDefinition}. 

Let $h$ instead be arbitrary, and let \ar{a_{-1},a_{0},g_{0}} be an arrow of $\cl{A}$. The following diagram in $\cl{A}$ commutes, since $f$ satisfies the compatibility of lifts condition of Definition \ref{NormallyClovenFibrationCoCylinderDefinition}. \tri{a_{-1},a_{0},\cocyl(a_{1}),g_{0},{k_{a_{0}}\big(g_{1},\mathsf{adj}(h)\big)}, {k_{a_{-1}}\big(g_{1} \circ g_{0}, \mathsf{adj}(h) \circ g_{0}\big)}} %
We also have that the following diagram in $\cl{A}$ commutes, by the naturality of $\mathsf{adj}$. \tri{a_{-1},a_{0},\cocyl(a_{2}),g_{0},\mathsf{adj}(h),\mathsf{adj}\big(h \circ \cyl(g_{0}) \big)} %
Putting the last two observations together, we have that \[ l_{a_{-1}}\big(g_{1} \circ g_{0}, h \circ \cyl(g_{0})\big) = \mathsf{adj}^{-1}\Big(k_{a_{0}}\big(g_{1},\mathsf{adj}(h)\big) \circ g_{0}\Big). \] %
Moreover, the following diagram in $\cl{A}$ commutes, by the naturality of $\mathsf{adj}^{-1}$. \tri[{4,3}]{\cyl(a_{-1}),\cyl(a_{0}),a_{1},\cyl(g_{0}),{\mathsf{adj}^{-1}\Big(k_{a_{0}}\big(g_{1},\mathsf{adj}(h)\big) \Big)},{\mathsf{adj}^{-1}\Big(k_{a_{0}}\big(g_{1},\mathsf{adj}(h)\big) \circ g_{0}\Big)}} %
Putting the last two observations together, we have that the following diagram in $\cl{A}$ commutes. \tri[{4,3}]{\cyl(a_{-1}),\cyl(a_{0}),a_{1},\cyl(g_{0}),{l_{a_{0}}(g_{1},h)},{l_{a_{-1}}\big(g_{1} \circ g_{0}, h \circ \cyl(g_{0}) \big)}} %
This proves that $f$ satisfies the compatibility of liftings condition of Definition \ref{NormallyClovenFibrationCylinderDefinition}, and concludes this direction of the proof. 

We now prove that if $f$ is a normally cloven fibration with respect to $\cylinder$, then $f$ is a normally cloven fibration with respect to $\cocylinder$. To this end, suppose that we have a commutative diagram in $\cl{A}$ as follows. \sq{a,a_{1},\cocyl(a_{2}),a_{2},g_{1},f,h,e_{0}(a_{2})}%
As in the proof of Proposition \ref{FibrationCoCylinderIffFibrationCylinderProposition}, we have that the following diagram in $\cl{A}$ commutes. \sq[{5,3}]{a,a_{1},\cyl(a),a_{2},g_{1},f,i_{0}(a_{0}),\mathsf{adj}^{-1}(h)} %
Let \ar{\Delta_{f,a}^{\cylinder},\Omega_{f,a}^{\cylinder},l_{a}} denote the map, of the cleavage with which $f$ is equipped, corresponding to the object $a$ of $\cl{A}$. Then the homotopy \ar[8]{\cyl(a),a_{1},{l_{a}\big(g_{1},\mathsf{adj}^{-1}(h)\big)}} with respect to $\cylinder$ fits into a commutative diagram in $\cl{A}$ as follows. \liftingsquare[{15,4}]{a,a_{1},\cyl(a),a_{2},g_{1},f,i_{0}(a),\mathsf{adj}^{-1}(h),{l_{a}\big(g_{1},\mathsf{adj}^{-1}(h)\big)}} %
Let $k_{a}(g_{1},h)$ denote the arrow \ar[11]{a,\cocyl(a_{1}),{\mathsf{adj}\Big(l_{a}\big(g_{1},\mathsf{adj}^{-1}(h)\big)\Big)}} of $\cl{A}$. 

We claim that the cleavage defined by the maps \ar{\Sigma_{f,a}^{\cocylinder},\Upsilon_{f,a}^{\cocylinder},k_{a}} defined by $(g_{1},h) \mapsto k_{a}(g_{1},h)$, for an object $a$ of $\cl{A}$, equips $f$ with the structure of a normally cloven fibration with respect to $\cocylinder$. 

Indeed, suppose that we have a commutative triangle in $\cl{A}$ as follows. \tri{a_{0},a_{2},\cocyl(a_{2}),g_{2},c(a_{2}),h} %
Then we have that \[ \mathsf{adj}^{-1}(h) = \mathsf{adj}^{-1}\big( c(a_{2}) \circ g_{2} \big). \] %
Moreover, the following diagram in $\cl{A}$ commutes, since the adjunction between $\cyl$ and $\cocyl$ is compatible with $p$ and $c$. \tri{\cyl(a_{0}),a_{0},a_{2},p(a_{0}),g_{2},\mathsf{adj}^{-1}\big(c(a_{2}) \circ g_{2}\big)} %
We deduce that the following diagram in $\cl{A}$ commutes. \tri{\cyl(a_{0}),a_{0},a_{2},p(a_{0}),g_{2},\mathsf{adj}^{-1}(h)} %
Hence the following diagram in $\cl{A}$ commutes, since $f$ satisfies the lifting of identities condition of Definition \ref{NormallyClovenFibrationCylinderDefinition}. \tri{\cyl(a_{0}),a_{0},a_{1},p(a_{0}),g_{1},{l_{a_{0}}\big(g_{1},\mathsf{adj}^{-1}(h)\big)}} %
Thus we have that \[ k_{a_{0}}(g_{1},h)  = \mathsf{adj}\big( g_{1} \circ p(a_{0})\big). \] %
Moreover, the following diagram in $\cl{A}$ commutes, since the adjunction between $\cyl$ and $\cocyl$ is compatible with $p$ and $c$. \tri{a_{0},a_{1},\cocyl(a_{1}),g_{1},c(a_{1}),\mathsf{adj}\big(g_{1} \circ p(a_{0})\big)} %
Putting the last two observations together, we have that the following diagram in $\cl{A}$ commutes. \tri{a_{0},a_{1},\cocyl(a_{1}),g_{1},c(a_{1}),{k_{a_{0}}(g_{1},h)}} %
This proves that $f$ satisfies the lifting of identities condition of Definition \ref{NormallyClovenFibrationCoCylinderDefinition}. 

Let $h$ instead be arbitrary, and let \ar{a_{-1},a_{0},g_{0}} be an arrow of $\cl{A}$. The following diagram in $\cl{A}$ commutes, since $f$ satisfies the compatibility of lifts condition of Definition \ref{NormallyClovenFibrationCylinderDefinition}. \tri[{4,3}]{\cyl(a_{-1}),\cyl(a_{0}),a_{1},\cyl(g_{0}),{l_{a_{0}}\big(g_{1},\mathsf{adj}^{-1}(h)\big)}, {l_{a_{-1}}\big(g_{1} \circ g_{0}, \mathsf{adj}^{-1}(h) \circ \cyl(g_{0})\big)}} %
We also have that the following diagram in $\cl{A}$ commutes, by the naturality of $\mathsf{adj}^{-1}$. \tri[{4,3}]{\cyl(a_{-1}),\cyl(a_{0}),a_{2},\cyl(g_{0}),\mathsf{adj}^{-1}(h),\mathsf{adj}^{-1}(h \circ g_{0})} %
Putting the last two observations together, we have that \[ k_{a_{-1}}\big(g_{1} \circ g_{0}, h \circ g_{0}\big) = \mathsf{adj}\Big(l_{a_{0}}\big(g_{1},\mathsf{adj}^{-1}(h)\big) \circ \cyl(g_{0}) \Big). \] %
Moreover, the following diagram in $\cl{A}$ commutes, by the naturality of $\mathsf{adj}$. \tri{a_{-1},a_{0},\cocyl(a_{1}),g_{0},{\mathsf{adj}\Big(l_{a_{0}}\big(g_{1},\mathsf{adj}^{-1}(h)\big) \Big)},{\mathsf{adj}\Big(l_{a_{0}}\big(g_{1},\mathsf{adj}^{-1}(h)\big) \circ \cyl(g_{0}) \Big)}} %
Thus the following diagram in $\cl{A}$ commutes. \tri{a_{-1},a_{0},\cocyl(a_{1}),g_{0},{k_{a_{0}}(g_{1},h)},{k_{a_{-1}}\big(g_{1} \circ g_{0}, h \circ g_{0} \big)}} %
This proves that $f$ satisfies the compatibility of liftings condition of Definition \ref{NormallyClovenFibrationCoCylinderDefinition}, and concludes this direction of the proof. \end{proof}

\begin{cor} Let $\cylinder = \big( \cyl, i_0, i_1, p \big)$ be a cylinder in $\cl{A}$ equipped with a contraction structure $p$, and let $\cocylinder = \big( \cocyl, e_0, e_1, c \big)$ be a co-cylinder in $\cl{A}$ equipped with a contraction structure $c$. Suppose that $\cylinder$ is left adjoint to $\cocylinder$, and that the adjunction between $\cyl$ and $\cocyl$ is compatible with $p$ and $c$. 

An arrow \ar{a_{0},a_{1},j} is a normally cloven cofibration with respect to $\cylinder$ if and only if it is a normally cloven cofibration with respect to $\cocylinder$. \end{cor}

\begin{proof} Follows immediately from Proposition \ref{NormallyClovenFibrationCoCylinderIffNormallyClovenFibrationCylinderProposition} by duality. \end{proof}

\begin{prpn} \label{RetractionNormallyClovenCofibrationIsNormallyClovenCofibrationProposition} Let $\cylinder = \big( \cyl, i_{0}, i_{1},p  \big)$ be a cylinder in $\cl{A}$ equipped with a contraction structure $p$. Let \ar{a_{0},a_{1},j} be an arrow of $\cl{A}$ which is a normally cloven cofibration with respect to $\cylinder$. 

Suppose that we have commutative diagrams \twosq{a_{2},a_{0},a_{3},a_{1},g_{0},j,j',g_{1},a_{0},a_{2},a_{1},a_{3},r_{0},j',j,r_{1}} in $\cl{A}$, such that $r_{0}$ is a retraction of $g_{0}$, and such that $r_{1}$ is a retraction of $g_{1}$. Then $j'$ is a normally cloven cofibration with respect to $\cylinder$. \end{prpn} 

\begin{proof} For any object $a_{4}$ of $\cl{A}$, let \ar{\Sigma^{\cylinder}_{j,a_{4}},\Upsilon^{\cylinder}_{j,a_{4}},k_{a_{4}}} denote the corresponding map of the cleavage with which $j$ is equipped. Suppose that we have a commutative diagram in $\cl{A}$ as follows. \sq{a_{2},\cyl(a_{2}),a_{3},a_{4},i_{0}(a_{2}),h,j',f} %
As in the proof of Proposition \ref{RetractionCofibrationIsCofibrationProposition}, we have that the following diagram in $\cl{A}$ commutes. \sq{a_{0},\cyl(a_{0}),a_{1},a_{4},i_{0}(a_{0}),h \circ \cyl(r_{0}),j,f \circ r_{1}} %
Let us define a map \ar{\Sigma^{\cylinder}_{j',a_{4}},\Upsilon^{\cylinder}_{j',a_{4}},k'_{a_{4}}} by 

\begin{diagram} 

$(f,h)  \mapsto k_{a_{4}}\big(f \circ r_{1}, h \circ \cyl(r_{0}) \big) \circ \cyl(g_{1})$. 

\end{diagram} %
We claim that the maps $k'_{a}$, for $a$ an object of $\cl{A}$, define a cleavage exhibiting $j'$ to be a normally cloven cofibration with respect to $\cylinder$. 

After the proof of Proposition \ref{RetractionCofibrationIsCofibrationProposition}, it remains to demonstrate that conditions (i) and (ii) of Definition \ref{NormallyClovenCofibrationCylinderDefinition} are satisfied. 

Let us first show that condition (i) holds. To this end, suppose that we have a commutative diagram in $\cl{A}$ as follows. \triother{a_{2},a_{3},a_{4},j',f_{1},f_{0}} %
Then the following diagram in $\cl{A}$ commutes. \squareabovetriangle{a_{0},a_{1},a_{2},a_{3},a_{4},j,r_{1},r_{0},j',f_{1},f_{0}} %
Since the cleavage defined by the maps $k_{a}$, for $a$ an object of $\cl{A}$, exhibits $j$ to be a normally cloven cofibration with respect to $\cylinder$, this cleavage satisfies condition (i) of Definition \ref{NormallyClovenCofibrationCylinderDefinition}. We deduce that the following diagram in $\cl{A}$ commutes. \tri{\cyl(a_{1}),a_{1},a_{4},p(a_{1}),f_{1} \circ r_{1},{k_{a_{4}}\big( f_{1} \circ r_{1}, f_{0} \circ r_{0} \circ p(a_{0})\big)}} %
Thus, by the commutativity of the diagram \sq{\cyl(a_{0}),a_{0},\cyl(a_{2}),a_{2},p(a_{0}),r_{0},\cyl(r_{0}),p(a_{2})} in $\cl{A}$, we have that the following diagram in $\cl{A}$ commutes. \tri{\cyl(a_{1}),a_{1},a_{4},p(a_{1}),f_{1} \circ r_{1},{k_{a_{4}}\big( f_{1} \circ r_{1}, f_{0} \circ p(a_{2}) \circ \cyl(r_{0}) \big)}} %
Moreover, the following diagram in $\cl{A}$ commutes. \squareabovetriangle{\cyl(a_{3}),\cyl(a_{1}),a_{3}, a_{1}, a_{3},\cyl(g_{1}),p(a_{1}),p(a_{3}),g_{1},r_{1},id} %
Putting the last two observations together, we have that the following diagram in $\cl{A}$ commutes. \threetrianglestwo[{6,3,1}]{\cyl(a_{3}),\cyl(a_{1}),a_{3},a_{4},p(a_{3}),\cyl(g_{1}), r_{1} \circ p(a_{1}),f_{1},{k_{a_{4}}\big( f_{1} \circ r_{1},f_{0} \circ p(a_{2}) \circ \cyl(r_{0}) \big)}} %
Thus the following diagram in $\cl{A}$ commutes, as required. \tri{\cyl(a_{3}),a_{3},a_{4},p(a_{3}),f_{1},{k'_{a_{4}}\big(f_{1},f_{0} \circ p(a_{2})\big)}} 

Let us now prove that condition (ii) of Definition \ref{NormallyClovenCofibrationCylinderDefinition} holds. Let \ar{a_{4},a_{5},f'} be an arrow of $\cl{A}$. Since the cleavage defined by the maps $k_{a}$, for $a$ an object of $\cl{A}$, exhibits $j$ to be a normally cloven cofibration with respect to $\cylinder$, this cleavage satisfies condition (ii) of Definition \ref{NormallyClovenCofibrationCylinderDefinition}. We deduce that the following diagram in $\cl{A}$ commutes. \tri[{10,3}]{\cyl(a_{1}),a_{4},a_{5},{k_{a_{4}}\big(f \circ r_{1},h \circ \cyl(r_{0})\big)},f',{k_{a_{5}}\big( f' \circ f \circ r_{1},  f' \circ h \circ \cyl(r_{0}) \big)}} %
By definition of $k'_{a_{4}}(f,h)$ and $k'_{a_{5}}(f' \circ f, f' \circ h)$, the diagrams \tri{\cyl(a_{3}),a_{4},\cyl(a_{1}),\cyl(g_{1}),{k_{a_{4}}\big(f \circ r_{1}, h \circ \cyl(r_{0})\big)},{k_{a_{4}}'(f,h)}} and \tri{\cyl(a_{3}),\cyl(a_{1}),a_{5}, \cyl(g_{1}),{k_{a_{5}}\big(f' \circ f \circ r_{1},f' \circ h \circ \cyl(r_{0}) \big)},{k_{a_{5}}'(f' \circ f, f'\circ h)}} in $\cl{A}$ commute. %
Putting the last two observations together, we have that the following diagram in $\cl{A}$ commutes, as required. \tri[{4,3}]{\cyl(a_{3}),a_{4},a_{5},{k'_{a_{4}}(f,h)},f',{k'_{a_{5}}(f' \circ f, f' \circ h)}} 

\end{proof}

\begin{cor} \label{RetractionTrivialNormallyClovenCofibrationIsTrivialNormallyClovenCofibrationCorollary} Let $\cylinder = \big( \cyl, i_{0}, i_{1},p  \big)$ be a cylinder in $\cl{A}$ equipped with a contraction structure $p$. Let \ar{a_{0},a_{1},j} be an arrow of $\cl{A}$ which is a trivial normally cloven cofibration with respect to $\cylinder$. 

Suppose that we have commutative diagrams \twosq{a_{2},a_{0},a_{3},a_{1},g_{0},j,j',g_{1},a_{0},a_{2},a_{1},a_{3},r_{0},j',j,r_{1}} in $\cl{A}$, such that $r_{0}$ is a retraction of $g_{0}$, and such that $r_{1}$ is a retraction of $g_{1}$. Then $j'$ is a trivial normally cloven cofibration with respect to $\cylinder$. \end{cor}

\begin{proof} Follows immediately from Proposition \ref{RetractionNormallyClovenCofibrationIsNormallyClovenCofibrationProposition} and Proposition \ref{RetractionHomotopyEquivalenceIsHomotopyEquivalenceProposition}. \end{proof}  

\begin{cor} \label{RetractionNormallyClovenFibrationIsNormallyClovenFibrationCorollary} Let $\cocylinder = \big( \cocyl, e_{0}, e_{1}, c  \big)$ be a co-cylinder in $\cl{A}$ equipped with a contraction structure $c$. Let \ar{a_{0},a_{1},f} be an arrow of $\cl{A}$ which is a normally cloven fibration with respect to $\cocylinder$. 

Suppose that we have commutative diagrams \twosq{a_{2},a_{0},a_{3},a_{1},g_{0},f,f',g_{1},a_{0},a_{2},a_{1},a_{3},r_{0},f',f,r_{1}} in $\cl{A}$, such that $r_{0}$ is a retraction of $g_{0}$, and such that $r_{1}$ is a retraction of $g_{1}$. Then $f'$ is a normally cloven fibration with respect to $\cocylinder$. \end{cor} 

\begin{proof} Follows immediately from Proposition \ref{RetractionNormallyClovenCofibrationIsNormallyClovenCofibrationProposition} by duality. \end{proof}

\begin{cor} \label{RetractionTrivialNormallyClovenFibrationIsTrivialNormallyClovenFibrationCorollary} Let $\cocylinder = \big( \cocyl, e_{0}, e_{1}, c  \big)$ be a co-cylinder in $\cl{A}$ equipped with a contraction structure $c$. Let \ar{a_{0},a_{1},f} be an arrow of $\cl{A}$ which is a trivial normally cloven fibration with respect to $\cocylinder$. 

Suppose that we have commutative diagrams \twosq{a_{2},a_{0},a_{3},a_{1},g_{0},f,f',g_{1},a_{0},a_{2},a_{1},a_{3},r_{0},f',f,r_{1}} in $\cl{A}$, such that $r_{0}$ is a retraction of $g_{0}$, and such that $r_{1}$ is a retraction of $g_{1}$. Then $f'$ is a trivial normally cloven fibration with respect to $\cocylinder$. \end{cor} 

\begin{proof} Follows immediately from Corollary \ref{RetractionNormallyClovenFibrationIsNormallyClovenFibrationCorollary} and Proposition \ref{RetractionHomotopyEquivalenceIsHomotopyEquivalenceProposition}. \end{proof}

\end{chapter}

\begin{chapter}{Mapping cylinders and mapping co-cylinders} \label{MappingCylindersAndMappingCoCylindersChapter}

Let $\cl{A}$ be a formal category. With respect to a cylinder $\cylinder$ in $\cl{A}$, we introduce the notion of a mapping cylinder $a^{\cylinder}_{j}$ of an arrow \ar{a_{0},a_{1},f} of $\cl{A}$. There is a canonical arrow $m^{\cylinder}_{f}$ from $a^{\cylinder}_{f}$ to $\cyl(a_{1})$. This arrow admits a retraction if $f$ is a cofibration with respect to $\cylinder$.

 We introduce the dual notion of a mapping co-cylinder $a^{\cocylinder}_{f}$ of $f$ with respect to a co-cylinder $\cocylinder$ in $\cl{A}$. There is a canonical arrow from $\cocyl(a_{0})$ to $a^{\cocylinder}_{f}$. This arrow admits a section if $f$ is a fibration with respect to $\cocylinder$.

In \ref{LiftingAxiomsChapter}, we shall prove that if $\cylinder$ admits sufficient structure, then the retraction of $m^{\cylinder}_{f}$ for $f$ a cofibration is a strong deformation retraction with respect to $\cylinder$. This will be a vital step towards establishing the lifting axioms for a model structure.

To return to the current section, we present a proof, given by van den Berg and Garner in \S{6} of \cite{VanDenBergGarnerTopologicalAndSimplicialModelsOfIdentityTypes}, that if $\cylinder$ has strictness of right identities, then every arrow $f$ of $\cl{A}$ factors through $a^{\cylinder}_{f}$ into a normally cloven cofibration followed by a strong deformation retraction. 

If $\cocylinder$ has strictness of right identities we deduce, dually, that $f$ factors through $a^{\cocylinder}_{f}$ into an arrow admitting a strong deformation retraction followed by a normally cloven fibration. 

Strictness of right identities is an indispensable hypothesis here. We shall discuss it a little more in \ref{FactorisationAxiomsChapter}, where we shall build upon our work here to establish the factorisation axioms for a model structure. 

To carry through the ideas of van den Berg and Garner, we assume that $\cylinder$ admits a lower right connection structure. Instead, van den Berg and Garner work with what is known as a strength, a structure of a slightly different nature to those which we considered in \ref{StructuresUponACylinderOrACoCylinderChapter}. 

The Moore co-cylinder in topological spaces is a key motivating example for van den Berg and Garner. It does not admit a lower connection structure, but does admit a strength. We refer the reader to \S{4.2} of \cite{VanDenBergGarnerTopologicalAndSimplicialModelsOfIdentityTypes} for more on this. Everywhere that we make use of a connection in this work, it should be possible to instead make use of a strength. 

In abstract homotopy theory, the observation that the mapping cylinder of an arrow \ar{a_{0},a_{1},f} of $\cl{A}$ gives rise to a factorisation into a cofibration followed by a strong deformation retraction goes back to Kamps. It can be found in \S{4} of \cite{KampsKanBedingungenUndAbstrakteHomotopietheorie}, and is also Theorem 5.11 in the book \cite{KampsPorterAbstractHomotopyAndSimpleHomotopyTheory} of Kamps and Porter. 

However, the cleavage of the arrow \ar{a_{0},a^{\cylinder}_{f},j} constructed in these works does not satisfy our conditions for $j$ to be a normally cloven cofibration. The cleavage we construct after van den Berg and Garner does demonstrate $j$ to be a normally cloven cofibration, which is crucial for us.

A different proof that the mapping cylinder of $f$ gives rise to a factorisation into a cofibration followed by a strong deformation retraction was given by Grandis in \S{4.6.5} -- \S{4.6.7} of the book \cite{GrandisDirectedAlgebraicTopology}, under the additional assumption that $\cylinder$ admits what is referred to as an extended acceleration structure.

A key step in this argument of Grandis appears as Theorem 1.8 in his earlier paper \cite{GrandisDirectedHomotopyTheoryII}, but under weaker hypotheses. The proof given in \cite{GrandisDirectedHomotopyTheoryII} is erroneous, but can be repaired if $\cylinder$ admits a zero collapse structure, in the sense of \cite{GrandisCategoricallyAlgebraicFoundationsForHomotopicalAlgebra}, or alternatively if $\cylinder$ has strictness of right identities. 

We also point the reader to Theorem 3.8 of the paper \cite{GrandisHomotopicalAlgebraInHomotopicalCategories} of Grandis.

\begin{assum} Let $\cl{C}$ be a 2-category with a final object. Suppose that pushouts and pullbacks of 2-arrows of $\cl{C}$ give rise to pushouts and pullbacks in formal categories, in the sense of Definition \ref{PushoutsPullbacks2ArrowsArePushoutsPullbacksInFormalCategoriesTerminology}. Let $\cl{A}$ be an object of $\cl{C}$. As before, we view $\cl{A}$ as a formal category, writing of objects and arrows of $\cl{A}$. \end{assum} 

\begin{defn} Let $\cylinder = \big( \cyl, i_0, i_1 \big)$ be a cylinder in $\cl{A}$. A {\em mapping cylinder} with respect to $\cylinder$ of an arrow \ar{a_{0},a_{1},f} of $\cl{A}$ is an object $a^{\cylinder}_{f}$ of $\cl{A}$, together with an arrow \ar{\cyl(a_{0}),a^{\cylinder}_{f},d^{0}_{f}} of $\cl{A}$, and an arrow \ar{a_{1},a^{\cylinder}_{f},d^{1}_{f}} of $\cl{A}$, such that the following diagram in $\cl{A}$ is co-cartesian. \sq{a_{0},\cyl(a_{0}),a_{1},a^{\cylinder}_{f},i_{0}(a_{0}),d^{0}_{f},f,d^{1}_{f}} \end{defn}

\begin{defn} Let $\cylinder = \big( \cyl, i_0, i_1 \big)$ be a cylinder in $\cl{A}$. Then $\cl{A}$ {\em has mapping cylinders} with respect to $\cylinder$ if, for every arrow $f$ of $\cl{A}$, we have a mapping cylinder $\big(a^{\cylinder}_{f}, d^{0}_{f},d^{1}_{f} \big)$ of $f$ with respect to $\cylinder$. \end{defn}

\begin{defn} Let $\cocylinder = \big( \cocyl, e_0, e_1 \big)$ be a co-cylinder in $\cl{A}$. A {\em mapping co-cylinder} of $f$ with respect to $\cocylinder$ of an arrow \ar{a_{0},a_{1},f} of $\cl{A}$ is an object $a^{\cocylinder}_{f}$ of $\cl{A}$, together with an arrow \ar{a^{\cocylinder}_{f},a_{0},d^{0}_{f}} of $\cl{A}$, and an arrow \ar{a^{\cocylinder}_{f},\cocyl(a_{1}),d^{1}_{f}} of $\cl{A}$, such that the following diagram in $\cl{A}$ is cartesian. \sq[{4,3}]{a^{\cocylinder}_{f},a_{0},\cocyl(a_{1}),a_{1},d^{0}_{f},f,d^{1}_{f},e_{0}(a_{1})} \end{defn}

\begin{rmk} Let $\cocylinder = \big( \cocyl, e_0, e_1 \big)$ be a co-cylinder in $\cl{A}$. Let \ar{a_{0},a_{1},f} be an arrow of $\cl{A}$. Then $\big(a^{\cocylinder}_{f},d^{0}_{f},d^{1}_{f} \big)$ defines a mapping co-cylinder of $f$ with respect to $\cocylinder$ if and only if $\big( a^{\cocylinder}_{f}, (d^{0}_{f})^{op}, (d^{1}_{f})^{op} \big)$ defines a mapping cylinder of $f$ with respect to the cylinder $\cocylinder^{op}$ in $\cl{A}^{op}$. \end{rmk}

\begin{defn} Let $\cocylinder = \big( \cocyl, e_0, e_1 \big)$ be a co-cylinder in $\cl{A}$. Then $\cl{A}$ {\em has mapping co-cylinders} with respect to $\cocylinder$ if, for every arrow $f$ of $\cl{A}$, we have a mapping co-cylinder $\big( a^{\cocylinder}_{f},d^{0}_{f},d^{1}_{f} \big)$ of $f$ with respect to $\cocylinder$. \end{defn}

\begin{defn} Let $\cylinder = \big( \cyl, i_0, i_1 \big)$ be a cylinder in $\cl{A}$. Then $\cyl$ {\em preserves mapping cylinders} with respect to $\cylinder$ if, for every pair of an arrow \ar{a_{0},a_{1},f} of $\cl{A}$ and a mapping cylinder $\big( a^{\cylinder}_{f}, d^{0}_{f}, d^{1}_{f} \big)$ of $f$ with respect to $\cylinder$, we have that $\big( \cyl\big(a^{\cylinder}_{f}\big), \cyl(d^{0}_{f}), \cyl(d^{1}_{f}) \big)$ defines a mapping cylinder of $\cyl(f)$ with respect to $\cylinder$. \end{defn} 

\begin{defn} Let $\cocylinder = \big( \cocyl, e_0, e_1 \big)$ be a co-cylinder in $\cl{A}$. Then $\cocyl$ {\em preserves mapping co-cylinders} with respect to $\cocylinder$ if $\cocyl$ preserves mapping cylinders with respect to the cylinder $\cocylinder^{op}$ in $\cl{A}^{op}$. \end{defn} 

\begin{notn} \label{CanonicalArrowFromMappingCylinderToCylinderDefinition} Let $\cylinder = \big( \cyl, i_{0}, i_{1} \big)$ be a cylinder in $\cl{A}$. Let \ar{a_{0},a_{1},f} be an arrow of $\cl{A}$, and suppose that $\big( a^{\cylinder}_{f}, d^{0}_{f}, d^{1}_{f} \big)$ defines a mapping cylinder of $f$ with respect to $\cylinder$. 

We denote by $m^{\cylinder}_{f}$ the canonical arrow of $\cl{A}$ such that the following diagram in $\cl{A}$ commutes. \pushout{a_{0},\cyl(a_{0}),a_{1},a^{\cylinder}_{f},\cyl(a_{1}),i_{0}(a_{0}),d^{0}_{f},f,d^{1}_{f},\cyl(f),i_{0}(a_{1}),m^{\cylinder}_{f}} \end{notn}

\begin{notn} Let $\cocylinder = \big( \cocyl, e_{0}, e_{1} \big)$ be a co-cylinder in $\cl{A}$. Let \ar{a_{0},a_{1},f} be an arrow of $\cl{A}$, and suppose that $\big( a^{\cocylinder}_{f}, d^{0}_{f}, d^{1}_{f} \big)$ defines a mapping co-cylinder of $f$ with respect to $\cocylinder$. 

We denote by $m^{\cocylinder}_{f}$ the canonical arrow of $\cl{A}$ such that the following diagram in $\cl{A}$ commutes. \pullback{a^{\cocylinder}_{f},a_{0},\cocyl(a_{1}),a_{1},\cocyl(a_{0}),d^{0}_{f},f,d^{1}_{f},e_{0}(a_{1}),e_{0}(a_{0}),\cocyl(f),m^{\cocylinder}_{f}} \end{notn}

\begin{prpn} \label{RetractionToMappingCylinderProposition} Let $\cylinder = \big( \cyl, i_0, i_1 \big)$ be a cylinder in $\cl{A}$. Let \ar{a_{0},a_{1},j} be an arrow of $\cl{A}$, and suppose that $\big( a^{\cylinder}_{j}, d^{0}_{j}, d^{1}_{j} \big)$ defines a mapping cylinder of $j$ with respect to $\cylinder$. 

Suppose that $j$ is a cofibration with respect to $\cylinder$, and let \ar{\cyl(a_{1}),a^{\cylinder}_{j},r^{\cylinder}_{j}} denote the corresponding arrow of $\cl{A}$ such that the following diagram in $\cl{A}$ commutes. \pushout{a_{0},\cyl(a_{0}),a_{1},\cyl(a_{1}),a^{\cylinder}_{j},i_{0}(a_{0}),\cyl(j),j,i_{0}(a_{1}),d^{0}_{j},d^{1}_{j},r^{\cylinder}_{j}} Then $r^{\cylinder}_{j}$ is a retraction of $m^{\cylinder}_{j}$. \end{prpn}

\begin{proof} The following diagram in $\cl{A}$ commutes.

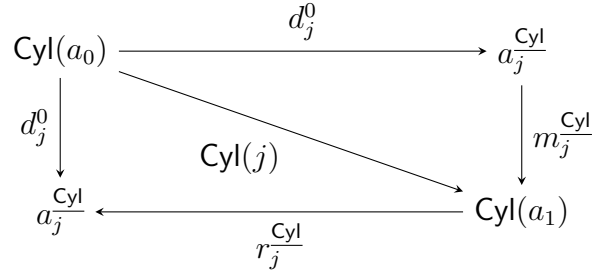
\begin{diagram}

\begin{tikzpicture} [>=stealth]

\matrix [ampersand replacement=\&, matrix of math nodes, column sep=11 em, row sep=3 em, nodes={anchor=center}]
{ 
|(0-0)| \cyl(a_{0})        \& |(1-0)| a^{\cylinder}_{j} \\ 
|(0-1)| a^{\cylinder}_{j} \& |(1-1)| \cyl(a_{1}) \\
};
	
\draw[->] (0-0) to node[auto] {$d^{0}_{j}$} (1-0);
\draw[->] (1-0) to node[auto] {$m^{\cylinder}_{j}$} (1-1);
\draw[->] (0-0) to node[auto,swap] {$\cyl(j)$} (1-1);
\draw[->] (1-1) to node[auto] {$r^{\cylinder}_{j}$} (0-1);
\draw[->] (0-0) to node[auto,swap] {$d^{0}_{j}$} (0-1);
\end{tikzpicture} 

\end{diagram} 

The following diagram in $\cl{A}$ also commutes.

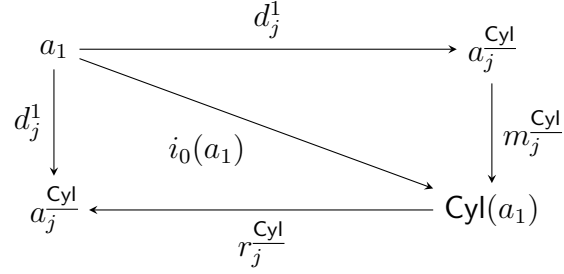
\begin{diagram}

\begin{tikzpicture} [>=stealth]

\matrix [ampersand replacement=\&, matrix of math nodes, column sep=11 em, row sep=3 em, nodes={anchor=center}]
{ 
|(0-0)| a_{1}        \& |(1-0)| a^{\cylinder}_{j} \\ 
|(0-1)| a^{\cylinder}_{j} \& |(1-1)| \cyl(a_{1}) \\
};
	
\draw[->] (0-0) to node[auto] {$d^{1}_{j}$} (1-0);
\draw[->] (1-0) to node[auto] {$m^{\cylinder}_{j}$} (1-1);
\draw[->] (0-0) to node[auto,swap] {$i_{0}(a_{1})$} (1-1);
\draw[->] (1-1) to node[auto] {$r^{\cylinder}_{j}$} (0-1);
\draw[->] (0-0) to node[auto,swap] {$d^{1}_{j}$} (0-1);
\end{tikzpicture} 

\end{diagram} %
Putting these observations together, we have that the following diagram in $\cl{A}$ commutes. \pushout[{3,3,1}]{a_{0},\cyl(a_{0}),a_{1},a^{\cylinder}_{j},a^{\cylinder}_{j},i_{0}(a_{0}),d^{0}_{j},j,d^{1}_{j},d^{0}_{j},d^{1}_{j},r^{\cylinder}_{j} \circ m^{\cylinder}_{j}} %
Appealing to the universal property of $a^{\cylinder}_{j}$, we deduce that the following diagram in $\cl{A}$ commutes. \tri{a^{\cylinder}_{j},\cyl(a_{1}),a^{\cylinder}_{j},m^{\cylinder}_{j},r^{\cylinder}_{j},id} \end{proof} 

\begin{prpn} \label{MappingCylinderFactorisationProposition} Let $\cylinder = \big( \cyl, i_0, i_1, p \big)$ be a cylinder in $\cl{A}$, equipped with a contraction structure $p$. Let \ar{a_{0},a_{1},f} be an arrow of $\cl{A}$, and suppose that $\big( a^{\cylinder}_{f}, d^{0}_{f}, d^{1}_{f} \big)$ defines a mapping cylinder of $f$ with respect to $\cylinder$. 

Let \ar{a^{\cylinder}_{f},a_{1},g} denote the canonical arrow of $\cl{A}$ such that the following diagram in $\cl{A}$ commutes. \pushout{a_{0},\cyl(a_{0}),a_{1},a^{\cylinder}_{f},a_{1},i_{0}(a_{0}),d^{0}_{f},f,d^{1}_{f},f \circ p(a_{0}),id,g} 

Let \ar{a_{0},a^{\cylinder}_{f},j} denote the arrow $d^{0}_{f} \circ i_{1}(a_{0})$ of $\cl{A}$. 

The following diagram in $\cl{A}$ commutes. \tri{a_{0},a^{\cylinder}_{f},a_{1},j,g,f} \end{prpn}

\begin{proof} The following diagram in $\cl{A}$ commutes. 

\begin{diagram}

\begin{tikzpicture} [>=stealth]

\matrix [ampersand replacement=\&, matrix of math nodes, column sep=4 em, row sep=4 em, nodes={anchor=center}]
{
|(0-0)| a_{0} \&                     \& |(2-0)| a^{\cylinder}_{f} \\ 
              \& |(1-1)| \cyl(a_{0}) \&                           \\ 
|(0-2)| a_{0} \&                     \& |(2-2)| a_{1}             \\
};
	
\draw[->] (0-0) to node[auto] {$j$} (2-0);
\draw[->] (0-0) to node[auto] {$i_{1}(a_{0})$} (1-1);
\draw[->] (1-1) to node[auto,swap] {$d^{0}_{f}$} (2-0);
\draw[->] (2-0) to node[auto] {$g$} (2-2);
\draw[->] (0-0) to node[auto,swap] {$id$} (0-2);
\draw[->] (1-1) to node[auto] {$p(a_{0})$} (0-2);
\draw[->] (0-2) to node[auto,swap] {$f$} (2-2);
\end{tikzpicture} 

\end{diagram} 

\end{proof}

\begin{defn} \label{MappingCylinderFactorisationDefinition} Let $\cylinder = \big( \cyl, i_0, i_1, p \big)$ be a cylinder in $\cl{A}$ equipped with a contraction structure $p$. Let \ar{a_{0},a_{1},f} be an arrow of $\cl{A}$, and suppose that $\big( a^{\cylinder}_{f}, d^{0}_{f}, d^{1}_{f} \big)$ defines a mapping cylinder of $f$ with respect to $\cylinder$. 

We refer to the factorisation of $f$ obtained via Proposition \ref{MappingCylinderFactorisationProposition} as the {\em mapping cylinder factorisation} of $f$ with respect to $\cylinder$. \end{defn} 

\begin{cor} \label{MappingCoCylinderFactorisationCorollary} Let $\cylinder = \big( \cocyl, e_0, e_1, c \big)$ be a co-cylinder in $\cl{A}$ equipped with a contraction structure $c$. Let \ar{a_{0},a_{1},f} be an arrow of $\cl{A}$, and suppose that $\big( a^{\cocylinder}_{f}, d^{0}_{f}, d^{1}_{f} \big)$ defines a mapping co-cylinder of $f$ with respect to $\cocylinder$. 

Let \ar{a_{0},a^{\cocylinder}_{f},j} denote the canonical arrow of $\cl{A}$ such that the following diagram in $\cl{A}$ commutes. \pullback[{3,3,-1}]{a^{\cocylinder}_{f},a_{0},\cocyl(a_{1}),a_{1},a_{0},d^{0}_{f},f,d^{1}_{f},e_{0}(a_{1}),id,c(a_{1}) \circ f,j} 

Let \ar{a^{\cocylinder}_{f},a_{1},g} denote the arrow $e_{1}(a_{1}) \circ d^{1}_{f}$ of $\cl{A}$. 

The following diagram in $\cl{A}$ commutes. \tri{a_{0},a^{\cocylinder}_{f},a_{1},j,g,f} \end{cor}

\begin{proof} Follows immediately from Proposition \ref{MappingCylinderFactorisationProposition} by duality. \end{proof}

\begin{defn} \label{MappingCoCylinderFactorisationDefinition} Let $\cocylinder = \big( \cocyl, e_0, e_1, c \big)$ be a co-cylinder in $\cl{A}$ equipped with a contraction structure $c$. Let \ar{a_{0},a_{1},f} be an arrow of $\cl{A}$, and suppose that $\big( a^{\cocylinder}_{f}, d^{0}_{f}, d^{1}_{f} \big)$ defines a mapping co-cylinder of $f$ with respect to $\cocylinder$. We refer to the factorisation of $f$ obtained via Proposition \ref{MappingCoCylinderFactorisationCorollary} as the {\em mapping co-cylinder factorisation} of $f$ with respect to $\cocylinder$. \end{defn}

\begin{prpn} \label{MappingCylinderFactorisationStrongDeformationRetractionWithRespectToCylinderProposition} Let $\cylinder = \big( \cyl, i_{0}, i_{1}, p, \Gamma_{lr} \big)$ be a cylinder in $\cl{A}$ equipped with a contraction structure $p$, and a lower right connection structure $\Gamma_{lr}$. Suppose that $\Gamma_{lr}$ is compatible with $p$, and that $\cyl$ preserves mapping cylinders with respect to $\cylinder$. 

Let \ar{a_{0},a_{1},f} be an arrow of $\cl{A}$, and let $\big( a^{\cylinder}_{f}, d^{0}_{f}, d^{1}_{f} \big)$ be a mapping cylinder of $f$ with respect to $\cylinder$. 

Let \tri{a_{0},a^{\cylinder}_{f},a_{1},j,g,f} denote the corresponding mapping cylinder factorisation of $f$. 

There is a homotopy over $a_{1}$ from $d^{1}_{f} \circ g$ to $id\big(a^{\cylinder}_{f}\big)$ with respect to $\cylinder$ and $(g,g)$. \end{prpn}

\begin{proof} By construction, $g$ is a retraction of $d^{1}_{f}$. The following diagram in $\cl{A}$ is co-cartesian, since $\cyl$ preserves mapping cylinders with respect to $\cylinder$. \sq[{5,3}]{\cyl(a_{0}),\cyl^{2}(a_{0}),\cyl(a_{1}),\cyl\big(a^{\cylinder}_{f}\big),\cyl\big(i_{0}(a_{0})\big),\cyl(d^{0}_{f}),\cyl(j),\cyl(d^{1}_{f})} Moreover, the following diagram in $\cl{A}$ commutes.  

\begin{diagram}

\begin{tikzpicture} [>=stealth]

\matrix [ampersand replacement=\&, matrix of math nodes, column sep=3 em, row sep=3 em, nodes={anchor=center}]
{ 
|(0-0)| \cyl(a_{0}) \&               \& |(2-0)| \cyl^{2}(a_{0}) \\
|(0-1)| \cyl(a_{1}) \& |(1-1)| a_{0} \& |(2-1)| \cyl(a_{0})     \\
|(0-2)| a_{1}       \&               \& |(2-2)| a^{\cylinder}_{f}    \\
};
	
\draw[->] (0-0) to node[auto] {$\cyl\big(i_{0}(a_{0})\big)$} (2-0);
\draw[->] (2-0) to node[auto] {$\Gamma_{lr}(a_{0})$} (2-1);
\draw[->] (0-0) to node[auto] {$p(a_{0})$} (1-1);
\draw[->] (1-1) to node[auto] {$i_{0}(a_{0})$} (2-1);
\draw[->] (1-1) to node[auto] {$f$} (0-2);
\draw[->] (0-2) to node[auto,swap] {$d^{1}_{f}$} (2-2);
\draw[->] (2-1) to node[auto] {$d^{0}_{f}$} (2-2);
\draw[->] (0-0) to node[auto,swap] {$\cyl(f)$} (0-1);
\draw[->] (0-1) to node[auto,swap] {$p(a_{1})$} (0-2);
\end{tikzpicture} 

\end{diagram} %
Thus there is a canonical arrow \ar{\cyl\big(a^{\cylinder}_{f}\big),a^{\cylinder}_{f},h} of $\cl{A}$ such that the following diagram in $\cl{A}$ commutes. \pushout[{5,3,-1}]{\cyl(a_{0}),\cyl^{2}(a_{0}),\cyl(a_{1}),\cyl\big(a^{\cylinder}_{f}\big),a^{\cylinder}_{f},\cyl\big(i_{0}(a_{0})\big),\cyl(d^{0}_{f}),\cyl(f),\cyl(d^{1}_{f}),d^{0}_{f} \circ \Gamma_{lr}(a_{0}),d^{1}_{f} \circ p(a_{1}),h} We claim that $h$ defines a homotopy over $a_{1}$ from $d^{1}_{f} \circ g$ to $id(a^{\cylinder}_{f})$ with respect to $\cylinder$ and $(d^{1}_{f},d^{1}_{f})$. 

Let us first prove that the following diagram in $\cl{A}$ commutes. \tri[{5,3}]{a^{\cylinder}_{f},\cyl\big(a^{\cylinder}_{f}\big),a^{\cylinder}_{f},i_{0}\big(a^{\cylinder}_{f}\big),h,d^{1}_{f} \circ g} We have that the following diagram in $\cl{A}$ commutes. \trapeziums[{3,3,1,0}]{a_{1},a^{\cylinder}_{f},\cyl(a_{1}),\cyl\big(a^{\cylinder}_{f}\big),a_{1},a^{\cylinder}_{f},d^{1}_{f},i_{0}\big(a^{\cylinder}_{f}\big),h,id,d^{1}_{f},i_{0}(a_{1}),\cyl(d^{1}_{f}),p(a_{1})} %
The following diagram in $\cl{A}$ also commutes.  

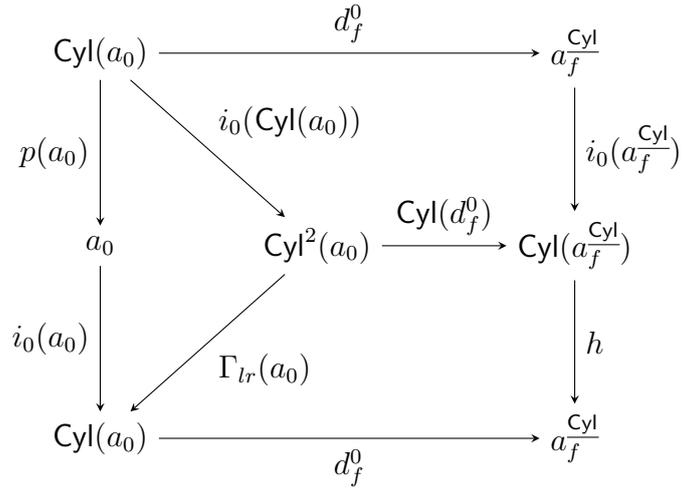
\begin{diagram}

\begin{tikzpicture} [>=stealth]

\matrix [ampersand replacement=\&, matrix of math nodes, column sep=3 em, row sep=4 em, nodes={anchor=center}]
{ 
|(0-0)| \cyl(a_{0}) \&                         \&[1em] |(2-0)| a^{\cylinder}_{f} \\
|(0-1)| a_{0}       \& |(1-1)| \cyl^{2}(a_{0}) \& |(2-1)| \cyl\big(a^{\cylinder}_{f}\big)     \\
|(0-2)| \cyl(a_{0}) \&                         \& |(2-2)| a^{\cylinder}_{f}    \\
};
	
\draw[->] (0-0) to node[auto] {$d^{0}_{f}$} (2-0);
\draw[->] (2-0) to node[auto] {$i_{0}\big(a^{\cylinder}_{f}\big)$} (2-1);
\draw[->] (0-0) to node[auto] {$i_{0}\big(\cyl(a_{0})\big)$} (1-1);
\draw[->] (1-1) to node[auto] {$\cyl(d^{0}_{f})$} (2-1);
\draw[->] (1-1) to node[auto] {$\Gamma_{lr}(a_{0})$} (0-2);
\draw[->] (0-2) to node[auto,swap] {$d^{0}_{f}$} (2-2);
\draw[->] (2-1) to node[auto] {$h$} (2-2);
\draw[->] (0-0) to node[auto,swap] {$p(a_{0})$} (0-1);
\draw[->] (0-1) to node[auto,swap] {$i_{0}(a_{0})$} (0-2);
\end{tikzpicture} 

\end{diagram} %
Putting these two observations together, we have that the following diagram in $\cl{A}$ commutes. \pushout[{4,3,-1}]{a_{0},\cyl(a_{0}),a_{1},a^{\cylinder}_{f},a^{\cylinder}_{f},i_{0}(a_{0}),d^{0}_{f},f,d^{1}_{f},d^{0}_{f} \circ i_{0}(a_{0}) \circ p(a_{0}),d^{1}_{f},h \circ i_{0}\big(a^{\cylinder}_{f}\big)} We also have that the diagram \tri{a_{1},a^{\cylinder}_{f},a^{\cylinder}_{f},d^{1}_{f},d^{1}_{f} \circ g,d^{1}_{f}} in $\cl{A}$ commutes, since $g$ is a retraction of $d^{1}_{f}$. Moreover, the following diagram in $\cl{A}$ commutes.

\begin{diagram}

\begin{tikzpicture} [>=stealth]

\matrix [ampersand replacement=\&, matrix of math nodes, column sep=3 em, row sep=3 em, nodes={anchor=center}]
{ 
|(0-0)| \cyl(a_{0}) \& |(1-0)| a^{\cylinder}_{f} \\
|(0-1)| a_{0}       \& |(1-1)| a_{1}        \\
|(0-2)| \cyl(a_{0}) \& |(1-2)| a^{\cylinder}_{f}    \\
};
	
\draw[->] (0-0) to node[auto] {$d^{0}_{f}$} (1-0);
\draw[->] (1-0) to node[auto] {$g$} (1-1);
\draw[->] (0-0) to node[auto,swap] {$p(a_{0})$} (0-1);
\draw[->] (1-1) to node[auto] {$d^{1}_{f}$} (1-2);
\draw[->] (0-2) to node[auto,swap] {$d^{0}_{f}$} (1-2);
\draw[->] (0-1) to node[auto,swap] {$i_{0}(a_{0})$} (0-2);
\draw[->] (0-1) to node[auto,swap] {$f$} (1-1);
\end{tikzpicture} 

\end{diagram} %
It follows from the commutativity of the last two diagrams that the following diagram in $\cl{A}$ commutes. \pushout[{4,3,-1}]{a_{0},\cyl(a_{0}),a_{1},a^{\cylinder}_{f},a^{\cylinder}_{f},i_{0}(a_{0}),d^{0}_{f},f,d^{1}_{f},d^{0}_{f} \circ i_{0}(a_{0}) \circ p(a_{0}),d^{1}_{f},d^{1}_{f} \circ g} %
Appealing to the universal property of $a^{\cylinder}_{f}$, we deduce that the following diagram in $\cl{A}$ commutes. \tri[{5,3}]{a^{\cylinder}_{f},\cyl\big(a^{\cylinder}_{f}\big),a^{\cylinder}_{f},i_{0}\big(a^{\cylinder}_{f}\big),h,d^{1}_{f} \circ g} 

Let us now prove that the following diagram in $\cl{A}$ commutes. \tri[{5,3}]{a^{\cylinder}_{f},\cyl\big(a^{\cylinder}_{f}\big),a^{\cylinder}_{f},i_{1}\big(a^{\cylinder}_{f}\big),h,id} We have that the following diagram in $\cl{A}$ commutes. 

\begin{diagram}

\begin{tikzpicture} [>=stealth]

\matrix [ampersand replacement=\&, matrix of math nodes, column sep=3 em, row sep=4 em, nodes={anchor=center}]
{ 
|(0-0)| a_{1} \&                     \&[1em] |(2-0)| a^{\cylinder}_{f} \\
              \& |(1-1)| \cyl(a_{1}) \& |(2-1)| \cyl(a^{\cylinder}_{f})     \\
|(0-2)| a_{1} \&                     \& |(2-2)| a^{\cylinder}_{f}    \\
};
	
\draw[->] (0-0) to node[auto] {$d^{1}_{f}$} (2-0);
\draw[->] (2-0) to node[auto] {$i_{1}(a^{\cylinder}_{f})$} (2-1);
\draw[->] (0-0) to node[auto,swap] {$id$} (0-2);
\draw[->] (1-1) to node[auto] {$\cyl(d^{1}_{f})$} (2-1);
\draw[->] (1-1) to node[auto] {$p(a_{1})$} (0-2);
\draw[->] (0-2) to node[auto,swap] {$d^{1}_{f}$} (2-2);
\draw[->] (2-1) to node[auto] {$h$} (2-2);
\draw[->] (0-0) to node[auto] {$i_{1}(a_{1})$} (1-1);
\end{tikzpicture} 

\end{diagram} %
The following diagram in $\cl{A}$ also commutes.  

\begin{diagram}

\begin{tikzpicture} [>=stealth]

\matrix [ampersand replacement=\&, matrix of math nodes, column sep=3 em, row sep=4 em, nodes={anchor=center}]
{ 
|(0-0)| \cyl(a_{0}) \&                         \&[1em] |(2-0)| a^{\cylinder}_{f} \\
                    \& |(1-1)| \cyl^{2}(a_{0}) \& |(2-1)| \cyl(a^{\cylinder}_{f})     \\
|(0-2)| \cyl(a_{0}) \&                         \& |(2-2)| a^{\cylinder}_{f}    \\
};
	
\draw[->] (0-0) to node[auto] {$d^{0}_{f}$} (2-0);
\draw[->] (2-0) to node[auto] {$i_{1}(a^{\cylinder}_{f})$} (2-1);
\draw[->] (0-0) to node[auto] {$i_{1}\big(\cyl(a_{0})\big)$} (1-1);
\draw[->] (1-1) to node[auto] {$\cyl(d^{0}_{f})$} (2-1);
\draw[->] (1-1) to node[auto] {$\Gamma_{lr}(a_{0})$} (0-2);
\draw[->] (0-2) to node[auto,swap] {$d^{0}_{f}$} (2-2);
\draw[->] (2-1) to node[auto] {$h$} (2-2);
\draw[->] (0-0) to node[auto,swap] {$id$} (0-2);
\end{tikzpicture} 

\end{diagram} %
Putting the last two observations together, we have that the following diagram in $\cl{A}$ commutes. \pushout[{4,3,0}]{a_{0},\cyl(a_{0}),a_{1},a^{\cylinder}_{f},a^{\cylinder}_{f},i_{0}(a_{0}),d^{0}_{f},f,d^{1}_{f},d^{0}_{f},d^{1}_{f},h \circ i_{1}(a^{\cylinder}_{f})} %
Appealing to the universal property of $a^{\cylinder}_{f}$, we deduce that the following diagram in $\cl{A}$ commutes, as we were aiming to show. \tri[{4,3}]{a^{\cylinder}_{f},\cyl(a^{\cylinder}_{f}),a^{\cylinder}_{f},i_{1}(a^{\cylinder}_{f}),h,id} 

We have now proven that $h$ defines a homotopy from $d^{1}_{f} \circ g$ to $id\big(a^{\cylinder}_{f}\big)$ with respect to $\cylinder$. It remains to demonstrate that the following diagram in $\cl{A}$ commutes. \sq{\cyl\big(a^{\cylinder}_{f}\big),a^{\cylinder}_{f},\cyl(a_{1}),a_{1},h,g,\cyl(g),p(a_{1})} %
Since $\Gamma_{lr}$ is compatible with $p$, the following diagram in $\cl{A}$ commutes. \sq[{4,3}]{\cyl^{2}(a_{0}),\cyl(a_{0}),\cyl(a_{0}),a_{0},\Gamma_{lr}(a_{0}),p(a_{0}),p\big(\cyl(a_{0})\big),p(a_{0})} By definition of $h$ and $g$, we thus have that the following diagram in $\cl{A}$ commutes.  

\begin{diagram}

\begin{tikzpicture} [>=stealth]

\matrix [ampersand replacement=\&, matrix of math nodes, column sep=3 em, row sep=4 em, nodes={anchor=center}]
{ 
|(0-0)| \cyl^{2}(a_{0}) \&                         \& |(2-0)| \cyl\big(a^{\cylinder}_{f}\big) \\
|(0-1)| \cyl(a_{0})     \& |(1-1)| \cyl^{2}(a_{0}) \& |(2-1)| a^{\cylinder}_{f}     \\
|(0-2)| a_{0}           \&                         \& |(2-2)| a_{1}    \\
};
	
\draw[->] (0-0) to node[auto] {$\cyl(d^{0}_{f})$} (2-0);
\draw[->] (2-0) to node[auto] {$h$} (2-1);
\draw[->] (0-0) to node[auto] {$\Gamma_{lr}(a_{0})$} (1-1);
\draw[->] (1-1) to node[auto] {$d^{0}_{f}$} (2-1);
\draw[->] (1-1) to node[auto] {$p(a_{0})$} (0-2);
\draw[->] (0-2) to node[auto,swap] {$f$} (2-2);
\draw[->] (2-1) to node[auto] {$g$} (2-2);
\draw[->] (0-0) to node[auto,swap] {$p\big(\cyl(a_{0})\big)$} (0-1);
\draw[->] (0-1) to node[auto,swap] {$p(a_{0})$} (0-2);
\end{tikzpicture} 

\end{diagram} %
By definition of $h$ and $g$, we also have that the following diagram in $\cl{A}$ commutes.  

\begin{diagram}

\begin{tikzpicture} [>=stealth]

\matrix [ampersand replacement=\&, matrix of math nodes, column sep=4 em, row sep=3 em, nodes={anchor=center}]
{ 
|(0-0)| \cyl(a_{1}) \& |(1-0)| \cyl\big(a^{\cylinder}_{f}\big) \\
|(0-1)| a_{1}       \& |(1-1)| a^{\cylinder}_{f}       \\
                    \& |(1-2)| a_{1}              \\
};
	
\draw[->] (0-0) to node[auto] {$\cyl(d^{1}_{f})$} (1-0);
\draw[->] (1-0) to node[auto] {$h$} (1-1);
\draw[->] (0-0) to node[auto,swap] {$p(a_{1})$} (0-1);
\draw[->] (1-1) to node[auto] {$g$} (1-2);
\draw[->] (0-1) to node[auto] {$d^{1}_{f}$} (1-1);
\draw[->] (0-1) to node[auto,swap] {$id$} (1-2);
\end{tikzpicture} 

\end{diagram} %
Putting the last two observations together, we have that the following diagram in $\cl{A}$ commutes. \pushout[{5,3,-1}]{\cyl(a_{0}),\cyl^{2}(a_{0}),\cyl(a_{1}),\cyl\big(a^{\cylinder}_{f}\big),a_{1},\cyl\big(i_{0}(a_{0})\big),\cyl(d^{0}_{f}),\cyl(j),\cyl(d^{1}_{f}),f \circ p(a_{0}) \circ p\big(\cyl(a_{0})\big),p(a_{1}),g \circ h} %
By definition of $g$, the following diagram in $\cl{A}$ also commutes. \tri[{4,3}]{\cyl(a_{1}),\cyl(a^{\cylinder}_{f}),\cyl(a_{1}),\cyl(d^{1}_{f}),\cyl(g),id} Hence the following diagram in $\cl{A}$ commutes. \tri[{4,3}]{\cyl(a_{1}),\cyl(a^{\cylinder}_{f}),a_{1},\cyl(d^{1}_{f}),p(a_{1}) \circ \cyl(g),p(a_{1})} %
Moreover, appealing to the definition of $g$, the following diagram in $\cl{A}$ commutes.   

\begin{diagram}

\begin{tikzpicture} [>=stealth]

\matrix [ampersand replacement=\&, matrix of math nodes, column sep=3 em, row sep=4 em, nodes={anchor=center}]
{ 
|(0-0)| \cyl^{2}(a_{0}) \&                         \& |(2-0)| \cyl\big(a^{\cylinder}_{f}\big) \\
|(0-1)| \cyl(a_{0})     \& |(1-1)| \cyl(a_{0}) \& |(2-1)| \cyl(a_{1})     \\
|(0-2)| a_{0}           \&                         \& |(2-2)| a_{1}    \\
};
	
\draw[->] (0-0) to node[auto] {$\cyl(d^{0}_{f})$} (2-0);
\draw[->] (2-0) to node[auto] {$\cyl(g)$} (2-1);
\draw[->] (0-0) to node[auto] {$\cyl\big(p(a_{0})\big)$} (1-1);
\draw[->] (1-1) to node[auto] {$\cyl(f)$} (2-1);
\draw[->] (1-1) to node[auto] {$p(a_{0})$} (0-2);
\draw[->] (0-2) to node[auto,swap] {$f$} (2-2);
\draw[->] (2-1) to node[auto] {$p(a_{1})$} (2-2);
\draw[->] (0-0) to node[auto,swap] {$p\big(\cyl(a_{0})\big)$} (0-1);
\draw[->] (0-1) to node[auto,swap] {$p(a_{0})$} (0-2);
\end{tikzpicture} 

\end{diagram} %
Putting the last two observations together, we have that the following diagram in $\cl{A}$ commutes. \pushout[{5,3,-1}]{\cyl(a_{0}),\cyl^{2}(a_{0}),\cyl(a_{1}),\cyl\big(a^{\cylinder}_{f}\big),a_{1},\cyl\big(i_{0}(a_{0})\big),\cyl(d^{0}_{f}),\cyl(j),\cyl(d^{1}_{f}),f \circ p(a_{0}) \circ p\big(\cyl(a_{0})\big),p(a_{1}),p(a_{1}) \circ \cyl(g)} %
Appealing to the universal property of $\cyl\big(a^{\cylinder}_{f}\big)$, we deduce that the following diagram in $\cl{A}$ commutes. \sq{\cyl\big(a^{\cylinder}_{f}\big),a^{\cylinder}_{f},\cyl(a_{1}),a_{1},h,g,\cyl(g),p(a_{1})} This completes the proof of the claim.

\end{proof}

\begin{cor} \label{MappingCylinderFactorisationStrongDeformationRetractionWithRespectToCoCylinderCorollary} Let $\cylinder = \big( \cyl, i_{0}, i_{1}, p, \Gamma_{lr} \big)$ be a cylinder in $\cl{A}$ equipped with a contraction structure $p$, and a lower right connection structure $\Gamma_{lr}$. Suppose that $\Gamma_{lr}$ is compatible with $p$. Let $\cocylinder = \big( \cocyl, e_{0}, e_{1}, c \big)$ be a co-cylinder in $\cl{A}$ equipped with a contraction structure $c$. Suppose that $\cylinder$ is left adjoint to $\cocylinder$, and that the adjunction between $\cyl$ and $\cocyl$ is compatible with $p$ and $c$. 

Let \ar{a_{0},a_{1},f} be an arrow of $\cl{A}$, and let $\big( a^{\cylinder}_{f}, d^{0}_{f}, d^{1}_{f} \big)$ be a mapping cylinder of $f$ with respect to $\cylinder$. Let \tri{a_{0},a^{\cylinder}_{f},a_{1},j,g,f} denote the corresponding mapping cylinder factorisation of $f$. Then $g$ is a strong deformation retraction of $d^{1}_{f}$ with respect to $\cocylinder$. \end{cor}

\begin{proof} Since $\cyl$ is left adjoint to $\cocyl$ we have that $\cyl$ preserves mapping cylinders with respect to $\cylinder$. Thus the result follows immediately from Proposition \ref{MappingCylinderFactorisationStrongDeformationRetractionWithRespectToCylinderProposition} and Proposition \ref{StrongDeformationRetractionCoCylinderCharacterisationProposition}. \end{proof}

\begin{lem} \label{NormallyClovenCofibrationCriterionLemma} Let $\cylinder = \big( \cyl, i_{0}, i_{1}, p, \Gamma_{lr} \big)$ be a cylinder in $\cl{A}$ equipped with a contraction structure $p$ and a lower right connection structure $\Gamma_{lr}$. Suppose that $\Gamma_{lr}$ is compatible with $p$, and that $\cyl$ preserves mapping cylinders with respect to $\cylinder$. 

Let \ar{a_{0},a_{1},f} be an arrow of $\cl{A}$, and let $\big( a^{\cylinder}_{f},d^{0}_{f},d^{1}_{f} \big)$ be a mapping cylinder of $f$ with respect to $\cylinder$. Let \tri{a_{0},a^{\cylinder}_{f},a_{1},j,g,f} denote the corresponding mapping cylinder factorisation of $f$. 

Suppose that there is an arrow \ar{a_{1},a^{\cylinder}_{f},l} of $\cl{A}$ such that the following diagram in $\cl{A}$ commutes. \liftingsquare{a_{0},a^{\cylinder}_{f},a_{1},a_{1},j,g,f,id,l} Then $f$ is a normally cloven cofibration with respect to $\cylinder$. \end{lem} 

\begin{proof} We first prove that $f$ is a cofibration with respect to $\cylinder$. To this end, suppose that we have a commutative diagram in $\cl{A}$ as follows. \sq{a_{0},\cyl(a_{0}),a_{1},a_{2},i_{0}(a_{0}),k,f,f'} %
Let \ar{\cyl\big(a^{\cylinder}_{f}\big),a^{\cylinder}_{f},h} denote the canonical arrow of $\cl{A}$ constructed as in the proof of Proposition \ref{MappingCylinderFactorisationStrongDeformationRetractionWithRespectToCylinderProposition}. We have that the following diagram in $\cl{A}$ commutes. \pushout[{5,3,-1}]{\cyl(a_{0}),\cyl^{2}(a_{0}),\cyl(a_{1}),\cyl\big(a^{\cylinder}_{f}\big),a^{\cylinder}_{f},\cyl\big(i_{0}(a_{0})\big),\cyl(d^{0}_{f}),\cyl(f),\cyl(d^{1}_{f}),d^{0}_{f} \circ \Gamma_{lr}(a_{0}),d^{1}_{f} \circ p(a_{1}),h} %
Let \ar{a^{\cylinder}_{f},a_{1},u} denote the canonical arrow of $\cl{A}$ such that the following diagram in $\cl{A}$ commutes. \pushout{a_{0},\cyl(a_{0}),a_{1},a^{\cylinder}_{f},a_{2},i_{0}(a_{0}),d^{0}_{f},f,d^{1}_{f},k,f',u} %
By definition of $l$, we have that the following diagram in $\cl{A}$ commutes. \triother{a_{0},a_{1},a^{\cylinder}_{f},f,l,j} Hence the following diagram in $\cl{A}$ commutes. \triother[{4,3}]{\cyl(a_{0}),\cyl(a_{1}),\cyl\big(a^{\cylinder}_{f}\big),\cyl(f),\cyl(l),\cyl(j)} %
Moreover, by definition of $j$, we have that the following diagram in $\cl{A}$ commutes. \tri[{5,3}]{\cyl(a_{0}),\cyl^{2}(a_{0}),\cyl\big(a^{\cylinder}_{f}\big),\cyl\big(i_{1}(a_{0})\big),\cyl(d^{0}_{f}),\cyl(j)} Putting the last two observations together, we have that the following diagram in $\cl{A}$ commutes. \sq[{5,3}]{\cyl(a_{0}),\cyl^{2}(a_{0}),\cyl(a_{1}),\cyl\big(a^{\cylinder}_{f}\big),\cyl\big(i_{1}(a_{0})\big),\cyl(d^{0}_{f}),\cyl(f),\cyl(l)} %
By definition of $h$ and $u$, we also have that the following diagram in $\cl{A}$ commutes. \squareabovetriangle[{4,3,0}]{\cyl^{2}(a_{0}),\cyl\big(a^{\cylinder}_{f}\big),\cyl(a_{0}),a^{\cylinder}_{f},a_{2},\cyl(d^{0}_{f}),h,\Gamma_{lr}(a_{0}),d^{0}_{f},u,k} %
Putting the last two observations together, we have that the following diagram in $\cl{A}$ commutes. \trapeziums[{3,{3.5},1,0}]{\cyl(a_{0}),\cyl(a_{1}),\cyl^{2}(a_{0}),\cyl\big(a^{\cylinder}_{f}\big),\cyl(a_{0}),a_{2},\cyl(f),\cyl(l),u \circ h,id,k,\cyl\big(i_{1}(a_{0})\big),\cyl(d^{0}_{f}),\Gamma_{lr}(a_{0})} %
As in the proof of Proposition \ref{MappingCylinderFactorisationStrongDeformationRetractionWithRespectToCylinderProposition}, we also have that the following diagram in $\cl{A}$ commutes. \sq[{5,3}]{a^{\cylinder}_{f},\cyl\big(a^{\cylinder}_{f}\big),a_{1},a^{\cylinder}_{f},i_{0}\big(a^{\cylinder}_{f}\big),h,g,d^{1}_{f}} %
Thus, appealing to the definition of $u$, we have that the following diagram in $\cl{A}$ commutes. \squareabovetriangle[{5,3,0}]{a^{\cylinder}_{f},\cyl\big(a^{\cylinder}_{f}\big),a_{1},a^{\cylinder}_{f},a_{2},i_{0}\big(a^{\cylinder}_{f}\big),h,g,d^{1}_{f},u,f'} %
Moreover, by definition of $l$, we have that the following diagram in $\cl{A}$ commutes. \tri{a_{1},a^{\cylinder}_{f},a_{1},l,g,id} Putting the last two observations together, we have that the following diagram in $\cl{A}$ commutes. \trapeziums[{3,{3.5},1,0}]{a_{1},\cyl(a_{1}),a^{\cylinder}_{f},\cyl\big(a^{\cylinder}_{f}\big),a_{1},a_{2},i_{0}(a_{1}),\cyl(l),u \circ h,id,f',l,i_{0}\big(a^{\cylinder}_{f}\big),g} %
We have now shown that the following diagram in $\cl{A}$ commutes. \pushout[{4,3,2}]{a_{0},\cyl(a_{0}),a_{1},\cyl(a_{1}),a_{2},i_{0}(a_{0}),\cyl(f),f,i_{0}(a_{1}),k,f',u \circ h \circ \cyl(l)} This completes our proof that $f$ is a cofibration with respect to $\cylinder$. 

We claim, moreover, that the cleavage which associates, to an object $a_{2}$ and a pair of arrows $(f',g)$ of $\cl{A}$ as above, the arrow \ar[7]{\cyl(a_{1}),a_{2},f' \circ h \circ \cyl(l)} of $\cl{A}$, equips $f$ with the structure of a normally cloven cofibration with respect to $\cylinder$. 

To this end, let us prove that this cleavage has property (i) of Definition \ref{NormallyClovenCofibrationCylinderDefinition}. Suppose that we have a commutative diagram in $\cl{A}$ as follows. \tri{\cyl(a_{0}),a_{0},a_{2},p(a_{0}),f'',k} By definition of $u$, we then have that the following diagram in $\cl{A}$ commutes. \sq{\cyl(a_{0}),a^{\cylinder}_{f},a_{0},a_{2},d^{0}_{f},u,p(a_{0}),f''} %
Since $\Gamma_{lr}$ is compatible with $p$, the following diagram in $\cl{A}$ also commutes. \sq[{4,3}]{\cyl^{2}(a_{0}),\cyl(a_{0}),\cyl(a_{0}),a_{0},\Gamma_{lr}(a_{0}),p(a_{0}),p\big(\cyl(a_{0})\big),p(a_{0})} Putting the last two observations together, and appealing to the definition of $h$, we have that the following diagram in $\cl{A}$ commutes. \trapeziumstwo[{3,3,0,1}]{\cyl^{2}(a_{0}),\cyl\big(a^{\cylinder}_{f}\big),\cyl(a_{0}),\cyl(a_{0}),a^{\cylinder}_{f},a_{0},a_{2},\cyl(d^{0}_{f}),h,u,p\big(\cyl(a_{0})\big),p(a_{0}),f'',\Gamma_{lr}(a_{0}),d^{0}_{f},p(a_{0})} %
By definition of $h$ and $u$, the following diagram in $\cl{A}$ also commutes. \squareabovetriangle[{4,3,0}]{\cyl(a_{1}),\cyl\big(a^{\cylinder}_{f}\big),a_{1},a^{\cylinder}_{f},a_{2},\cyl(d^{1}_{f}),h,p(a_{1}),d^{1}_{f},u,f'} %
Putting the last two observations together, we have that the following diagram in $\cl{A}$ commutes. \pushout[{5,3,-1}]{\cyl(a_{0}),\cyl^{2}(a_{0}),\cyl(a_{1}),\cyl\big(a^{\cyl}_{f}\big),a_{2},\cyl\big(i_{0}(a_{0})\big),\cyl(d^{0}_{f}),\cyl(f),\cyl(d^{1}_{f}),f'' \circ p(a_{0}) \circ p\big(\cyl(a_{0})\big),f' \circ p(a_{1}),u \circ h} %
We also have that the following diagram in $\cl{A}$ commutes. \squareofthreetriangles[{3,4,0,0}]{a_{0},a_{1},\cyl(a_{0}),a_{0},a_{2},f,f',id,f'',i_{0}(a_{0}),p(a_{0}),k} Hence the following diagram in $\cl{A}$ commutes. \squareabovetriangle{\cyl(a_{0}),\cyl(a_{1}),a_{0},a_{1},a_{2},\cyl(f),p(a_{1}),p(a_{0}),f,f',f''} %
By definition of $g$, the following diagram in $\cl{A}$ commutes. \sq[{4,3}]{\cyl^{2}(a_{0}),\cyl\big(a^{\cylinder}_{f}\big),\cyl(a_{0}),\cyl(a_{1}),\cyl(d^{0}_{f}),\cyl(g),\cyl\big(p(a_{0})\big),\cyl(f)} Putting the last two observations together, we have that the following diagram in $\cl{A}$ commutes. \trapeziumstwo[{3,3,0,1}]{\cyl^{2}(a_{0}),\cyl\big(a^{\cylinder}_{f}\big),\cyl(a_{0}),\cyl(a_{0}),\cyl(a_{1}),a_{0},a_{2},\cyl(d^{0}_{f}),\cyl(g),f' \circ p(a_{1}),p\big(\cyl(a_{0})\big),p(a_{0}),f'',\cyl\big(p(a_{0})\big),\cyl(f),p(a_{0})} %
Moreover, by definition of $g$, the following diagram in $\cl{A}$ commutes. \tri[{4,3}]{\cyl(a_{1}),\cyl\big(a^{\cylinder}_{f}\big),\cyl(a_{1}),\cyl(d^{1}_{f}),\cyl(g),id} Thus the following diagram in $\cl{A}$ commutes. \tri[{4,3}]{\cyl(a_{1}),\cyl\big(a^{\cylinder}_{f}\big),a_{2},\cyl(d^{1}_{f}),f' \circ p(a_{1}) \circ \cyl(g),f' \circ p(a_{1})} %
We have now shown that the following diagram in $\cl{A}$ commutes.  \pushout[{5,3,-1}]{\cyl(a_{0}),\cyl^{2}(a_{0}),\cyl(a_{1}),\cyl\big(a^{\cyl}_{f}\big),a_{2},\cyl\big(i_{0}(a_{0})\big),\cyl(d^{0}_{f}),\cyl(f),\cyl(d^{1}_{f}),f'' \circ p(a_{0}) \circ p\big(\cyl(a_{0})\big),f' \circ p(a_{1}),f' \circ p(a_{1}) \circ \cyl(g)} %
Appealing to the universal property which $\cyl\big(a^{\cylinder}_{f}\big)$ possesses due to our assumption that $\cyl$ preserves mapping cylinders, we deduce that the following diagram in $\cl{A}$ commutes. \tri{\cyl\big(a^{\cylinder}_{f}\big),a^{\cylinder}_{f},a_{2},h,u,f' \circ p(a_{1}) \circ \cyl(g)} %
Moreover, by definition of $l$, the following diagram in $\cl{A}$ commutes. \tri{\cyl(a_{1}),\cyl\big(a^{\cylinder}_{f}\big),\cyl(a_{1}),\cyl(l),\cyl(g),id} Putting the last two observations together, we have that the following diagram in $\cl{A}$ commutes. \tri{\cyl(a_{1}),\cyl\big(a^{\cylinder}_{f}\big),a_{2},\cyl(l),u \circ h,f' \circ p(a_{1})} 

This completes the proof that the cleavage satisfies property (i) of Definition \ref{NormallyClovenCofibrationCylinderDefinition}. That it moreover satisfies property (ii) of Definition \ref{NormallyClovenCofibrationCylinderDefinition} is clear. 
\end{proof}

\begin{lem} \label{MappingCylinderFactorisationGivesArrowSatisfyingNormallyClovenCofibrationCriterionLemma} Let $\cylinder = \big( \cyl, i_{0}, i_{1}, p, \subdiv, r_{0}, r_{1}, s \big)$ be a cylinder in $\cl{A}$ equipped with a contraction structure $p$, and a subdivision structure $\big( \subdiv, r_{0}, r_{1}, s \big)$. Suppose that $\cylinder$ has strictness of right identities. 

Let \ar{a_{0},a_{1},f} be an arrow of $\cl{A}$, and let $\big( a^{\cylinder}_{f},d^{0}_{f},d^{1}_{f} \big)$ be a mapping cylinder of $f$ with respect to $\cylinder$. Let \tri{a_{0},a^{\cylinder}_{f},a_{1},j,g,f} denote the corresponding mapping cylinder factorisation of $f$. 

Let $\big( a^{\cylinder}_{j},d^{0}_{j},d^{1}_{j} \big)$ be a mapping cylinder of $j$ with respect to $\cylinder$, and let \tri{a_{0},a^{\cylinder}_{j},a^{\cylinder}_{f},j',g',j} denote the corresponding mapping cylinder factorisation of $j$. 

There is an arrow \ar{a^{\cylinder}_{f},a^{\cylinder}_{j},l} of $\cl{A}$ such that the following diagram in $\cl{A}$ commutes. \liftingsquare{a_{0},a^{\cylinder}_{j},a^{\cylinder}_{f},a^{\cylinder}_{f},j',g',j,id,l} \end{lem}

\begin{proof} Since $\big( a^{\cylinder}_{j}, d^{0}_{j}, d^{1}_{j} \big)$ defines a mapping cylinder of $j$ with respect to $\cylinder$, the following diagram in $\cl{A}$ commutes. \sq{a_{0},\cyl(a_{0}),a^{\cylinder}_{f},a^{\cylinder}_{j},i_{0}(a_{0}),d^{0}_{j},j,d^{1}_{j}} %
By definition of $j$, we also have that the following diagram in $\cl{A}$ commutes. \tri{a_{0},\cyl(a_{0}),a^{\cylinder}_{f},i_{1}(a_{0}),d^{0}_{f},j} Putting the last two observations together, we have that the following diagram in $\cl{A}$ commutes. \sq[{4,3}]{a_{0},\cyl(a_{0}),\cyl(a_{0}),a^{\cylinder}_{j},i_{0}(a_{0}),d^{0}_{j},i_{1}(a_{0}),d^{1}_{j} \circ d^{0}_{f}} %
Thus there is a canonical arrow \ar{\subdiv(a_{0}),a^{\cylinder}_{j},u} of $\cl{A}$ such that the following diagram in $\cl{A}$ commutes. \pushout{a_{0},\cyl(a_{0}),\cyl(a_{0}),\subdiv(a_{0}),a^{\cylinder}_{j},i_{0}(a_{0}),r_{0}(a_{0}),i_{1}(a_{0}),r_{1}(a_{0}),d^{0}_{j},d^{1}_{j} \circ d^{0}_{f},u}  By definition of $g'$, the following diagram in $\cl{A}$ commutes. \sq{\cyl(a_{0}),a^{\cylinder}_{j},a_{0},a^{\cylinder}_{f},d^{0}_{j},g',p(a_{0}),j} %
Appealing again to the commutativity, by definition of $j$, of the diagram \triother{a_{0},\cyl(a_{0}),a^{\cylinder}_{f},i_{1}(a_{0}),d^{0}_{f},j} in $\cl{A}$, we thus have that the following diagram in $\cl{A}$ commutes. \tri{\cyl(a_{0}),a^{\cylinder}_{j},a^{\cylinder}_{f},d^{0}_{j},g',d^{0}_{f} \circ i_{1}(a_{0}) \circ p(a_{0})} %
Appealing to the commutativity, by definition of $u$, of the diagram \tri[{4,3}]{\cyl(a_{0}),\subdiv(a_{0}),a^{\cylinder}_{j},r_{0}(a_{0}),u,d^{0}_{j}} in $\cl{A}$, we deduce that the following diagram in $\cl{A}$ commutes. \tri[{4,3}]{\cyl(a_{0}),\subdiv(a_{0}),a^{\cylinder}_{f},r_{0}(a_{0}),g' \circ u, d^{1}_{f} \circ i_{1}(a_{0}) \circ p(a_{0})} %
We also have that the following diagram in $\cl{A}$ commutes, by definition of $u$ and $g'$. \squareabovetriangle[{4,3,0,0}]{\cyl(a_{0}),\subdiv(a_{0}),a^{\cylinder}_{f},a^{\cylinder}_{j},a^{\cylinder}_{f},r_{1}(a_{0}),u,d^{0}_{f},d^{1}_{j},g',id} %
Putting the last two observations together, we have that the following diagram in $\cl{A}$ commutes. \pushout{a_{0},\cyl(a_{0}),\cyl(a_{0}),\subdiv(a_{0}),a^{\cylinder}_{f},i_{0}(a_{0}),r_{0}(a_{0}),i_{1}(a_{0}),r_{1}(a_{0}),d^{0}_{f} \circ i_{1}(a_{0}) \circ p(a_{0}), d^{0}_{f}, g' \circ u} %
Let \ar{\subdiv,\cyl,q_{r}} denote the canonical 2-arrow of $\cl{C}$ constructed as in Definition \ref{StrictnessLeftIdentitiesCylinderDefinition}. We have that the following diagram in $\cl{A}$  commutes. \pushout{a_{0},\cyl(a_{0}),\cyl(a_{0}),\subdiv(a_{0}),\cyl(a_{0}),i_{0}(a_{0}),r_{0}(a_{0}),i_{1}(a_{0}),r_{1}(a_{0}),i_{1}(a_{0}) \circ p(a_{0}),id,q_{r}(a_{0})} %
Then the following diagram in $\cl{A}$ commutes. \pushout{a_{0},\cyl(a_{0}),\cyl(a_{0}),\subdiv(a_{0}),a^{\cylinder}_{f},i_{0}(a_{0}),r_{0}(a_{0}),i_{1}(a_{0}),r_{1}(a_{0}),d^{0}_{f} \circ i_{1}(a_{0}) \circ p(a_{0}),d^{0}_{f},d^{0}_{f} \circ q_{r}(a_{0})} %
Appealing to the universal property of $\subdiv(a_{0})$, we deduce that the following diagram in $\cl{A}$ commutes. \sq{\subdiv(a_{0}),a^{\cylinder}_{j},\cyl(a_{0}),a^{\cylinder}_{f},u,g',q_{r}(a_{0}),d^{0}_{f}} By definition of the homotopy $(d^{1}_{j} \circ d^{0}_{f}) + d^{0}_{j}$ with respect to $\cylinder$, we have that the following diagram in $\cl{A}$ commutes. \tri{\cyl(a_{0}),\subdiv(a_{0}),a^{\cylinder}_{j},s(a_{0}),u,{(d^{1}_{j} \circ d^{0}_{f}) + d^{0}_{j}}} %
The following diagram in $\cl{A}$ also commutes, since $\cylinder$ has strictness of right identities. \tri{\cyl(a_{0}),\subdiv(a_{0}),\cyl(a_{0}),s(a_{0}),q_{r}(a_{0}),id} Putting the last three observations together, we have that the following diagram in $\cl{A}$ commutes. \tri[{7,3}]{\cyl(a_{0}),a^{\cylinder}_{j},a^{\cylinder}_{f},{(d^{1}_{j} \circ d^{0}_{f}) + d^{0}_{j}},g',d^{0}_{f}} %
Since $\big( a^{\cylinder}_{f}, d^{0}_{f}, d^{1}_{f} \big)$ defines a mapping cylinder of $f$ with respect to $\cylinder$, the following diagram in $\cl{A}$ commutes. \sq{a_{0},\cyl(a_{0}),a_{1},a^{\cylinder}_{f},i_{0}(a_{0}),d^{0}_{f},f,d^{1}_{f}} %
Hence the following diagram in $\cl{A}$ commutes. \sq[{4,3}]{a_{0},\cyl(a_{0}),a_{1},a^{\cylinder}_{j},i_{0}(a_{0}),d^{1}_{j} \circ d^{0}_{f},f,d^{1}_{j} \circ d^{1}_{f}} We deduce that the following diagram in $\cl{A}$ commutes. \sq[{4,3}]{a_{0},\cyl(a_{0}),a_{1},a^{\cylinder}_{j},i_{0}(a_{0}),{(d^{1}_{j} \circ d^{0}_{f}) + d^{0}_{j}},f,d^{1}_{j} \circ d^{1}_{f}} %
Thus there is a canonical arrow \ar{a^{\cylinder}_{f},a^{\cylinder}_{j},l} of $\cl{A}$ such that the following diagram in $\cl{A}$ commutes. \pushout{a_{0},\cyl(a_{0}),a_{1},a^{\cylinder}_{f},a^{\cylinder}_{j},i_{0}(a_{0}),d^{0}_{f},f,d^{1}_{f},{(d^{1}_{j} \circ d^{0}_{f}) + d^{0}_{j}},d^{1}_{j} \circ d^{1}_{f},l} %
We claim that $l$ fits into a commutative diagram as in the statement of the proposition. 

Firstly, it follows from the commutativity of the diagrams \tri{\cyl(a_{0}),a^{\cylinder}_{f},a^{\cylinder}_{j},d^{0}_{f},l,{(d^{1}_{j} \circ d^{0}_{f}) + d^{0}_{j}}} and \tri[{7,3}]{\cyl(a_{0}),a^{\cylinder}_{j},a^{\cylinder}_{f},{(d^{1}_{j} \circ d^{0}_{f}) + d^{0}_{j}},g',d^{0}_{f}} in $\cl{A}$ that the following diagram in $\cl{A}$ commutes. \tri{\cyl(a_{0}),a^{\cylinder}_{f},a^{\cylinder}_{f},d^{0}_{f},g' \circ l,d^{0}_{f}} %
The following diagram in $\cl{A}$ also commutes, by definition of $l$ and $g'$. \squareabovetriangle{a_{1},a^{\cylinder}_{f},a^{\cylinder}_{f},a^{\cylinder}_{j},a^{\cylinder}_{j},d^{1}_{f},l,d^{1}_{f},d^{1}_{j},g',id} %
Putting the last two observations together, we have that the following diagram in $\cl{A}$ commutes. \pushout{a_{0},\cyl(a_{0}),a_{1},a^{\cylinder}_{f},a^{\cylinder}_{f},i_{0}(a_{0}),d^{0}_{f},f,d^{1}_{f},d^{0}_{f},d^{1}_{f},g' \circ l} Appealing to the universal property of $a^{\cylinder}_{f}$, we deduce that the following diagram in $\cl{A}$ commutes. \tri{a^{\cylinder}_{f},a^{\cylinder}_{j},a^{\cylinder}_{f},l,g',id} 

Secondly, the following diagram in $\cl{A}$ commutes, by definition of $j'$. \tri{a_{0},\cyl(a_{0}),a^{\cylinder}_{f},i_{1}(a_{0}),d^{0}_{j},j'} Hence the following diagram in $\cl{A}$ commutes. \tri{\cyl(a_{0}),\cyl(a_{0}),a^{\cylinder}_{j},i_{1}(a_{0}),{(d^{1}_{j} \circ d^{0}_{f}) + d^{0}_{j}},j'} %
By definition of $j$ and $l$, we also have that the following diagram in $\cl{A}$ commutes. \threetriangles{a_{0},\cyl(a_{0}),\cyl(a_{0}),a^{\cylinder}_{j},j,l,i_{1}(a_{0}),d^{0}_{f},{(d^{1}_{j} \circ d^{0}_{f}) + d^{0}_{j}}} Putting the last two observations together, we have that the following diagram in $\cl{A}$ commutes. \tri{a_{0},\cyl(a_{0}),a^{\cylinder}_{j},j,l,j'} This completes the proof of the claim.    
\end{proof}

\begin{prpn} \label{MappingCylinderFactorisationGivesNormallyClovenCofibrationProposition} Let $\cylinder = \big( \cyl, i_{0}, i_{1}, p, \subdiv, r_{0}, r_{1}, s, \Gamma_{lr} \big)$ be a cylinder in $\cl{A}$ equipped with a contraction structure $p$, a subdivision structure $\big(\subdiv,r_{0},r_{1},s\big)$, and a lower right connection structure $\Gamma_{lr}$. Suppose that $\Gamma_{lr}$ is compatible with $p$, and that $\cylinder$ has strictness of right identities. Suppose moreover that $\cyl$ preserves mapping cylinders with respect to $\cylinder$. 

Let \ar{a_{0},a_{1},f} be an arrow of $\cl{A}$, and let $\big( a^{\cylinder}_{f},d^{0}_{f},d^{1}_{f} \big)$ be a mapping cylinder of $f$ with respect to $\cylinder$. 

Let \tri{a_{0},a^{\cylinder}_{f},a_{1},j,g,f} denote the corresponding mapping cylinder factorisation of $f$. Then $j$ is a normally cloven cofibration with respect to $\cylinder$. \end{prpn} 

\begin{proof} Follows immediately from Lemma \ref{NormallyClovenCofibrationCriterionLemma} and Lemma \ref{MappingCylinderFactorisationGivesArrowSatisfyingNormallyClovenCofibrationCriterionLemma}. \end{proof}

\begin{cor} Let $\cocylinder = \big( \cocyl, e_{0}, e_{1}, c, \subdiv, r_{0}, r_{1},s,\Gamma_{lr} \big)$ be a co-cylinder in $\cl{A}$ equipped with a contraction structure $c$, a subdivision structure $\big(\subdiv, r_{0}, r_{1}, s\big)$, and a lower right connection structure $\Gamma_{lr}$. Suppose that $\Gamma_{lr}$ is compatible with $c$, and that $\cocylinder$ has strictness of right identities. Suppose moreover that $\cocyl$ preserves mapping co-cylinders with respect to $\cocylinder$. 

Let \ar{a_{0},a_{1},f} be an arrow of $\cl{A}$, and let $\big( a^{\cocylinder}_{f},d^{0}_{f},d^{1}_{f} \big)$ be a mapping co-cylinder of $f$ with respect to $\cocylinder$. 

Let \tri{a_{0},a^{\cocylinder}_{f},a_{1},j,g,f} denote the corresponding mapping co-cylinder factorisation of $f$. Then $g$ is a normally cloven fibration with respect to $\cocylinder$. \end{cor}

\begin{proof} Follows immediately from Proposition \ref{MappingCylinderFactorisationGivesNormallyClovenCofibrationProposition} by duality. \end{proof}

\end{chapter}

\begin{chapter}{Covering homotopy extension property} \label{CoveringHomotopyExtensionPropertyChapter}

We introduce the covering homotopy extension property with respect to a cylinder $\cylinder$ in a formal category $\cl{A}$, and to an arrow $j$ of $\cl{A}$. 

Suppose that $\cylinder$ is equipped with a contraction structure $p$. Let $\cocylinder$ be a co-cylinder in $\cl{A}$ equipped with a contraction structure $c$. Suppose that $\cylinder$ is left adjoint to $\cocylinder$, and that the adjunction between $\cyl$ and $\cocyl$ is compatible with $p$ and $c$. If an arrow $f$ of $\cl{A}$ has the covering homotopy extension property with respect to $\cylinder$ and $j$, and if $f$ is, moreover, a strong deformation retraction with respect to $\cocylinder$, we prove that $f$ has the right lifting property with respect to $j$. This will be vital to us in \ref{LiftingAxiomsChapter}.

That the covering homotopy extension property could imply a right lifting property goes back to the proof of Theorem 9 in the paper \cite{StromNoteOnCofibrationsII} of Str{\o}m, and was explored in \S{6} of the paper \cite{HastingsFibrationsOfCompactlyGeneratedSpaces} of Hastings. Our proof is abstracted from these two papers. Already in \cite{HastingsFibrationsOfCompactlyGeneratedSpaces}, Hastings observed that an abstraction of this kind should be possible. A proof in an abstract setting is also given towards the end of \S{3} of Chapter II of the book \cite{KampsPorterAbstractHomotopyAndSimpleHomotopyTheory} of Kamps and Porter.  

In the category of topological spaces, Str{\o}m proved as Theorem 4 of \cite{StromNoteOnCofibrations} that fibrations have the covering homotopy extension property with respect to closed cofibrations. In Theorem 2.1 of \cite{HastingsFibrationsOfCompactlyGeneratedSpaces}, Hastings proved that in the category of compactly generated Hausdorff spaces, fibrations have the covering homotopy extension property with respect to arbitrary cofibrations. Related theorems go back to around 1960, if not further.

More recently, fibrations in the category of topological spaces satisfying the covering homotopy extension property with respect to cofibrations were investigated in the paper \cite{SchwanzelVogtStrongCofibrationsAndFibrationsInEnrichedCategories} of Schw{\"a}nzel and Vogt, in which they are referred to as strong fibrations. They are also discussed in Chapter 4 of the book \cite{MaySigurdssonParametrizedHomotopyTheory} of May and Sigurdsson, in which the same terminology is adopted. 

\begin{assum} Let $\cl{C}$ be a 2-category with a final object. Suppose that pushouts and pullbacks of 2-arrows of $\cl{C}$ give rise to pushouts and pullbacks in formal categories, in the sense of Definition \ref{PushoutsPullbacks2ArrowsArePushoutsPullbacksInFormalCategoriesTerminology}. Let $\cl{A}$ be an object of $\cl{C}$. As before, we view $\cl{A}$ as a formal category, writing of objects and arrows of $\cl{A}$. \end{assum} 

\begin{defn} Let $\cylinder = \big( \cyl, i_0, i_1 \big)$ be a cylinder in $\cl{A}$. Let \ar{a_{0},a_{1},j} be an arrow of $\cl{A}$, and suppose that $\big( a^{\cylinder}_{j}, d^{0}_{j}, d^{1}_{j} \big)$ defines a mapping cylinder of $j$ with respect to $\cylinder$. 

An arrow \ar{a_{2},a_{3},f} of $\cl{A}$ has the {\em covering homotopy extension property} with respect to $j$ and $\cylinder$ if, for any commutative diagram \sq{a^{\cylinder}_{j},a_{2},\cyl(a_{1}),a_{3},g,f,m^{\cylinder}_{j},h} in $\cl{A}$, there is an arrow \ar{\cyl(a_{1}),a_{2},l} of $\cl{A}$, such that the following diagram in $\cl{A}$ commutes. \liftingsquare{a^{\cylinder}_{j},a_{2},\cyl(a_{1}),a_{3},g,f,m^{\cylinder}_{j},h,l} \end{defn}

\begin{rmk} This definition is equivalent to that given towards the end of \S{3} of Chapter II of the book \cite{KampsPorterAbstractHomotopyAndSimpleHomotopyTheory} of Kamps and Porter. We shall not need this. \end{rmk}

\begin{prpn} \label{CHEPImpliesLiftingAxiomProposition} Let $\cylinder = \big( \cyl, i_0, i_1, p \big)$ be a cylinder in $\cl{A}$ equipped with a contraction structure $p$, and let $\cocylinder = \big( \cocyl, e_{0}, e_{1}, c \big)$ be a co-cylinder in $\cl{A}$ equipped with a contraction structure $c$. Suppose that $\cylinder$ is left adjoint to $\cocylinder$, and that the adjunction between $\cyl$ and $\cocyl$ is compatible with $p$ and $c$. 

Let \ar{a_{0},a_{1},j} be an arrow of $\cl{A}$, and suppose that $\big( a^{\cylinder}_{j}, d^{0}_{j}, d^{1}_{j} \big)$ defines a mapping cylinder of $j$ with respect to $\cylinder$. Let \ar{a_{2},a_{3},f} be an arrow of $\cl{A}$ which has the covering homotopy extension property with respect to $j$ and $\cylinder$. Moreover, let \ar{a_{3},a_{2},j'} be an arrow of $\cl{A}$ and suppose that $f$ is a strong deformation retraction of $j'$ with respect to $\cocylinder$. 

Then, for any commutative diagram \sq{a_{0},a_{2},a_{1},a_{3},g_{0},f,j,g_{1}} in $\cl{A}$, there is an arrow \ar{a_{1},a_{2},l} of $\cl{A}$ such that the following diagram in $\cl{A}$ commutes. \liftingsquare{a_{0},a_{2},a_{1},a_{3},g_{0},f,j,g_{1},l} \end{prpn} 

\begin{proof} Since $f$ is a strong deformation retraction of $j'$ with respect to $\cocylinder$, we have by Proposition \ref{StrongDeformationRetractionCoCylinderCharacterisationProposition} that there is a homotopy \ar{\cyl(a_{2}),a_{2},h} over $a_{3}$ from $j'f$ to $id(a_{2})$ with respect to $\cylinder$ and $(f,f)$. The following diagram in $\cl{A}$ commutes. \trapeziumstwo[{3,3,1,0}]{a_{0},\cyl(a_{0}),a_{1},a_{2},\cyl(a_{2}),a_{3},a_{2},i_{0}(a_{0}),\cyl(g_{0}),h,j,g_{1},j',g_{0},i_{0}(a_{2}),f} %
Thus there is a canonical arrow \ar{a^{\cyl}_{j},a_{2},u} of $\cl{A}$ such that the following diagram in $\cl{A}$ commutes. \pushout{a_{0},\cyl(a_{0}),a_{1},a^{\cylinder}_{j},a_{2},i_{0}(a_{0}),d^{0}_{j},j,d^{1}_{j},h \circ \cyl(g_{0}),j' \circ g_{1},u} %
By definition of $h$ as a homotopy over $a_{3}$ with respect to $\cylinder$ and $(f,f)$, the following diagram in $\cl{A}$ commutes. \sq{\cyl(a_{2}),a_{2},\cyl(a_{3}),a_{3},h,f,\cyl(f),p(a_{3})} Moreover, the following diagram in $\cl{A}$ commutes. \sq{\cyl(a_{2}),\cyl(a_{3}),a_{2},a_{3},\cyl(f),p(a_{3}),p(a_{2}),f} Putting the last two observations together, we have that the following diagram in $\cl{A}$ commutes. \sq{\cyl(a_{2}),a_{2},a_{2},a_{3},h,f,p(a_{2}),f} %
We now have that the following diagram in $\cl{A}$ commutes. 

\begin{diagram}

\begin{tikzpicture} [>=stealth]

\matrix [ampersand replacement=\&, matrix of math nodes, column sep=3 em, row sep=3 em, nodes={anchor=center}]
{ 
|(0-0)| \cyl(a_{0}) \&                     \& |(2-0)| a^{\cylinder}_{j} \\ 
                    \& |(1-1)| \cyl(a_{2}) \& |(2-1)| a_{2}        \\
                    \& |(1-2)| a_{2}       \&                     \\
|(0-3)| a_{0}       \& |(1-3)| a_{1}       \& |(2-3)| a_{3}       \\
};
	
\draw[->] (0-0) to node[auto] {$d^{0}_{j}$} (2-0);
\draw[->] (0-0) to node[auto,swap] {$p(a_{0})$} (0-3);
\draw[->] (0-0) to node[auto] {$\cyl(g_{0})$} (1-1);
\draw[->] (1-1) to node[auto] {$h$} (2-1);
\draw[->] (1-1) to node[auto,swap] {$p(a_{2})$} (1-2);
\draw[->] (1-2) to node[auto] {$f$} (2-3);
\draw[->] (0-3) to node[auto] {$g_{0}$} (1-2);
\draw[->] (2-0) to node[auto] {$u$} (2-1);
\draw[->] (2-1) to node[auto] {$f$} (2-3);
\draw[->] (0-3) to node[auto,swap] {$j$} (1-3);
\draw[->] (1-3) to node[auto,swap] {$g_{1}$} (2-3);
\end{tikzpicture} 

\end{diagram} %
The following diagram in $\cl{A}$ also commutes. \squareabovetriangle{a_{1},a^{\cylinder}_{j},a_{3},a_{2},a_{3},d^{1}_{j},u,g_{1},j',f,id}%
Putting the last two observations together, we have that the following diagram in $\cl{A}$ commutes. \pushout[{4,3,0}]{a_{0},\cyl(a_{0}),a_{1},a^{\cylinder}_{j},a_{3},i_{0}(a_{0}),d^{0}_{j},j,d^{1}_{j},g_{1} \circ j \circ p(a_{0}), g_{1}, f \circ u} %
The following diagram in $\cl{A}$ also commutes.  

\begin{diagram}

\begin{tikzpicture} [>=stealth]

\matrix [ampersand replacement=\&, matrix of math nodes, column sep=6 em, row sep=3 em, nodes={anchor=center}]
{ 
|(0-0)| \cyl(a_{0}) \& |(1-0)| a^{\cylinder}_{j} \\ 
                    \& |(1-1)| \cyl(a_{1})   \\
|(0-2)| a_{0}       \& |(1-2)| a_{1}        \\
};
	
\draw[->] (0-0) to node[auto] {$d^{0}_{j}$} (1-0);
\draw[->] (0-0) to node[auto,swap] {$p(a_{0})$} (0-2);
\draw[->] (1-0) to node[auto] {$m^{\cylinder}_{j}$} (1-1);
\draw[->] (0-0) to node[auto,swap] {$\cyl(j)$} (1-1);
\draw[->] (1-1) to node[auto] {$p(a_{1})$} (1-2);
\draw[->] (0-2) to node[auto,swap] {$j$} (1-2);
\end{tikzpicture} 

\end{diagram} %
Hence the following diagram in $\cl{A}$ commutes. \tri{\cyl(a_{0}),a^{\cylinder}_{j},a_{3},d^{0}_{j},g_{1} \circ p(a_{1}) \circ m^{\cylinder}_{j},g_{1} \circ j \circ p(a_{0})} %
Moreover, the following diagram in $\cl{A}$ commutes.   

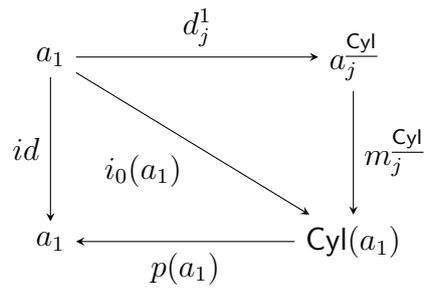
\begin{diagram}

\begin{tikzpicture} [>=stealth]

\matrix [ampersand replacement=\&, matrix of math nodes, column sep=7 em, row sep=4 em, nodes={anchor=center}]
{ 
|(0-0)| a_{1} \& |(1-0)| a^{\cylinder}_{j} \\ 
|(0-1)| a_{1} \& |(1-1)| \cyl(a_{1})   \\
};
	
\draw[->] (0-0) to node[auto] {$d^{1}_{j}$} (1-0);
\draw[->] (0-0) to node[auto,swap] {$id$} (0-1);
\draw[->] (1-0) to node[auto] {$m^{\cylinder}_{j}$} (1-1);
\draw[->] (1-1) to node[auto] {$p(a_{1})$} (0-1);
\draw[->] (0-0) to node[auto,swap] {$i_{0}(a_{1})$} (1-1);
\end{tikzpicture} 

\end{diagram} %
Thus the following diagram in $\cl{A}$ commutes. \tri{a_{1},a^{\cylinder}_{j},a_{3},d^{1}_{j},g_{1} \circ p(a_{1}) \circ m^{\cylinder}_{j},g_{1}} %
We now have that the following diagram in $\cl{A}$ commutes. \pushout[{4,3,3}]{a_{0},\cyl(a_{0}),a_{1},a^{\cylinder}_{j},a_{3},i_{0}(a_{0}),d^{0}_{j},j,d^{1}_{j},g_{1} \circ j \circ p(a_{0}), g_{1}, g_{1} \circ p(a_{1}) \circ m^{\cylinder}_{j}} %
Appealing to the universal property of $a^{\cylinder}_{j}$, it follows that the diagram \sq[{5,3}]{a^{\cylinder}_{j},a_{2},\cyl(a_{1}),a_{3},u,f,m^{\cylinder}_{j},g_{1} \circ p(a_{1})} in $\cl{A}$ commutes. %
Since $f$ has the covering homotopy extension property with respect to $j$, we deduce there is an arrow \ar{\cyl(a_{1}),a_{2},l} of $\cl{A}$ such that the following diagram in $\cl{A}$ commutes. \liftingsquare[{5,3.5}]{a^{\cylinder}_{j},a_{2},\cyl(a_{1}),a_{3},u,f,m^{\cylinder}_{j},g_{1} \circ p(a_{1}),l} %
Let $x$ denote the arrow \ar[5]{a_{1},a_{2},l \circ i_{1}(a_{1})} of $\cl{A}$. We claim that the following diagram in $\cl{A}$ commutes. \liftingsquare{a_{0},a_{2},a_{1},a_{3},g_{0},f,j,g_{1},x} 

Firstly, note that the following diagram in $\cl{A}$ commutes. \liftingsquare[{4,4}]{\cyl(a_{0}),\cyl(a_{1}),a^{\cylinder}_{j},a_{2},\cyl(j),l,d^{0}_{j},u,m^{\cylinder}_{j}} Hence the following diagram in $\cl{A}$ commutes. \trapeziumsthree[{4,3.5,-1,0}]{a_{0},a_{1},\cyl(a_{0}),\cyl(a_{1}),a^{\cylinder}_{j},a_{2},j,x,i_{1}(a_{0}),d^{0}_{j},u,i_{1}(a_{1}),\cyl(j),l} %
Moreover, the following diagram in $\cl{A}$ commutes, by appeal to the definition of $u$ and $h$. 

\begin{diagram}

\begin{tikzpicture} [>=stealth]

\matrix [ampersand replacement=\&, matrix of math nodes, column sep=3 em, row sep=3 em, nodes={anchor=center}]
{ 
|(0-0)| a_{0}  \&[1em]                     \& |(2-0)| \cyl(a_{0}) \\
|(0-1)| a_{2}  \& |(1-1)| \cyl(a_{2}) \& |(2-1)| a^{\cylinder}_{j} \\
               \&                     \& |(2-2)| a_{2} \\
};
	
\draw[->] (0-0) to node[auto] {$i_{1}(a_{0})$} (2-0);
\draw[->] (0-0) to node[auto,swap] {$g_{0}$} (0-1);
\draw[->] (0-1) to node[auto,swap] {$i_{1}(a_{2})$} (1-1);
\draw[->] (2-0) to node[auto,swap] {$\cyl(g_{0})$} (1-1);
\draw[->] (2-0) to node[auto] {$d^{0}_{j}$} (2-1);
\draw[->,bend right=45] (0-1) to node[auto,swap] {$id$} (2-2);
\draw[->] (1-1) to node[auto,swap] {$h$} (2-2);
\draw[->] (2-1) to node[auto] {$u$} (2-2);
\end{tikzpicture} 

\end{diagram} %
Putting the last two observations together, we have that the following diagram in $\cl{A}$ commutes. \tri{a_{0},a_{1},a_{2},j,x,g_{0}} 

Secondly, the following diagram in $\cl{A}$ commutes. \squareofthreetrianglestwo[{3,3.5,0,0}]{a_{1},a_{2},\cyl(a_{1}),a_{1},a_{3},x,f,id,g_{1},i_{1}(a_{1}),l,p(a_{1})} %
This completes the proof of the claim.

\end{proof}

\end{chapter}

\begin{chapter}{Dold's theorem} \label{DoldTheoremChapter}

Let $\cylinder$ be a cylinder in a formal category $\cl{A}$, equipped with a contraction structure, an involution structure compatible with contraction, a subdivision structure compatible with contraction, and an upper left connection structure. Suppose that we have a commutative diagram \triother{a_{0},a_{1},a,f,j_{1},j_{0}} in $\cl{A}$, in which $j_{0}$ and $j_{1}$ are fibrations with respect to $\cylinder$. We prove that if $f$ is a homotopy equivalence with respect to $\cylinder$, then $f$ is, moreover, a homotopy equivalence over $a$ with respect to $\cylinder$ and $(j_{0},j_{1})$. 

For topological spaces, this Theorem 3.1 of the paper \cite{DoldPartitionsOfUnityInTheTheoryOfFibrations} of Dold. Our proof is an abstraction of a hybrid of Dold's proof and the proof presented in \S{5} of Chapter 6 of the book \cite{MayAConciseCourseInAlgebraicTopology} of May. 

We exhibit a double homotopy \doublehomotopy{h_{0},h_{1},h_{2},h_{3},\sigma,f_{0},f_{1},f_{2},f_{3}} for which $h_{1}$, $h_{2}$, and $h_{3}$ can be proven to be homotopies over $a$. One can then construct a homotopy over $a$ from $f_{0}$ to $f_{1}$ by taking the indirect route around the square, after reversing $h_{1}$. 

In our construction of this double homotopy, the key role is played by an upper left connection structure. The reader may observe that the map \ar[12]{I^{2},{I,},{(t_{0},t_{1}) \mapsto t_{0} + (1-t_{0})t_{1}}} in which $I$ is the unit interval, underlies the construction of the double homotopy in \cite{MayAConciseCourseInAlgebraicTopology}. This map defines an upper left connection structure with respect to the topological interval.

Given a cylinder whose associated cubical set satisfies low dimensional Kan conditions, a proof of Dold's theorem was given by Kamps in \S{6} of \cite{KampsKanBedingungenUndAbstrakteHomotopietheorie}. A variation is presented in \S{6} of Chapter I of \cite{KampsPorterAbstractHomotopyAndSimpleHomotopyTheory}. There is a fundamental difference between our proof and these two.

We demonstrate that the double homotopy can be constructed if $\cylinder$ admits certain structures. By contrast, requiring that the cubical set associated to $\cylinder$ satisfies low dimensional Kan conditions ensures the existence of this double homotopy, but the Kan conditions must themselves be proven to hold. To put it another way, the proofs of \cite{KampsKanBedingungenUndAbstrakteHomotopietheorie} and \cite{KampsPorterAbstractHomotopyAndSimpleHomotopyTheory} can be thought of as a plan for the construction of the double homotopy, whereas we identify structures upon a cylinder which allow us to carry out this plan.

A different proof of Dold's theorem can be given by identifying structures upon $\cylinder$ which allow one to prove that the objects, arrows, and homotopies up to homotopy of $\cl{A}$ with respect to $\cylinder$ assemble into a 2-category. For one can then appeal to the argument presented in \S{1} and \S{2} of Chapter IV of the book \cite{KampsPorterAbstractHomotopyAndSimpleHomotopyTheory} of Kamps and Porter.

We deduce from Dold's theorem that a trivial cofibration admits a strong deformation retraction. This will be vital for us in \ref{LiftingAxiomsChapter}, when we establish the lifting axioms for a model structure. As observed by Dold as Satz 3.6 of \cite{DoldHalbexakteHomotopiefunktoren}, his theorem dualises, from which we deduce that a trivial fibration is a strong deformation retraction. 

If we have strictness of identities, we shall prove in \ref{LiftingAxiomsChapter} that trivial cofibrations are exactly sections of strong deformation retractions, and dually that trivial fibrations are exactly strong deformation retractions. 

\begin{assum} Let $\cl{C}$ be a 2-category with a final object. Suppose that pushouts and pullbacks of 2-arrows of $\cl{C}$ give rise to pushouts and pullbacks in formal categories, in the sense of Definition \ref{PushoutsPullbacks2ArrowsArePushoutsPullbacksInFormalCategoriesTerminology}. Let $\cl{A}$ be an object of $\cl{C}$. As before, we view $\cl{A}$ as a formal category, writing of objects and arrows of $\cl{A}$. \end{assum} 

\begin{lem} \label{DoldLemma1} Let $\cylinder = \big( \cyl, i_0, i_1, p, v, \subdiv, r_{0}, r_{1}, s \big)$ be a cylinder in $\cl{A}$ equipped with a contraction structure $p$, an involution structure $v$, and a subdivision structure $\big(\subdiv, r_{0}, r_{1}, s \big)$. 

Let \ar{a,a_{0},j_{0}} be an arrow of $\cl{A}$ which is a fibration with respect to $\cylinder$. Let \ar{a,a_{1},j_{1}} and \ar{a_{0},a_{1},f} be arrows of $\cl{A}$, such that the diagram \triother{a_{0},a_{1},a,f,j_{1},j_{0}} in $\cl{A}$ commutes, and such that $f$ is a homotopy equivalence with respect to $\cylinder$. 

There is an arrow \ar{a_{1},a_{0},g} of $\cl{A}$, and a homotopy from $fg$ to $id(a_{1})$ with respect to $\cylinder$, such that the following diagram in $\cl{A}$ commutes. \triother{a_{1},a_{0},a,g,j_{0},j_{1}} \end{lem} 

\begin{proof} Let \ar{\cyl(a_{1}),a_{1},h} be a homotopy from $ff^{-1}$ to $id(a_{1})$ with respect to $\cylinder$. The following diagram in $\cl{A}$ commutes. \squareabovetriangle{a_{1},\cyl(a_{1}),a_{0},a_{1},a,i_{0}(a_{1}),h,f^{-1},f,j_{1},j_{0}} %
Since $j_{0}$ is a fibration with respect to $\cylinder$, we deduce that there is an arrow \ar{\cyl(a_{1}),a_{0},k} of $\cl{A}$ such that the following diagram in $\cl{A}$ commutes. \liftingsquare[{4,3}]{a_{1},a_{0},\cyl(a_{1}),a,{f^{-1}},j_{0},i_{0}(a_{1}),j_{1} \circ h,k} %
Let $g$ denote the arrow \ar[5]{a_{1},a_{0},k \circ i_{1}(a_{1})} of $\cl{A}$. The following diagram in $\cl{A}$ commutes. \squarebeneathtriangle{a_{1},\cyl(a_{1}),a_{1},a_{1},a,i_{1}(a_{1}),h,id,j_{1},k,j_{0}} %
Thus the following diagram in $\cl{A}$ commutes. \triother{a_{1},a_{0},a,g,j_{0},j_{1}} %
It remains to construct a homotopy from $fg$ to $id(a_{1})$ with respect to $\cylinder$. 

Firstly, the following diagram in $\cl{A}$ commutes. 

\begin{diagram}

\begin{tikzpicture} [>=stealth]

\matrix [ampersand replacement=\&, matrix of math nodes, column sep=6 em, row sep=3 em, nodes={anchor=center}]
{ 
|(0-0)| a_{1} \& |(1-0)| \cyl(a_{1}) \\
|(0-1)| a_{0} \& |(1-1)| \cyl(a_{1}) \\ 
};
	
\draw[->] (0-0) to node[auto] {$i_{1}(a_{1})$} (1-0);
\draw[->] (0-0) to node[auto,swap] {$f^{-1}$} (0-1);
\draw[->] (0-0) to node[auto,swap] {$i_{0}(a_{1})$} (1-1);
\draw[->] (1-0) to node[auto] {$v(a_{1})$} (1-1);
\draw[->] (1-1) to node[auto] {$k$} (0-1);
\end{tikzpicture} 

\end{diagram} %
We also have that the diagram \sq{a_{1},\cyl(a_{1}),a_{0},a_{1},i_{0}(a_{1}),h,f^{-1},f} in $\cl{A}$ commutes. Putting the last two observations together, we have that the following diagram in $\cl{A}$ commutes. \sq[{5,3}]{a_{1},\cyl(a_{1}),\cyl(a_{1}),a_{1},i_{0}(a_{1}),h,i_{1}(a_{1}),f \circ k \circ v(a_{1})} %
Let us denote by \ar{\cyl(a_{1}),a_{1},l} the homotopy $\big(f \circ k \circ v(a_{1}) \big) + h$ with respect to $\cylinder$. The following diagram in $\cl{A}$ commutes. 

\begin{diagram}

\begin{tikzpicture} [>=stealth]

\matrix [ampersand replacement=\&, matrix of math nodes, column sep=6 em, row sep=3 em, nodes={anchor=center}]
{ 
|(0-0)| a_{1} \& |(1-0)| \cyl(a_{1}) \\
|(0-1)| a_{0} \& |(1-1)| \cyl(a_{1}) \\ 
};
	
\draw[->] (0-0) to node[auto] {$i_{0}(a_{1})$} (1-0);
\draw[->] (0-0) to node[auto,swap] {$g$} (0-1);
\draw[->] (0-0) to node[auto,swap] {$i_{1}(a_{1})$} (1-1);
\draw[->] (1-0) to node[auto] {$v(a_{1})$} (1-1);
\draw[->] (1-1) to node[auto] {$k$} (0-1);
\end{tikzpicture} 

\end{diagram} %
Thus, since the diagram \sq{a_{1},\cyl(a_{1}),\cyl(a_{1}),a_{1},i_{0}(a_{1}),l,i_{0}(a_{1}),h} in $\cl{A}$ commutes, we have that the following diagram in $\cl{A}$ commutes. \tri{a_{1},\cyl(a_{1}),a_{1},i_{0}(a_{1}),l,fg} %
In addition, since the diagram \sq{a_{1},\cyl(a_{1}),\cyl(a_{1}),a_{1},i_{1}(a_{1}),l,i_{1}(a_{1}),h} in $\cl{A}$ commutes, and since the diagram \tri{a_{1},\cyl(a_{1}),a_{1},i_{1}(a_{1}),h,id} in $\cl{A}$ commutes, we have that the following diagram in $\cl{A}$ commutes. \tri{a_{1},\cyl(a_{1}),a_{1},i_{1}(a_{1}),l,id}

\end{proof}

\begin{lem} \label{DoldLemma2} Let $\cylinder = \big( \cyl, i_0, i_1, p, v, \subdiv, r_{0}, r_{1}, s, \Gamma_{ul} \big)$ be a cylinder in $\cl{A}$ equipped with a contraction structure $p$, an involution structure $v$, a subdivision structure $\big(\subdiv, r_{0}, r_{1}, s \big)$, and an upper left connection structure $\Gamma_{ul}$. Suppose that $\cyl$ preserves subdivision with respect to $\cylinder$. 

Let \ar{a,a_{0},j_{0}} be an arrow of $\cl{A}$ which is a fibration with respect to $\cylinder$. Let \ar{a,a_{1},j_{1}} and \ar{a_{0},a_{1},f} be arrows of $\cl{A}$, such that the diagram \triother{a_{0},a_{1},a,f,j_{1},j_{0}} in $\cl{A}$ commutes, and such that $f$ is a homotopy equivalence with respect to $\cylinder$. 

Let \ar{a_{1},a_{0},g} and \ar{\cyl(a_{1}),a_{1},l} denote the arrows of $\cl{A}$ constructed in Lemma \ref{DoldLemma1}, such that the diagram \triother{a_{1},a_{0},a,g,j_{0},j_{1}} in $\cl{A}$ commutes, and such that $l$ defines a homotopy from $fg$ to $id(a_{1})$ with respect to $\cylinder$. 

There is an arrow \ar{\cyl^{2}(a_{1}),a,\tau} of $\cl{A}$ such that the following diagram in $\cl{A}$ commutes. \sq[{5,3}]{\cyl(a_{1}),\cyl^{2}(a_{1}),a_{1},a,i_{0}\big(\cyl(a_{1})\big),\tau,l,j_{1}} \end{lem} 

\begin{proof} Let \ar{\cyl(a_{1}),a_{0},k} denote the arrow of $\cl{A}$ constructed in the proof of Lemma \ref{DoldLemma1}. In particular, the following diagram in $\cl{A}$ commutes. \sq{\cyl(a_{1}),a_{0},a_{1},a,k,j_{0},h,j_{1}} %
By definition of $\Gamma_{ul}$ as an upper left connection structure, the following diagram in $\cl{A}$ also commutes. \tri[{5,3}]{\cyl(a_{1}),\cyl^{2}(a_{1}),\cyl(a_{1}),\cyl\big(i_{0}(a_{1})\big),\Gamma_{ul}(a_{1}),id} %
Putting the last two observations together, we have that the following diagram in $\cl{A}$ commutes. \tri[{5,3}]{\cyl(a_{1}),\cyl^{2}(a_{1}),a,\cyl\big(i_{0}(a_{1})\big),j_{1} \circ h \circ \Gamma_{ul}(a_{1}),j_{0} \circ k} %
By definition of $v$ as an involution structure, the following diagram in $\cl{A}$ commutes. \tri{a_{1},\cyl(a_{1}),\cyl(a_{1}),i_{1}(a_{1}),v(a_{1}),i_{0}(a_{1})} Thus, appealing once more to the commutativity of the diagram \tri[{5,3}]{\cyl(a_{1}),\cyl^{2}(a_{1}),\cyl(a_{1}),\cyl\big(i_{0}(a_{1})\big),\Gamma_{ul}(a_{1}),id} in $\cl{A}$, we have that the following diagram in $\cl{A}$ commutes. \tri{\cyl(a_{1}),\cyl^{2}(a_{1}),a,\cyl(a_{1}),\cyl\big(i_{1}(a_{1})\big),\Gamma_{ul}(a_{1}),\cyl\big(v(a_{1})\big),id} %
Hence the following diagram in $\cl{A}$ commutes. \tri[{5,3}]{\cyl(a_{1}),\cyl^{2}(a_{1}),a,\cyl\big(i_{1}(a_{1})\big),j_{0} \circ k \circ \Gamma_{ul}(a_{1}) \circ \cyl\big(v(a_{1})\big),j_{0} \circ k} %
Putting everything together, we have now shown that the following diagram in $\cl{A}$ commutes. \sq[{11,3}]{\cyl(a_{1}),\cyl^{2}(a_{1}),\cyl^{2}(a_{1}),a,\cyl\big(i_{0}(a_{1})\big),j_{1} \circ h \circ \Gamma_{ul}(a_{1}),\cyl\big(i_{1}(a_{1})\big),j_{0} \circ k \circ \Gamma_{ul}(a_{1}) \circ \cyl\big(v(a_{1})\big)} %
Since $\cyl$ preserves subdivision with respect to $\cylinder$, the following diagram in $\cl{A}$ is co-cartesian. \sq[{5,3}]{\cyl(a_{1}),\cyl^{2}(a_{1}),\cyl^{2}(a_{1}),\cyl\big(\subdiv(a_{1})\big),\cyl\big(i_{0}(a_{1})\big),\cyl\big(r_{0}(a_{1})\big),\cyl\big(i_{1}(a_{1})\big),\cyl\big(r_{1}(a_{1})\big)} %
Thus there is an arrow \ar{\cyl\big(\subdiv(a_{1})\big),a,u} of $\cl{A}$ such that the following diagram in $\cl{A}$ commutes. \pushout[{5,3,0}]{\cyl(a_{1}),\cyl^{2}(a_{1}),\cyl^{2}(a_{1}),\cyl\big(\subdiv(a_{1})\big),a,\cyl\big(i_{0}(a_{1})\big),\cyl\big(r_{0}(a_{1})\big),\cyl\big(i_{1}(a_{1})\big),\cyl\big(r_{1}(a_{1})\big),j_{1} \circ h \circ \Gamma_{ul}(a_{1}),j_{0} \circ k \circ \Gamma_{ul}(a_{1}) \circ \cyl\big(v(a_{1})\big),u} %
Let \ar{\cyl^{2}(a_{1}),a,\tau} denote the arrow $u \circ \cyl\big(s(a_{1})\big)$ of $\cl{A}$. We claim that the following diagram in $\cl{A}$ commutes. \sq[{5,3}]{\cyl(a_{1}),\cyl^{2}(a_{1}),a_{1},a,i_{0}\big(\cyl(a_{1})\big),\tau,l,j_{1}} 

We have that the following diagram in $\cl{A}$ commutes. \trapeziumstwo[{3,3.5,2,0}]{\cyl(a_{1}),\subdiv(a_{1}),\cyl(a_{1}),\cyl^{2}(a_{1}),\cyl\big(\subdiv(a_{1})\big),\cyl(a_{1}),a,r_{1}(a_{1}),i_{0}\big(\subdiv(a_{1})\big),u,v(a_{1}),i_{0}\big(\cyl(a_{1})\big),j_{0} \circ k \circ \Gamma_{ul}(a_{1}),i_{0}\big(\cyl(a_{1})\big),\cyl\big(r_{1}(a_{1})\big),\cyl\big(v(a_{1})\big)} %
Since the diagram \tri[{5,3}]{\cyl(a_{1}),\cyl^{2}(a_{1}),\cyl(a_{1}),i_{0}\big(\cyl(a_{1})\big),\Gamma_{ul}(a_{1}),id} in $\cl{A}$ commutes, we deduce that the following diagram in $\cl{A}$ commutes. \tri{\cyl(a_{1}),\subdiv(a_{1}),a,r_{1}(a_{1}),u \circ i_{0}\big(\subdiv(a_{1})\big),j_{0} \circ k \circ v(a_{1})} %
We also have that the following diagram in $\cl{A}$ commutes. \trapeziums[{3,3.5,2,0}]{\cyl(a_{1}),\subdiv(a_{1}),\cyl^{2}(a_{1}),\cyl\big(\subdiv(a_{1})\big),\cyl(a_{1}),a,r_{0}(a_{1}),i_{0}\big(\subdiv(a_{1})\big),u,id,j_{1} \circ h,i_{0}\big(\cyl(a_{1})\big),\cyl\big(r_{0}(a_{1})\big),\Gamma_{ul}(a_{1})} %
Putting the last two observations together, we have that the following diagram in $\cl{A}$ commutes. \pushout{a_{1},\cyl(a_{1}),\cyl(a_{1}),\subdiv(a_{1}),a,i_{0}(a_{1}),r_{0}(a_{1}),i_{1}(a_{1}),r_{1}(a_{1}),j_{1} \circ h,j_{0} \circ k \circ v(a_{1}),u \circ i_{0}\big(\subdiv(a_{1})\big)} %
In the proof of Lemma \ref{DoldLemma1}, we showed that the following diagram in $\cl{A}$ commutes. \sq[{5,3}]{a_{1},\cyl(a_{1}),\cyl(a_{1}),a_{1},i_{0}(a_{1}),h,i_{1}(a_{1}),f \circ k \circ v(a_{1})} %
Thus there is an arrow \ar{\subdiv(a_{1}),a_{1},r} of $\cl{A}$ such that the following diagram in $\cl{A}$ commutes. \pushout{a_{1},\cyl(a_{1}),\cyl(a_{1}),\subdiv(a_{1}),a_{1},i_{0}(a_{1}),r_{0}(a_{1}),i_{1}(a_{1}),r_{1}(a_{1}),h,f \circ k \circ v(a_{1}),r} %
In particular, since the diagram \tri{\cyl(a_{1}),\subdiv(a_{1}),a_{1},r_{0}(a_{1}),r,h} in $\cl{A}$ commutes, we have that the following diagram in $\cl{A}$ commutes. \tri{\cyl(a_{1}),\subdiv(a_{1}),a,r_{0}(a_{1}),j_{1} \circ r, j_{1} \circ h} %
Moreover, the following diagram in $\cl{A}$ commutes. \squareabovetriangle{\cyl(a_{1}),\subdiv(a_{1}),a_{0},a_{1},a,r_{1}(a_{1}),r,k \circ v(a_{1}),f,j_{1},j_{0}} %
Putting the last two observations together, we have that the following diagram in $\cl{A}$ commutes. \pushout{a_{1},\cyl(a_{1}),\cyl(a_{1}),\subdiv(a_{1}),a,i_{0}(a_{1}),r_{0}(a_{1}),i_{1}(a_{1}),r_{1}(a_{1}),j_{1} \circ h,j_{0} \circ k \circ v(a_{1}),j_{1} \circ r} %
Appealing to the universal property of $\subdiv(a_{1})$, we deduce that the following diagram in $\cl{A}$ commutes. \sq[{5,3}]{\subdiv(a_{1}),\cyl\big(\subdiv(a_{1})\big),a_{1},a,i_{0}\big(\subdiv(a_{1})\big),u,r,j_{1}} %
Hence the following diagram in $\cl{A}$ commutes.\trapeziumsthree[{5,3.5,-2,0}]{\cyl(a_{1}),\cyl^{2}(a_{1}),\subdiv(a_{1}),\cyl\big(\subdiv(a_{1})\big),a_{1},a,i_{0}\big(\subdiv(a_{1})\big),\tau,s(a_{1}),r,j_{1},\cyl\big(s(a_{1})\big),i_{0}\big(\cyl(a_{1})\big),u} %
By definition of the homotopy $l$ with respect to $\cylinder$ constructed in the proof of Lemma \ref{DoldLemma1}, the following diagram in $\cl{A}$ commutes. \tri{\cyl(a_{1}),\subdiv(a_{1}),a_{1},s(a_{1}),r,l} %
Putting the last two observations together, we have that the following diagram in $\cl{A}$ commutes. \sq[{5,3}]{\cyl(a_{1}),\cyl^{2}(a_{1}),a_{1},a,i_{0}\big(\cyl(a_{1})\big),\tau,l,j_{1}} This completes the proof of the claim.

\end{proof} 

\begin{lem} \label{DoldLemma3} Let $\cylinder = \big( \cyl, i_0, i_1, p, v, \subdiv, r_{0}, r_{1}, s, \Gamma_{ul} \big)$ be a cylinder in $\cl{A}$ equipped with a contraction structure $p$, an involution structure $v$ compatible with $p$, a subdivision structure $\big(\subdiv, r_{0}, r_{1}, s \big)$ compatible with $p$, and an upper left connection structure $\Gamma_{ul}$. Suppose that $\cyl$ preserves subdivision with respect to $\cylinder$. 

Let \ar{a,a_{0},j_{0}} and \ar{a,a_{1},j_{1}} be arrows of $\cl{A}$, which are fibrations with respect to $\cylinder$. 

Let \ar{a_{0},a_{1},f} be an arrow of $\cl{A}$, such that the diagram \triother{a_{0},a_{1},a,f,j_{1},j_{0}} in $\cl{A}$ commutes, and such that $f$ is a homotopy equivalence with respect to $\cylinder$. 

Let \ar{a_{1},a_{0},g} and \ar{\cyl(a_{1}),a_{1},l} denote the arrows of $\cl{A}$ constructed in Lemma \ref{DoldLemma1}, such that the diagram \triother{a_{1},a_{0},a,g,j_{0},j_{1}} in $\cl{A}$ commutes, and such that $l$ defines a homotopy from $fg$ to $id(a_{1})$ with respect to $\cylinder$. 

There is an arrow \ar{\cyl^{2}(a_{1}),a,\sigma} of $\cl{A}$ with the following properties. 

\begin{itemize}[topsep=1em,itemsep=1em]

\item[(i)] The following diagram in $\cl{A}$ commutes. \sq[{5,3}]{\cyl(a_{1}),\cyl^{2}(a_{1}),a_{1},a,i_{0}\big(\cyl(a_{1})\big),\sigma,l,j_{1}}

\item[(ii)] Let $h_{1}$, $h_{2}$, and $h_{3}$ denote the right, left, and bottom boundary homotopies of $\sigma$ respectively, so that we may depict $\sigma$ as follows, in the pictorial notation of Remark \ref{DoubleHomotopyPictorialNotationRemark}. \doublehomotopy{l,h_{1},h_{2},h_{3},\sigma} %
Then the following diagrams in $\cl{A}$ commute. \twosq{\cyl(a_{1}),a_{1},\cyl(a),a,h_{1},j_{1},\cyl(j_{1}),p(a),\cyl(a_{1}),a_{1},\cyl(a),a,h_{2},j_{1},\cyl(j_{1}),p(a)} %
\sq{\cyl(a_{1}),a_{1},\cyl(a),a,h_{3},j_{1},\cyl(j_{1}),p(a)}

\end{itemize}

\end{lem} 
 
\begin{proof} Let \ar{\cyl^{2}(a_{1}),a,\tau} denote the double homotopy with respect to $\cylinder$ constructed in the proof of Lemma \ref{DoldLemma2}. In particular, the following diagram in $\cl{A}$ commutes. \sq{\cyl(a_{1}),a_{1},\cyl^{2}(a_{1}),a,l,\tau,i_{0}\big(\cyl(a_{1})\big),j_{1}} %
Since $j_{1}$ is a fibration with respect to $\cylinder$, there is thus an arrow \ar{\cyl^{2}(a_{1}),a_{1},\sigma} of $\cl{A}$, such that the following diagram in $\cl{A}$ commutes. \liftingsquare[{5,3}]{\cyl(a_{1}),a_{1},\cyl^{2}(a_{1}),a,l,\tau,i_{0}\big(\cyl(a_{1})\big),j_{1},\sigma} %
Let $h_{1}$, $h_{2}$, and $h_{3}$ denote the right, left, and bottom boundary homotopies of $\sigma$ respectively, so that we may depict $\sigma$ as follows, in the pictorial notation of Remark \ref{DoubleHomotopyPictorialNotationRemark}. \doublehomotopy{l,h_{1},h_{2},h_{3},\sigma}

Firstly, let us prove that the following diagram in $\cl{A}$ commutes. \sq{\cyl(a_{1}),a_{1},\cyl(a),a,h_{1},j_{1},\cyl(j_{1}),p(a)} %
By definition of $\Gamma_{ul}$ as an upper left connection structure, the following diagram in $\cl{A}$ commutes. \sq[{5,3}]{\cyl(a_{1}),\cyl^{2}(a_{1}),a_{1},\cyl(a_{1}),\cyl\big(i_{1}(a_{1})\big),\Gamma_{ul}(a_{1}),p(a_{1}),i_{1}(a_{1})} %
Let \ar{\cyl\big(\subdiv(a_{1})\big),a,u} denote the arrow of $\cl{A}$ constructed in the proof of Lemma \ref{DoldLemma2}. We have that the following diagram in $\cl{A}$ commutes. \trapeziumstwo[{3,3.5,2,0}]{\cyl(a_{1}),\cyl^{2}(a_{1}),a_{1},\cyl^{2}(a_{1}),\cyl\big(\subdiv(a_{1})\big),\cyl(a_{1}),a,\cyl\big(i_{1}(a_{1})\big),\cyl\big(s(a_{1})\big),u,p(a_{1}),i_{1}(a_{1}),j_{1} \circ h,\cyl\big(i_{1}(a_{1})\big),\cyl\big(r_{0}(a_{1})\big),\Gamma_{ul}(a_{1})} %
We also have that the following diagram in $\cl{A}$ commutes. \tri{a_{1},\cyl(a_{1}),a_{1},i_{1}(a_{1}),h,id} %
Putting the last two observations together, we have that the following diagram in $\cl{A}$ commutes. \sq[{5,3}]{\cyl(a_{1}),\cyl^{2}(a_{1}),a_{1},a,\cyl\big(i_{1}(a_{1})\big),u \circ \cyl\big(s(a_{1})\big),p(a_{1}),j_{1}} %
Moreover, by definition of $\tau$, the following diagram in $\cl{A}$ commutes. \tri[{5,3}]{\cyl^{2}(a_{1}),\cyl\big(\subdiv(a_{1})\big),a,\cyl\big(s(a_{1})\big),u,\tau} %
Thus we have that the following diagram in $\cl{A}$ commutes. \sq[{5,3}]{\cyl(a_{1}),\cyl^{2}(a_{1}),a_{1},a,\cyl\big(i_{1}(a_{1})\big),\tau,p(a_{1}),j_{1}} %
By definition of $h_{1}$, the following diagram in $\cl{A}$ commutes. \tri[{5,3}]{\cyl(a_{1}),\cyl^{2}(a_{1}),a,\cyl\big(i_{1}(a_{1})\big),\sigma,h_{1}} %
Putting the last two observations together, we have that the following diagram in $\cl{A}$ commutes. \squareofthreetrianglesthree[{4,3.5,0,-0.5}]{\cyl(a_{1}),a,\cyl^{2}(a_{1}),a_{1},a,h_{1},j_{1},p(a_{1}),j_{1},\cyl\big(i_{1}(a_{1})\big),\sigma,\tau} %
We also have that the following diagram in $\cl{A}$ commutes. \sq{\cyl(a_{1}),a_{1},\cyl(a),a,p(a_{1}),j_{1},\cyl(j_{1}),p(a)} %
Putting the last two observations together, we have that the following diagram in $\cl{A}$ commutes, as required.  \sq{\cyl(a_{1}),a_{1},\cyl(a),a,h_{1},j_{1},\cyl(j_{1}),p(a)} 

Secondly, let us prove that the following diagram in $\cl{A}$ commutes. \sq{\cyl(a_{1}),a_{1},\cyl(a),a,h_{2},j_{1},\cyl(j_{1}),p(a)} %
Let \ar{\cyl(a_{1}),a_{0},k} denote the arrow of $\cl{A}$ of the proof of Lemma \ref{DoldLemma1}. We have that the following diagram in $\cl{A}$ commutes. \trapeziums[{3,3.5,2,0}]{\cyl(a_{1}),\cyl^{2}(a_{1}),\cyl^{2}(a_{1}),\cyl\big(\subdiv(a_{1})\big),\cyl^{2}(a_{1}),a,\cyl\big(i_{0}(a_{1})\big),\cyl\big(s(a_{1})\big),u,\cyl\big(i_{1}(a_{1})\big),j_{0} \circ k \circ \Gamma_{ul}(a_{1}),\cyl\big(i_{0}(a_{1})\big),\cyl\big(r_{1}(a_{1})\big),\cyl\big(v(a_{1})\big)} %
By definition of $\Gamma_{ul}$ as an upper left connection structurei, the following diagram in $\cl{A}$ commutes. \sq[{5,3}]{\cyl(a_{1}),\cyl^{2}(a_{1}),a_{1},\cyl(a_{1}),\cyl\big(i_{1}(a_{1})\big),\Gamma_{ul}(a_{1}),p(a_{1}),i_{1}(a_{1})} %
Again, we also have that the following diagram in $\cl{A}$ commutes, by definition of $\tau$. \tri[{5,3}]{\cyl^{2}(a_{1}),\cyl\big(\subdiv(a_{1})\big),a,\cyl\big(s(a_{1})\big),u,\tau} %
Putting the last three observations together, we have that the following diagram in $\cl{A}$ commutes. \sq[{5,3}]{\cyl(a_{1}),\cyl^{2}(a_{1}),a_{1},a,\cyl\big(i_{0}(a_{1})\big),\tau,p(a_{1}),j_{0} \circ k \circ i_{1}(a_{1})} %
By definition of $g$, the following diagram in $\cl{A}$ also commutes. \tri{a_{1},\cyl(a_{1}),a_{0},i_{1}(a_{1}),k,g} %
Moreover, we have that the following diagram in $\cl{A}$ commutes. \tri{a_{1},a_{0},a,g,j_{0},j_{1}} %
Putting the last three observations together, we have that the following diagram in $\cl{A}$ commutes. \sq[{5,3}]{\cyl(a_{1}),\cyl^{2}(a_{1}),a_{1},a,\cyl\big(i_{0}(a_{1})\big),\tau,p(a_{1}),j_{1}} %
By definition of $h_{2}$, the following diagram in $\cl{A}$ also commutes. \tri[{5,3}]{\cyl(a_{1}),\cyl^{2}(a_{1}),a,\cyl\big(i_{0}(a_{1})\big),\sigma,h_{2}} %
Putting the last two observations together, we have that the following diagram in $\cl{A}$ commutes. \squareofthreetrianglesthree[{4,3.5,0,-0.5}]{\cyl(a_{1}),a,\cyl^{2}(a_{1}),a_{1},a,h_{2},j_{1},p(a_{1}),j_{1},\cyl\big(i_{0}(a_{1})\big),\sigma,\tau} %
Appealing to the commutativity of the diagram \sq{\cyl(a_{1}),a_{1},\cyl(a),a,p(a_{1}),j_{1},\cyl(j_{1}),p(a)} in $\cl{A}$, we deduce that the following diagram in $\cl{A}$ commutes, as required. \sq{\cyl(a_{1}),a_{1},\cyl(a),a,h_{2},j_{1},\cyl(j_{1}),p(a)} 

Thirdly, let us prove that the following diagram in $\cl{A}$ commutes. \sq{\cyl(a_{1}),a_{1},\cyl(a),a,h_{3},j_{1},\cyl(j_{1}),p(a)} %
The following diagram in $\cl{A}$ commutes. \trapeziumstwo[{3,3.5,2,0}]{\cyl(a_{1}),\subdiv(a_{1}),\cyl(a_{1}),\cyl^{2}(a_{1}),\cyl\big(\subdiv(a_{1})\big),\cyl^{2}(a_{1}),a,r_{1}(a_{1}),i_{1}\big(\subdiv(a_{1})\big),u,v(a_{1}),i_{1}\big(\cyl(a_{1})\big),j_{0} \circ k \circ \Gamma_{ul}(a_{1}),i_{1}\big(\cyl(a_{1})\big),\cyl\big(r_{1}(a_{1})\big),\cyl\big(v(a_{1})\big)} %
By definition of $\Gamma_{ul}$ as an upper left connection structure, the following diagram in $\cl{A}$ commutes. \sq[{5,3}]{\cyl(a_{1}),\cyl^{2}(a_{1}),a_{1},\cyl(a_{1}),i_{1}\big(\cyl(a_{1})\big),\Gamma_{ul}(a_{1}),p(a_{1}),i_{1}(a_{1})} %
Putting the last two observations together, we have that the following diagram in $\cl{A}$ commutes. \sq[{4,3}]{\cyl(a_{1}),\subdiv(a_{1}),a_{1},a,r_{1}(a_{1}),u \circ i_{1}\big(\subdiv(a_{1})\big),p(a_{1}) \circ v(a_{1}),j_{0} \circ k \circ i_{1}(a_{1})} %
Moreover, as earlier in the proof, the following diagram in $\cl{A}$ commutes. \tri{a_{1},\cyl(a_{1}),a,i_{1}(a_{1}),k \circ j_{0},j_{1}} %
In addition, since $v$ is compatible with $p$, the following diagram in $\cl{A}$ commutes. \tri{\cyl(a_{1}),\cyl(a_{1}),a_{1},v(a_{1}),p(a_{1}),p(a_{1})} %
Putting the last three observations together, we have that the following diagram in $\cl{A}$ commutes. \sq[{5,3}]{\cyl(a_{1}),\subdiv(a_{1}),a_{1},a,r_{1}(a_{1}),u \circ i_{1}\big(\subdiv(a_{1})\big),p(a_{1}),j_{1}} %
Let \ar{\cyl(a_{1}),a_{1},h} denote the homotopy from $ff^{-1}$ to $id(a_{1})$ of the proof of Lemma \ref{DoldLemma1}. In particular, the following diagram in $\cl{A}$ commutes. \tri{a_{1},\cyl(a_{1}),a_{1},i_{1}(a_{1}),h,id} %
We also have that the following diagram in $\cl{A}$ commutes. \trapeziumstwo[{3,3.5,2,0}]{\cyl(a_{1}),\subdiv(a_{1}),a_{1},\cyl^{2}(a_{1}),\cyl\big(\subdiv(a_{1})\big),\cyl(a_{1}),a,r_{0}(a_{1}),i_{1}\big(\subdiv(a_{1})\big),u,p(a_{1}),i_{1}(a_{1}),j_{1} \circ h,i_{1}\big(\cyl(a_{1})\big),\cyl\big(r_{0}(a_{1})\big),\Gamma_{ul}(a_{1})} %
Putting the last two observations together, we have that the following diagram in $\cl{A}$ commutes. \sq{\cyl(a_{1}),\subdiv(a_{1}),a_{1},a,r_{0}(a_{1}),u \circ i_{1}\big(\subdiv(a_{1})\big),p(a_{1}),j_{1}} %
Putting everything together, we have now shown that the following diagram in $\cl{A}$ commutes. \pushout[{3,3,1}]{a_{1},\cyl(a_{1}),\cyl(a_{1}),\subdiv(a_{1}),a,i_{0}(a_{1}),r_{0}(a_{1}),i_{1}(a_{1}),r_{1}(a_{1}),j_{1} \circ p(a_{1}),j_{1} \circ p(a_{1}),u \circ i_{1}\big(\subdiv(a_{1})\big)} %
Let \ar{\subdiv,\id_{\cl{A}},\overline{p}} denote the 2-arrow of $\cl{C}$ of Definition \ref{SubdivisionCompatibleWithContractionDefinition}. By definition, the following diagram in $\cl{A}$ commutes.  \pushout{a_{1},\cyl(a_{1}),\cyl(a_{1}),\subdiv(a_{1}),a_{1},i_{0}(a_{1}),r_{0}(a_{1}),i_{1}(a_{1}),r_{1}(a_{1}),p(a_{1}),p(a_{1}),\overline{p}(a_{1})} %
Hence the following diagram in $\cl{A}$ commutes. \pushout{a_{1},\cyl(a_{1}),\cyl(a_{1}),\subdiv(a_{1}),a,i_{0}(a_{1}),r_{0}(a_{1}),i_{1}(a_{1}),r_{1}(a_{1}),j_{1} \circ p(a_{1}),j_{1} \circ p(a_{1}),j_{1} \circ \overline{p}(a_{1})} %
Appealing to the universal property of $\subdiv(a_{1})$, we deduce that the following diagram in $\cl{A}$ commutes. \sq[{5,3}]{\subdiv(a_{1}),\cyl\big(\subdiv(a_{1})\big),a_{1},a,i_{1}\big(\subdiv(a_{1})\big),u,\overline{p}(a_{1}),j_{1}} %
Since the subdivision structure $\big(\subdiv,r_{0},r_{1},s\big)$ is compatible with $p$, we also have that the following diagram in $\cl{A}$ commutes. \tri{\cyl(a_{1}),\subdiv(a_{1}),a_{1},s(a_{1}),\overline{p}(a_{1}),p(a_{1})} %
Putting the last two observations together, we have that the following diagram in $\cl{A}$ commutes. \trapeziums[{3,3.5,2,0}]{\cyl(a_{1}),\subdiv(a_{1}),\subdiv(a_{1}),\cyl\big(\subdiv(a_{1})\big),a_{1},a,i_{1}\big(\cyl(a_{1})\big),\cyl\big(s(a_{1})\big),u,p(a_{1}),j_{1},s(a_{1}),i_{1}\big(\subdiv(a_{1})\big),\overline{p}(a_{1})} %
Once more, we also have that the following diagram in $\cl{A}$ commutes, by definition of $\tau$. \tri[{5,3}]{\cyl^{2}(a_{1}),\cyl\big(\subdiv(a_{1})\big),a,\cyl\big(s(a_{1})\big),u,\tau} %
Putting the last two observations together, we have that the following diagram in $\cl{A}$ commutes. \sq[{5,3}]{\cyl(a_{1}),\subdiv(a_{1}),a_{1},a,i_{1}\big(\cyl(a_{1})\big),\tau,p(a_{1}),j_{1}} %
By definition of $h_{3}$i, the following diagram in $\cl{A}$ also commutes.\tri[{5,3}]{\cyl(a_{1}),\cyl^{2}(a_{1}),a,i_{1}\big(\cyl(a_{1})\big),\sigma,h_{3}} %
Putting the last two observations together, we have that the following diagram in $\cl{A}$ commutes. \squareofthreetrianglesthree[{4,3.5,0,-0.5}]{\cyl(a_{1}),a,\cyl^{2}(a_{1}),a_{1},a,h_{3},j_{1},p(a_{1}),j_{1},i_{0}\big(\cyl(a_{1})\big),\sigma,\tau} %
Appealing to the commutativity of the diagram \sq{\cyl(a_{1}),a_{1},\cyl(a),a,p(a_{1}),j_{1},\cyl(j_{1}),p(a)} in $\cl{A}$, we deduce that the following diagram in $\cl{A}$ commutes, as required. \sq{\cyl(a_{1}),a_{1},\cyl(a),a,h_{3},j_{1},\cyl(j_{1}),p(a)} %

\end{proof}

\begin{lem} \label{DoldRightOverHomotopyInverseLemma} Let $\cylinder = \big( \cyl, i_0, i_1, p, v, \subdiv, r_{0}, r_{1}, s, \Gamma_{ul} \big)$ be a cylinder in $\cl{A}$ equipped with a contraction structure $p$, an involution structure $v$ compatible with $p$, a subdivision structure $\big(\subdiv,r_{0},r_{1},s \big)$ compatible with $p$, and an upper left connection structure $\Gamma_{ul}$. Suppose that $\cyl$ preserves subdivision with respect to $\cylinder$. 

Let \ar{a,a_{0},j_{0}} and \ar{a,a_{1},j_{1}} be arrows of $\cl{A}$ which are fibrations with respect to $\cylinder$. 

Let \ar{a_{0},a_{1},f} be an arrow of $\cl{A}$, such that the diagram \triother{a_{0},a_{1},a,f,j_{1},j_{0}} in $\cl{A}$ commutes, and such that $f$ is a homotopy equivalence with respect to $\cylinder$. 

There is an arrow \ar{a_{1},a_{0},g} of $\cl{A}$, such that the diagram \triother{a_{1},a_{0},a,g,j_{0},j_{1}} in $\cl{A}$ commutes, and such that there is a homotopy over $a$ from $fg$ to $id(a_{1})$ with respect to $\cylinder$ and $(j_{1},j_{1})$.  \end{lem}

\begin{proof} Let \ar{a_{1},a_{0},g} and \ar{\cyl(a_{1}),a_{1},l} denote the arrows of $\cl{A}$ constructed in Lemma \ref{DoldLemma1}. Let \ar{\cyl^{2}(a_{1}),a_{1},\sigma} denote the arrow of $\cl{A}$ constructed in Lemma \ref{DoldLemma3}. In the pictorial notation of Remark \ref{DoubleHomotopyPictorialNotationRemark}, the boundary of $\sigma$ is as follows, where $h_{1}$, $h_{2}$, and $h_{3}$ are all homotopies over $a$ with respect to $\cylinder$ and $(j_{1},j_{1})$. \doublehomotopy{l,h_{1},h_{2},h_{3},\sigma} %
By Proposition \ref{ReverseHomotopyIsOverHomotopyProposition} and Proposition \ref{CompositeHomotopyIsOverHomotopyProposition}, the following diagram in $\cl{A}$ commutes. \sq[{7,3}]{\cyl(a_{1}),a_{1},\cyl(a),a,{(h_{2} + h_{3}) + h_{1}^{-1}},j_{1},\cyl(j_{1}),p(a)} %
The following diagram in $\cl{A}$ also commutes. \sq{a_{1},\cyl(a_{1}),\cyl(a_{1}),a_{1},i_{0}(a_{1}),{(h_{2}) + h_{3}) + h_{1}^{-1}},i_{0}(a_{1}),l} %
Since the diagram \tri{a_{1},\cyl(a_{1}),a_{1},i_{0}(a_{1}),l,gf} in $\cl{A}$ commutes, we deduce that the following diagram in $\cl{A}$ commutes. \tri{a_{1},\cyl(a_{1}),a_{1},i_{0}(a_{1}),{(h_{2} + h_{3}) + h_{1}^{-1}},gf} %
Moreover, the following diagram in $\cl{A}$ commutes. \sq{a_{1},\cyl(a_{1}),\cyl(a_{1}),a_{1},i_{1}(a_{1}),{(h_{2}) + h_{3}) + h_{1}^{-1}},i_{1}(a_{1}),l} %
Since the diagram \tri{a_{1},\cyl(a_{1}),a_{1},i_{1}(a_{1}),l,id} in $\cl{A}$ commutes, we deduce that the following diagram in $\cl{A}$ commutes. \tri{a_{1},\cyl(a_{1}),a_{1},i_{1}(a_{1}),{(h_{2} + h_{3}) + h_{1}^{-1}},id} %
Putting everything together, we have that $(h_{2} + h_{3}) + h_{1}^{-1}$ defines a homotopy over $a$ from $gf$ to $id(a_{1})$ with respect to $\cylinder$ and $(j_{1},j_{1})$. \end{proof}

\begin{prpn} \label{DoldOverHomotopiesCylinderProposition} Let $\cylinder = \big( \cyl, i_0, i_1, p, v, \subdiv, r_{0}, r_{1}, s, \Gamma_{ul} \big)$ be a cylinder in $\cl{A}$ equipped with a contraction structure $p$, an involution structure $v$ compatible with $p$, a subdivision structure $\big( \subdiv, r_{0}, r_{1}, s \big)$ compatible with $p$, and an upper left connection structure $\Gamma_{ul}$. Suppose that $\cyl$ preserves subdivision with respect to $\cylinder$. 

Let \ar{a,a_{0},j_{0}} and \ar{a,a_{1},j_{1}} be arrows of $\cl{A}$ which are fibrations with respect to $\cylinder$. 

Let \ar{a_{0},a_{1},f} of $\cl{A}$ be an arrow of $\cl{A}$, such that the diagram \triother{a_{0},a_{1},a,f,j_{1},j_{0}} in $\cl{A}$ commutes, and such that $f$ is a homotopy equivalence with respect to $\cylinder$. 

Then $f$ is, moreover, a homotopy equivalence over $a$ with respect to $\cylinder$ and $(j_{0},j_{1})$. \end{prpn}

\begin{proof} By Lemma \ref{DoldRightOverHomotopyInverseLemma}, there is an arrow \ar{a_{1},a_{0},g} of $\cl{A}$ such that the diagram \triother{a_{1},a_{0},a,g,j_{0},j_{1}} in $\cl{A}$ commutes, together with a homotopy \ar{\cyl(a_{1}),a_{1},h} over $a$ from $fg$ to $id(a_{1})$ with respect to $\cylinder$ and $(j_{1},j_{1})$. It remains to construct a homotopy over $a$ from $gf$ to $id(a_{0})$ with respect to $\cylinder$ and $(j_{0},j_{0})$. 

By Lemma \ref{RightHomotopyInverseIsHomotopyInverseLemma}, we have that $g$ is a homotopy equivalence with respect to $\cylinder$. Thus $g$ satisfies the hypotheses of Lemma \ref{DoldRightOverHomotopyInverseLemma}. We deduce that there is an arrow \ar{a_{0},a_{1},g'} of $\cl{A}$, such that the diagram \triother{a_{0},a_{1},a,g',j_{1},j_{0}} in $\cl{A}$ commutes, and such that there is a homotopy \ar{\cyl(a_{0}),a_{0},h'} over $a$ from $gg'$ to $id(a_{0})$ with respect to $\cylinder$ and $(j_{0},j_{0})$.

By Corollary \ref{ReverseHomotopyIsOverHomotopyCorollary}, we have that $(h')^{-1}$ defines a homotopy over $a$ from $id(a_{0})$ to $gg'$ with respect to $\cylinder$ and $(j_{0},j_{0})$. Thus, by Lemma \ref{HomotopyPreAndPostCompositionIsOverHomotopyLemma}, the arrow \ar[8]{\cyl(a_{1}),a_{1},{(h')^{-1}} \circ \cyl(gf)} of $\cl{A}$ defines a homotopy over $a$ from $gf$ to $gfgg'$ with respect to $\cylinder$ and $(j_{0},j_{0})$. 

Appealing again to Lemma \ref{HomotopyPreAndPostCompositionIsOverHomotopyLemma}, we also have that the arrow \ar[7]{\cyl(a_{1}),a_{1},g \circ h \circ \cyl(g')} of $\cl{A}$ defines a homotopy over $a$ from $gfgg'$ to $gg'$ with respect to $\cylinder$ and $(j_{0},j_{0})$. 

By virtue of Corollary \ref{CompositeHomotopyIsOverHomotopyCorollary}, we have that the arrow \ar[17]{\cyl(a_{1}),a_{1},\Big( (h')^{-1} \circ \cyl(gf) \Big) + \Big( g \circ h \circ \cyl(g') \Big)} of $\cl{A}$ defines a homotopy over $a$ from $gf$ to $gg'$ with respect to $\cylinder$ and $(j_{0},j_{0})$. %
Let us denote it by $k$ for brevity. Appealing to Corollary \ref{CompositeHomotopyIsOverHomotopyCorollary} once more, we have that the arrow \ar[4]{\cyl(a_{1}),a_{1},k + h'} of $\cl{A}$ defines a homotopy over $a$ from $gf$ to $id(a_{0})$. \end{proof}

\begin{cor} \label{DoldUnderHomotopiesCoCylinderCorollary} Let $\cocylinder = \big( \cocyl, e_0, e_1, c, v, \subdiv, r_{0}, r_{1}, s, \Gamma_{ul} \big)$ be a co-cylinder in $\cl{A}$ equipped with a contraction structure $c$, an involution structure $v$ compatible with $c$, a subdivision structure $\big( \subdiv, r_{0}, r_{1}, s \big)$ compatible with $c$, and an upper left connection structure $\Gamma_{ul}$. Suppose that $\cocyl$ preserves subdivision with respect to $\cocylinder$. 

Let \ar{a,a_{0},j_{0}} and \ar{a,a_{1},j_{1}} be arrows of $\cl{A}$ which are cofibrations with respect to $\cocylinder$. 

Let \ar{a_{0},a_{1},f} of $\cl{A}$ be an arrow of $\cl{A}$, such that the diagram \tri{a,a_{0},a_{1},j_{0},f,j_{1}} in $\cl{A}$ commutes, and such that $f$ is a homotopy equivalence with respect to $\cocylinder$. 

Then $f$ is, moreover, a homotopy equivalence under $a$ with respect to $\cocylinder$ and $(j_{0},j_{1})$. \end{cor}

\begin{proof} Follows immediately from Proposition \ref{DoldOverHomotopiesCylinderProposition} by duality. \end{proof}

\begin{cor} \label{DoldUnderHomotopiesCylinderCorollary} Let $\cylinder = \big( \cyl, i_0, i_1, p \big)$ be a cylinder in $\cl{A}$ equipped with a contraction structure $p$. Let $\cocylinder = \big( \cocyl, e_0, e_1, c, v, \subdiv, r_{0}, r_{1}, s, \Gamma_{ul}  \big)$ be a co-cylinder in $\cl{A}$ equipped with a contraction structure $c$, an involution structure $v$ compatible with $c$, a subdivision structure $\big( \subdiv, r_{0}, r_{1}, s \big)$ compatible with $c$, and an upper left connection structure $\Gamma_{ul}$. Suppose that $\cylinder$ is left adjoint to $\cocylinder$, and that the adjunction between $\cyl$ and $\cocyl$ is compatible with $p$ and $c$. 

Let \ar{a,a_{0},j_{0}} and \ar{a,a_{1},j_{1}} be arrows of $\cl{A}$ which are cofibrations with respect to $\cylinder$. 

Let \ar{a_{0},a_{1},f} of $\cl{A}$ be an arrow of $\cl{A}$, such that the diagram \tri{a,a_{0},a_{1},j_{0},f,j_{1}} in $\cl{A}$ commutes, and such that $f$ is a homotopy equivalence with respect to $\cylinder$. Then $f$ is, moreover, a homotopy equivalence under $a$ with respect to $\cylinder$ and $(j_{0},j_{1})$. \end{cor}

\begin{proof} Since $\cylinder$ is left adjoint to $\cocylinder$, we have that $\cocylinder$ preserves subdivision. By Proposition \ref{CofibrationCoCylinderIffCofibrationCylinderCorollary}, we have that $j_{0}$ and $j_{1}$ are cofibrations with respect to $\cocylinder$. %
By Proposition \ref{HomotopyEquivalenceCylinderIffHomotopyEquivalenceCoCylinderProposition}, we have that $f$ is a homotopy equivalence with respect to $\cocylinder$. 

We deduce, by Corollary \ref{DoldUnderHomotopiesCoCylinderCorollary}, that $f$ is, moreover, a homotopy equivalence under $a$ with respect to $\cocylinder$ and $(j_{0},j_{1})$. Hence, by Corollary \ref{UnderHomotopyCoCylinderGivesUnderHomotopyCylinderCorollary}, $f$ is a homotopy equivalence under $a$ with respect to $\cylinder$ and $(j_{0},j_{1})$. \end{proof}

\begin{rmk} Assuming that $\cylinder$ preserves subdivision, but not necessarily that it is left adjoint to a co-cylinder, it is possible to prove Corollary \ref{DoldUnderHomotopiesCylinderCorollary} directly for a cylinder $\cylinder = \big(\cyl, i_{0}, i_{1}, p, v, \subdiv, r_{0}, r_{1}, s, \Gamma_{ul} \big)$ equipped with a contraction structure $p$, an involution structure $v$ compatible with $p$, a subdivision structure $\big( \subdiv, r_{0}, r_{1}, s \big)$ compatible with $p$, and an upper left connection structure $\Gamma_{ul}$. 

For this, we need that if an arrow $j$ of $\cl{A}$ is a cofibration, then so is $\cyl(j)$. This can be proven, and is the approach taken in the book \cite{MayAConciseCourseInAlgebraicTopology} of May and the book \cite{KampsPorterAbstractHomotopyAndSimpleHomotopyTheory} of Kamps and Porter. 

However, the proof relies upon the assumption that $\cylinder$ admits a {\em transposition} structure, namely a 2-arrow \ar{\cyl^{2},\cyl^{2},t} of $\cl{C}$ such that the following diagrams in $\underline{\mathsf{Hom}}_{\cl{C}}(\cl{A},\cl{A})$ commute. \twotriangles[{4,3,-1}]{\cyl,\cyl^{2},\cyl^{2},i_{0} \cdot \cyl,t,\cyl \cdot i_0,\cyl,\cyl^{2},\cyl^{2},\cyl \cdot i_0, t, i_0 \cdot \cyl} %
We might also require that the two analogous diagrams involving $i_{1}$ commute, but this is not necessary for the proof of Dold's theorem.

For example, suppose that we are working in the 2-category of categories, and that $\cylinder$ arises from an interval $\widehat{I}$ in a braided monoidal category $\cl{A}$. Then the arrow \ar{I^{2},I^{2}} of $\cl{A}$ which defines the braiding gives a transposition structure with respect to $\cylinder$. 

We shall not need a transposition structure anywhere else in this work. Moreover, it will later be indispensable for us to assume that we have an adjoint cylinder and co-cylinder. For these reasons, we have chosen to give a different proof. \end{rmk}

\begin{cor} \label{DoldOverHomotopiesCoCylinderCorollary} Let $\cylinder = \big( \cyl, i_0, i_1, p, v, \subdiv, r_{0}, r_{1}, s, \Gamma_{ul} \big)$ be a cylinder in $\cl{A}$ equipped with a contraction structure $p$, an involution structure $v$ compatible with $p$, a subdivision structure $\big( \subdiv, r_{0}, r_{1}, s, \big)$ compatible with $p$, and an upper left connection structure $\Gamma_{ul}$. Let $\cocylinder = \big( \cocyl, e_0, e_1, c \big)$ be a co-cylinder in $\cl{A}$ equipped with a contraction structure $c$. Suppose that $\cylinder$ is left adjoint to $\cocylinder$, and that the adjunction between $\cyl$ and $\cocyl$ is compatible with $p$ and $c$. 

Let \ar{a,a_{0},j_{0}} and \ar{a,a_{1},j_{1}} be arrows of $\cl{A}$ which are fibrations with respect to $\cocylinder$. 

Let \ar{a_{0},a_{1},f} of $\cl{A}$ be an arrow of $\cl{A}$, such that the diagram \triother{a_{0},a_{1},a,f,j_{1},j_{0}} in $\cl{A}$ commutes, and such that $f$ is a homotopy equivalence with respect to $\cocylinder$. 

Then $f$ is, moreover, a homotopy equivalence over $a$ with respect to $\cocylinder$ and $(j_{0},j_{1})$. \end{cor}

\begin{proof} Follows immediately from Corollary \ref{DoldUnderHomotopiesCylinderCorollary} by duality. \end{proof}

\begin{cor} \label{TrivialCofibrationAdmitsStrongDeformationRetractionCorollary} Let $\cylinder = \big( \cyl, i_0, i_1, p \big)$ be a cylinder in $\cl{A}$ equipped with a contraction structure $p$. Let $\cocylinder = \big( \cocyl, e_0, e_1, c, v, \subdiv, r_{0}, r_{1}, s, \Gamma_{ul}  \big)$ be a co-cylinder in $\cl{A}$ equipped with a contraction structure $c$, an involution structure $v$ compatible with $c$, a subdivision structure $\big( \subdiv, r_{0}, r_{1}, s \big)$ compatible with $p$, and an upper left connection structure $\Gamma_{ul}$. Suppose that $\cylinder$ is left adjoint to $\cocylinder$, and that the adjunction between $\cyl$ and $\cocyl$ is compatible with $p$ and $c$. 

Let \ar{a_{0},a_{1},j} be an arrow of $\cl{A}$ which is a trivial cofibration with respect to $\cylinder$. Then $j$ admits a strong deformation retraction with respect to $\cylinder$. \end{cor}

\begin{proof} By Proposition \ref{IdentityIsCofibrationProposition}, $id(a_{0})$ is a cofibration with respect to $\cylinder$. Moreover, the following diagram in $\cl{A}$ commutes. \tri{a_{0},a_{0},a_{1},id,j,j} %
Thus, by Corollary \ref{DoldUnderHomotopiesCylinderCorollary}, $j$ defines a homotopy equivalence under $a_{0}$ with respect to $\cylinder$ and $\big( id(a_{0}),j \big)$. This means exactly that there is an arrow \ar{a_{1},a_{0},j'} of $\cl{A}$ which is a strong deformation retraction of $j$ with respect to $\cylinder$.  \end{proof}

\begin{cor} \label{TrivialFibrationIsStrongDeformationRetractionCorollary} Let $\cylinder = \big( \cyl, i_0, i_1, p, v, \subdiv, r_0, r_1, s, \Gamma_{ul} \big)$ be a cylinder in $\cl{A}$ equipped with a contraction structure $p$, an involution structure $v$ compatible with $p$, a subdivision structure $\big( \subdiv, r_{0}, r_{1}, s \big)$ compatible with $p$, and an upper left connection structure $\Gamma_{ul}$. Let $\cocylinder = \big( \cocyl, e_0, e_1, c \big)$ be a co-cylinder in $\cl{A}$ equipped with a contraction structure $c$. Suppose that $\cylinder$ is left adjoint to $\cocylinder$, and that the adjunction between $\cyl$ and $\cocyl$ is compatible with $p$ and $c$. 

Let \ar{a_{0},a_{1},f} be an arrow of $\cl{A}$ which is a trivial fibration with respect to $\cocylinder$. Then there is an arrow \ar{a_{0},a_{1},j} of $\cl{A}$ such that $f$ is a strong deformation retraction of $j$ with respect to $\cocylinder$. \end{cor}

\begin{proof} Follows immediately from Corollary \ref{TrivialCofibrationAdmitsStrongDeformationRetractionCorollary} by duality. \end{proof}

\end{chapter}

\begin{chapter}{Lifting axioms} \label{LiftingAxiomsChapter}

Let $\cylinder$ be a cylinder in a formal category $\cl{A}$, and let $\cocylinder$ be a co-cylinder in $\cl{A}$. Suppose that $\cylinder$ is left adjoint to $\cocylinder$. We prove that if an arrow $j$ of $\cl{A}$ is a cofibration, then the canonical arrow $m^{\cylinder}_{j}$ of $\cl{A}$ defined in \ref{MappingCylindersAndMappingCoCylindersChapter} admits a strong deformation retraction. For this we assume, for the first time, that our cylinder is equipped with upper and lower right connection structures, which we require to be compatible with subdivision. 

For topological spaces, the fact that $m^{\cylinder}_{j}$ admits a strong deformation retraction if $j$ is a cofibration is due to Str{\o}m, proven in \S{2} of the paper \cite{StromNoteOnCofibrations}. Str{\o}m's argument is, however, quite different to ours. It relies on the fact that the homotopy theory of topological spaces is defined with respect to a cartesian monoidal structure.

Next, we prove that a normally cloven fibration has the right lifting property with respect to arrows admitting a strong deformation retraction. We deduce that normally cloven fibrations have the covering homotopy extension property with respect to cofibrations. 

In a similar way, we prove that trivial normally cloven fibrations have the right lifting property with respect to cofibrations. Dualising, we deduce that fibrations have the right lifting property with respect to trivial normally cloven cofibrations, and that trivial fibrations have the right lifting property with respect to normally cloven cofibrations. 

\begin{assum} Let $\cl{C}$ be a 2-category with a final object. Suppose that pushouts and pullbacks of 2-arrows of $\cl{C}$ give rise to pushouts and pullbacks in formal categories, in the sense of Definition \ref{PushoutsPullbacks2ArrowsArePushoutsPullbacksInFormalCategoriesTerminology}. Let $\cl{A}$ be an object of $\cl{C}$. As before, we view $\cl{A}$ as a formal category, writing of objects and arrows of $\cl{A}$. \end{assum} 

\begin{prpn} \label{RetractionIsStrongDeformationRetractionProposition} Let $\cylinder = \big( \cyl, i_{0}, i_{1}, p, v, \subdiv, r_{0}, r_{1}, s, \Gamma_{lr}, \Gamma_{ur} \big)$ be a cylinder in $\cl{A}$ equipped with a contraction structure $p$, an involution structure $v$, a subdivision structure $\big( \subdiv, r_{0}, r_{1}, s \big)$ compatible with $p$, a lower right connection structure $\Gamma_{lr}$, and an upper right connection structure $\Gamma_{ur}$. Suppose that $\Gamma_{lr}$ and $\Gamma_{ur}$ are compatible with the subdivision structure $\big( \subdiv, r_{0}, r_{1}, s \big)$, and that $\cyl$ preserves mapping cylinders with respect to $\cylinder$. 

Let \ar{a_{0},a_{1},j} be an arrow of $\cl{A}$ which is a cofibration with respect to $\cylinder$, and suppose that $\big( a^{\cylinder}_{j}, d^{0}_{j}, d^{1}_{j} \big)$ defines a mapping cylinder of $j$ with respect to $\cylinder$. Let \ar{a^{\cylinder}_{j},\cyl(a_{1}),m^{\cylinder}_{j}} denote the corresponding canonical arrow of $\cl{A}$ of Notation \ref{CanonicalArrowFromMappingCylinderToCylinderDefinition}. Let \ar{\cyl(a_{1}),a^{\cylinder}_{j},r^{\cylinder}_{j}} be the arrow of $\cl{A}$ of Proposition \ref{RetractionToMappingCylinderProposition}. %
Then $r^{\cylinder}_{j}$ is a strong deformation retraction of $m^{\cylinder}_{j}$ with respect to $\cylinder$. \end{prpn}

\begin{proof} By Proposition \ref{RetractionToMappingCylinderProposition}, we have that $r^{\cylinder}_{j}$ is a retraction of $m^{\cylinder}_{j}$. It remains to prove that there is a homotopy \ar{\cyl^{2}(a_{1}),\cyl(a_{1}),\sigma} from $m^{\cylinder}_{j} \circ r^{\cylinder}_{j}$ to $id\big( \cyl(a_{1}) \big)$ with respect to $\cylinder$, such that the following diagram in $\cl{A}$ commutes. \sq{\cyl\big(a^{\cylinder}_{j}\big),a^{\cylinder}_{j},\cyl^{2}(a_{1}),\cyl(a_{1}),p\big(a^{\cylinder}_{j}\big),m^{\cylinder}_{j},\cyl\big(m^{\cylinder}_{j}\big),\sigma} 

Let us construct $\sigma$. We have that the following diagram in $\cl{A}$ commutes. \squarewithdiagonal{a_{1},\cyl(a_{1}),\cyl(a_{1}),a^{\cylinder}_{j},i_{0}(a_{1}),r^{\cylinder}_{j},i_{0}(a_{1}),m^{\cylinder}_{j},d^{1}_{j}} %
By definition of $\Gamma_{ur}$ as an upper right connection structure, we also have that the following diagram in $\cl{A}$ commutes. \sq[{5,3}]{\cyl(a_{1}),\cyl^{2}(a_{1}),a_{1},\cyl(a_{1}),i_{1}\big(\cyl(a_{1})\big),\Gamma_{ur}(a_{1}),p(a_{1}),i_{0}(a_{1})} %
Putting the last two observations together, we have that the following diagram in $\cl{A}$ commutes. \sq[{5,3}]{\cyl(a_{1}),\cyl^{2}(a_{1}),a_{1},\cyl(a_{1}),i_{1}\big(\cyl(a_{1})\big),m^{\cylinder}_{j} \circ r^{\cylinder}_{j} \circ \Gamma_{ur}(a_{1}),p(a_{1}),i_{0}(a_{1})} %
By definition of $\Gamma_{lr}$ as a lower right connection structure, the following diagram in $\cl{A}$ commutes. \sq[{5,3}]{\cyl(a_{1}),\cyl^{2}(a_{1}),a_{1},\cyl(a_{1}),i_{0}\big(\cyl(a_{1})\big),\Gamma_{ur}(a_{1}),p(a_{1}),i_{0}(a_{1})} %
Putting the last two observations together, we have that the following diagram in $\cl{A}$ commutes. \sq[{9,3}]{\cyl(a_{1}),\cyl^{2}(a_{1}),\cyl^{2}(a_{1}),\cyl(a_{1}),i_{0}\big(\cyl(a_{1})\big),\Gamma_{lr}(a_{1}),i_{1}\big(\cyl(a_{1})\big),m^{\cylinder}_{j} \circ r^{\cylinder}_{j} \circ \Gamma_{ur}(a_{1})} %
We define $\sigma$ to be the arrow \ar[15]{\cyl^{2}(a_{1}),\cyl(a_{1}),\big( m^{\cylinder}_{j} \circ r^{\cylinder}_{j} \circ \Gamma_{ur}(a_{1}) \big) + \Gamma_{lr}(a_{1})} of $\cl{A}$. 

Let us first prove that the following diagram in $\cl{A}$ commutes. \tri[{5,3}]{\cyl(a_{1}),\cyl^{2}(a_{1}), \cyl(a_{1}),i_{0}\big(\cyl(a_{1})\big),\sigma,m^{\cylinder}_{j} \circ r^{\cylinder}_{j}} %
By definition of $\Gamma_{ur}$ as an upper right connection structure, the following diagram in $\cl{A}$ commutes. \tri[{5,3}]{\cyl(a_{1}),\cyl^{2}(a_{1}),\cyl(a_{1}),i_{0}\big(\cyl(a_{1})\big),\Gamma_{ur}(a_{1}),id} %
Thus the following diagram in $\cl{A}$ commutes. \tri[{5,3}]{\cyl(a_{1}),\cyl^{2}(a_{1}),\cyl(a_{1}),i_{0}\big(\cyl(a_{1})\big),m^{\cylinder}_{j} \circ r^{\cylinder}_{j} \circ \Gamma_{ur}(a_{1}),m^{\cylinder}_{j} \circ r^{\cylinder}_{j}} %
We also have that the following diagram in $\cl{A}$ commutes. \sq[{9,3}]{\cyl(a_{1}),\cyl^{2}(a_{1}),\cyl^{2}(a_{1}),\cyl(a_{1}),i_{0}\big(\cyl(a_{1})\big),\sigma,i_{0}\big(\cyl(a_{1})\big),m^{\cylinder}_{j} \circ r^{\cylinder}_{j} \circ \Gamma_{ur}(a_{1})} %
Putting the last two observations together, we have that the following diagram in $\cl{A}$ commutes, as required. \tri[{5,3}]{\cyl(a_{1}),\cyl^{2}(a_{1}), \cyl(a_{1}),i_{0}\big(\cyl(a_{1})\big),\sigma,m^{\cylinder}_{j} \circ r^{\cylinder}_{j}} 

Next, let us prove that the following diagram in $\cl{A}$ commutes. \tri[{5,3}]{\cyl(a_{1}),\cyl^{2}(a_{1}),\cyl(a_{1}),i_{1}\big(\cyl(a_{1})\big),\sigma,id} %
By definition of $\Gamma_{lr}$ as a lower right connection structure, we have that the following diagram in $\cl{A}$ commutes. \tri[{5,3}]{\cyl(a_{1}),\cyl^{2}(a_{1}),\cyl(a_{1}),i_{1}\big(\cyl(a_{1})\big),\Gamma_{lr}(a_{1}),id} %
We also have that the following diagram in $\cl{A}$ commutes. \sq[{5,3}]{\cyl(a_{1}),\cyl^{2}(a_{1}),\cyl^{2}(a_{1}),\cyl(a_{1}),i_{1}\big(\cyl(a_{1})\big),\sigma,i_{1}\big(\cyl(a_{1})\big),\Gamma_{lr}(a_{1})} %
Putting the last two observations together, we have that the following diagram in $\cl{A}$ commutes, as required. \tri[{5,3}]{\cyl(a_{1}),\cyl^{2}(a_{1}),\cyl(a_{1}),i_{1}\big(\cyl(a_{1})\big),\Gamma_{lr}(a_{1}),id} 

Let us now prove that the following diagram in $\cl{A}$ commutes. \sq{\cyl\big(a^{\cylinder}_{j}\big),a^{\cylinder}_{j},\cyl^{2}(a_{1}),\cyl(a_{1}),p\big(a^{\cylinder}_{j}\big),m^{\cylinder}_{j},\cyl\big(m^{\cylinder}_{j}\big),\sigma} %
Let \ar{\subdiv\big(\cyl(a_{1})\big),\cyl(a_{1}),u} denote the canonical arrow of $\cl{A}$ such that the following diagram in $\cl{A}$ commutes. \pushout[{5,3,0}]{\cyl(a_{1}),\cyl^{2}(a_{1}),\cyl^{2}(a_{1}),\subdiv\big(\cyl(a_{1})\big),\cyl(a_{1}),i_{0}\big(\cyl(a_{1})\big),r_{0}\big(\cyl(a_{1})\big),i_{1}\big(\cyl(a_{1})\big),r_{1}\big(\cyl(a_{1})\big),\Gamma_{lr}(a_{1}),m^{\cylinder}_{j} \circ r^{\cylinder}_{j} \circ \Gamma_{ur}(a_{1}),u} %
Let \ar{\subdiv,\id,\overline{p}} denote the canonical 2-arrow of $\cl{C}$ of Definition \ref{SubdivisionCompatibleWithContractionDefinition}. We claim that the following diagram in $\cl{A}$ commutes. \sq{\cyl(a_{1}),\subdiv(a_{1}),a_{1},\cyl(a_{1}),s(a_{1}),u \circ \subdiv\big(i_{0}(a_{1})\big),p(a_{1}),i_{0}(a_{1})} 

By definition of $\Gamma_{lr}(a_{1})$ as a lower right connection structure, the following diagram in $\cl{A}$ commutes. \sq[{5,3}] {\cyl(a_{1}),\cyl^{2}(a_{1}),a_{1},\cyl(a_{1}),\cyl\big(i_{0}(a_{1})\big),\Gamma_{ur}(a_{1}),p(a_{1}),i_{0}(a_{1})} %
We also have that the following diagram in $\cl{A}$ commutes. \squareabovetriangle[{5,3,0}]{\cyl(a_{1}),\subdiv(a_{1}),\cyl^{2}(a_{1}),\subdiv\big(\cyl(a_{1})\big),\cyl(a_{1}),r_{0}(a_{1}),\subdiv\big(i_{0}(a_{1})\big),\cyl\big(i_{0}(a_{1})\big),r_{0}\big(\cyl(a_{1})\big),u,\Gamma_{lr}(a_{1})} %
Putting the last two observations together, we have that the following diagram in $\cl{A}$ commutes. \sq{\cyl(a_{1}),\subdiv(a_{1}),a_{1},\cyl(a_{1}),r_{0}(a_{1}),u \circ \subdiv\big(i_{0}(a_{1})\big),p(a_{1}),i_{0}(a_{1})}
By definition of $\Gamma_{ur}(a_{1})$ as an upper right connection structure, the following diagram in $\cl{A}$ commutes. \sq[{5,3}] {\cyl(a_{1}),\cyl^{2}(a_{1}),a_{1},\cyl(a_{1}),\cyl\big(i_{0}(a_{1})\big),\Gamma_{ur}(a_{1}),p(a_{1}),i_{0}(a_{1})} %
Thus we have that the following diagram in $\cl{A}$ commutes. \trapeziumstwo[{3,3,2,0}]{\cyl(a_{1}),\subdiv(a_{1}),a_{1},\cyl^{2}(a_{1}),\subdiv\big(\cyl(a_{1})\big),\cyl(a_{1}),\cyl(a_{1}),r_{1}(a_{1}),\subdiv\big(i_{0}(a_{1})\big),u,p(a_{1}),i_{0}(a_{1}),m^{\cylinder}_{j} \circ r^{\cylinder}_{j},\cyl\big(i_{0}(a_{1})\big),r_{1}\big(\cyl(a_{1})\big),\Gamma_{ur}(a_{1})} %
Earlier in the proof, we observed that the following diagram in $\cl{A}$ commutes. \tri{a_{1},\cyl(a_{1}),\cyl(a_{1}),i_{0}(a_{1}),m^{\cylinder}_{j} \circ r^{\cylinder}_{j},i_{0}(a_{1})} %
Putting the last two observations together, we have that the following diagram in $\cl{A}$ commutes. \sq{\cyl(a_{1}),\subdiv(a_{1}),a_{1},\cyl(a_{1}),r_{1}(a_{1}),u \circ \subdiv\big(i_{0}(a_{1})\big),p(a_{1}),i_{0}(a_{1})} %
We have now shown that the following diagram in $\cl{A}$ commutes. \pushout{a_{1},\cyl(a_{1}),\cyl(a_{1}),\subdiv(a_{1}),\cyl(a_{1}),i_{0}(a_{1}),r_{0}(a_{1}),i_{1}(a_{1}),r_{1}(a_{1}),i_{0}(a_{1}) \circ p(a_{1}),i_{0}(a_{1}) \circ p(a_{1}), u \circ \subdiv\big(i_{0}(a_{1})\big)} %
By definition of $\overline{p}$, the following diagram in $\cl{A}$ commutes. \pushout{a_{1},\cyl(a_{1}),\cyl(a_{1}),\subdiv(a_{1}),a_{1},i_{0}(a_{1}),r_{0}(a_{1}),i_{1}(a_{1}),r_{1}(a_{1}),p(a_{1}), p(a_{1}), \overline{p}(a_{1})} %
Thus the following diagram in $\cl{A}$ commutes. \pushout{a_{1},\cyl(a_{1}),\cyl(a_{1}),\subdiv(a_{1}),\cyl(a_{1}),i_{0}(a_{1}),r_{0}(a_{1}),i_{1}(a_{1}),r_{1}(a_{1}),i_{0}(a_{1}) \circ p(a_{1}),i_{0}(a_{1}) \circ p(a_{1}), i_{0}(a_{1}) \circ \overline{p}(a_{1})} %
Appealing to the universal property of $\subdiv(a_{1})$, we deduce that the following diagram in $\cl{A}$ commutes.\sq[{5,3}]{\subdiv(a_{1}),\subdiv\big(\cyl(a_{1})\big),a_{1},\cyl(a_{1}),\subdiv\big(i_{0}(a_{1})\big),u,\overline{p}(a_{1}),i_{0}(a_{1})} %
Since the subdivision structure $\big( \subdiv, r_{0}, r_{1}, s \big)$ is compatible with $p$, we also have that the following diagram in $\cl{A}$ commutes. \tri{\cyl(a_{1}),\subdiv(a_{1}),a_{1},s(a_{1}),\overline{p}(a_{1}),p(a_{1})} %
Putting the last two observations together, we have that the following diagram in $\cl{A}$ commutes, completing the proof of the claim. \squarewithdiagonaltwo[{4,3.5}]{\cyl(a_{1}),\subdiv(a_{1}),a_{1},\cyl(a_{1}),s(a_{1}),u \circ \subdiv\big(i_{0}(a_{1})\big),p(a_{1}),i_{0}(a_{1}),\overline{p}(a_{1})} 

We also claim that the following diagram in $\cl{A}$ commutes. \sq[{5,3}]{\cyl^{2}(a_{0}),\subdiv\big(\cyl(a_{0})\big),\cyl(a_{0}),\cyl(a_{1}),s\big(\cyl(a_{0})\big),u \circ \subdiv\big(\cyl(j)\big),p\big(\cyl(a_{0})\big),\cyl(j)} %
We have that the following diagram in $\cl{A}$ commutes. \trapeziumstwo[{3,3,2,0}]{\cyl^{2}(a_{0}),\subdiv\big(\cyl(a_{0})\big),\cyl(a_{0}),\cyl^{2}(a_{1}),\subdiv\big(\cyl(a_{1})\big),\cyl(a_{1}),\cyl(a_{1}),r_{1}\big(\cyl(a_{0})\big),\subdiv\big(\cyl(j)\big),u,\Gamma_{ur}(a_{0}),\cyl(j),m^{\cylinder}_{j} \circ r^{\cylinder}_{j},\cyl^{2}(j),r_{1}\big(\cyl(a_{1})\big),\Gamma_{ur}(a_{1})} %
The following diagram in $\cl{A}$ also commutes. \squarewithdiagonal{\cyl(a_{0}),\cyl(a_{1}),\cyl(a_{1}),a^{\cylinder}_{j},\cyl(j),r^{\cylinder}_{j},\cyl(j),m^{\cylinder}_{j},d^{0}_{j}} %
Putting the last two observations together, we have that the following diagram in $\cl{A}$ commutes. \sq[{5,3}]{\cyl^{2}(a_{0}),\subdiv\big(\cyl(a_{0})\big),\cyl(a_{0}),\cyl(a_{1}),r_{1}\big(\cyl(a_{0})\big),u \circ \subdiv\big(\cyl(j)\big),\Gamma_{ur}(a_{0}),\cyl(j)} %
The following diagram in $\cl{A}$ commutes. \squareabovetriangle[{5,3,0}]{\cyl^{2}(a_{0}),\subdiv\big(\cyl(a_{0})\big),\cyl^{2}(a_{1}),\subdiv\big(\cyl(a_{1})\big),\cyl(a_{1}),r_{0}\big(\cyl(a_{0})\big),\subdiv\big(\cyl(j)\big),\cyl^{2}(j),r_{0}\big(\cyl(a_{1})\big),u,\Gamma_{lr}(a_{1})} %
The following diagram in $\cl{A}$ also commutes. \sq{\cyl^{2}(a_{0}),\cyl^{2}(a_{1}),\cyl(a_{0}),\cyl(a_{1}),\cyl^{2}(j),\Gamma_{lr}(a_{1}),\Gamma_{lr}(a_{0}),\cyl(j)} %
Putting the last two observations together, we have that the following diagram in $\cl{A}$ commutes. \sq[{5,3}]{\cyl^{2}(a_{0}),\subdiv\big(\cyl(a_{0})\big),\cyl(a_{0}),\cyl(a_{1}),r_{0}\big(\cyl(a_{0})\big),u \circ \subdiv\big(\cyl(j)\big),\Gamma_{lr}(a_{0}),\cyl(j)} %
We have now shown that the following diagram in $\cl{A}$ commutes. \pushout[{5,3,0}]{\cyl(a_{0}),\cyl^{2}(a_{0}),\cyl^{2}(a_{0}),\subdiv\big(\cyl(a_{0})\big),\cyl(a_{1}),i_{0}\big(\cyl(a_{0})\big),r_{0}\big(\cyl(a_{0})\big),i_{1}\big(\cyl(a_{0})\big),r_{1}\big(\cyl(a_{0})\big),\cyl(j) \circ \Gamma_{lr}(a_{0}),\cyl(j) \circ \Gamma_{ur}(a_{0}), u \circ \subdiv\big(\cyl(j)\big)} %
Let \ar{\subdiv \circ \cyl,\cyl,x} denote the canonical 2-arrow of $\cl{C}$ of Definition \ref{RightConnectionsCylinderCompatibilityDefinition}. By definition of $x$, the following diagram in $\cl{A}$ commutes. \pushout[{5,3,0}]{\cyl(a_{0}),\cyl^{2}(a_{0}),\cyl^{2}(a_{0}),\subdiv\big(\cyl(a_{0})\big),\cyl(a_{0}),i_{0}\big(\cyl(a_{0})\big),r_{0}\big(\cyl(a_{0})\big),i_{1}\big(\cyl(a_{0})\big),r_{1}\big(\cyl(a_{0})\big),\Gamma_{lr}(a_{0}),\Gamma_{ur}(a_{0}),x(a_{0})} %
Thus the following diagram in $\cl{A}$ commutes. \pushout[{5,3,0}]{\cyl(a_{0}),\cyl^{2}(a_{0}),\cyl^{2}(a_{0}),\subdiv\big(\cyl(a_{0})\big),\cyl(a_{1}),i_{0}\big(\cyl(a_{0})\big),r_{0}\big(\cyl(a_{0})\big),i_{1}\big(\cyl(a_{0})\big),r_{1}\big(\cyl(a_{0})\big),\cyl(j) \circ \Gamma_{lr}(a_{0}),\cyl(j) \circ \Gamma_{ur}(a_{0}),\cyl(j) \circ x(a_{0})} %
Appealing to the universal property of $\subdiv\big(\cyl(a_{0})\big)$, we deduce that the following diagram in $\cl{A}$ commutes. \sq[{4,3}]{\subdiv\big(\cyl(a_{0})\big),\subdiv\big(\cyl(a_{1})\big),\cyl(a_{0}),\cyl(a_{1}),\subdiv\big(\cyl(j)\big),u,x(a_{0}),\cyl(j)} %
Since $\Gamma_{lr}$ and $\Gamma_{ur}$ are compatible with the subdivision structure $\big( \subdiv,r_{0}, r_{1}, s \big)$, we also have that the following diagram in $\cl{A}$ commutes. \tri[{5,4}]{\cyl^{2}(a_{0}),\subdiv\big(\cyl(a_{0})\big),\cyl(a_{0}),s\big(\cyl(a_{0})\big),x(a_{0}),p\big(\cyl(a_{0})\big)} %
Putting the last two observations together, we have that the following diagram in $\cl{A}$ commutes, completing the proof of the claim. \squarewithdiagonaltwo[{5,3}]{\cyl^{2}(a_{0}),\subdiv\big(\cyl(a_{0})\big),\cyl(a_{0}),\cyl(a_{1}),s\big(\cyl(a_{0})\big),u \circ \subdiv(\cyl(j)\big),p\big(\cyl(a_{0})\big), \cyl(j), x(a_{0})} 

Next, we claim that the following diagram in $\cl{A}$ commutes. \sq[{4,3}]{\cyl^{2}(a_{0}),\cyl\big(a^{\cylinder}_{j}\big),\cyl(a_{0}),\cyl(a_{1}),\cyl(d^{0}_{j}),\sigma \circ \cyl\big(m^{\cylinder}_{j}\big),p\big(\cyl(a_{0})\big),\cyl(j)} %
We have that the following diagram in $\cl{A}$ commutes. \tri{\cyl(a_{0}),a^{\cylinder}_{j},\cyl(a_{1}),d^{0}_{j},m^{\cylinder}_{j},\cyl(j)} %
Thus the following diagram in $\cl{A}$ commutes. \tri[{4,3}]{\cyl^{2}(a_{0}),\cyl\big(a^{\cylinder}_{j}\big),\cyl^{2}(a_{1}),\cyl(d^{0}_{j}),\cyl\big(m^{\cylinder}_{j}\big),\cyl^{2}(j)} %
The following diagram in $\cl{A}$ also commutes. \sq[{5,3}]{\cyl^{2}(a_{0}),\cyl^{2}(a_{1}),\subdiv\big(\cyl(a_{0})\big),\subdiv\big(\cyl(a_{1})\big),\cyl^{2}(j),s\big(\cyl(a_{1})\big),s\big(\cyl(a_{0})\big),\subdiv\big(\cyl(j)\big)} %
Putting the last two observations together, we have that the following diagram in $\cl{A}$ commutes. \sq[{5,3}]{\cyl^{2}(a_{0}),\cyl\big(a^{\cylinder}_{j}\big),\subdiv\big(\cyl(a_{0})\big),\subdiv\big(\cyl(a_{1})\big),\cyl(d^{0}_{j}),s\big(\cyl(a_{1})\big) \circ \cyl\big(m^{\cylinder}_{j}\big),s\big(\cyl(a_{0})\big),\subdiv\big(\cyl(j)\big)} %
Earlier in the proof, we established that the following diagram in $\cl{A}$ commutes. \sq[{5,3}]{\cyl^{2}(a_{0}),\subdiv\big(\cyl(a_{0})\big),\cyl(a_{0}),\cyl(a_{1}),s\big(\cyl(a_{0})\big),u \circ \subdiv\big(\cyl(j)\big),p\big(\cyl(a_{0})\big),\cyl(j)} %
In addition, by definition of $\sigma$, the following diagram in $\cl{A}$ commutes. \tri[{5,3}]{\cyl^{2}(a_{1}),\subdiv\big(\cyl(a_{1})\big),\cyl(a_{1}),s\big(\cyl(a_{1})\big),u,\sigma} %
Putting the last three observations together, we have that the following diagram in $\cl{A}$ commutes, completing our proof of the claim. \sq[{4,3}]{\cyl^{2}(a_{0}),\cyl\big(a^{\cylinder}_{j}\big),\cyl(a_{0}),\cyl(a_{1}),\cyl(d^{0}_{j}),\sigma \circ \cyl\big(m^{\cylinder}_{j}\big),p\big(\cyl(a_{0})\big),\cyl(j)} %
Next, we claim that the following diagram in $\cl{A}$ commutes. \sq[{4,3}]{\cyl(a_{1}),\cyl\big(a^{\cylinder}_{j}\big),a_{1},\cyl(a_{1}),\cyl(d^{1}_{j}),\sigma \circ \cyl\big(m^{\cylinder}_{j}\big),p(a_{1}),i_{0}(a_{1})} %
We have that the following diagram in $\cl{A}$ commutes. \tri{a_{1},a^{\cylinder}_{j},\cyl(a_{1}),d^{1}_{j},m^{\cylinder}_{j},i_{0}(a_{1})} %
Thus the following diagram in $\cl{A}$ commutes. \tri[{4,3}]{\cyl(a_{1}),\cyl\big(a^{\cylinder}_{j}\big),\cyl^{2}(a_{1}),\cyl(d^{1}_{j}),\cyl\big(m^{\cylinder}_{j}\big),\cyl\big(i_{0}(a_{1})\big)} %
The following diagram in $\cl{A}$ also commutes. \sq[{5,3}]{\cyl(a_{1}),\cyl^{2}(a_{1}),\subdiv(a_{1}),\subdiv\big(\cyl(a_{1})\big),\cyl\big(i_{0}(a_{1})\big),s\big(\cyl(a_{1})\big),s(a_{1}),\subdiv\big(i_{0}(a_{1})\big)} %
Putting the last two observations together, we have that the following diagram in $\cl{A}$ commutes. \sq[{4,3}]{\cyl(a_{1}),\cyl\big(a^{\cylinder}_{j}\big),a_{1},\cyl(a_{1}),\cyl(d^{1}_{j}),s\big(\cyl(a_{1})\big) \circ \cyl\big(m^{\cylinder}_{j}\big),s(a_{1}),\subdiv\big(i_{0}(a_{1})\big)} %
Earlier in the proof, we established that the following diagram in $\cl{A}$ commutes. \sq{\cyl(a_{1}),\subdiv(a_{1}),a_{1},\cyl(a_{1}),s(a_{1}),u \circ \subdiv\big(i_{0}(a_{1})\big),p(a_{1}),i_{0}(a_{1})} %
In addition, by definition of $\sigma$, the following diagram in $\cl{A}$ commutes. \tri[{5,3}]{\cyl^{2}(a_{1}),\subdiv\big(\cyl(a_{1})\big),\cyl(a_{1}),s\big(\cyl(a_{1})\big),u,\sigma} %
Putting the last three observations together, we have that the following diagram in $\cl{A}$ commutes, completing the proof of the claim. \sq[{4,3}]{\cyl(a_{1}),\cyl\big(a^{\cylinder}_{j}\big),a_{1},\cyl(a_{1}),\cyl(d^{1}_{j}),\sigma \circ \cyl\big(m^{\cylinder}_{j}\big),p(a_{1}),i_{0}(a_{1})} %
We have now shown that the following diagram in $\cl{A}$ commutes. \pushout[{5,3,-2}]{\cyl(a_{0}),\cyl^{2}(a_{0}),\cyl(a_{1}),\cyl(a^{\cylinder}_{j}),\cyl(a_{1}),\cyl\big(i_{0}(a_{0})\big),\cyl(d^{0}_{j}),\cyl(j),\cyl(d^{1}_{j}),\cyl(j) \circ p\big(\cyl(a_{0})\big),i_{0}(a_{1}) \circ p(a_{1}), \sigma \circ \cyl\big(m^{\cylinder}_{j}\big)} %
The following diagram in $\cl{A}$ commutes. \squareabovetriangle{\cyl^{2}(a_{0}),\cyl\big(a^{\cylinder}_{j}\big),\cyl(a_{0}),a^{\cylinder}_{j},\cyl(a_{1}),\cyl(d^{0}_{j}),p\big(a^{\cylinder}_{j}\big),p\big(\cyl(a_{0})\big),d^{0}_{j},m^{\cylinder}_{j},\cyl(j)} %
The following diagram in $\cl{A}$ also commutes. \squareabovetriangle{\cyl(a_{1}),\cyl\big(a^{\cylinder}_{j}\big),a_{1},a^{\cylinder}_{j},\cyl(a_{1}),\cyl(d^{1}_{j}),p\big(a^{\cylinder}_{j}\big),p(a_{1}),d^{1}_{j},m^{\cylinder}_{j},i_{0}(a_{1})} %
Putting the last two observations together, we have that the following diagram in $\cl{A}$ commutes. \pushout[{5,3,-2}]{\cyl(a_{0}),\cyl^{2}(a_{0}),\cyl(a_{1}),\cyl(a^{\cylinder}_{j}),\cyl(a_{1}),\cyl\big(i_{0}(a_{0})\big),\cyl(d^{0}_{j}),\cyl(j),\cyl(d^{1}_{j}),\cyl(j) \circ p\big(\cyl(a_{0})\big),i_{0}(a_{1}) \circ p(a_{1}), m^{\cylinder}_{j} \circ p(a^{\cylinder}_{j})} %
Since $\cyl$ preserves mapping cylinders with respect to $\cylinder$, the following diagram in $\cl{A}$ is co-cartesian. \sq[{5,3}]{\cyl(a_{0}),\cyl^{2}(a_{0}),\cyl(a_{1}),\cyl(a^{\cylinder}_{j}),\cyl\big(i_{0}(a_{0})\big),\cyl(d^{0}_{j}),\cyl(j),\cyl(d^{1}_{j})} %
Appealing to the universal property of $\cyl\big(a^{\cylinder}_{j}\big)$, we deduce that the following diagram in $\cl{A}$ commutes, as required. \sq{\cyl(a^{\cylinder}_{j}),a^{\cylinder}_{j},\cyl^{2}(a_{1}),\cyl(a_{1}),p(a^{\cylinder}_{j}),m^{\cylinder}_{j},\cyl(m^{\cylinder}_{j}),\sigma} 

\end{proof}

\begin{prpn} \label{NormallyClovenFibrationsHaveRLPWithRespectToArrowsAdmittingAStrongDeformationRetractionProposition} Let $\cylinder = \big( \cyl, i_{0}, i_{1}, p \big)$ be a cylinder in $\cl{A}$ equipped with a contraction structure $p$. Let \ar{a_{0},a_{1},j} be an arrow of $\cl{A}$ which admits a strong deformation retraction \ar{a_{1},a_{0},r} with respect to $\cylinder$, and let \ar{a_{2},a_{3},f} be an arrow of $\cl{A}$ which is a normally cloven fibration with respect to $\cylinder$. 

For any arrows \ar{a_{0},a_{2},g_{0}} and \ar{a_{1},a_{3},g_{1}} of $\cl{A}$ such that the diagram \sq{a_{0},a_{2},a_{1},a_{3},g_{0},f,j,g_{1}} in $\cl{A}$ commutes, there is an arrow \ar{a_{1},a_{2},l} of $\cl{A}$ such that the following diagram in $\cl{A}$ commutes. \liftingsquare{a_{0},a_{2},a_{1},a_{3},g_{0},f,j,g_{1},l}  \end{prpn}

\begin{proof} Since $r$ is a strong deformation retraction of $j$ with respect to $\cylinder$, there is a homotopy \ar{\cyl(a_{1}),a_{1},h} under $a_{0}$ from $jr$ to $id(a_{1})$ with respect to $\cylinder$. In particular, the following diagram in $\cl{A}$ commutes. \tri{a_{1},\cyl(a_{1}),a_{1},i_{0}(a_{1}),h, j \circ r} %
By assumption, we have that the following diagram in $\cl{A}$ commutes. \sq{a_{0},a_{2},a_{1},a_{3},g_{0},f,j,g_{1}} %
Together, the commutativity of these two diagrams implies that the following diagram in $\cl{A}$ commutes. \sq{a_{1},a_{2},\cyl(a_{1}),a_{3},g_{0} \circ r,f, i_{0}(a_{1}), g_{1} \circ h} %
For any object $a$ of $\cl{A}$, let \ar{\Delta^{\cylinder}_{f,a},\Omega^{\cylinder}_{f,a},k_{a}} denote the map of the cleavage with which $f$ is equipped. %
Let \ar{\cyl(a_{1}),a_{2},k} denote the arrow $k_{a_{1}}(g_{0} \circ r, g_{1} \circ h)$ of $\cl{A}$. We have that the following diagram in $\cl{A}$ commutes. \liftingsquare{a_{1},a_{2},\cyl(a_{1}),a_{3},g_{0}\circ r,f,i_{0}(a_{1}),g_{1} \circ h,k} %
Let \ar{a_{1},a_{2},l} denote the arrow $k \circ i_{1}(a_{1})$ of $\cl{A}$. We claim that the following diagram in $\cl{A}$ commutes. \liftingsquare{a_{0},a_{2},a_{1},a_{3},g_{0},f,j,g_{1},l} 

Firstly, we have that the following diagram in $\cl{A}$ commutes. \squareofthreetrianglestwo{a_{1},a_{2},\cyl(a_{1}),a_{1},a_{3},l,f,id,g_{1},i_{1}(a_{1}),k,h} Thus the triangle \tri{a_{1},a_{2},a_{3},l,f,g_{1}} in $\cl{A}$ commutes. It remains to prove the commutativity of the triangle \tri{a_{0},a_{1},a_{2},j,l,g_{0}} in $\cl{A}$. 

Let \ar{\cyl(a_{0}),a_{2},k'} denote the arrow $k_{a_{0}}\big(g_{0} \circ r \circ j, g_{1} \circ h \circ \cyl(j) \big)$ of $\cl{A}$. Since $f$ is a normally cloven fibration with respect to $\cylinder$, its cleavage satisfies property (ii) of Definition \ref{NormallyClovenFibrationCylinderDefinition}. Thus the following diagram in $\cl{A}$ commutes. \tri{\cyl(a_{0}),\cyl(a_{1}),a_{3},\cyl(j),k,k'} %
By definition of $k'$, we have that the following diagram in $\cl{A}$ commutes. \liftingsquare[{7,3}]{a_{0},a_{2},\cyl(a_{0}),a_{3},g_{0} \circ r \circ j,f,i_{0}(a_{1}),g_{1} \circ h \circ \cyl(j),k'} %
By the commutativity of the diagram \tri{a_{0},a_{1},a_{0},j,r,id} in $\cl{A}$, we thus have that the following diagram in $\cl{A}$ commutes. \liftingsquare[{7,3}]{a_{0},a_{2},\cyl(a_{0}),a_{3},g_{0}, f, i_{0}(a_{0}), g_{1} \circ h \circ \cyl(j),k'} %
The following diagram in $\cl{A}$ also commutes, by definition of $h$. \sq{\cyl(a_{0}),a_{0},\cyl(a_{1}),a_{1},p(a_{0}),j,\cyl(j),h} %
Hence the following diagram in $\cl{A}$ commutes. \tri{\cyl(a_{0}),a_{0},a_{3},p(a_{0}),g_{1} \circ j,g_{1} \circ h \circ \cyl(j)} %
We now have that the following diagram in $\cl{A}$ commutes. \liftingsquare[{7,3}]{a_{0},a_{2},\cyl(a_{0}),a_{3},g_{0}, f, i_{0}(a_{0}), g_{1} \circ j \circ p(a_{0}),k'} %
Since $f$ is a normally cloven fibration with respect to $\cylinder$, its cleavage satisfies property (i) of Definition \ref{NormallyClovenFibrationCylinderDefinition}. Thus the following diagram in $\cl{A}$ commutes. \tri{\cyl(a_{0}),a_{0},a_{2},p(a_{0}),g_{0},k'} %
Putting everything together, we have that the following diagram in $\cl{A}$ commutes. \squarewithdiagonalthree{\cyl(a_{0}),\cyl(a_{1}),a_{0},a_{2},\cyl(j),k,p(a_{0}),g_{0},k'} %
Thus the following diagram in $\cl{A}$ commutes, as required. \trapeziumsfive[{3,3.5,0,0,0}]{a_{0},a_{1},\cyl(a_{0}),\cyl(a_{1}),a_{0},a_{2},j,l,id,g_{0},i_{1}(a_{0}),i_{1}(a_{1}),\cyl(j),p(a_{0}),k} 
\end{proof} 

\begin{cor} \label{NormallyClovenFibrationRLPTrivialCofibrationCorollary} Let $\cylinder = \big( \cyl, i_{0}, i_{1}, p, v, \subdiv, r_{0}, r_{1}, s, \Gamma_{lr}, \Gamma_{ur} \big)$ be a cylinder in $\cl{A}$ equipped with a contraction structure $p$, an involution structure $v$, a subdivision structure $\big( \subdiv, r_{0}, r_{1}, s \big)$ compatible with $p$, a lower right connection structure $\Gamma_{lr}$, and an upper right connection structure $\Gamma_{ur}$. Suppose that $\Gamma_{lr}$ and $\Gamma_{ur}$ are compatible with the subdivision structure $\big( \subdiv, r_{0}, r_{1}, s \big)$, and that $\cyl$ preserves mapping cylinders with respect to $\cylinder$. 

Let $\cocylinder = \big( \cocyl, e_0, e_1, c, v', \subdiv', r_{0}', r_{1}', s', \Gamma_{ul}'  \big)$ be a co-cylinder in $\cl{A}$ equipped with a contraction structure $c$, an involution structure $v'$ compatible with $c$, a subdivision structure $\big( \subdiv', r_{0}', r_{1}', s' \big)$ compatible with $c$, and an upper left connection structure $\Gamma_{ul}'$. Suppose that $\cylinder$ is left adjoint to $\cocylinder$, and that the adjunction between $\cyl$ and $\cocyl$ is compatible with $p$ and $c$. 

Let \ar{a_{0},a_{1},j} be an arrow of $\cl{A}$ which is a trivial cofibration with respect to $\cylinder$, and let \ar{a_{2},a_{3},f} be an arrow of $\cl{A}$ which is a normally cloven fibration with respect to $\cocylinder$. 

For any commutative diagram \sq{a_{0},a_{1},a_{2},a_{3},g_{0},f,j,g_{1}} in $\cl{A}$, there is an arrow \ar{a_{1},a_{2},l} of $\cl{A}$ such that the following diagram in $\cl{A}$ commutes. \liftingsquare{a_{0},a_{1},a_{2},a_{3},g_{0},f,j,g_{1},l} \end{cor}

\begin{proof} By Corollary \ref{TrivialCofibrationAdmitsStrongDeformationRetractionCorollary}, we have that $j$ admits a strong deformation retraction with respect to $\cylinder$. By Proposition \ref{NormallyClovenFibrationCoCylinderIffNormallyClovenFibrationCylinderProposition}, we have that $f$ is a normally cloven fibration with respect to $\cylinder$. Thus we may appeal to Proposition \ref{NormallyClovenFibrationsHaveRLPWithRespectToArrowsAdmittingAStrongDeformationRetractionProposition} for a construction of $l$. \end{proof} 

\begin{cor} \label{TrivialFibrationRLPNormallyClovenCofibrationCorollary} Let $\cylinder = \big( \cyl, i_{0}, i_{1}, p, v, \subdiv, r_{0}, r_{1}, s, \Gamma_{ul} \big)$ be a cylinder in $\cl{A}$ equipped with a contraction structure $p$, an involution structure $v$ compatible with $p$, a subdivision structure $\big( \subdiv, r_{0}, r_{1}, s \big)$ compatible with $p$, and an upper left connection structure $\Gamma_{ul}$. 

Let $\cocylinder = \big( \cocyl, e_0, e_1, c, v', \subdiv', r_{0}', r_{1}', s', \Gamma_{lr}', \Gamma_{ur}' \big)$ be a co-cylinder in $\cl{A}$ equipped with a contraction structure $c$, an involution structure $v'$ compatible with $c$, a subdivision structure $\big( \subdiv', r_{0}', r_{1}', s' \big)$ compatible with $c$, a lower right connection structure $\Gamma_{lr}'$, and an upper right connection structure $\Gamma_{ur}'$. 

Suppose that $\Gamma_{lr}'$ and $\Gamma_{ur}'$ are compatible with the subdivision structure $\big( \subdiv', r_{0}', r_{1}', s' \big)$, and that $\cocyl$ preserves mapping co-cylinders with respect to $\cocylinder$. %
Suppose that $\cylinder$ is left adjoint to $\cocylinder$, and that the adjunction between $\cyl$ and $\cocyl$ is compatible with $p$ and $c$. 

Let \ar{a_{0},a_{1},j} be an arrow of $\cl{A}$ which is a normally cloven cofibration with respect to $\cylinder$, and let \ar{a_{2},a_{3},f} be an arrow of $\cl{A}$ which is a trivial fibration with respect to $\cocylinder$. 

For any commutative diagram \sq{a_{0},a_{1},a_{2},a_{3},g_{0},f,j,g_{1}} in $\cl{A}$, there is an arrow \ar{a_{1},a_{2},l} of $\cl{A}$ such that the following diagram in $\cl{A}$ commutes. \liftingsquare{a_{0},a_{1},a_{2},a_{3},g_{0},f,j,g_{1},l} \end{cor}

\begin{proof} Follows immediately from Corollary \ref{NormallyClovenFibrationRLPTrivialCofibrationCorollary} by duality. \end{proof} 

\begin{cor} \label{CHEPHoldsCorollary} Let $\cylinder = \big( \cyl, i_{0}, i_{1}, p, v, \subdiv, r_{0}, r_{1}, s, \Gamma_{lr}, \Gamma_{ur} \big)$ be a cylinder in $\cl{A}$ equipped with a contraction structure $p$, an involution structure $v$, a subdivision structure $\big( \subdiv, r_{0}, r_{1}, s \big)$ compatible with $p$, a lower right connection structure $\Gamma_{lr}$, and an upper right connection structure $\Gamma_{ur}$. %
Suppose that $\Gamma_{lr}$ and $\Gamma_{ur}$ are compatible with the subdivision structure $\big( \subdiv, r_{0}, r_{1}, s \big)$, and that $\cyl$ preserves mapping cylinders with respect to $\cylinder$. 

Let $f$ be an arrow of $\cl{A}$ which is a normally cloven fibration with respect to $\cylinder$, and let $j$ be an arrow of $\cl{A}$ which is a cofibration with respect to $\cylinder$. Then $f$ has the covering homotopy extension property with respect to $j$ and $\cylinder$. \end{cor} 

\begin{proof} Follows immediately from Proposition \ref{RetractionIsStrongDeformationRetractionProposition} and Proposition \ref{NormallyClovenFibrationsHaveRLPWithRespectToArrowsAdmittingAStrongDeformationRetractionProposition}. \end{proof}

\begin{cor} \label{TrivialNormallyClovenFibrationRLPCofibrationCorollary} Let $\cylinder = \big( \cyl, i_{0}, i_{1}, p, v, \subdiv, r_{0}, r_{1}, s, \Gamma_{ul}, \Gamma_{lr}, \Gamma_{ur} \big)$ be a cylinder in $\cl{A}$ equipped with a contraction structure $p$, an involution structure $v$ compatible with $p$, a subdivision structure $\big( \subdiv, r_{0}, r_{1}, s \big)$ compatible with $p$, an upper left connection structure $\Gamma_{ul}$, a lower right connection structure $\Gamma_{lr}$, and an upper right connection structure $\Gamma_{ur}$. %
Suppose that $\Gamma_{lr}$ and $\Gamma_{ur}$ are compatible with the subdivision structure $\big( \subdiv, r_{0}, r_{1}, s \big)$, and that $\cyl$ preserves mapping cylinders with respect to $\cylinder$. 

Let $\cocylinder = \big( \cocyl, e_{0}, e_{1}, c \big)$ be a co-cylinder in $\cl{A}$ equipped with a contraction structure $c$. Suppose that $\cylinder$ is left adjoint to $\cocylinder$, and that the adjunction between $\cyl$ and $\cocyl$ is compatible with $p$ and $c$. %

Let \ar{a_{0},a_{1},j} be an arrow of $\cl{A}$ which is a cofibration with respect to $\cylinder$, and let \ar{a_{2},a_{3},f} be an arrow of $\cl{A}$ which is a trivial normally cloven fibration with respect to $\cocylinder$. 

For any commutative diagram \sq{a_{0},a_{2},a_{1},a_{3},g_{0},f,j,g_{1}} in $\cl{A}$, there is an arrow \ar{a_{1},a_{2},l} of $\cl{A}$ such that the following diagram in $\cl{A}$ commutes. \liftingsquare{a_{0},a_{2},a_{1},a_{3},g_{0},f,j,g_{1},l} \end{cor}

\begin{proof} By Proposition \ref{NormallyClovenFibrationCoCylinderIffNormallyClovenFibrationCylinderProposition}, we have that $f$ is a normally cloven fibration with respect to $\cylinder$. Thus by Corollary \ref{CHEPHoldsCorollary}, $f$ has the covering homotopy extension property with respect to $j$ and $\cylinder$. 

Moreover, by Corollary \ref{TrivialFibrationIsStrongDeformationRetractionCorollary}, there is an arrow \ar{a_{3},a_{2},j'} of $\cl{A}$ such that $f$ is a strong deformation retraction of $j'$ with respect to $\cocylinder$. Thus we may appeal to Proposition \ref{CHEPImpliesLiftingAxiomProposition} for a construction of $l$. \end{proof}

\begin{cor} \label{FibrationRLPTrivialNormallyClovenCofibrationCorollary} Let $\cylinder = \big( \cyl, i_{0}, i_{1}, p \big)$ be a cylinder in $\cl{A}$ equipped with a contraction structure $p$. Let $\cocylinder = \big( \cocyl, e_{0}, e_{1}, c, v, \subdiv, r_{0}, r_{1}, s, \Gamma_{ul}, \Gamma_{lr}, \Gamma_{ur} \big)$ be a co-cylinder in $\cl{A}$ equipped with a contraction structure $c$, an involution structure $v$ compatible with $c$, a subdivision structure $\big( \subdiv, r_{0}, r_{1}, s \big)$ compatible with $c$, an upper left connection structure $\Gamma_{ul}$, a lower right connection structure $\Gamma_{lr}$, and an upper right connection structure $\Gamma_{ur}$. 

Suppose that $\Gamma_{lr}$ and $\Gamma_{ur}$ are compatible with the subdivision structure $\big( \subdiv, r_{0}, r_{1}, s \big)$, and that $\cocyl$ preserves mapping co-cylinders with respect to $\cocylinder$. %
Suppose that $\cylinder$ is left adjoint to $\cocylinder$, and that the adjunction between $\cyl$ and $\cocyl$ is compatible with $p$ and $c$. %

Let \ar{a_{0},a_{1},j} be an arrow of $\cl{A}$ which is a trivial normally cloven cofibration with respect to $\cylinder$, and let \ar{a_{2},a_{3},f} be an arrow of $\cl{A}$ which is a fibration with respect to $\cocylinder$. 

For any commutative diagram \sq{a_{0},a_{2},a_{1},a_{3},g_{0},f,j,g_{1}} in $\cl{A}$, there is an arrow \ar{a_{1},a_{2},l} of $\cl{A}$ such that the following diagram in $\cl{A}$ commutes. \liftingsquare{a_{0},a_{2},a_{1},a_{3},g_{0},f,j,g_{1},l} \end{cor}

\begin{proof} Follows immediately from Corollary \ref{TrivialNormallyClovenFibrationRLPCofibrationCorollary} by duality. \end{proof} 

\end{chapter}

\begin{chapter}{Factorisation axioms} \label{FactorisationAxiomsChapter}

Let $\cylinder$ be a cylinder in a formal category $\cl{A}$. In \ref{MappingCylindersAndMappingCoCylindersChapter}, we showed that if $\cylinder$ is equipped with certain structures, and has strictness of right identities, then the mapping cylinder with respect to $\cylinder$ of an arrow $f$ gives rise to a factorisation into a normally cloven cofibration followed by a strong deformation retraction.

We now prove that, if $\cylinder$ has strictness of left identities, then a strong deformation retraction with respect to $\cylinder$ is a trivial fibration with respect to $\cylinder$. Thus, if $\cylinder$ has strictness of both left and right identities, then the mapping cylinder of $f$ with respect to $\cylinder$ yields a factorisation of $f$ into a normally cloven cofibration followed by a trivial fibration. 

Dually, if a co-cylinder $\cocylinder$ in $\cl{A}$ is equipped with sufficient structures and has strictness of identities, the mapping co-cylinder of $f$ with respect to $\cocylinder$ yields a factorisation of $f$ into a trivial cofibration followed by a normally cloven fibration. 

Morever, building upon this, we construct a factorisation of $f$ into a cofibration followed by a trivial normally cloven fibration, and into a trivial normally cloven cofibration followed by a fibration.  

Neither the strictness of left identities hypothesis nor the strictness of right identities hypothesis holds with respect to the usual cylinder and co-cylinder in the category of topological spaces. Essentially, this is the observation that by glueing a path $g$ to a constant path, we obtain a path homotopic to $g$, but not $g$ itself. 

Whilst the mapping cylinder of a map between topological spaces gives a factorisation into a cofibration followed by a homotopy equivalence, this homotopy equivalence is not necessarily a fibration. An example is given by the mapping cylinder factorisation of the inclusion of a circle into a disc. Dually, whilst the mapping co-cylinder of a map between topological spaces gives a factorisation into a homotopy equivalence followed by a fibration, this homotopy equivalence is not necessarily a cofibration. An example is given by the mapping co-cylinder factorisation of any inclusion of spaces which is not closed. We refer the reader to \ref{IntroductionChapter} for a discussion of a way around this. 

We shall explore in \ref{ExampleChapter} a guiding example in which strictness of identities does hold, namely the homotopy theory of categories or groupoids.

\begin{assum} Let $\cl{C}$ be a 2-category with a final object. Suppose that pushouts and pullbacks of 2-arrows of $\cl{C}$ give rise to pushouts and pullbacks in formal categories, in the sense of Definition \ref{PushoutsPullbacks2ArrowsArePushoutsPullbacksInFormalCategoriesTerminology}. Let $\cl{A}$ be an object of $\cl{C}$. As before, we view $\cl{A}$ as a formal category, writing of objects and arrows of $\cl{A}$. \end{assum} 

\begin{prpn} \label{StrongDeformationRetractionIsFibrationProposition} Let $\cylinder = \big( \cyl, i_{0}, i_{1}, p, \subdiv, r_{0}, r_{1}, s \big)$ be a cylinder in $\cl{A}$ equipped with a contraction structure $p$, and a subdivision structure $\big( \subdiv, r_{0}, r_{1}, s \big)$. Suppose that $\cylinder$ has strictness of left identities. 

Let \ar{a_{1},a_{2},j} be an arrow of $\cl{A}$, and let \ar{a_{2},a_{1},f} be an arrow of $\cl{A}$ which is a retraction of $j$. Suppose that \ar{\cyl(a_{2}),a_{2},h} defines a homotopy over $a_{1}$ from $id(a_{2})$ to $jf$ with respect to $\cylinder$ and $(f,f)$. Then $f$ is a fibration with respect to $\cylinder$. \end{prpn}

\begin{proof} Suppose that we have a commutative diagram in $\cl{A}$ as follows. \sq{a_{0},a_{2},\cyl(a_{0}),a_{1},g,f,i_{0}(a_{0}),k} %
By definition of $h$, the following diagram in $\cl{A}$ commutes. \tri{a_{2},\cyl(a_{2}),a_{2},i_{1}(a_{2}),h,j \circ f} %
Appealing to the commutativity of the diagram \sq{a_{0},\cyl(a_{0}),a_{2},\cyl(a_{2}),i_{1}(a_{0}),\cyl(g),g,i_{1}(a_{2})} in $\cl{A}$, we deduce that the following diagram in $\cl{A}$ commutes. \tri{a_{0},\cyl(a_{0}),a_{2},i_{1}(a_{0}),h \circ \cyl(g),j \circ f \circ g} %
Moreover, the following diagram in $\cl{A}$ commutes. \tri{a_{0},\cyl(a_{0}),a_{2},i_{0}(a_{0}),j \circ k,j \circ f \circ g} %
Putting the last two observations together, we have that the following diagram in $\cl{A}$ commutes. \sq[{4,3}]{a_{0},\cyl(a_{0}),\cyl(a_{0}),a_{2},i_{0}(a_{0}), j \circ k,i_{1}(a_{0}),h \circ \cyl(g)} %
Let \ar{\cyl(a_{0}),a_{2},l} denote the homotopy $\big(h \circ \cyl(g) \big) + (j \circ k)$ with respect to $\cylinder$. We claim that the following diagram in $\cl{A}$ commutes. \liftingsquare{a_{0},a_{2},\cyl(a_{0}),a_{1},g,f,i_{0}(a_{0}),k,l} 

Firstly, the following diagram in $\cl{A}$ commutes. \squareabovetriangle{a_{0},\cyl(a_{0}),a_{2},\cyl(a_{2}),a_{2},i_{0}(a_{0}),\cyl(g),g,i_{0}(a_{2}),h,id} %
By definition of $l$, we also have that the following diagram in $\cl{A}$ commutes. \tri{a_{0},\cyl(a_{0}),a_{2},i_{0}(a_{0}),l,h \circ \cyl(g)} %
Hence the following diagram in $\cl{A}$ commutes, as required. \tri{a_{0},\cyl(a_{0}),a_{2},i_{0}(a_{0}),l,g} 

Secondly, let us prove that the diagram \tri{\cyl(a_{0}),a_{2},a_{1},l,f,k} in $\cl{A}$ commutes. Let \ar{\subdiv(a_{0}),a_{2},u} denote the canonical arrow of $\cl{A}$ such that the following diagram in $\cl{A}$ commutes. \pushout{a_{0},\cyl(a_{0}),\cyl(a_{0}),\subdiv(a_{0}),a_{2},i_{0}(a_{0}),r_{0}(a_{0}),i_{1}(a_{0}),r_{1}(a_{0}),j \circ k, h \circ \cyl(g),u} %
By definition of $h$, the following diagram in $\cl{A}$ commutes. \sq{\cyl(a_{1}),a_{1},\cyl(a_{2}),a_{2},h,f,\cyl(f),p(a_{2})} %
Appealing to the commutativity of the diagram \tri{\cyl(a_{0}),\subdiv(a_{0}),a_{2},r_{1}(a_{0}),u,h \circ \cyl(g)} in $\cl{A}$, we deduce that the following diagram in $\cl{A}$ commutes. \tri{\cyl(a_{0}),\subdiv(a_{0}),a_{1},s(a_{0}),f \circ u, p(a_{1}) \circ \cyl(f \circ g)} %
We also have that the following diagram in $\cl{A}$ commutes. \trapeziumstwo[{3,3.5,2,0}]{\cyl(a_{0}),\cyl(a_{2}),a_{0},\cyl^{2}(a_{0}),\cyl(a_{1}),\cyl(a_{0}),a_{1},\cyl(g),\cyl(f),p(a_{1}),p(a_{0}),i_{0}(a_{0}),k,\cyl\big(i_{0}(a_{0})\big),\cyl(k),p\big(\cyl(a_{0})\big)} %
Putting the last two observations together, we have that the following diagram in $\cl{A}$ commutes. \tri{\cyl(a_{0}),\subdiv(a_{0}),a_{1},s(a_{0}),f \circ u,k \circ i_{0}(a_{0}) \circ p(a_{0})} %
Moreover, the following diagram in $\cl{A}$ commutes. \squareabovetriangle{\cyl(a_{0}),\subdiv(a_{0}),a_{1},a_{2},a_{1},r_{0}(a_{0}),u,k,j,f,id} %
Putting the last two observations together, we have that the following diagram in $\cl{A}$ commutes. \pushout{a_{0},\cyl(a_{0}),\cyl(a_{0}),\subdiv(a_{0}),a_{1},i_{0}(a_{0}),r_{0}(a_{0}),i_{1}(a_{0}),r_{1}(a_{0}),k,k \circ i_{0}(a_{0}) \circ p(a_{0}), f \circ u} %
Let \ar{\subdiv,\cyl,q_{l}} denote the canonical 2-arrow of $\cl{C}$ of Definition \ref{StrictnessLeftIdentitiesCylinderDefinition}. We have that the following diagram in $\cl{A}$ commutes. \pushout{a_{0},\cyl(a_{0}),\cyl(a_{0}),\subdiv(a_{0}),\cyl(a_{0}),i_{0}(a_{0}),r_{0}(a_{0}),i_{1}(a_{0}),r_{1}(a_{0}),id, i_{0}(a_{0}) \circ p(a_{0}), q_{l}(a_{0})} %
Then the following diagram in $\cl{A}$ commutes. \pushout{a_{0},\cyl(a_{0}),\cyl(a_{0}),\subdiv(a_{0}),a_{1},i_{0}(a_{0}),r_{0}(a_{0}),i_{1}(a_{0}),r_{1}(a_{0}),k, k \circ i_{0}(a_{0}) \circ p(a_{0}), k \circ q_{l}(a_{0})} %
Appealing to the universal property of $\subdiv(a_{0})$, we deduce that the following diagram in $\cl{A}$ commutes. \tri{\subdiv(a_{0}),\cyl(a_{0}),a_{1},q_{l}(a_{0}),k,f \circ u} %
By definition of $l$, the following diagram in $\cl{A}$ commutes. \tri{\cyl(a_{0}),\subdiv(a_{0}),a_{2},s(a_{0}),u,l} %
Putting the last two observations together, we have that the following diagram in $\cl{A}$ commutes. \tri{\cyl(a_{0}),a_{2},a_{1},l,f,k \circ q_{l}(a_{0}) \circ s(a_{0})} %
Since $\cylinder$ has strictness of left identities, the following diagram in $\cl{A}$ commutes. \tri{\cyl(a_{0}),\subdiv(a_{0}),\cyl(a_{0}),s(a_{0}),q_{l}(a_{0}),id} %
Putting the last two observations together, we have that the following diagram in $\cl{A}$ commutes, as required. \tri{\cyl(a_{0}),a_{2},a_{1},l,f,k}
\end{proof}

\begin{cor} \label{TrivialFibrationCylinderIffStrongDeformationRetractionCorollary} Let $\cylinder = \big( \cyl, i_{0}, i_{1}, p, \subdiv, r_{0}, r_{1}, s \big)$ be a cylinder in $\cl{A}$ equipped with a contraction structure $p$, and a subdivision structure $\big( \subdiv, r_{0}, r_{1}, s \big)$. Suppose that $\cylinder$ has strictness of left identities. 

An arrow \ar{a_{1},a_{0},f} of $\cl{A}$ is a trivial fibration with respect to $\cylinder$ if and only if there is an arrow \ar{a_{0},a_{1},j} of $\cl{A}$, such that $f$ is a retraction of $j$, and such that there exists a homotopy over $a_{0}$ from $jf$ to $id(a_{1})$ with respect to $\cylinder$ and $(f,f)$. \end{cor} 

\begin{proof} Follows immediately from Proposition \ref{DoldOverHomotopiesCylinderProposition} and Proposition \ref{StrongDeformationRetractionIsFibrationProposition}. \end{proof}

\begin{cor} \label{TrivialFibrationCoCylinderIffStrongDeformationRetractionCorollary} Let $\cylinder = \big( \cyl, i_{0}, i_{1}, p, v, \subdiv, r_{0}, r_{1}, s \big)$ be a cylinder in $\cl{A}$ equipped with a contraction structure $p$, an involution structure $v$ compatible with $p$, and a subdivision structure $\big( \subdiv, r_{0}, r_{1}, s \big)$. Suppose that $\cylinder$ has strictness of left identities. 

Let $\cocylinder = \big( \cocyl, e_0, e_1, $c$ \big)$ be a co-cylinder in $\cl{A}$ equipped with a contraction structure $c$. Suppose that $\cylinder$ is left adjoint to $\cocylinder$, and that the adjunction between $\cyl$ and $\cocyl$ is compatible with $p$ and $c$. 

An arrow \ar{a_{0},a_{1},f} of $\cl{A}$ is a trivial fibration with respect to $\cocylinder$ if and only if it is a strong deformation retraction with respect to $\cocylinder$. \end{cor} 

\begin{proof} Follows immediately from Proposition \ref{StrongDeformationRetractionCoCylinderCharacterisationProposition}, Proposition \ref{ReverseHomotopyIsOverHomotopyProposition}, and Corollary \ref{TrivialFibrationCylinderIffStrongDeformationRetractionCorollary}. \end{proof}

\begin{cor} \label{TrivialCofibrationCylinderIffStrongDeformationRetractionCorollary} Let $\cocylinder = \big( \cocyl, e_{0}, e_{1}, c, v, \subdiv, r_{0}, r_{1}, s \big)$ be a co-cylinder in $\cl{A}$ equipped with a contraction structure $c$, an involution structure $v$ compatible with $c$, and a subdivision structure $\big(\subdiv, r_{0}, r_{1}, s \big)$. Suppose that $\cocylinder$ has strictness of left identities. 

Let $\cylinder = \big( \cyl, i_0, i_1,p \big)$ be a cylinder in $\cl{A}$ equipped with a contraction structure $p$. Suppose that $\cylinder$ is left adjoint to $\cocylinder$, and that the adjunction between $\cyl$ and $\cocyl$ is compatible with $p$ and $c$. 

An arrow \ar{a_{0},a_{1},f} of $\cl{A}$ is a trivial cofibration with respect to $\cylinder$ if and only if it admits a strong deformation retraction with respect to $\cylinder$. \end{cor} 

\begin{proof} Follows immediately from Corollary \ref{TrivialFibrationCoCylinderIffStrongDeformationRetractionCorollary} by duality. \end{proof}

\begin{cor} \label{NormallyClovenCofibrationFollowedByTrivialFibrationCorollary} Let $\cylinder = \big( \cyl, i_{0}, i_{1}, p, v, \subdiv, r_{0}, r_{1}, s, \Gamma_{lr} \big)$ be a cylinder in $\cl{A}$ equipped with a contraction structure $p$, an involution structure $p$ compatible with $p$, a subdivision structure $\big( \subdiv, r_{0}, r_{1}, s \big)$, and a lower right connection structure $\Gamma_{lr}$. Suppose that $\Gamma_{lr}$ is compatible with $p$, and that $\cylinder$ has strictness of identities. Suppose moreover that $\cyl$ preserves mapping cylinders with respect to $\cylinder$. 

Let $\cocylinder = \big( \cocyl, e_0, e_1, $c$ \big)$ be a co-cylinder in $\cl{A}$ equipped with a contraction structure $c$. Suppose that $\cylinder$ is left adjoint to $\cocylinder$, and that the adjunction between $\cyl$ and $\cocyl$ is compatible with $p$ and $c$. 

Let \ar{a_{0},a_{1},f} be an arrow of $\cl{A}$, and let $\big( a^{\cylinder}_{f},d^{0}_{f},d^{1}_{f} \big)$ be a mapping cylinder of $f$ with respect to $\cylinder$. Let \tri{a_{0},a^{\cylinder}_{f},a_{1},j,g,f} denote the corresponding mapping cylinder factorisation of $f$. Then $j$ is a normally cloven cofibration with respect to $\cylinder$, and $g$ is a trivial fibration with respect to $\cocylinder$. \end{cor}

\begin{proof} Follows immediately from Proposition \ref{MappingCylinderFactorisationGivesNormallyClovenCofibrationProposition}, Corollary \ref{MappingCylinderFactorisationStrongDeformationRetractionWithRespectToCoCylinderCorollary}, and Corollary \ref{TrivialFibrationCoCylinderIffStrongDeformationRetractionCorollary}. \end{proof} 

\begin{cor} \label{TrivialCofibrationFollowedByNormallyClovenFibrationCorollary} Let $\cocylinder = \big( \cocyl, e_{0}, e_{1}, c, v, \subdiv, r_{0}, r_{1}, s, \Gamma_{lr} \big)$ be a co-cylinder in $\cl{A}$ equipped with a contraction structure $c$, an involution structure compatible with $c$, a subdivision structure $\big( \subdiv, r_{0}, r_{1}, s \big)$, and a lower right connection structure $\Gamma_{lr}$. Suppose that $\Gamma_{lr}$ is compatible with $c$, and that $\cocylinder$ has strictness of identities. Suppose moreover that $\cocyl$ preserves mapping co-cylinders with respect to $\cocylinder$. 

Let $\cylinder = \big( \cyl, i_0, i_1,p \big)$ be a cylinder in $\cl{A}$ equipped with a contraction structure $p$. Suppose that $\cylinder$ is left adjoint to $\cocylinder$, and that the adjunction between $\cyl$ and $\cocyl$ is compatible with $p$ and $c$.
 
Let \ar{a_{0},a_{1},f} be an arrow of $\cl{A}$, and let $\big( a^{\cocylinder}_{f},d^{0}_{f},d^{1}_{f} \big)$ be a mapping co-cylinder of $f$ with respect to $\cocylinder$. Let \tri{a_{0},a^{\cocylinder}_{f},a_{1},j,g,f} denote the corresponding mapping co-cylinder factorisation of $f$. 

Then $j$ is a trivial cofibration with respect to $\cylinder$, and $g$ is a normally cloven fibration with respect to $\cocylinder$. \end{cor} 

\begin{proof} Follows immediately from Corollary \ref{NormallyClovenCofibrationFollowedByTrivialFibrationCorollary} by duality. \end{proof}

\begin{cor} \label{CofibrationFollowedByTrivialNormallyClovenFibrationCorollary} Let $\cylinder = \big( \cyl, i_{0}, i_{1}, p, v, \subdiv, r_{0}, r_{1}, s, \Gamma_{lr} \big)$ be a cylinder in $\cl{A}$ equipped with a contraction structure $p$, an involution structure $p$ compatible with $p$, a subdivision structure $\big( \subdiv, r_{0}, r_{1}, s \big)$, and a lower right connection structure $\Gamma_{lr}$. Suppose that $\Gamma_{lr}$ is compatible with $p$, and that $\cylinder$ has strictness of identities. Suppose moreover that $\cyl$ preserves mapping cylinders with respect to $\cylinder$.

Let $\cocylinder = \big( \cocyl, e_{0}, e_{1}, c, v, \subdiv, r_{0}, r_{1}, s, \Gamma_{lr} \big)$ be a co-cylinder in $\cl{A}$ equipped with a contraction structure $c$, an involution structure compatible with $c$, a subdivision structure $\big( \subdiv, r_{0}, r_{1}, s \big)$, and a lower right connection structure $\Gamma_{lr}$. Suppose that $\Gamma_{lr}$ is compatible with $c$, and that $\cocylinder$ has strictness of identities. Suppose moreover that $\cocyl$ preserves mapping co-cylinders with respect to $\cocylinder$. 

Suppose that $\cylinder$ is left adjoint to $\cocylinder$, and that the adjunction between $\cyl$ and $\cocyl$ is compatible with $p$ and $c$.

Let \ar{a_{0},a_{1},f} be an arrow of $\cl{A}$. There is an object $a$ of $\cl{A}$, an arrow \ar{a_{0},a,j} of $\cl{A}$ which is a cofibration with respect to $\cylinder$, and an arrow \ar{a,a_{1},g} of $\cl{A}$ which is a trivial normally cloven fibration with respect to $\cocylinder$, such that the following diagram in $\cl{A}$ commutes. \tri{a_{0},a,a_{1},j,g,f} \end{cor}

\begin{proof} Let $\big( a^{\cylinder}_{f},d^{0}_{f},d^{1}_{f} \big)$ be a mapping cylinder of $f$ with respect to $\cylinder$. Let \tri{a_{0},a^{\cylinder}_{f},a_{1},j',g',f} denote the corresponding mapping cylinder factorisation of $f$. By Proposition \ref{MappingCylinderFactorisationGivesNormallyClovenCofibrationProposition}, we have that $j'$ is a cofibration with respect to $\cylinder$. 

Let $\big( a^{\cocylinder}_{g'},d^{0}_{g'},d^{1}_{g'} \big)$ be a mapping co-cylinder of $g'$ with respect to $\cocylinder$. Let \tri{a^{\cylinder}_{f},a^{\cocylinder}_{g'},a_{1},j'',g,g'} denote the corresponding mapping co-cylinder factorisation of $g'$. Let us take $a$ to be $a^{\cocylinder}_{g'}$. By Corollary \ref{TrivialCofibrationFollowedByNormallyClovenFibrationCorollary}, we have that $j''$ is a trivial cofibration with respect to $\cylinder$, and that $g$ is a normally cloven fibration with respect to $\cocylinder$.  

Since $g'$ and $j''$ are homotopy equivalences with respect to $\cylinder$, it follows from Proposition \ref{TwoOutOfThreeHomotopyEquivalencesProposition} that $g$ is a homotopy equivalence with respect to $\cylinder$. Hence, by Proposition \ref{HomotopyEquivalenceCylinderIffHomotopyEquivalenceCoCylinderProposition}, $g$ is a homotopy equivalence with respect to $\cocylinder$. Thus $g$ is a trivial normally cloven fibration with respect to $\cocylinder$.

Let \ar{a_{0},a^{\cocylinder}_{g'},j} denote the arrow $j'' \circ j'$ of $\cl{A}$. Since both $j'$ and $j''$ are cofibrations with respect to $\cylinder$, we conclude that $j$ is a cofibration with respect to $\cylinder$ by Proposition \ref{CompositionOfCofibrationsIsACofibrationProposition}.

\end{proof}

\begin{cor} \label{TrivialNormallyClovenCofibrationFollowedByFibrationCorollary} Let $\cylinder = \big( \cyl, i_{0}, i_{1}, p, v, \subdiv, r_{0}, r_{1}, s, \Gamma_{lr} \big)$ be a cylinder in $\cl{A}$ equipped with a contraction structure $p$, an involution structure $p$ compatible with $p$, a subdivision structure $\big( \subdiv, r_{0}, r_{1}, s \big)$, and a lower right connection structure $\Gamma_{lr}$. Suppose that $\Gamma_{lr}$ is compatible with $p$, and that $\cylinder$ has strictness of identities. Suppose moreover that $\cyl$ preserves mapping cylinders with respect to $\cylinder$.

Let $\cocylinder = \big( \cocyl, e_{0}, e_{1}, c, v, \subdiv, r_{0}, r_{1}, s, \Gamma_{lr} \big)$ be a co-cylinder in $\cl{A}$ equipped with a contraction structure $c$, an involution structure compatible with $c$, a subdivision structure $\big( \subdiv, r_{0}, r_{1}, s \big)$, and a lower right connection structure $\Gamma_{lr}$. Suppose that $\Gamma_{lr}$ is compatible with $c$, and that $\cocylinder$ has strictness of identities. Suppose moreover that $\cocyl$ preserves mapping co-cylinders with respect to $\cocylinder$. 

Suppose that $\cylinder$ is left adjoint to $\cocylinder$, and that the adjunction between $\cyl$ and $\cocyl$ is compatible with $p$ and $c$.

Let \ar{a_{0},a_{1},f} be an arrow of $\cl{A}$. There is an object $a$ of $\cl{A}$, an arrow \ar{a_{0},a,j} of $\cl{A}$ which is a trivial normally cloven cofibration with respect to $\cylinder$, and an arrow \ar{a,a_{1},g} of $\cl{A}$ which is a fibration with respect to $\cocylinder$, such that the following diagram in $\cl{A}$ commutes. \tri{a_{0},a,a_{1},j,g,f} \end{cor}

\begin{proof} Follows immediately from Corollary \ref{TrivialNormallyClovenCofibrationFollowedByFibrationCorollary} by duality. \end{proof}

\end{chapter}

\begin{chapter}{Model category recollections} \label{ModelCategoryRecollectionsChapter} 

In \ref{ModelStructureChapter}, we shall bring together all the theory we have developed so far. In order to do so, we now present a few recollections on model categories. 

The notion of a model category was introduced by Quillen in \cite{QuillenHomotopicalAlgebra}. We recall two definitions, and prove that they are equivalent. Our arguments comprise part of the proof of Proposition 2 of \S{5} of \cite{QuillenHomotopicalAlgebra}. Our definitions are equivalent to the definition of a closed model category given in \S{5} of \cite{QuillenHomotopicalAlgebra}.

\begin{defn} \label{ModelStructureDefinition} Let $\cl{A}$ be a category with finite limits and colimits.  A {\em model structure} upon $\cl{A}$ consists of three sets $W$, $F$, and $C$ of arrows of $\cl{A}$, such that the following conditions are satisfied. 

\begin{itemize}[topsep=1em,itemsep=1em]

\item[(i)] If any two of the arrows in a commutative diagram \tri{a_{0},a_{1},a_{2},g_{0},g_{1},g_{2}} in $\cl{A}$ belong to $W$, so does the third. 

\item[(ii)] An arrow \ar{a_{2},a_{3},f} of $\cl{A}$ belongs to $F$ if and only if, for every commutative diagram \sq{a_{0},a_{1},a_{2},a_{3},g_{0},f,j,g_{1}} in $\cl{A}$ such that $j$ belongs to both $W$ and $C$, there is an arrow \ar{a_{2},a_{1},l} of $\cl{A}$ such that the diagram \liftingsquare{a_{0},a_{2},a_{1},a_{3},g_{0},f,j,g_{1},l} in $\cl{A}$ commutes.

\item[(iii)] An arrow \ar{a_{2},a_{3},f} of $\cl{A}$ belongs to both $F$ and $W$ if and only if, for every commutative diagram \sq{a_{0},a_{1},a_{2},a_{3},g_{0},f,j,g_{1}} in $\cl{A}$ such that $j$ belongs to $C$, there is an arrow \ar{a_{2},a_{1},l} of $\cl{A}$ such that the diagram \liftingsquare{a_{0},a_{2},a_{1},a_{3},g_{0},f,j,g_{1},l} in $\cl{A}$ commutes.

\item[(iv)] An arrow \ar{a_{0},a_{1},j} of $\cl{A}$ belongs to $C$ if and only if, for every commutative diagram \sq{a_{0},a_{1},a_{2},a_{3},g_{0},f,j,g_{1}} in $\cl{A}$ such that $f$ belongs to both $F$ and $W$, there is an arrow \ar{a_{2},a_{1},l} of $\cl{A}$ such that the diagram \liftingsquare{a_{0},a_{2},a_{1},a_{3},g_{0},f,j,g_{1},l} in $\cl{A}$ commutes.

\item[(v)] An arrow \ar{a_{0},a_{1},j} belongs to both $C$ and $W$ if and only if, for every commutative diagram \sq{a_{0},a_{1},a_{2},a_{3},g_{0},f,j,g_{1}} in $\cl{A}$ such that $f$ belongs to $F$, there is an arrow \ar{a_{2},a_{1},l} of $\cl{A}$ such that the diagram \liftingsquare{a_{0},a_{2},a_{1},a_{3},g_{0},f,j,g_{1},l} in $\cl{A}$ commutes.

\item[(vi)] For every arrow \ar{a_{0},a_{1},f} of $\cl{A}$, there is an arrow \ar{a_{0},a_{2},j} of $\cl{A}$ which belongs to $C$, and an arrow \ar{a_{2},a_{1},g} of $\cl{A}$ which belongs to $W$ and $F$, such that the following diagram in $\cl{A}$ commutes. \tri{a_{0},a_{2},a_{1},j,g,f}

\item[(vii)] For every arrow \ar{a_{0},a_{1},f} of $\cl{A}$, there is an arrow \ar{a_{0},a_{2},j} of $\cl{A}$ which belongs to $W$ and $C$, and an arrow \ar{a_{2},a_{1},g} of $\cl{A}$ which belongs to $F$, such that the following diagram in $\cl{A}$ commutes. \tri{a_{0},a_{2},a_{1},j,g,f}

\end{itemize}

\end{defn}

\begin{defn} Let $\cl{A}$ be a category with finite limits and colimits. Let $W$, $F$, and $C$ be sets of arrows of $\cl{A}$ which equip $\cl{A}$ with a model structure. We refer to an arrow of $\cl{A}$ which belongs to $W$ as a {\em weak equivalence}, to an arrow of $\cl{A}$ which belongs to $F$ as a {\em fibration}, and to an arrow of $\cl{A}$ which belongs to $C$ as a {\em cofibration}. We refer to an arrow of $\cl{A}$ which belongs to both $W$ and $C$ as a {\em trivial cofibration}, and to an arrow of $\cl{A}$ which belongs to both $W$ and $F$ as a {\em trivial fibration}. \end{defn}

\begin{defn} A {\em model category} is a category $\cl{A}$ which has finite limits and colimits, together with a model structure upon $\cl{A}$. \end{defn}

\begin{prpn} \label{EquivalentDefinitionsOfModelStructureProposition} Let $\cl{A}$ be a category with finite limits and colimits. Let $W$, $F$, $C$ be sets of arrows of $\cl{A}$. Then $(W,F,C)$ equips $\cl{A}$ with a model structure if and only if the following conditions are satisfied. 

\begin{itemize}[itemsep=1em,topsep=1em]

\item[(i)] If any two of the arrows in a commutative diagram \tri{a_{0},a_{1},a_{2},g_{0},g_{1},g_{2}} in $\cl{A}$ belong to $W$, so does the third. 

\item[(ii)] Suppose that we have commutative diagrams \twosq{a_{2},a_{0},a_{3},a_{1},g_{0},j,j',g_{1},a_{0},a_{2},a_{1},a_{3},r_{0},j',j,r_{1}} in $\cl{A}$, such that $r_{0}$ is a retraction of $g_{0}$, and such that $r_{1}$ is a retraction of $g_{1}$. If $j$ belongs to $C$, then $j'$ belongs to $C$. If $j$ belongs to both $C$ and $W$, then $j'$ belongs to both $C$ and $W$.

\item[(iii)] Suppose that we have commutative diagrams \twosq{a_{2},a_{0},a_{3},a_{1},g_{0},f,f',g_{1},a_{0},a_{2},a_{1},a_{3},r_{0},f',f,r_{1}} in $\cl{A}$ such that $r_{0}$ is a retraction of $g_{0}$, and such that $r_{1}$ is a retraction of $g_{1}$. If $f$ belongs to $F$ then $f'$ belongs to $F$. If $f$ belongs to both $F$ and $W$ then $f'$ belongs to both $F$ and $W$.

\item[(iv)] For every diagram \sq{a_{0},a_{1},a_{2},a_{3},g_{0},f,j,g_{1}} in $\cl{A}$, such that $j$ belongs to $W$ and $C$, and $f$ belongs to $F$, there is an arrow \ar{a_{2},a_{1},l} of $\cl{A}$ such that the diagram \liftingsquare{a_{0},a_{1},a_{2},a_{3},g_{0},f,j,g_{1},l} in $\cl{A}$ commutes.

\item[(v)] For every diagram \sq{a_{0},a_{1},a_{2},a_{3},g_{0},f,j,g_{1}} in $\cl{A}$, such that $j$ belongs to $C$, and $f$ belongs to $W$ and $F$, there is an arrow \ar{a_{2},a_{1},l} of $\cl{A}$ such that the diagram \liftingsquare{a_{0},a_{1},a_{2},a_{3},g_{0},f,j,g_{1},l} in $\cl{A}$ commutes.

\item[(vi)] For every arrow \ar{a_{0},a_{1},f} of $\cl{A}$, there is an arrow \ar{a_{0},a_{2},j} of $\cl{A}$ which belongs to $C$, and an arrow \ar{a_{2},a_{1},g} of $\cl{A}$ which belongs to $W$ and $F$, such that the following diagram in $\cl{A}$ commutes. \tri{a_{0},a_{2},a_{1},j,g,f}

\item[(vii)] For every arrow \ar{a_{0},a_{1},f} of $\cl{A}$, there is an arrow \ar{a_{0},a_{2},j} of $\cl{A}$ which belongs to $W$ and $C$, and an arrow \ar{a_{2},a_{1},g} of $\cl{A}$ which belongs to $F$, such that the following diagram in $\cl{A}$ commutes. \tri{a_{0},a_{2},a_{1},j,g,f}

\end{itemize}

\end{prpn}

\begin{proof} We first prove that if the conditions of Proposition \ref{EquivalentDefinitionsOfModelStructureProposition} are satisfied, then $(W,F,C)$ equips $\cl{A}$ with a model structure. Let us demonstrate that condition (ii) of Definition \ref{ModelStructureDefinition} holds. 

Given that condition (iv) of Proposition \ref{EquivalentDefinitionsOfModelStructureProposition} holds, it suffices to show that if \ar{a_{2},a_{3},f} is an arrow of $\cl{A}$ with the property that, for every commutative diagram \sq{a_{0},a_{1},a_{2},a_{3},g_{0},f,j,g_{1}} in $\cl{A}$ such that $j$ belongs to both $W$ and $C$, there is an arrow \ar{a_{2},a_{1},l} of $\cl{A}$ such that the diagram \liftingsquare{a_{0},a_{2},a_{1},a_{3},g_{0},f,j,g_{1},l} in $\cl{A}$ commutes, then $f$ belongs to $F$.

To this end, by condition (vii) of Proposition \ref{EquivalentDefinitionsOfModelStructureProposition}, there is an arrow \ar{a_{2},a_{4},j'} of $\cl{A}$ which belongs to both $W$ and $C$, and an arrow \ar{a_{4},a_{3},f'} of $\cl{A}$ which belongs to $F$, such that the following diagram in $\cl{A}$ commutes. \tri{a_{2},a_{4},a_{3},j',f',f} %
By assumption, there is an arrow \ar{a_{4},a_{2},l'} of $\cl{A}$ such that the following diagram in $\cl{A}$ commutes. \liftingsquare{a_{2},a_{2},a_{4},a_{3},id,f,j',f',l'} %
In other words, we have a pair of commutative diagrams in $\cl{A}$ as follows such that $l'$ is a retraction of $j'$. \twosq{a_{2},a_{4},a_{3},a_{3},j',f',f,id,a_{4},a_{2},a_{3},a_{3},l',f,f',id} %
Appealing to condition (iii) of Proposition \ref{EquivalentDefinitionsOfModelStructureProposition}, we deduce that $f$ belongs to $F$. 

Next, let us demonstrate that condition (iii) of Definition \ref{ModelStructureDefinition} holds. Given that condition (v) of Proposition \ref{EquivalentDefinitionsOfModelStructureProposition} holds, it suffices to show that if \ar{a_{2},a_{3},f} is now an arrow of $\cl{A}$ with the property that, for every commutative diagram \sq{a_{0},a_{1},a_{2},a_{3},g_{0},f,j,g_{1}} in $\cl{A}$ such that $j$ belongs to $C$, there is an arrow \ar{a_{2},a_{1},l} of $\cl{A}$ such that the diagram \liftingsquare{a_{0},a_{2},a_{1},a_{3},g_{0},f,j,g_{1},l} in $\cl{A}$ commutes, then $f$ belongs to both $W$ and $F$.

To this end, by condition (vi) of Proposition \ref{EquivalentDefinitionsOfModelStructureProposition}, there is an arrow \ar{a_{2},a_{4},j'} of $\cl{A}$ which belongs to $C$, and an arrow \ar{a_{4},a_{3},f'} of $\cl{A}$ which belongs to both $F$ and $W$, such that the following diagram in $\cl{A}$ commutes. \tri{a_{2},a_{4},a_{3},j',f',f} %
By assumption, there is an arrow \ar{a_{4},a_{2},l'} of $\cl{A}$ such that the following diagram in $\cl{A}$ commutes. \liftingsquare{a_{2},a_{2},a_{4},a_{3},id,f,j',f',l'} %
In other words, we have a pair of commutative diagrams in $\cl{A}$ as follows such that $l'$ is a retraction of $j'$. \twosq{a_{2},a_{4},a_{3},a_{3},j',f',f,id,a_{4},a_{2},a_{3},a_{3},l',f,f',id} %
Appealing to condition (iii) of Proposition \ref{EquivalentDefinitionsOfModelStructureProposition}, we deduce that $f$ belongs to both $F$ and $W$. 

That conditions (iv) and (v) of Definition \ref{ModelStructureDefinition} hold, given that conditions (ii), (iv), (v), (vi), and (vii) of Proposition \ref{EquivalentDefinitionsOfModelStructureProposition} hold, follows formally, by duality, from the two arguments we have already given in this proof. 

Conversely, suppose that $(W,F,C)$ equips $\cl{A}$ with a model structure. We must demonstrate that conditions (ii) and (iii) of Proposition \ref{EquivalentDefinitionsOfModelStructureProposition} are satisfied. 

Suppose that we have commutative diagrams \twosq{a_{2},a_{0},a_{3},a_{1},g_{0}',f,f',g_{1}',a_{0},a_{2},a_{1},a_{3},r_{0},f',f,r_{1}} in $\cl{A}$, such that $r_{0}$ is a retraction of $g_{0}$, such that $r_{1}$ is a retraction of $g_{1}$, and such that $f$ belongs to $F$. %
Suppose that we have a commutative diagram in $\cl{A}$ as follows, in which $j$ belongs to both $C$ and $W$. \sq{a_{0}',a_{2},a_{1}',a_{3}',g_{0},f',j,g_{1}} %
Then the following diagram in $\cl{A}$ commutes. \sq[{4,3}]{a_{0}',a_{0},a_{1}',a_{1},g_{0} \circ g_{0}',f,j,g_{1} \circ g_{1}'} %
Since $f$ belongs to $F$, by condition (ii) of Definition \ref{ModelStructureDefinition} there is an arrow \ar{a_{1}',a_{0},l} of $\cl{A}$ such that the following diagram in $\cl{A}$ commutes. \liftingsquare[{4,3}]{a_{0}',a_{0},a_{1}',a_{1},g_{0} \circ g_{0}',f,j,g_{1} \circ g_{1}',l} %
Thus the following diagram in $\cl{A}$ commutes. \liftingsquare[{6,3}]{a_{0}',a_{2},a_{1}',a_{3},r_{0} \circ g_{0} \circ g_{0}',f,j,r_{1} \circ g_{1} \circ g_{1}',r_{0} \circ l} %
Since $r_{0}$ is a retraction of $g_{0}$, and since $r_{1}$ is a retraction of $g_{1}$, we thus have that the following diagram in $\cl{A}$ commutes. \liftingsquare[{6,3}]{a_{0}',a_{2},a_{1}',a_{3},g_{0}',f,j,g_{1}',r_{0} \circ l} 

An entirely similar argument, appealing to condition (iii) rather than condition (ii) of Definition \ref{ModelStructureDefinition}, proves that if $f$ belongs to both $F$ and $W$, then $f'$ belongs to both $F$ and $W$. This completes our proof that condition (ii) of Proposition \ref{EquivalentDefinitionsOfModelStructureProposition} is satisfied. 

That condition (iii) of Proposition \ref{EquivalentDefinitionsOfModelStructureProposition} is satisfied, given that conditions (iv) and (v) of Definition \ref{ModelStructureDefinition} hold, follows formally, by duality, from the proof we have just given that condition (ii) of Proposition \ref{EquivalentDefinitionsOfModelStructureProposition} holds. 
 
\end{proof} 

\end{chapter}

\begin{chapter}{Model structure} \label{ModelStructureChapter}

Suppose that we have a cylinder $\cylinder$ and a co-cylinder $\cocylinder$ in a category $\cl{A}$, such that the following hold:

\begin{itemize}[itemsep=1em,topsep=1em]

\item[(i)] both $\cylinder$ and $\cocylinder$ are equipped with all the structures we have considered in this work, and have strictness of identities;

\item[(ii)] $\cylinder$ is left adjoint to $\cocylinder$, and the adjunction between $\cyl$ and $\cocyl$ is compatible with their respective contraction structures.

\end{itemize} %
We bring all our theory together, to prove that we obtain a model structure upon $\cl{A}$ by taking:

\begin{itemize} [itemsep=1em,topsep=1em]

\item[(i)] weak equivalences to be homotopy equivalences with respect to $\cylinder$, or equivalently with respect to $\cocylinder$,

\item[(ii)] fibrations to be fibrations with respect to $\cocylinder$,

\item[(iii)] cofibrations to be normally cloven cofibrations with respect to $\cylinder$.

\end{itemize} %
Equally, we prove that we obtain a model structure upon $\cl{A}$ by taking:

\begin{itemize} [itemsep=1em,topsep=1em]

\item[(i)] weak equivalences to be homotopy equivalences with respect to $\cylinder$, or equivalently with respect to $\cocylinder$;

\item[(ii)] fibrations to be normally cloven fibrations with respect to $\cocylinder$;

\item[(iii)] cofibrations to be cofibrations with respect to $\cylinder$.

\end{itemize} %
An interval $\interval$ with respect to a monoidal structure upon $\cl{A}$ gives rise, as in \ref{StructuresUponAnIntervalChapter}, to a cylinder $\cylinderinterval$ and a co-cylinder $\cocylinderinterval$ in $\cl{A}$, under certain conditions. In this way, we also obtain two model structures upon $\cl{A}$ from an interval $\interval$ in a monoidal category, equipped with all the structures we have considered in this work, and satisfying strictness of identities. 

\begin{assum} Let $\cl{A}$ be a category with finite limits and colimits. \end{assum} 

\begin{rmk} We make this assumption to be consistent with the definition of a model category which was recalled in \ref{ModelCategoryRecollectionsChapter}. Our work in fact relies only upon the existence of mapping cylinders and mapping co-cylinders in $\cl{A}$. 

This is a significant difference. Mapping cylinders and mapping co-cylinders exist in the category of chain complexes in any additive category, for example, whereas arbitrary finite limits and colimits do not. \end{rmk}

\begin{thm} \label{CofibrationsNormallyClovenFibrationsModelStructureTheorem} Let $\cylinder = \big( \cyl, i_{0}, i_{1}, p, v, \subdiv, r_{0}, r_{1}, s, \Gamma_{ul}, \Gamma_{lr}, \Gamma_{ur} \big)$ be a cylinder in $\cl{A}$ equipped with: 

\begin{itemize} [itemsep=1em,topsep=1em]

\item[(i)] a contraction structure $p$,

\item[(ii)] an involution structure $v$ compatible with $p$,

\item[(iii)] a subdivision structure $\big( \subdiv, r_{0}, r_{1}, s \big)$ compatible with $p$,

\item[(iv)] an upper left connection structure $\Gamma_{ul}$, 

\item[(v)] a lower right connection structure $\Gamma_{lr}$ compatible with $p$,

\item[(vi)] an upper right connection structure $\Gamma_{ur}$.

\end{itemize} %
Suppose that: 

\begin{itemize}[itemsep=1em,topsep=1em]

\item[(i)] $\Gamma_{lr}$ and $\Gamma_{ur}$ are compatible with $\big( \subdiv, r_{0}, r_{1}, s \big)$, 

\item[(ii)] $\cyl$ preserves mapping cylinders with respect to $\cylinder$,

\item[(iii)] $\cylinder$ has strictness of identities.

\end{itemize} %
Let $\cocylinder = \big( \cocyl, e_{0}, e_{1}, c, v', \subdiv',r_{0}', r_{1}', s', \Gamma_{ul}', \Gamma_{lr}' \big)$ be a co-cylinder in $\cl{A}$ equipped with:

\begin{itemize}[itemsep=1em,topsep=1em]

\item[(i)] a contraction structure $c$, 

\item[(ii)] an involution structure $v'$ compatible with $c$, 

\item[(iii)] a subdivision structure $\big(\subdiv', r_{0}', r_{1}', s' \big)$ compatible with $c$, 

\item[(iv)] an upper left connection structure $\Gamma_{ul}'$,

\item[(v)]  a lower right connection structure $\Gamma_{lr}'$ compatible with $c$. 

\end{itemize} %
Suppose that:

\begin{itemize}[itemsep=1em,topsep=1em]

\item[(i)] $\cocyl$ preserves mapping co-cylinders with respect to $\cocylinder$,

\item[(ii)] $\cocylinder$ has strictness of identities.

\end{itemize} %
Suppose moreover that $\cylinder$ is left adjoint to $\cocylinder$, and that the adjunction between $\cyl$ and $\cocyl$ is compatible with $p$ and $c$. 

We obtain a model structure upon $\cl{A}$ by taking: 

\begin{itemize}[itemsep=1em,topsep=1em]

\item[(i)] weak equivalences to be the homotopy equivalences with respect to $\cylinder$, or equivalently, by Proposition \ref{HomotopyEquivalenceCylinderIffHomotopyEquivalenceCoCylinderProposition}, to be the homotopy equivalences with respect to $\cocylinder$;

\item[(ii)] fibrations to be the normally cloven fibrations with respect to $\cocylinder$;

\item[(iii)] cofibrations to be the cofibrations with respect to $\cylinder$.

\end{itemize} 

\end{thm}

\begin{proof} That the conditions of Proposition \ref{EquivalentDefinitionsOfModelStructureProposition} hold has been established as follows:

\begin{itemize}[itemsep=1em,topsep=1em]

\item[(i)] Proposition \ref{TwoOutOfThreeHomotopyEquivalencesProposition},

\item[(ii)] Proposition \ref{RetractionCofibrationIsCofibrationProposition} and Corollary \ref{RetractionTrivialCofibrationIsTrivialCofibrationCorollary},

\item[(iii)] Corollary \ref{RetractionNormallyClovenFibrationIsNormallyClovenFibrationCorollary} and Corollary \ref{RetractionTrivialNormallyClovenFibrationIsTrivialNormallyClovenFibrationCorollary},

\item[(iv)] Corollary \ref{NormallyClovenFibrationRLPTrivialCofibrationCorollary},

\item[(v)] Corollary \ref{TrivialNormallyClovenFibrationRLPCofibrationCorollary},

\item[(vi)] Corollary \ref{CofibrationFollowedByTrivialNormallyClovenFibrationCorollary},

\item[(vii)] Corollary \ref{TrivialCofibrationFollowedByNormallyClovenFibrationCorollary}. 

\end{itemize}
\end{proof}

\begin{thm} \label{NormallyClovenCofibrationsFibrationsModelStructureTheorem} Let $\cylinder = \big( \cyl, i_{0}, i_{1}, p, v, \subdiv, r_{0}, r_{1}, s, \Gamma_{ul}, \Gamma_{lr} \big)$ be a cylinder in $\cl{A}$ equipped with: 

\begin{itemize}[itemsep=1em,topsep=1em]

\item[(i)] a contraction structure $p$,

\item[(ii)] an involution structure $v$ compatible with $p$,

\item[(iii)] a subdivision structure $\big( \subdiv, r_{0}, r_{1}, s \big)$ compatible with $p$,

\item[(iv)] an upper left connection structure $\Gamma_{ul}$, 

\item[(v)] a lower right connection structure $\Gamma_{lr}$ compatible with $p$.
\end{itemize} %
Suppose that: 

\begin{itemize}[itemsep=1em,topsep=1em]

\item[(i)] $\cyl$ preserves mapping cylinders with respect to $\cylinder$,

\item[(ii)] $\cylinder$ has strictness of identities.

\end{itemize} %
Let $\cocylinder = \big( \cocyl, e_{0}, e_{1}, c, v', \subdiv',r_{0}', r_{1}', s', \Gamma_{ul}', \Gamma_{lr}', \Gamma_{ur}' \big)$ be a co-cylinder in $\cl{A}$ equipped with:

\begin{itemize}[itemsep=1em,topsep=1em]

\item[(i)] a contraction structure $c$, 

\item[(ii)] an involution structure $v'$ compatible with $c$, 

\item[(iii)] a subdivision structure $\big(\subdiv', r_{0}', r_{1}', s' \big)$ compatible with $c$, 

\item[(iv)] an upper left connection structure $\Gamma_{ul}'$,

\item[(v)] a lower right connection structure $\Gamma_{lr}'$ compatible with $c$,

\item[(vi)] an upper right connection structure $\Gamma_{ur}'$.

\end{itemize} %
Suppose that:

\begin{itemize}[itemsep=1em,topsep=1em]

\item[(i)] $\Gamma_{lr}'$ and $\Gamma_{ur}'$ are compatible with $\big( \subdiv', r_{0}', r_{1}', s' \big)$. 

\item[(ii)] $\cocyl$ preserves mapping co-cylinders with respect to $\cocylinder$,

\item[(iii)] $\cocylinder$ has strictness of identities.

\end{itemize} %
Suppose moreover that $\cylinder$ is left adjoint to $\cocylinder$, and that the adjunction between $\cyl$ and $\cocyl$ is compatible with $p$ and $c$.

We obtain a model structure upon $\cl{A}$ by taking:

\begin{itemize}[itemsep=1em,topsep=1em]

\item[(i)] weak equivalences to be the homotopy equivalences with respect to $\cylinder$, or equivalently, by Proposition \ref{HomotopyEquivalenceCylinderIffHomotopyEquivalenceCoCylinderProposition}, the homotopy equivalences with respect to $\cocylinder$;

\item[(ii)] fibrations to be the fibrations with respect to $\cocylinder$;

\item[(iii)] cofibrations to be the normally cloven cofibrations with respect to $\cylinder$.

\end{itemize} 

\end{thm}

\begin{proof} That the conditions of Proposition \ref{EquivalentDefinitionsOfModelStructureProposition} hold has been established as follows:

\begin{itemize}[itemsep=1em,topsep=1em]

\item[(i)] Proposition \ref{TwoOutOfThreeHomotopyEquivalencesProposition},

\item[(ii)] Proposition \ref{RetractionNormallyClovenCofibrationIsNormallyClovenCofibrationProposition} and Corollary \ref{RetractionTrivialNormallyClovenCofibrationIsTrivialNormallyClovenCofibrationCorollary},

\item[(iii)] Corollary \ref{RetractionFibrationIsFibrationCorollary} and Corollary \ref{RetractionTrivialFibrationIsTrivialFibrationCorollary},

\item[(iv)] Corollary \ref{FibrationRLPTrivialNormallyClovenCofibrationCorollary},

\item[(v)] Corollary \ref{TrivialFibrationRLPNormallyClovenCofibrationCorollary}, 

\item[(vi)] Corollary \ref{NormallyClovenCofibrationFollowedByTrivialFibrationCorollary}, 

\item[(vii)] Corollary \ref{TrivialNormallyClovenCofibrationFollowedByFibrationCorollary}. 

\end{itemize}
\end{proof}

\begin{assum} Let $\otimes$ be a monoidal structure upon $\cl{A}$. \end{assum}

\begin{cor} \label{CofibrationsNormallyClovenFibrationsIntervalModelStructureCorollary} Let $\interval = \big( I, i_{0}, i_{1}, p, v, S, r_{0}, r_{1}, s, \Gamma_{ul}, \Gamma_{lr}, \Gamma_{ur} \big)$ be an interval in $\cl{A}$ equipped with:

\begin{itemize}[itemsep=1em,topsep=1em]

\item[(i)] a contraction structure $p$,

\item[(ii)] an involution structure $v$ compatible with $p$,

\item[(iii)] a subdivision structure $\big( S, r_{0}, r_{1}, s \big)$ compatible with $p$,

\item[(iv)] an upper left connection structure $\Gamma_{ul}$, 

\item[(v)] a lower right connection structure $\Gamma_{lr}$ compatible with $p$,

\item[(vi)] an upper right connection structure $\Gamma_{ur}$.

\end{itemize} %
Suppose that:

\begin{itemize}[itemsep=1em,topsep=1em]

\item[(i)] $\Gamma_{lr}$ and $\Gamma_{ur}$ are compatible with $\big( S, r_{0}, r_{1}, s \big)$,

\item[(ii)] $I$ and $S$ are exponentiable with respect to $\otimes$,

\item[(iii)] Requirement \ref{SubdivisionRequirement} holds,

\item[(iv)] $\interval$ has strictness of identities.  
 
\end{itemize} %
We obtain a model structure upon $\cl{A}$ by taking:

\begin{itemize}[itemsep=1em,topsep=1em]

\item[(i)] weak equivalences to be the homotopy equivalences with respect to $\cylinderinterval$, or equivalently, by Proposition \ref{HomotopyEquivalenceCylinderIffHomotopyEquivalenceCoCylinderProposition}, the homotopy equivalences with respect to $\cocylinderinterval$;

\item[(ii)] fibrations to be the normally cloven fibrations with respect to $\cocylinderinterval$;

\item[(iii)] cofibrations to be the cofibrations with respect to $\cylinderinterval$.

\end{itemize}

\end{cor}

\begin{proof} Follows immediately from Theorem \ref{CofibrationsNormallyClovenFibrationsModelStructureTheorem} by the observations of \ref{StructuresUponAnIntervalChapter}. \end{proof}

\begin{cor} \label{NormallyClovenCofibrationsFibrationsIntervalModelStructureCorollary} Let $\interval = \big( I, i_{0}, i_{1}, p, v, S, r_{0}, r_{1}, s, \Gamma_{ul}, \Gamma_{lr}, \Gamma_{ur} \big)$ be an interval in $\cl{A}$ equipped with:

\begin{itemize}[itemsep=1em,topsep=1em]

\item[(i)] a contraction structure $p$,

\item[(ii)] an involution structure $v$ compatible with $p$,

\item[(iii)] a subdivision structure $\big( S, r_{0}, r_{1}, s \big)$ compatible with $p$,

\item[(iv)] an upper left connection structure $\Gamma_{ul}$, 

\item[(v)] a lower right connection structure $\Gamma_{lr}$ compatible with $p$,

\item[(vi)] an upper right connection structure $\Gamma_{ur}$.

\end{itemize} %
Suppose that:

\begin{itemize}[itemsep=1em,topsep=1em]

\item[(i)] $\Gamma_{lr}$ and $\Gamma_{ur}$ are compatible with $\big( S, r_{0}, r_{1}, s \big)$,

\item[(ii)] $I$ and $S$ are exponentiable with respect to $\otimes$,

\item[(iii)] Requirement \ref{SubdivisionRequirement} holds,

\item[(iv)] $\interval$ has strictness of identities.  
\end{itemize} %
We obtain a model structure upon $\cl{A}$ by taking:

\begin{itemize}[itemsep=1em,topsep=1em]

\item[(i)] weak equivalences to be the homotopy equivalences with respect to $\cylinderinterval$, or equivalently, by Proposition \ref{HomotopyEquivalenceCylinderIffHomotopyEquivalenceCoCylinderProposition} the homotopy equivalences with respect to $\cocylinderinterval$;

\item[(ii)] fibrations to be the fibrations with respect to $\cocylinderinterval$;

\item[(iii)] cofibrations to be the normally cloven cofibrations with respect to $\cylinderinterval$.

\end{itemize}

\end{cor}

\begin{proof} Follows immediately from Theorem \ref{NormallyClovenCofibrationsFibrationsModelStructureTheorem} by the observations of \ref{StructuresUponAnIntervalChapter}. \end{proof}

\end{chapter}

\begin{chapter}{Example --- categories and groupoids} \label{ExampleChapter}

We define an interval in the category $\mathsf{Cat}$ of categories, equipped with its cartesian monoidal structure. It admits all the structures of \ref{StructuresUponAnIntervalChapter}, and has strictness of identities. By \ref{ModelStructureChapter}, we thus obtain two model structures upon $\mathsf{Cat}$. In a non-constructive setting, both model structures can be proven to coincide with folk model structure. 

In the same way, we obtain two model structures upon the category $\mathsf{Grpd}$ of groupoids. Again both may be demonstrated, by a non-constructive argument, to coincide with the folk model structure.

The folk model structure on $\mathsf{Cat}$ was constructed by Joyal and Tierney in \cite{JoyalTierneyStrongStacksAndClassifyingSpaces}. Independently, a construction was given by Rezk in \cite{RezkAModelCategoryForCategories}. 

The folk model structure on $\mathsf{Grpd}$ appeared in the literature earlier. It was first described by Anderson in \S{5} of \cite{AndersonFibrationsAndGeometricRealizations}, and is also discussed in \S{14.1} of the paper \cite{BousfieldHomotopySpectralSequencesAndObstructions} of Bousfield. A detailed construction is given in \S{6.1} of the article \cite{StricklandKnLocalDualityForFiniteGroupsAndGroupoids} of Strickland, built upon in \S{3} of the thesis \cite{HollanderAHomotopyTheoryForStacks} of Hollander. 

The folk model structure on groupoids can also be seen to arise as the restriction to groupoids of the model structure on $\mathsf{Cat}$ constructed by Thomason in \cite{ThomasonCatAsAClosedModelCategory}. This is observed, for example, in \S{1} of the paper \cite{CasacubertaGolasinskiTonksHomotopyLocalizationOfGroupoids} of Casacuberta, Golasi{\'n}ski, and Tonks.

In all these works, the non-constructive characterisation of equivalences of categories as functors which are fully faithful and essentially surjective is essential. From this point of view, the folk model structure on $\mathsf{Cat}$ or $\mathsf{Grpd}$ is akin to the Serre model structure on topological spaces, which was first constructed in \S{II.3} of \cite{QuillenHomotopicalAlgebra}. 

The conceptual approach we have taken is significantly different. Our two model structures are instead akin to the model structure on topological spaces, which was constructed by Str{\o}m in \cite{StromTheHomotopyCategoryIsAHomotopyCategory}. The fact that we may non-constructively identify the two model structures on categories or groupoids which we construct with the folk model structure might reasonably, we think, be viewed as something of a coincidence.

\begin{notn} Let $\mathsf{Cat}$ denote the category of categories, and let $\mathsf{Grpd}$ denote the category of groupoids. We denote by $1$ the final object of $\mathsf{Cat}$ and $\mathsf{Grpd}$, the category with a unique object $\bullet$ and a unique arrow. \end{notn}

\begin{rmk} We shall regard $\mathsf{Cat}$ and $\mathsf{Grpd}$ as equipped with their cartesian monoidal structures. These monoidal structures are closed, and thus Requirement \ref{SubdivisionRequirement} is satisfied. \end{rmk}

\begin{notn} Let $\cl{I}$ denote the free groupoid on the following directed graph. \ar{0,1} \end{notn}

\begin{notn} Let \ar{1,\cl{I},i_{0}} denote the unique functor which maps $\bullet$ to $0$. \end{notn}

\begin{notn} Let \ar{1,\cl{I},i_{1}} denote the unique functor which maps $\bullet$ to $1$. \end{notn}

\begin{notn} Let $\interval$ denote the interval $\big( \cl{I},i_{0},i_{1} \big)$ in $\mathsf{Cat}$ or $\mathsf{Grpd}$. \end{notn}

\begin{observation} The canonical functor \ar{\cl{I},1,p} equips $\interval$ with a contraction structure. \end{observation} 

\begin{notn} Let \ar{\cl{I},\cl{I},v} denote the unique functor which maps the arrow \ar{0,1} of $\cl{I}$ to the arrow \ar{1,0} of $\cl{I}$. \end{notn}

\begin{observation} The functor $v$ equips $\interval$ with an involution structure which is compatible with $p$. \end{observation}

\begin{notn} Let $\cl{S}$ denote the free groupoid on the following directed graph. \twoar{0,1,2} \end{notn}

\begin{notn} Let \ar{\cl{I},\cl{S},r_{0}} denote the unique functor which maps the arrow \ar{0,1} of $\cl{I}$ to the arrow \ar{1,2} of $\cl{S}$. \end{notn}

\begin{notn} Let \ar{\cl{I},\cl{S},r_{1}} denote the unique functor which maps the arrow \ar{0,1} of $\cl{I}$ to the arrow \ar{0,1} of $\cl{S}$. \end{notn}

\begin{notn} Let \ar{\cl{I},\cl{S},s} denote the unique functor which maps the arrow \ar{0,1} of $\cl{I}$ to the arrow \ar{0,2} of $\cl{S}$. \end{notn}

\begin{observation} We have that $\big( \cl{S}, r_{0}, r_{1}, s \big)$ equips $\interval$ with a subdivision structure, which is compatible with $p$.
\end{observation}

\begin{observation} With respect to the involution structure $v$ and the subdivision structure $\big( S, r_{0}, r_{1}, s \big)$, the interval $\interval$ has strictness of identities and strictness of left inverses. \end{observation}

\begin{observation} The groupoid $\cl{I}^{2}= \cl{I} \times \cl{I}$ is the unique groupoid with objects and arrows as follows, excluding the four identity arrows. 

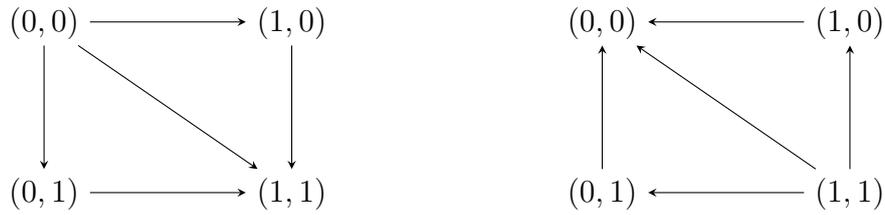
\begin{diagram}

\begin{tikzpicture} [>=stealth]

\matrix [ampersand replacement=\&, matrix of math nodes, column sep=5 em, row sep=4 em, text height=1.6ex, text depth=0.25ex]
{ 
|(0-0)| (0,0) \& |(1-0)| (1,0) \&[2em] |(2-0)| (0,0) \& |(3-0)| (1,0) \\ 
|(0-1)| (0,1) \& |(1-1)| (1,1) \& |(2-1)| (0,1) \& |(3-1)| (1,1) \\
};
	
\draw[->] (0-0) to node[auto] {} (1-0);
\draw[->] (1-0) to node[auto] {} (1-1);
\draw[->] (0-0) to node[auto,swap] {} (0-1);
\draw[->] (0-1) to node[auto,swap] {} (1-1);
\draw[->] (3-0) to node[auto,swap] {} (2-0);
\draw[->] (3-1) to node[auto,swap] {} (3-0);
\draw[->] (2-1) to node[auto] {} (2-0);
\draw[->] (3-1) to node[auto] {} (2-1);
\draw[->] (0-0) to node[auto,swap] {} (1-1);
\draw[->] (3-1) to node[auto] {} (2-0);

\end{tikzpicture} 

\end{diagram} 

\end{observation}

\begin{notn} Let \ar{\cl{I}^{2},\cl{I},\Gamma_{ul}} denote the unique functor with the following properties: 

\begin{itemize}[itemsep=1em,topsep=1em]

\item[(i)] the arrow \ar{{(0,0)},{(1,0)}} of $\cl{I}^{2}$ maps to the arrow \ar{0,1} of $\cl{I}$. 

\item[(ii)] the arrow \ar{{(0,0)},{(0,1)}} of $\cl{I}^{2}$ maps to the arrow \ar{0,1} of $\cl{I}$. 

\item[(iii)] the arrow \ar{{(1,0)},{(1,1)}} of $\cl{I}^{2}$ maps to the arrow \ar{1,1} of $\cl{I}$. 

\item[(iv)] the arrow \ar{{(0,1)},{(1,1)}} of $\cl{I}^{2}$ maps to the arrow \ar{1,1} of $\cl{I}$.

\end{itemize}

\end{notn} 

\begin{observation} The functor $\Gamma_{ul}$ equips $\interval$ with an upper left connection structure. \end{observation}

\begin{notn} Let \ar{\cl{I}^{2},\cl{I},\Gamma_{lr}} denote the unique functor with the following properties: 

\begin{itemize}[itemsep=1em,topsep=1em]

\item[(i)] the arrow \ar{{(0,0)},{(1,0)}} of $\cl{I}^{2}$ maps to the arrow \ar{0,0} of $\cl{I}$. 

\item[(ii)] the arrow \ar{{(0,0)},{(0,1)}} of $\cl{I}^{2}$ maps to the arrow \ar{0,0} of $\cl{I}$. 

\item[(iii)] the arrow \ar{{(1,0)},{(1,1)}} of $\cl{I}^{2}$ maps to the arrow \ar{0,1} of $\cl{I}$. 

\item[(iv)] the arrow \ar{{(0,1)},{(1,1)}} of $\cl{I}^{2}$ maps to the arrow \ar{0,1} of $\cl{I}$.

\end{itemize}

\end{notn} 

\begin{observation} The functor $\Gamma_{lr}$ equips $\interval$ with a lower right connection structure, which is compatible with $p$. \end{observation}

\begin{notn} Let \ar{\cl{I}^{2},\cl{I},\Gamma_{ur}} denote the unique functor with the following properties: 

\begin{itemize}[itemsep=1em,topsep=1em]

\item[(i)] the arrow \ar{{(0,0)},{(1,0)}} of $\cl{I}^{2}$ maps to the arrow \ar{0,1} of $\cl{I}$. 

\item[(ii)] the arrow \ar{{(0,0)},{(0,1)}} of $\cl{I}^{2}$ maps to the arrow \ar{0,0} of $\cl{I}$. 

\item[(iii)] the arrow \ar{{(1,0)},{(1,1)}} of $\cl{I}^{2}$ maps to the arrow \ar{1,0} of $\cl{I}$. 

\item[(iv)] the arrow \ar{{(0,1)},{(1,1)}} of $\cl{I}^{2}$ maps to the arrow \ar{0,0} of $\cl{I}$.

\end{itemize}

\end{notn} 

\begin{observation} The functor $\Gamma_{ur}$ equips $\interval$ with an upper right connection structure. We have that $\Gamma_{lr}$ and $\Gamma_{ur}$ are compatible with $\big(\cl{S},r_{0},r_{1},s\big)$. \end{observation}

\begin{observation} A functor is a homotopy equivalence with respect to $\cylinderinterval$ if and only if it is an equivalence of categories. \end{observation}  

\begin{recollection} An {\em iso-fibration} is a functor \ar{\cl{A}_{0},\cl{A}_{1},F} with the property that, for every commutative diagram \sq{1,\cl{A}_{0},\cl{I},\cl{A}_{1},a,F,i_{0},g} in $\mathsf{Cat}$, there is a functor \ar{\cl{I},\cl{A}_{0},l} such that the following diagram in $\mathsf{Cat}$ commutes. \liftingsquare{1,\cl{A}_{0},\cl{I},\cl{A}_{1},a,F,i_{0},g,l} \end{recollection}

\begin{recollection} A {\em normally cloven iso-fibration} is a functor \ar{\cl{A}_{0},\cl{A}_{1},F} with the property that, to every commutative diagram \sq{1,\cl{A}_{0},\cl{I},\cl{A}_{1},a,F,i_{0},g} in $\mathsf{Cat}$, we can associate a functor \ar{\cl{I},\cl{A}_{0},l} such that the following hold.

\begin{itemize}[itemsep=1em,topsep=1em] 

\item[(i)] The diagram \liftingsquare{1,\cl{A}_{0},\cl{I},\cl{A}_{1},a,F,i_{0},g,l} in $\mathsf{Cat}$ commutes.

\item[(ii)] If the diagram \tri{\cl{I},1,\cl{A}_{1},p,F(a),g} in $\mathsf{Cat}$ commutes, then the diagram \tri{\cl{I},1,\cl{A}_{0},p,a,l} in $\mathsf{Cat}$ commutes.

\end{itemize}
\end{recollection}

\begin{observation} A functor \ar{\cl{A}_{0},\cl{A}_{1},F} is a fibration with respect to $\cylinderinterval$ if and only if it is an iso-fibration. This goes back to Proposition 2.1 of the paper \cite{BrownFibrationsOfGroupoids} of Brown. 

Moreover, $F$ is a normally cloven fibration with respect to $\cylinderinterval$ if and only if it is a normally cloven iso-fibration. \end{observation}

\begin{defn} An {\em iso-cofibration} is a functor \ar{\cl{A}_{0},\cl{A}_{1},j} such that $j$ is a cofibration with respect to $\cylinderinterval$. \end{defn} 

\begin{rmk} Non-constructively, it is possible to characterise an iso-cofibration as a functor which is injective on objects. \end{rmk}

\begin{defn} A {\em normally cloven iso-cofibration} is a functor \ar{\cl{A}_{0},\cl{A}_{1},j} such that $j$ is a normally cloven cofibration with respect to $\cylinderinterval$. \end{defn} 

\begin{thm} We obtain a model structure on $\mathsf{Cat}$ and $\mathsf{Grpd}$ by taking:

\begin{itemize}[itemsep=1em, topsep=1em]

\item[(i)] weak equivalences to be equivalences of categories,

\item[(ii)] fibrations to be iso-fibrations,

\item[(iii)] cofibrations to be normally cloven iso-cofibrations.

\end{itemize}
\end{thm}

\begin{proof} Follows immediately from Corollary \ref{NormallyClovenCofibrationsFibrationsIntervalModelStructureCorollary}. \end{proof}

\begin{thm} We obtain a model structure on $\mathsf{Cat}$ and $\mathsf{Grpd}$ by taking:

\begin{itemize}[itemsep=1em, topsep=1em]

\item[(i)] weak equivalences to be equivalences of categories,

\item[(ii)] fibrations to be normally cloven iso-fibrations,

\item[(iii)] cofibrations to be iso-cofibrations.

\end{itemize}
\end{thm}

\begin{proof} Follows immediately from Corollary \ref{CofibrationsNormallyClovenFibrationsIntervalModelStructureCorollary}. \end{proof}

\end{chapter}

\bibliography{ref}
\bibliographystyle{siam}

\end{document}